\documentclass[11pt,reqno]{amsart}
\usepackage{geometry} 

\usepackage{amsmath,amssymb,amsthm, mathrsfs,appendix,ulem}
\usepackage{xcolor}
\usepackage{graphicx} 
\usepackage{tikz}
\usetikzlibrary{positioning} 
\usepackage{pgfplots}

\usepackage{marvosym}
\usepackage{cancel}

\usepackage{latexsym}
\usepackage[nobysame]{amsrefs}[11pts]
\usepackage{bm}

\usepackage{enumerate}
\usepackage{indentfirst}
\usepackage[colorlinks,
            linkcolor=black,
            anchorcolor=black,
            citecolor=black
            ]{hyperref}

\newcommand{\into}{\hookrightarrow}

\newcommand{\ds}{\mathrm{d}s}
\newcommand{\dt}{\mathrm{d}t}
\newcommand{\dx}{\mathrm{d}x}
\newcommand{\dy}{\mathrm{d}y}

\newcommand{\intt}{\int_{I(t)}}

\newcommand\abs[1]{\left|#1\right|}
\newcommand\absf[1]{|#1|}
\newcommand\absb[1]{\big|#1\big|}
\newcommand\absB[1]{\Big|#1\Big|}
\newcommand\norm[1]{\left\|#1\right\|}
\newcommand\normf[1]{\|#1\|}
\newcommand{\ef}{\eqref}

\DeclareMathOperator{\dive}{div}

\renewcommand{\AA}{\mathbb{A}}
\newcommand{\BB}{\mathbb{B}}
\newcommand{\CC}{\mathbb{C}}

\newcommand{\II}{\mathbb{I}}

\newcommand{\NN}{\mathbb{N}}

\newcommand{\PP}{\mathbb{P}}

\newcommand{\RR}{\mathbb{R}}

\newcommand{\ZZ}{\mathbb{Z}}

\newcommand{\cC}{\mathcal{C}}

\newcommand{\cE}{\mathcal{E}}

\newcommand{\cG}{\mathcal{G}}

\newcommand{\cJ}{\mathcal{J}}

\newcommand{\cL}{\mathcal{L}}

\newcommand{\cQ}{\mathcal{Q}}
\newcommand{\cR}{\mathcal{R}}

\newcommand{\cV}{\mathcal{V}}

\newcommand{\wE}{\widehat{\cE}}
\newcommand{\wU}{\widehat{U}}

\newcommand{\weta}{\widehat{\eta}}

\newtheorem{Definition}{Definition}
\newtheorem{Theorem}{Theorem}[section]
\newtheorem{Lemma}{Lemma}[section]
\newtheorem{Proposition}{Proposition}
\newtheorem{Remark}{Remark}
\newtheorem{Corollary}{Corollary}[section]
\numberwithin{Remark}{section}
\numberwithin{Proposition}{section}
\numberwithin{Definition}{section}
\numberwithin{Lemma}{section}
\numberwithin{equation}{section}
\numberwithin{Theorem}{section}

\allowdisplaybreaks[3]

\title[Viscous Saint-Venant System]{Global-in-time Well-posedness of Classical Solutions to the Vacuum Free Boundary Problem for the 1-D Viscous Saint-Venant System with Large Data}
\date{\today}

\author{Zhouping Xin}
\address[Zhouping Xin]{The Institute of Mathematical Sciences and
Department of Mathematics, The Chinese University of Hong Kong, Shatin, N.T., Hong Kong.}
\email{\tt zpxin@ims.cuhk.edu.hk}

\author{Jiawen Zhang}
\address[Jiawen Zhang]{School of Mathematical Sciences, Shanghai Jiao Tong University,
Shanghai 200240, P. R. China} \email{\tt zhangjiawen317@sjtu.edu.cn}

\author{Shengguo Zhu}
\address[Shengguo Zhu]{School of Mathematical Sciences, CMA-Shanghai  and MOE-LSC,  Shanghai Jiao Tong University, Shanghai 200240, P. R. China}
 \email{\tt  zhushengguo@sjtu.edu.cn}

\pgfplotsset{compat=1.18}

\begin{document}
\begin{abstract}

This paper establishes the global existence and uniqueness of classical
solutions to the vacuum free boundary problem for the 1-D viscous
Saint-Venant system for laminar shallow water with a general class of
large initial data. Since the fluid depth $\rho$ vanishes on the moving
boundary, the momentum equations degenerate in both time evolution and
spatial dissipation, potentially causing singularities in velocity
derivatives and making classical solutions difficult to construct.
By exploiting the intrinsic degenerate-singular structures, we identify
admissible initial depth profiles for which $\rho_0^\alpha$ belongs to
$H^3$ and vanishes as the distance to the moving boundary, with
$\frac{1}{3}<\alpha\le 1$. For $\frac{1}{3}<\alpha<1$, $\rho_0$ satisfies
the BD entropy condition; for $\alpha=1$, it satisfies the physical vacuum
boundary condition but violates the BD entropy condition. The solutions
are classical up to the moving boundary. First, we introduce new weighted
nonlinear energy functionals involving lower- and higher-order derivatives,
based on the balance between pressure and spatial dissipation.
A new cross-derivatives embedding theorem yields the required elliptic
estimates and, together with a modified Galerkin scheme, establishes local
existence and uniqueness. Second, we derive new global weighted $L^p$
estimates, $2\le p\le\infty$, for the effective velocity
$v=u+(\log\rho)_y$, where $y$ is the Eulerian coordinate, using the
transport properties of its evolution equation. These estimates differ
from those of Bresch-Desjardins (Comm. Math. Phys. \textbf{238} (2003),
211-223) and Kanel (Differ. Uravn. \textbf{4} (1968), 721-734) and are
crucial for $\alpha=1$, where the BD entropy condition fails.
Finally, we establish weighted energy estimates for first-order velocity
derivatives with weights involving powers of $\rho_0$ and $\eta_x$,
where $x$ is the Lagrangian coordinate and $\eta$ the flow map.
These estimates give an upper bound for $\eta_x$ on every finite time
interval without any smallness assumption on the initial data.
Further singular or degenerate weighted energy estimates yield the desired global
regularity.

\end{abstract}
\maketitle

\tableofcontents

\section{Introduction}\label{section1}

The time evolution of a general viscous isentropic compressible fluid occupying a spatial domain $\Omega\subset \mathbb{R}^N$ with  the mass density $\rho\geq 0$ and the velocity $\textbf{u}=(u^{(1)},\cdots, u^{(N)})^\top$ $\in \mathbb{R}^N$  is governed by the following isentropic compressible Navier-Stokes system (\textbf{CNS}):
\begin{equation}\label{eq:1.1}
\begin{cases}
\rho_t+\dive(\rho \textbf{u})=0,\\[2pt]
(\rho \textbf{u})_t+\dive(\rho \textbf{u}\otimes \textbf{u})
+\nabla P =\dive\mathbb{T}.
\end{cases}
\end{equation}
Here, $\textbf{y}=(y_1,\cdots, y_N)^{\top}\in \Omega$, $t\geq 0$ are the space and time variables, respectively. For the polytropic gases, the constitutive relation is given by
\begin{equation}\label{eq:1.1p}
P=A\rho^{\gamma}, \quad A>0,\quad  \gamma> 1,
\end{equation}
where $A$ is an entropy constant and $\gamma$ is the adiabatic exponent.  The viscous stress tensor $\mathbb{T}$ is given by 
\begin{equation}\label{eq:1.1t}
\mathbb{T}=2\mu(\rho)D(\textbf{u})+\lambda(\rho)\dive \textbf{u}\,\mathbb{I}_N,
\end{equation} 
where $D(\textbf{u})=\frac{1}{2}\left(\nabla \textbf{u}+(\nabla \textbf{u})^\top\right)$ is the deformation tensor, $\mathbb{I}_N$ is the $N\times N$ identity matrix,
\begin{equation}\label{fandan}
\mu(\rho)=a_1  \rho^\nu,\quad \lambda(\rho)=a_2 \rho^\nu,
\end{equation}
for some  constant $\nu\geq 0$, $\mu(\rho)$ is the shear viscosity coefficient, $\lambda(\rho)+\frac{2}{N}\mu(\rho)$ is the bulk viscosity coefficient,  $a_1$ and $a_2$ are both constants satisfying
\begin{equation}\label{kelaoxiusi}
a_1>0 \quad\text{and}\quad 2a_1+Na_2\geq 0.
\end{equation}

In  the rarefied gas dynamics, the \textbf{CNS} can be derived from the Boltzmann equation through the Chapman-Enskog expansion, cf. Chapman-Cowling \cite{chap} and Li-Qin \cite{tlt}. Under some proper physical assumptions,   the viscosity coefficients $(\mu,\lambda)$ and the heat  conductivity coefficient $\kappa$ are  functions of the absolute temperature $\theta$. Actually, 
for the cut-off inverse power force models (see \cite{chap}), if the intermolecular potential varies as $\ell^{-\varkappa}$,
where $\ell$ is the intermolecular distance and $\varkappa$ is a positive constant, then   
\begin{equation}\label{eq:1.6g}
\mu(\theta)=b_1 \theta^{\frac{1}{2}+b},\quad \lambda(\theta)=b_2 \theta^{\frac{1}{2}+b} \quad \text{and} \quad \kappa(\theta)=b_3 \theta^{\frac{1}{2}+b}\quad \text{with}\quad b=\frac{2}{\varkappa} \in [0,\infty),
\end{equation}
for some constants $(b_1,b_2,b_3)$.
In particular, for the ionized gas,
$\varkappa=1$;
for Maxwellian molecules,
$\varkappa=4$; 
while for rigid elastic spherical molecules,
$\varkappa=\infty$.
In  Liu-Xin-Yang \cite{taiping}, for isentropic and polytropic fluids,  according to the laws of Boyle and Gay-Lussac:
\begin{equation*}
P=\cR\rho \theta=A\rho^\gamma \quad \text{for a constant } \cR>0,
\end{equation*}
i.e., $\theta=A\cR^{-1}\rho^{\gamma-1}$, one finds that $(\mu,\lambda)$ become functions of $\rho$ taking   the form $(\ref{fandan})$.  Note that there exist other  physical models satisfying the density-dependent viscosities  assumption \eqref{fandan}, such as the Korteweg system, the shallow water equations, the lake equations and the quantum Navier-Stokes system and so on (see \cites{ bd6,bd2,gp,jun}).

The current paper concerns the following vacuum free boundary problem ({\rm\textbf{VFBP}}) for  the one-dimensional (\text{1-D}) viscous Saint-Venant system for laminar shallow water,
\begin{equation}\label{shallow}
\begin{cases}
\rho_t+(\rho u)_y=0& \text{in} \ \ I(t),\\[2pt]
(\rho u)_t+(\rho u^2+\rho^2)_y-(\rho u_y)_y=0&\text{in} \ \ I(t),\\[2pt]
\rho>0& \text{in} \ \ I(t),\\[2pt]
\rho=0& \text{on} \ \ \Gamma(t),\\[2pt]
\cV(\Gamma(t))=u& \text{on} \ \ \Gamma(t),\\[2pt]
(\rho,u)|_{t=0}=(\rho_0,u_0)& \text{on} \ \ I:=I(0)=(0,1),
\end{cases}
\end{equation}
where $\rho\geq 0$ denotes the depth of  the fluid, $u$  the Eulerian velocity of the fluid, the open and bounded interval  $I(t)\subset \mathbb{R}$  the changing domain occupied by the fluid, $\Gamma(t):=\partial I(t)$ the moving vacuum boundary, $y\in I(t)$  the Eulerian spatial coordinate, $t\geq 0$  the time coordinate, and $\cV(\Gamma(t))$ the velocity of $\Gamma(t)$, respectively. The viscous Saint-Venant system $\eqref{shallow}_1$-$\eqref{shallow}_2$ in \eqref{shallow} can be derived rigorously from the incompressible Navier-Stokes equations with a free moving interface by Gerbeau-Perthame \cite{gp}, which corresponds to the degenerate \textbf{CNS} \eqref{eq:1.1}-\eqref{kelaoxiusi} with $N=\nu=A=1$, $a_1=1/2$, $a_2=0$ and  $\gamma=2$. Indeed, such models appear naturally and frequently in geophysical flows \cites{bd6,bd2}. It is worth pointing out that the assumption that $(a_1,A)=(1/2,1)$ in $\eqref{shallow}_1$-$\eqref{shallow}_2$ is  just used to simplify the  description in our analysis, and one can regard $(a_1,A)$ as any two positive constants. $\eqref{shallow}_3$ states that there is no vacuum inside  the fluid;  $\eqref{shallow}_4$ states that the depth vanishes along the moving vacuum boundary $\Gamma(t)$; $\eqref{shallow}_5$ states that the vacuum boundary $\Gamma(t)$ is moving with speed equal to the fluid velocity, and $\eqref{shallow}_6$ provides  the initial conditions for the depth, velocity, and domain.

The  main goal here is to  establish the local/global-in-time existence  of the unique classical  solution  to  problem \eqref{shallow} for general data with  the initial depth profile such that 
\begin{equation}\label{distance}
\rho_0^\alpha\in H^3(I) \ \ \text{and} \ \ \cC_1d(y)\leq \rho_0^\alpha(y)\leq \cC_2d(y) \ \ \text{for all }y\in \bar I,
\end{equation}
for some  constants $\cC_1>0$, $\cC_2>0$ and $0<\alpha\leq 1$, where $d(y):=\mathrm{dist}\, (y,\Gamma)$ ($\Gamma:=\Gamma(0)$) is the distance  function from $y\in \bar I$ to $\Gamma$. It is interesting to note that the set of $\rho_0$ defined by  \eqref{distance}   contains two  different kinds  of  initial profiles. Indeed, for  $0<\alpha<1$, \eqref{distance} implies that  $\rho_0$  satisfies the  so-called BD entropy condition, i.e.,
\begin{equation}\label{BDcondition}
\|(\sqrt{\rho_0})_y\|_{L^2(I)}<\infty,
\end{equation}
which was  initiated with a series of papers by Bresch-Desjardins \cites{bd6,Bresch,bd8} (started in 2003 with Lin \cite{bd2} in the context of Navier-Stokes-Korteweg with a linear shear viscosity).
On the other hand, one denotes by 
$c=\sqrt{P'(\rho)}$
the speed of the sound, $c_0=c|_{t=0}$, and $\textbf{n}$ the outward unit normal vector to the initial  boundary. Then when $\alpha=1$, $\rho_0$   satisfies the so-called  physical vacuum boundary  condition
\begin{equation}\label{PVcondition}
-\infty<\frac{\partial c^2_0}{\partial \mathbf{n}}<0 \quad \text{on} \ \Gamma,
\end{equation}
which was first proposed by Liu \cite{LiuTP} when he studied the self-similar solutions to compressible Euler with damping. This assumption means that   
the pressure will induce  a  nontrivial finite acceleration in the outward normal direction on the moving boundary.
Furthermore, we can check that these two types of initial conditions on the depth shown in \eqref{BDcondition}-\eqref{PVcondition} are not compatible.

The study  of the vacuum  is crucial in the analysis of the dynamics of  viscous compressible fluids (\cites{BJ, fu3,JZ,lions,MPI,zx}). In fact, when $(\mu,\lambda,\kappa)$  are all constants, some  singular behaviors of solutions with vacuum to the Cauchy problem of \textbf{CNS} have already been observed. In particular, Hoff-Serre \cite{hoffserre} shows that the weak solutions of the 1-D isentropic \textbf{CNS} need not depend continuously on their initial data when the initial density contains an interval of vacuum states; Li-Wang-Xin \cite{LWX3} proves the instantaneous blow up of $L^2$-norm of $H^s(\RR^N)$ solutions for $s>[\frac{N}{2}]+1$; and Xin-Zhu \cite{zz2} and Duan-Xin-Zhu \cite{DXZ1} prove that,  for both the isentropic and  non-isentropic flow, the classical solutions with vacuum of the three-dimensional (3-D) \textbf{CNS} cannot preserve the conservation of the momentum.  These counterintuitive behaviors can be attributed to the unphysical assumption that $(\mu,\lambda,\kappa)$ are all constants when one utilizes \textbf{CNS} to deal with the vacuum problems in fluids \cite{taiping}, which indicates that the vacuum exerts a force on the fluid on the vacuum boundary. Thus, viscous compressible fluids near vacuum should be better modeled by the \textbf{CNS} with degenerate viscosities and heat conductivity, as was mentioned in \eqref{fandan}-\eqref{eq:1.6g}. 

However, for the isentropic \textbf{CNS} \eqref{eq:1.1}-\eqref{kelaoxiusi} with $\nu>0$, the momentum equations are degenerate both in the time evolution and spatial dissipation near the vacuum, 
\begin{equation}\label{dengshang}
\displaystyle
\underbrace{\rho(\textbf{u}_t+\textbf{u}\cdot \nabla \textbf{u})}_{\circledast}+\nabla P= \underbrace{\dive (\rho^{\nu}Q(\textbf{u}))}_{\Diamond},
\end{equation}
where   $\circledast$ denotes the degenerate time evolution, $\Diamond$ the degenerate dissipation, and $Q(\textbf{u})=2a_1D(\textbf{u})+a_2\text{div}\,\textbf{u}\,\mathbb{I}_N$.
Such a double degenerate structure  in \eqref{dengshang} may lead to singular behaviors of solutions compared with  the  uniformly   parabolic systems, which makes it challenging to study the well-posedness of large solutions with vacuum.  This degenerate system has attracted extensive attention recently, and 
some important achievements both on weak and strong solutions with vacuum to its Cauchy problem have been obtained, cf. \cites{bvy,zhenhua,lz,vayu,sz3,zz2,zz,clz,germain,sz333}.

It is worth pointing out that, for the important physical model, the shallow water equations that correspond to \eqref{eq:1.1}-\eqref{kelaoxiusi} with $\nu=1$ and $\gamma=2$, the well-posedness theories of  classical solutions  in \cites{clz,sz3}   allow vacuum only at  far fields, and it is still unclear how to deal with the corresponding Cauchy  problem with vacuum appearing in some open sets with nonzero measures. Actually, in  the derivation of hydrodynamic equations from physical principles, the underlying assumption is that the fluid is non-dilute and can be described as a continuum, which means that one can not use hydrodynamic equations to study  the time evolution of  thermodynamical states in the vacuum region. Such   considerations  leads to studies on  the  vacuum problem for  compressible fluids by the {\rm\textbf{VFBP}} instead of the Cauchy problem, which  arises in many important physical situations such as astrophysics, shallow water waves, etc., and has received much attention.  For the {\rm\textbf{VFBP}} of the isentropic compressible Euler equations (\eqref{eq:1.1}-\eqref{kelaoxiusi} with $a_1=a_2=0$), some significant progress  on the well-posedness of smooth solutions satisfying \eqref{PVcondition} have been obtained. The local existence  theory was developed by Coutand-Shkoller \cites{coutand1,coutand3} and Jang-Masmoudi \cites{Jang-M1,Jang-M2},
and the unconditional uniqueness was proved by Luo-Xin-Zeng \cites{LXZ1}. Recently,  Jang-Hadžić \cite{Jang-Hadzic} constructed global unique solutions when $\gamma \in (1,\frac{5}{3}]$, and the initial data lie sufficiently close to the expanding compactly supported affine motions constructed by Sideris \cite{sideris} and they satisfy \eqref{PVcondition}. We also refer readers to \cites{coutand2,GL1,LZ4,tpy,Oli1, Tra-Wang} and the references therein for some other related progress.

For the {\rm\textbf{VFBP}} of the degenerate \textbf{CNS}, the key issue is whether  the double degenerate structure shown in \eqref{dengshang} can propagate the initial regularity of $u$, which is subtle and surprisingly different from  the inviscid case. Actually, on the one hand, due to the appearance of the degenerate dissipation, the classical argument on the div-curl type estimates that is used in inviscid flows for establishing the normal estimates fails here. On the other hand, the viscosity degenerates at vacuum, which makes it difficult to adapt the standard regularity estimates theory of elliptic equations to the current case. Until now, only a few papers have addressed the well-posedness theory of strong or classical solutions to the {\rm\textbf{VFBP}} of the degenerate isentropic \textbf{CNS} \eqref{eq:1.1}-\eqref{kelaoxiusi} and some  related physical models. By taking the effect of gravity force into account, when the initial datum is a small perturbation of the steady solution, the global existence of the 1-D strong solution to \eqref{PVcondition} was proved by Ou-Zeng \cite{OZ}. Later, under a suitable  smallness assumption, Luo-Xin-Zeng \cite{LXZ3}  established the global existence of  strong solutions to \eqref{PVcondition}  of the 3-D spherically symmetric  compressible Navier-Stokes-Poisson system with degenerate viscosities. Recently, assuming that $\rho_0\in H^5(I)$ and $u_0$ stays in a weighted $H^6(I)$ space, Li-Wang-Xin \cites{LWX} established the local well-posedness of classical solutions  satisfying \eqref{PVcondition} to \eqref{shallow}, and then extended this theory to the two-dimensional (2-D) shallow water equations in \cites{LWX2} under the assumption that $\rho_0\in H^7$ and $u_0$ stays in a weighted $H^8$ space. Some other related progress can also be found in \cites{Fang-Zhang, Guo-Zhu,VYZ, Yang-Zhu} and the references therein.

Despite these important progress on the {\rm\textbf{VFBP}} for viscous compressible fluids, the global well-posedness of smooth solutions with large data  is extremely difficult due to the degeneracies in the presence of the vacuum.   Indeed, almost all the known results either on the local well-posedness with large data of classical solutions  or global well-posedness for perturbed data of strong solutions to the \textbf{VFBP} of \eqref{eq:1.1}-\eqref{kelaoxiusi} (\cites{LWX,LWX2,LXZ3,OZ}) were obtained under the assumption of physical vacuum condition \eqref{PVcondition}, which makes it possible to exclude the singularity formation near the vacuum boundary. Yet it seems difficult  to generalize the techniques in \cites{LWX,LWX2,LXZ3,OZ} to the case in which  \eqref{PVcondition} fails, or use them to establish global well-posedness of classical solutions  even \eqref{PVcondition} is satisfied. 
Due to the double degenerate structures in \eqref{dengshang} in the presence of vacuum,  it is  challenging to establish global uniform estimates on high order derivatives in general unless some additional constraints, such as the BD entropy condition \eqref{BDcondition}, are imposed. In fact, as far as we know, all the known theories for global well-posedness of strong or classical solutions to \eqref{eq:1.1} with either general Cauchy data or initial boundary data on a fixed domain require that the initial density satisfies the BD entropy condition \eqref{BDcondition}, see \cites{clz,cons, BH,vassu2}. These seem to indicate that it is plausible to obtain the global well-posedness of classical solutions to the \textbf{VFBP} of \eqref{eq:1.1}-\eqref{kelaoxiusi} for general large data by exploiting the effects of both physical vacuum and BD entropy conditions. Unfortunately, such an approach fails to apply to our case since, as discussed earlier, for general initial density profiles satisfying \eqref{distance}, these two constraints are not compatible. Thus new ideas and techniques are needed to achieve the global well-posedness of classical solutions to the \textbf{VFBP} of \eqref{eq:1.1}-\eqref{kelaoxiusi} under either the physical vacuum condition or BD entropy condition alone (but not both). Fortunately, by exploiting the underlying intrinsic degenerate-singular structure of \eqref{shallow} and some elaborate analysis, we are able to identify a class of initial data defined in \eqref{distance} for $\alpha\in \left(\frac{1}{3},1\right]$ so that a global existence and uniqueness theory for classical solutions to the  \textbf{VFBP} of \eqref{eq:1.1}-\eqref{kelaoxiusi} holds without restrictions on the size of the initial data.

For simplicity, in the rest of  this paper, for any function space $X$  appearing  in this paper, unless otherwise specified, $X=X(I)$, and the following conventions are used:
\begin{align*}
&H^k=W^{k,2},\quad H^k_0=W^{k,2}_0,\quad H^{-k}=(H^k_0)^*,\quad \int f\,\dx=\int_I f\,\dx,\\
&\abs{f}_p=\norm{f}_{L^p},\quad \norm{f}_{m,p}=\norm{f}_{W^{m,p}},\quad \norm{f}_s=\norm{f}_{H^s},\\
&H^k_{\omega^p}=\big\{f\in L^1_{\mathrm{loc}}: \omega^\frac{p}{2}\partial_x^j f\in L^2,\,\,0\leq j\leq k\big\},\quad L^2_{\omega^p}=H^0_{\omega^p},\quad H^{-k}_{\omega^p}:=(H^k_{\omega^p})^*,\\
&\abs{f}_{2,\omega^p}=\norm{f}_{L^2_{\omega^p}}=\absb{\omega^\frac{p}{2}f}_2, \ \ \norm{f}_{k,\omega^p}=\norm{f}_{H^k_{\omega^p}}=\sum_{j=0}^k \absb{\omega^\frac{p}{2}\partial_x^j f}_2, \\
& d(x)\!:=\mathrm{dist}(x,\Gamma), \quad 
X([0,T]; Y)=X([0,T]; Y(I)),\\
&\norm{f}_{X_t(Y)}=\norm{f}_{X([0,T];Y)},\ \  \|\varphi (f,g)\|_X=\|\varphi f\|_{X}+\|\varphi g\|_{X}, 
\end{align*}
where $\omega\in L^1_{\text{loc}}$ stands for a  generic weight function. Moreover, we denote by $X^*$ the dual space of $X$. More details on   weighted Sobolev spaces can be found in  Kufner  \cite{kufner}.

\subsection{Main results in Lagrangian coordinates}

Denote  by $\eta(t,x)$ the position of the fluid  particle $x\in I$ at time $t$ so that
\begin{equation}\label{flow-map}
\eta_t(t,x)=u(t,\eta(t,x)) \ \ \text{for} \ t>0 \ \ \text{and} \ \ 
\eta(0,x)=x,
\end{equation}
and $(t,x)$ are the Lagrangian coordinates. Set
\begin{equation}
 H(t,x):=\rho(t,\eta(t,x)),\quad U(t,x):=u(t,\eta(t,x)).\end{equation}
Then the {\rm\textbf{VFBP}} \eqref{shallow} can be rewritten into the following initial-boundary value problem in the fixed domain $I$  in  Lagrangian coordinates $(t,x)$:
\begin{equation}\label{lagrange}
\begin{cases}
\displaystyle H_t+H\frac{U_x}{\eta_x}=0&\text{in}  \ \ (0,T]\times I,\\[2pt]
\displaystyle \eta_x  H U_t +(H^2)_x -\left(H\frac{U_x}{\eta_x}\right)_x=0&\text{in} \ \  (0,T]\times I,\\[2pt]
\displaystyle \eta_t= U&\text{in} \ \  (0,T]\times I,\\[2pt]
H>0&\text{in} \ \  (0,T]\times I,\\[2pt]
H=0&\text{on}  \ \  (0,T]\times \Gamma,\\[2pt]
(H,U,\eta) =(\rho_0,u_0,\mathrm{id})&\text{on}  \ \  \{t=0\}\times I.
\end{cases}
\end{equation}

$\eqref{lagrange}_1$ and $\eqref{lagrange}_3$ imply that
\begin{equation}\label{HHH}
H(t,x)=\frac{\rho_0(x)}{\eta_x   (t,x)}.
\end{equation}
Thus \eqref{lagrange} becomes  the following initial boundary value problem for $(U,\eta)$,
\begin{equation}\label{secondreformulation}
\begin{cases}
\displaystyle\rho_0 U_t+ \left(\frac{\rho_0^2}{\eta_x^2}\right)_x-\left(\frac{\rho_0 U_x}{\eta_x^2}\right)_x=0&\text{in}  \ \ (0,T]\times I,\\[2pt]
\displaystyle \eta_t= U&\text{in} \ \  (0,T]\times I,\\[2pt]
(U,\eta) =(u_0,\mathrm{id})&\text{on}  \ \  \{t=0\}\times I.
\end{cases}
\end{equation}

The classical solutions to  \eqref{secondreformulation} can be defined as follows.
\begin{Definition}\label{definition-lag}
Let $T$ be any positive number. $(U(t,x),\eta(t,x))$ is called a classical solution in $[0,T]\times \bar I $ to  problem \eqref{secondreformulation}, if 
\begin{gather*}
U\in C([0,T];C^2(\bar I))\cap C^1([0,T];C(\bar I)),\ \
\eta\in C^1([0,T];C^2(\bar I))\cap C^2([0,T];C(\bar I)),
\end{gather*}
satisfy the equations $\eqref{secondreformulation}_1$-$\eqref{secondreformulation}_2$ pointwise in $(0,T]\times I $, and take the initial data $\eqref{secondreformulation}_3$ continuously.
\end{Definition}

In order to construct classical  solutions  to \eqref{secondreformulation},  we consider the following two types of high-order energy functions:

\begin{equation}\label{E-1}
\begin{aligned}
E (t,U)&:=\sum_{k=0}^2 \big|\rho_0^\alpha\partial_t^k U(t)\big|_2^2+\sum_{k=0}^1 \big|\rho_0^\alpha\partial_t^k U_x(t)\big|_2^2\\
&\quad +\abs{\rho_0^\alpha\partial_t U_{xx}(t)}_2^2+\sum_{k=2}^4 \big|\rho_0^\alpha\partial_x^k U(t)\big|_2^2,\\
\widetilde{E} (t,U)&:=\sum_{k=0}^2 \big|\rho_0^\frac{1}{2}\partial_t^k U(t)\big|_2^2+\sum_{k=0}^1 \big|\rho_0^\frac{1}{2}\partial_t^k U_x(t)\big|_2^2\\
&\quad +\big|\rho_0^{\left(\frac{3}{2}-\varepsilon_0\right)\alpha}\partial_t U_{xx}(t)\big|_2^2+\sum_{k=2}^4 \big|\rho_0^{\left(\frac{3}{2}-\varepsilon_0\right)\alpha}\partial_x^k U(t)\big|_2^2,
\end{aligned}
\end{equation}
where 
\begin{equation}\label{varepsilon0}
\begin{cases}
0<\varepsilon_0\leq \frac{3\alpha-1}{2\alpha}\text{ and }0<\varepsilon_0<\frac{1}{\alpha}-1 &\text{for }\frac{1}{3}<\alpha<1,\\
0<\varepsilon_0<1 &\text{for }\alpha=1.
\end{cases}
\end{equation}
In addition, we define the following spaces: for $\ell\in \NN$,
\begin{equation*}
\begin{split}
\mathscr{C}^\ell([0,T];E):=&\left\{
F\left|\begin{array}{cc}
\rho_0^\alpha\partial_t^j F,\,\, \rho_0^\alpha\partial_t^k F_x\in C^\ell([0,T];L^2),\quad j=0,1,2,\,\,k=0,1.\\[4pt]
\rho_0^\alpha\partial_t F_{xx},\,\,\rho_0^\alpha\partial_x^j F\in C^\ell([0,T];L^2),\quad j=2,3,4.
\end{array}\right.
\right\},\\
\mathscr{C}^\ell([0,T];\widetilde E):=&\Bigg\{
F\Bigg|\begin{array}{cc}
\rho_0^\frac{1}{2}\partial_t^j F,\,\, \rho_0^\frac{1}{2}\partial_t^k F_x\in C^\ell([0,T];L^2),\quad j=0,1,2,\,\,k=0,1.\\[4pt]
\rho_0^{\left(\frac{3}{2}-\varepsilon_0\right)\alpha}\partial_t F_{xx},\,\,\rho_0^{\left(\frac{3}{2}-\varepsilon_0\right)\alpha}\partial_x^j F\in C^\ell([0,T];L^2),\quad j=2,3,4.
\end{array}\Bigg.
\Bigg\}.
\end{split}
\end{equation*}

Now, we are ready to state the main results in Lagrangian coordinates. The first one is the  local-in-time existence  of  the unique classical solution to \eqref{secondreformulation}.
\begin{Theorem}\label{theorem3.1}
Assume  that \eqref{distance} holds for $0<\alpha \leq 1$.
\begin{enumerate}
\item [$\mathrm{i)}$] If $0<\alpha\leq\frac{1}{3}$ and $(\rho_0,u_0)$ satisfies 
\begin{equation}\label{a1}
 E(0,U)<\infty,
\end{equation}
then there exist a time $T_*>0$ and a unique classical solution $(U,\eta)$  in $[0,T_*]\times \bar I $ to  \ef{secondreformulation} such that 
\begin{equation}\label{b111}
\begin{gathered}
U\in \mathscr{C}([0,T_*];E),\quad \eta\in \mathscr{C}^1([0,T_*];E);\\
\frac{1}{2}\leq \eta_x(t,x)\leq  \frac{3}{2} \ \ \text{for all } (t,x)\in [0,T_*]\times\bar I.
\end{gathered}
\end{equation}
In particular, 
\begin{equation}\label{continuous}
\begin{gathered}
\ \quad \quad\quad U\in C([0,T_*];H^3)\cap C^1([0,T_*];H^1), \
\eta\in C^1([0,T_*];H^3)\cap C^2([0,T_*];H^1).
\end{gathered}
\end{equation}

Moreover, such a classical solution admits the following Neumann boundary condition,
\begin{equation}\label{N111}
U_x(t,x)=0\ \ \text{on} \  (0,T_*]\times \Gamma,
\end{equation}
and the asymptotic behavior,
\begin{equation}\label{AN111}
\abs{U_x(t,x)}\leq C d(x)\ \ \text{in} \  (0,T_*]\times \bar{I}.
\end{equation}
\item [$\mathrm{ii)}$] If $\frac{1}{3}<\alpha\leq 1$ and $(\rho_0,u_0)$ satisfies  
\begin{equation}\label{a1'}
\widetilde E(0,U)<\infty,
\end{equation}
then there exist a time $T_*>0$ and a unique classical solution $(U,\eta)$  in $[0,T_*]\times \bar I $ to  \ef{secondreformulation} such that
\begin{equation}\label{b111'}
\begin{gathered}
U\in \mathscr{C}([0,T_*];\widetilde E),\quad \eta\in \mathscr{C}^1([0,T_*];\widetilde E);\\
\frac{1}{2}\leq \eta_x(t,x)\leq  \frac{3}{2} \ \ \text{for all } (t,x)\in [0,T_*]\times\bar I.
\end{gathered}
\end{equation}
Moreover,  \eqref{N111}-\eqref{AN111} hold, and 
\begin{equation}\label{continuous'}
\begin{gathered}
U\in C([0,T_*];W^{3,1})\cap C^1([0,T_*];W^{1,1}),\\
\eta\in C^1([0,T_*];W^{3,1})\cap C^2([0,T_*];W^{1,1}).
\end{gathered}
\end{equation}
\end{enumerate}
\end{Theorem}

The second  one is the  global-in-time existence  of the unique  classical solution with large data  to \eqref{secondreformulation}.
\begin{Theorem}\label{Theorem1.1} 
Assume   that $\frac{1}{3}<\alpha\leq 1$ and $(\rho_0,u_0)$ satisfies  \eqref{distance} and \eqref{a1'}.
Then for arbitrarily large time $T>0$, there  exists a unique classical solution $(U,\eta)$  in $[0,T]\times \bar I $ to \ef{secondreformulation}, satisfying  
\begin{equation}\label{b1}
\begin{gathered}
U\in \mathscr{C}([0,T];\widetilde E),\quad \eta\in \mathscr{C}^1([0,T];\widetilde E);\\
C^{-1}(T)\leq  \eta_x(t,x) \leq C(T) \ \ \text{for all }(t,x)\in [0,T]\times \bar I,
\end{gathered}
\end{equation}
where $C(T)$ is a  positive constant  depending  only on  $\alpha$, $\varepsilon_0$, $|I|$, $(\rho_0,u_0)$ and  $T$.
Moreover, \eqref{N111}-\eqref{AN111} and \eqref{continuous'}  hold with  $T_*$ replaced by  $T$.
\end{Theorem}

\subsection{Main results in Eulerian coordinates}
Denote $\mathbb{I}(T)=\{(t,y)|  t\in (0,T], \ y\in I(t)\}$. The classical solutions to the {\rm\textbf{VFBP}} \eqref{shallow} in $\overline{\mathbb{I}(T)}$ can be defined as follows.
\begin{Definition}
Let $T$ be any positive number. A triple $(\rho(t,y),u(t,y),\Gamma(t))$ is said to be a classical solution to the {\rm\textbf{VFBP}} \eqref{shallow} in $\overline{\mathbb{I}(T)}$, if 
\begin{equation*}
\begin{split}
\rho, \ \rho_t, \ \rho_y,\ u,\ u_y, \ u_{yy}, \ u_t \in C( \overline{\mathbb{I}(T)}),\ \Gamma(t)\in C^2([0,T]),
\end{split}
\end{equation*}
$(\rho,u,\Gamma)$ satisfies the equations $\eqref{shallow}_1$-$\eqref{shallow}_3$  pointwise in $\mathbb{I}(T)$, takes the initial data $\eqref{shallow}_6$, and satisfies the boundary conditions $\eqref{shallow}_4$-$\eqref{shallow}_5$ continuously.
\end{Definition}

Now, the main results shown  in the previous subsection can be given  in Eulerian coordinates as follows.
\begin{Theorem}\label{theorem1.3}
Assume  that \eqref{distance} and \eqref{a1} hold  for  $0<\alpha\leq \frac{1}{3}$;  while \eqref{distance} and \eqref{a1'} hold for $\frac{1}{3}<\alpha\leq 1$. Then there exist a time $T_*>0$ and a unique classical solution $(\rho(t,y),u(t,y),\Gamma(t))$ in $\overline{\mathbb{I}(T_*)}$ to the {\rm\textbf{VFBP}} \ef{shallow}  such that for  $0<\alpha\leq \frac{1}{3}$,
\begin{equation}\label{eulerregularity1}
\begin{split}
\sup_{t\in [0,T_*]}\big(\|\rho^\alpha\|_{H^3(I(t))}+\|(\rho^\alpha)_t\|_{H^2(I(t))}+\|u\|_{H^3(I(t))}+\|u_t\|_{H^1(I(t))}\big)<\infty,
\end{split}
\end{equation}
while for 
$\frac{1}{3}<\alpha\leq 1$,
\begin{equation}\label{eulerregularity2}
\begin{split}
\sup_{t\in [0,T_*]}\big(\|\rho^\alpha\|_{W^{3,1}(I(t))}+\|(\rho^\alpha)_t\|_{W^{2,1}(I(t))}+\|u\|_{W^{3,1}(I(t))}+\|u_t\|_{W^{1,1}(I(t))}\big)<\infty.
\end{split}
\end{equation}
Moreover, the velocity  satisfies 
\begin{equation}\label{N111euler}
u_y(t,y)=0 \ \ \text{for all }t\in [0,T_*]\text{ and }y\in \Gamma(t).
\end{equation}
\end{Theorem}

\begin{Theorem}\label{Theorem1.4} 
Assume that $\frac{1}{3}<\alpha\leq 1$,   \eqref{distance} and \eqref{a1'} hold. Then for arbitrarily large time $T>0$, there  exists a unique classical solution $(\rho(t,y),u(t,y),\Gamma(t))$ in $\overline{\mathbb{I}(T)}$  to the {\rm\textbf{VFBP}} \ef{shallow}. Moreover, \eqref{eulerregularity2}-\eqref{N111euler} hold with  $T_*$ replaced by  $T$.
\end{Theorem}

Furthermore, in contrast to the classical theory  \cites{KA2, mat}, if $|\int_{I}\rho_0u_0\mathrm{d}y|>0$, the velocity of  solutions
obtained in Theorem \ref{Theorem1.4}  does not decay to zero in the $L^\infty$ norm as $t\rightarrow \infty$.

\begin{Theorem}\label{th:2.20-c}
Assume that   $|\int_{I}\rho_0u_0\mathrm{d}y|>0$.
Then the  global classical solution $(\rho,u)$ to \eqref{shallow} obtained in Theorem \ref{Theorem1.4} does not 
satisfy 
\begin{equation}\label{eq:2.15}
\limsup_{t\rightarrow \infty} \sup_{y\in I(t)}|u(t,y)|=0.
\end{equation}
\end{Theorem}

We make some comments on the results of this paper.

\begin{Remark} It should be noted that \eqref{N111} in the case $\alpha=1$ has been observed in \cites{LWX,LWX2}. Now we show how to derive \eqref{N111}  in the general case. Take the case $\frac{1}{3}<\alpha\leq 1$ for example. First, it follows from \eqref{distance}, \eqref{flow-map}, \eqref{continuous'} and Lemma \ref{sobolev-embedding} that 
\begin{equation}\label{eq117}
\begin{gathered}
\rho_0^\alpha\in C^{2}(\bar I),\  U_x, \ U_{xx}, \ U_t, \ \eta_x,\ \eta_{xx}\in  C([0,T_*]\times \bar I).
\end{gathered}
\end{equation}
Next, multiplying both sides of $\eqref{secondreformulation}_1$ by $\eta_x^2\rho_0^{\alpha-1}$ gives
\begin{equation}\label{remark1.1}
\rho_0^\alpha U_{xx}= \frac{2}{\alpha} \rho_0(\rho_0^\alpha)_x+\rho_0^\alpha\eta_x^2 U_t -\frac{2\rho_0^{\alpha+1}\eta_{xx}}{\eta_x}-\frac{1}{\alpha}(\rho_0^\alpha)_xU_x+\frac{2\rho_0^\alpha \eta_{xx} U_x}{\eta_x}. 
\end{equation}
Then  letting $x$ approach the  boundary $\Gamma$ in \eqref{remark1.1}, one obtains from \eqref{b111'} and \eqref{eq117} that
\begin{equation}
(\rho_0^\alpha)_xU_x=0\ \  \text{for }(t,x)\in (0,T_*]\times \Gamma,
\end{equation}
which, along with 
$(\rho_0^\alpha)_x|_{x\in\Gamma}\neq 0$,
yields that $U_x=0$ for $(t,x)\in (0,T_*]\times \Gamma$. It is worth noting that \eqref{N111} can be thought of as inheriting from the homogeneous Neumann boundary condition of $u_0$ which is shown in Lemma \ref{Reduction}, and \eqref{N111} will play an important role in establishing the uniform lower and upper bounds of $\eta_x$ in \S \ref{Section6}. 
\end{Remark}

\begin{Remark}
For the  {\rm\textbf{VFBP}} \eqref{shallow}, it follows from \eqref{distance}, \eqref{N111} and   \eqref{continuous'}  that  the usual
stress free boundary  condition holds automatically, i.e.,
$$ \mathbb{S} = \rho^2 -\rho u_y=0 \ \ \text{for }t\in (0,T] \text{ and } y\in \Gamma(t).$$
\end{Remark}

\begin{Remark}\label{initialexample}
The initial assumptions \eqref{distance} and \eqref{a1} or \eqref{a1'} in Theorem \ref{theorem3.1} identify a class of admissible initial data that ensure unique solvability of \eqref{secondreformulation}. In Appendix \ref{AppB}, we give an equivalent form of \eqref{a1} or \eqref{a1'} in terms of $(\rho_0,u_0)$ themselves and their spatial derivatives in Lemma \ref{Reduction}. Indeed, it follows from  Lemma \ref{Reduction}  that the assumptions \eqref{distance} and \eqref{a1} or \eqref{a1'} in Theorem \ref{theorem3.1} can be fulfilled if $\rho^\alpha_0\in H^3$ and $u_0$ stays in a weighted $H^4$ space and a special compatibility condition is satisfied. In particular,  Lemma \ref{Reduction} implies the following facts.

First, for the case  $0<\alpha\leq \frac{1}{3}$, \eqref{distance} and  \eqref{a1} can be  satisfied by:
\begin{equation}\label{initalexample1}
\rho_0(x):=C\left(x(1-x)\right)^\frac{1}{\alpha},\quad u_0(x)\in C^\infty_c,
\end{equation}
where $C>0$ denotes a generic constant.  The details can be found in Remark \ref{Reduction}.

Second, for $\frac{1}{3}< \alpha \leq 1$, \eqref{distance} and   \eqref{a1'} are fulfilled by the set of initial data given by 
\begin{equation}\label{initalexample2}\rho_0(x):=C\left(x(1-x)\right)^\frac{1}{\alpha},\quad u_0(x):= \int_0^x \rho_0(z)\,\mathrm{d}z + f_0(x),
\end{equation}
for arbitrary  $f_0\in C_c^\infty(I)$.  The details can be found in Remark \ref{Reduction}.
\end{Remark}

We make some comments on the methodology of this paper.

\begin{Remark} \label{constructionE}We give some comments on the forms of energy functions  $(E(t,U),\widetilde{E}(t,U))$ in \eqref{E-1}. In order to get the solution which is classical uniformly up to the moving boundary, inspired by  the Sobolev embedding theorem and the Hardy inequality, we should establish some weighted $H^4$ estimates on $U$. So a natural energy function takes the form
\begin{equation}\label{E^*}
E^*(t,U)=\sum_{k=0}^2 \big|\rho_0^p\partial_t^k U(t)\big|_2^2+\sum_{k=0}^1 \big|\rho_0^p\partial_t^k U_x(t)\big|_2^2+\abs{\rho_0^q\partial_t U_{xx}(t)}_2^2+\sum_{k=2}^4 \big|\rho_0^q\partial_x^k U(t)\big|_2^2.
\end{equation}
To derive the highest order elliptic estimates, by formally applying $\rho_0^{q-\alpha}\partial_t$ and $\rho_0^{q-\alpha}\partial_x^2$, respectively, to both sides of \eqref{remark1.1}, one has 
\begin{equation}\label{Utxx-Uxxxx}
\begin{aligned}
\rho_0^q \partial_t U_{xx}&=\rho_0^q\eta_x^2 U_{tt}+(R_{q}^1),\\
\rho_0^q \partial_x^4 U &=\frac{2(1-\alpha)}{\alpha^3}\rho_0^{1-3\alpha+q}(\rho_0^\alpha)_x^{3}+\rho_0^q\eta_x^2 \partial_t U_{xx}+(R_{q}^2),
\end{aligned}
\end{equation}
which, by substituting $\eqref{Utxx-Uxxxx}_1$ into $\eqref{Utxx-Uxxxx}_2$, leads to
\begin{equation}\label{q-Uxxxx}
\rho_0^q \partial_x^4 U  = \underline{\frac{2(1-\alpha)}{\alpha^3}\rho_0^{1-3\alpha+q}(\rho_0^\alpha)_x^{3}}_{(\star_1)} +\underline{\rho_0^q\eta_x^4U_{tt}}_{(\star_2)}+(R_{q}^3),
\end{equation}
where $((R_q^1),(R_q^2),(R_q^3))$ denote the remaining terms. As can be checked, $(\star_1)$ is the most singular part of the derivatives of the pressure, $(\star_2)$ is the highest-order tangential derivatives, and $(R_q^3)$, compared with $(\star_1)$-$(\star_2)$, possesses either higher-order weights or lower-order derivatives of $U$. It turns out that, to control $|\rho_0^q\partial_x^4 U|_2$ by \eqref{q-Uxxxx}, the main obstacles are $(\star_1)$-$(\star_2)$. 
Hence it follows from \eqref{distance} that $(\star_1)$-$(\star_2)$ belong to $L^2$ whenever 
\begin{equation}\label{value-q-alpha}
q> \frac{5\alpha}{2}-1 \ \text{ or } \ \alpha=1,\ \ \text{ and } \ q\geq p.
\end{equation}
On the other hand, the embedding relation  $H^4_{\rho_0^{2\alpha}}\into H^3\into C^2(\bar I)$ (see Lemmas \ref{sobolev-embedding}-\ref{hardy-inequality} in Appendix \ref{appendix A}) implies that $q\leq \alpha$ in \eqref{q-Uxxxx}. Therefore, in order to get the maximum range of $\alpha$, we determine the energy function $E(t,U)$ in \eqref{E-1} by setting $p=q=\alpha$ in \eqref{E^*}, and finally  obtain from \eqref{value-q-alpha} that $\alpha\in\left(0,\frac{2}{3}\right)$ or $\alpha=1$. Fortunately, it follows from direct calculations that $|(R_q^3)|_2$ can be controlled by $E(t,U)$ with the above well-chosen weights. 

When  $\alpha\in \left[\frac{2}{3},1\right)$, based on Lemmas \ref{sobolev-embedding}-\ref{hardy-inequality}, we consider the embedding relation  $H^4_{\rho_0^{2q}} \into W^{3,1} \into C^2(\bar I)$ for $q<\frac{3\alpha}{2}$. On the other hand, it seems that one can determine $p=\frac{1}{2}$ in \eqref{E^*} from establishing the tangential estimates via $\eqref{secondreformulation}_1$, and hence it follows from \eqref{value-q-alpha} and $q<\frac{3\alpha}{2}$ that the maximum range of $\alpha$ is $\alpha\in \left(\frac{1}{3},1\right]$. As can be checked, $|(R_q^3)|_2$ is controlled by $\widetilde E(t,U)$ with such a weight. Actually, only from the perspective of the local existence, one can still obtain the tangential estimates with $p\neq \frac{1}{2}$ via a simple reformulation on $\eqref{secondreformulation}_1$ by multiplying  both of its sides by $\rho_0^{2p-1}$, while such a reformulation will break some intrinsic structure of $\eqref{secondreformulation}_1$ and make it hard to establish the global-in-time energy estimates.  Therefore, we determine the energy function $\widetilde E(t,U)$ in \eqref{E-1} by setting $p=\frac{1}{2}$ and $q<\frac{3\alpha}{2}$ in \eqref{E^*}, then deduce from \eqref{value-q-alpha} that $\frac{1}{3}<\alpha\leq 1$.

Based on the above considerations, we will establish the desired solutions in two different energy functions: $E(t,U)$ when $\alpha\in\left(0,\frac{1}{3}\right]$, and $\widetilde E(t,U)$ when $\alpha\in\left(\frac{1}{3},1\right]$.
\end{Remark}

\begin{Remark}
For proving the local existence   stated in Theorem \ref{theorem3.1}, a key step is to establish the  existence and uniqueness  of the corresponding linearized problems \eqref{lp''}-\eqref{lp} in \S \ref{Section3} via the Galerkin method.  However, the standard Galerkin method (see \cites{evans}) is not applicable here, since there will be some issue that arises in the approximation of initial data $U_{tt}(0,x)$, which only belongs to  $L^2_{\rho_0}$, via the standard Hilbert basis $\{e_j\}_{j=1}^\infty$ which is orthonormal in $L^2$ and orthogonal in $H^1$ generated by the Laplace operator. To solve this problem, we apply a modified Galerkin method by using a sequence of $\{\ddot U^\delta_0\}_{\delta>0}\subset C^\infty(\bar I)$ to approximate $U_{tt}(0,x)$, such that each $\ddot U^\delta_0$ can be expanded by $\{e_j\}_{j=1}^\infty$, then considering the linearized problem with the initial data $\ddot U^\delta_0$, and finally recovering the original linearized problems    \eqref{lp''}-\eqref{lp} by the standard density arguments. Note that this issue has been initially noticed by Li-Wang-Xin \cites{LWX,LWX2} when $\alpha=1$, where they developed a different approach to overcome the difficulties  by constructing a new Galerkin basis $\{w_j\}_{j=1}^\infty$ which is orthonormal in $L^2_{\rho_0}$ and orthogonal in $H^1_{\rho_0}$. More details  can be found in  \S\ref{subsection2.3}.
\end{Remark}

\begin{Remark} As mentioned before, the   BD entropy condition \eqref{BDcondition}  and physical vacuum  condition \eqref{PVcondition}  are not compatible in \eqref{distance}. Then for the case $\alpha=1$, since \eqref{BDcondition} fails,  the classical global-in-time estimates on the first order derivatives of the depth obtained in \cites{bd6,bd8, Bresch, bd2, Kanel} are not available here.
Yet we can still establish the global existence and uniqueness of classical solutions  to the {\rm\textbf{VFBP}} \eqref{shallow}  with large data for $\alpha \in (\frac{1}{3},1]$. The key tool used here is the so-called effective velocity $
V=U+\frac{H_x}{\rho_0}$. By taking full advantage of the transport mechanism of the evolution equation \eqref{eq:effective2} of $V$, we exploit some  new global  weighted $L^p$ ($2\leq p \leq \infty$) estimates for $V$ (see Lemmas \ref{boundv} and \ref{boundv'}), 
which do not require the initial BD entropy condition or  the physical vacuum boundary condition.
Such new estimates enable one to  deal effectively with the terms related with  $H_x$ and $\eta_{xx}$ that appear in the lower order  estimates of  $U$,  and then  establish the global-in-time upper bound of $\eta_x$ and the weighted energy estimates of $U$. More details  can be found in  \S\ref{subsection2.3}.

\end{Remark}

The rest of this paper is organized as follows. In \S \ref{Section2}, we first  introduce a new reformulation of \eqref{secondreformulation} in \S\ref{Subsection2.1},  and then outline the main strategy  to establish the local/global existence and uniqueness. \S\ref{Section3}-\S\ref{Section5} are devoted to proving  the  local existence and uniqueness of  classical solutions to   problem \eqref{secondreformulation} and hence the {\rm\textbf{VFBP}} \eqref{shallow} as follows: 
\begin{enumerate}
\item  construct global smooth approximate solutions  for the corresponding   degenerate linearized problems via the  Galerkin method (\S \ref{Section3});
\item establish  the uniform weighted estimates for  the linearized problems (\S \ref{Section4});
\item  establish  the  local existence and uniqueness  of classical solutions to the nonlinear problems \eqref{secondreformulation} and \eqref{shallow}, respectively, through the classical Picard iteration (\S \ref{Section5}).
\end{enumerate}
In \S \ref{Section6}-\S \ref{Section8}, we  show the global existence of classical solutions to the problem \eqref{secondreformulation} and the {\rm\textbf{VFBP}} \eqref{shallow} in the following three steps:    
\begin{enumerate}
\item  derive  the global-in-time a priori lower and upper bounds for  $\eta_x$, and the weighted boundedness of the effective velocity (\S \ref{Section6});
\item establish the global-in-time a priori  weighted estimates for the velocity (\S \ref{Section7});
\item obtain  the  global existence and uniqueness of classical solutions to the nonlinear problems \eqref{secondreformulation} and \eqref{shallow}, respectively, by  the standard continuity  method (\S \ref{Section8}).
\end{enumerate} 
The non-decay result stated in   Theorem \ref{th:2.20-c}  is proved in \S\ref{section9}.
 Finally, for the convenience of readers, we  list some basic facts and auxiliary lemmas which have been  used frequently in this paper in  Appendixes A-E.

\section{Reformulations and the main strategy}\label{Section2}

In this section, we first reformulate  $\eqref{secondreformulation}_1$ according to the value of $\alpha$,  and then sketch the main strategy of the  analysis.  Throughout the rest of this paper, $C(\geq 1)$ will denote
a generic  constant which depends only on fixed constants $\alpha$, $\varepsilon_0$, $|I|$ and $(\rho_0,u_0)$, which may be different from line to line; we will also use $C(\nu_1,\cdots,\nu_k)$ to emphasize the dependency of $C$ on the additional parameters  $\nu_1,\cdots, \nu_k$;  $A\sim B$ means $C^{-1}A\leq B\leq CA$; and  $\langle\cdot,\cdot\rangle$ denotes the inner product in $L^2$. 

\subsection{Reformulations}\label{Subsection2.1}

Set  $\phi_0:=\rho_0^{\alpha}$, and rewrite \eqref{secondreformulation} into the following two forms: 
\begin{itemize}
\item for $0<\alpha\leq \frac{1}{3}$,
\begin{equation}\label{fp''}
\begin{cases}
\quad\displaystyle \phi_0^{2} U_t-\left(\frac{\phi_0^{2} U_{x}}{\eta_x^2}\right)_x\\
\displaystyle=\Big(\frac{1}{\alpha}-2\Big)\frac{\phi_0(\phi_0)_x U_x}{\eta_x^2}-\bigg(\frac{\phi_0^{2+\frac{1}{\alpha}}}{\eta_x^2}\bigg)_x+\Big(2-\frac{1}{\alpha}\Big)\frac{\phi_0^{1+\frac{1}{\alpha}} (\phi_0)_x}{\eta_x^2}&\text{in }(0,T]\times I,\\
\quad \displaystyle \eta_t= U&\text{in} \ \  (0,T]\times I,\\
\quad (U,\eta)=(u_0,\mathrm{id})&\text{on }\{t=0\}\times I;
\end{cases}
\end{equation}
\item for  $\frac{1}{3}<\alpha\leq 1$,
\begin{equation}\label{fp}
\begin{cases}
\displaystyle \phi_0^{\frac{1}{\alpha}} U_t-\bigg(\frac{\phi_0^{\frac{1}{\alpha}} U_{x}}{\eta_x^2}\bigg)_x+\bigg(\frac{\phi_0^\frac{2}{\alpha}}{\eta_x^2}\bigg)_x=0 &\text{in }(0,T]\times I,\\
\displaystyle \displaystyle \eta_t= U&\text{in} \ \  (0,T]\times I,\\
(U,\eta)=(u_0,\mathrm{id})&\text{on }\{t=0\}\times I.
\end{cases}
\end{equation}
\end{itemize}

Under the above  reformulations, the condition \eqref{distance} is equivalent to 
\begin{equation}\label{dist}
\phi_0\in H^3\ \ \text{and} \ \ \cC_1d(x)\leq \phi_0(x)\leq \cC_2d(x)\quad \text{for all }x\in I,
\end{equation}
and the energy functionals \eqref{E-1} can be rewritten as
\begin{equation}\label{energy2.5}
\begin{aligned}
E (t,U)&=\sum_{k=0}^2 \absf{\phi_0 \partial_t^k U(t)}_2^2+\sum_{k=0}^1 \absf{\phi_0 \partial_t^k U_x(t)}_2^2 +\abs{\phi_0 \partial_t U_{xx}(t)}_2^2+\sum_{k=2}^4 \absf{\phi_0 \partial_x^k U(t)}_2^2,\\
\widetilde E (t,U)&=\sum_{k=0}^2 \absb{\phi_0^\frac{1}{2\alpha}\partial_t^k U(t)}_2^2+\sum_{k=0}^1 \absb{\phi_0^\frac{1}{2\alpha}\partial_t^k U_x(t)}_2^2
\\
&\quad
+\absb{\phi_0^{\frac{3}{2}-\varepsilon_0}\partial_t U_{xx}(t)}_2^2
 +\sum_{k=2}^4 \absb{\phi_0^{\frac{3}{2}-\varepsilon_0}\partial_x^k U(t)}_2^2.
\end{aligned}
\end{equation}

\subsection{Main strategy}\label{subsection2.3} 
Our  main strategy and new ideas  will be given in this section.   To overcome the difficulty caused by the degeneracy and singular behavior near the vacuum free boundary for large data problems, the main ingredients of the analysis consist of
\begin{itemize}
\item  some new weighted energy functionals;
\smallskip
\item a modified Galerkin scheme for  local existence;
\smallskip
\item  a new cross-derivatives embedding theorem for the elliptic estimates;
\smallskip
\item some new global  weighted $L^p$ ($2\leq p \leq \infty$) estimates  for the effective velocity (that  are different from the classical  ones in  Bresch-Desjardins \cites{bd6}
and Kanel \cite{Kanel}), which are crucial for the case $\alpha=1$ in \eqref{distance} since the initial  BD entropy condition fails;
\smallskip
 \item  the global-in-time lower and upper bounds of $\eta_x$ for large data problems that are given by some novel $(\rho_0, \eta_x)$-weighted estimates for $U_x$.
\end{itemize}

\subsubsection{Some new weighted energy functionals} \label{subsec-2.0}

The strong degeneracy of  the momentum  equation $\eqref{secondreformulation}_1$
  makes it intricate 
to provide an effective propagation  mechanism for the  regularity of  $U$ near the vacuum    in general Sobolev  spaces. The first key step in constructing classical solutions is to introduce suitable weighted energy
functionals.
By  considering  the balance of the pressure and the  spatial dissipation near the vacuum, we give  $E(t,U)$ for $\alpha \in (0,\frac{1}{3}]$ and $\widetilde{E}(t,U)$ for $\alpha \in (\frac{1}{3},1]$ by \eqref{E-1}, which turn out to be  reasonable later. The details on how to construct  these  energy functionals   can be found in Remark \ref{constructionE}.

\subsubsection{A modified Galerkin scheme}
Starting from   $(E(t,U), \widetilde{E}(t,U))$, for the local existence,   we introduce 
a modified  Galerkin scheme  to overcome the degeneracy of the time evolution. Actually, the classical Galerkin scheme does not work here. We only take the case $\alpha=1$ for example to show this issue, and the other cases can be dealt with similarly.

 Usually,  one may choose a smooth  and  orthogonal basis $\{e_j\}_{j=1}^\infty$ of $H^1$ with $(e_j)_x|_{x\in \Gamma}=0$, which is orthonormal in $L^2$, and the Galerkin approximate solutions have the form  $X^n(t,x)=\sum_{j=1}^n \mu_j^n(t) e_j(x)$ with $\mu^n_j(0)=\langle u_0,e_j\rangle$ (see \cite{evans}). For obtaining the local existence   in $\widetilde E (t,U)$ to   \eqref{lp},  the classical  Galerkin method  will  first establish the  uniform energy estimates for $X^n$, then obtain the unique  weak solution $U$ of  \eqref{lp}, and finally repeat this process sequentially for $(X^n_x,X^n_t,X^n_{tx},X^n_{tt})$ to improve the tangential regularities for $U$. However,
 a  question naturally arises when establishing such kind of  energy estimate:
 does there exist a  constant $C>0$, independent of $n$, such that 
\begin{equation}\label{initialcon}
\sum_{k=0}^2 \big|\sqrt{\rho_0}\partial_t^k X^n(0)\big|_2^2+\sum_{k=0}^1 \big|\sqrt{\rho_0}\partial_t^k X^n_x(0)\big|_2^2\leq C\widetilde E(0,U)\,? 
\end{equation}
For example,  $X^n_{tt}(0)=\sum_{j=1}^n (\mu_j^n)''(0) e_j$ is  defined  by the equation of $\mu_j^n$ at $t=0$, while the form of $(\mu_j^n)''(0)$ is too complicated to obtain \eqref{initialcon}. In fact, a similar issue has been noticed in  \cites{LWX,LWX2}, where they constructed a Hilbert basis that is orthogonal in a suitable weighted Sobolev space, so that \eqref{initialcon} can be derived directly via Bessel's inequality. 

In the current paper,  we developed a new way to establish the regularities of $(U,U_t,U_{tt})$ by a modified Galerkin scheme. Specifically,  we consider three
linear problems:  the linear problem \eqref{lp} for $U=w^{(0)}$ itself, the linear problem for  $U_t=w^{(1)}$: 
\begin{equation}\label{eq:U_t}
\begin{cases}
\displaystyle \rho_0 w^{(1)}_{t}-\bigg(\frac{\rho_0 w^{(1)}_{x}}{\bar\eta_x^2}\bigg)_x=\bigg(\frac{2\rho_0^2 \bar U_x}{\bar \eta_x^3}-\frac{2\rho_0 \bar U_x U_x}{\bar \eta_x^3}\bigg)_x&\text{in }(0,T]\times I,\\
w^{(1)} = U_t(0,x)&\text{on }\{t=0\}\times I,
\end{cases}
\end{equation}
and the linear problem for $U_{tt}=w^{(2)}$: 
\begin{equation}\label{eq:U_tt}
\begin{cases}
\displaystyle \quad \rho_0 w^{(2)}_{t}-\bigg(\frac{\rho_0 w^{(2)}_{x}}{\bar\eta_x^2}\bigg)_x\\[8pt]
\displaystyle=\left(\frac{2\rho_0^2 \bar U_{tx}}{\bar \eta_x^3}-\frac{6\rho_0^2 \bar U_x^2}{\bar\eta_x^4}-\frac{2\rho_0 \bar U_{tx} U_x}{\bar\eta_x^3}-\frac{4\rho_0 \bar U_x U_{tx}}{\bar\eta_x^3}+\frac{6\rho_0 \bar U_x^2 U_x}{\bar \eta_x^4}\right)_x&\text{in }(0,T]\times I,\\
\quad w^{(2)} = U_{tt}(0,x)&\text{on }\{t=0\}\times I,
\end{cases}
\end{equation}
where $(\bar{\eta},\bar{U})$ are some known functions. 
\begin{enumerate}
\item[$\mathrm{i)}$]
 first, we  prove   the global existence of  weak solutions to \eqref{lp}, \eqref{eq:U_t} and  \eqref{eq:U_tt} in $L^2_{\rho_0}$ space  by  the Galerkin method, respectively;

\item[$\mathrm{ii)}$]
second, since $(w^{(1)},w^{(2)})$ have the same regularity as $w^{(0)}$ as a weak solution in  $L^2_{\rho_0}$ space,  the issue for deriving the regularities of $(U_t,U_{tt})$ is naturally converted into proving $(w^{(1)},w^{(2)})=(U_t,U_{tt})$.
\end{enumerate}

On the one hand, for the theory in ${\rm i)}$, we just take \eqref{eq:U_tt} for example, and \eqref{lp}, \eqref{eq:U_t} can be solved similarly.
We first choose a sequence of smooth functions $\{\ddot{U}_0^\delta\}_{\delta>0}$ that converges to $U_{tt}(0,x)$ under the $L^2_{\rho_0}$-norm (see Lemma \ref{W-space}). Second, let the corresponding Galerkin approximate solutions of the problem \eqref{eq:U_tt} have the forms $Y^{n,\delta}(t,x)=\sum_{j=1}^n \lambda_j^{n,\delta}(t)e_j(x)$ with $\lambda_j^{n,\delta}(0) =\langle\ddot{U}_0^\delta,e_j\rangle$. By the same diagonal density argument used after \eqref{wndelta}, the pairs $(n,\delta)$ can be selected with $n,\delta^{-1}\to\infty$ so that
\begin{equation*}
\absf{\sqrt{\rho_0}Y^{n,\delta}(0)}_2 \leq C\absf{\sqrt{\rho_0}U_{tt}(0)}_2,
\end{equation*}
where $C>0$ is independent of the pairs in this diagonal sequence. We then establish the desired uniform estimates for $Y^{n,\delta}$ in $L^2_{\rho_0}$ and take the limit as $n,\delta^{-1}\to\infty$ along the same diagonal sequence to obtain a weak solution $w^{(2)}$ to \eqref{eq:U_tt}.

On the other hand,  for the theory in ${\rm ii)}$,
we  notice that  $(w^{(0)},w^{(1)},w^{(2)})$ satisfy the equations in \eqref{lp}, \eqref{eq:U_t}, \eqref{eq:U_tt}  in the sense of distributions, respectively, i.e., for all $\varphi\in H^1_{\rho_0}$,
\begin{equation}\label{WF-X}
\big<\rho_0 w^{(i)}_t, \varphi\big>_{H^{-1}_{\rho_0}\times H^1_{\rho_0}}+\Big<\frac{\rho_0 w^{(i)}_{x}}{\bar\eta_x^2},\varphi_x\Big>=\big<\sqrt{\rho_0}\bar R^{(i)},\varphi_x\big>,\quad i=0,1,2,
\end{equation}
where $\bar R^{(i)}$ ($i=0,1,2$) denote some   remainders. Denoting 
$$Y^{(1)}=\int_0^t w^{(1)}\,\mathrm{d}s+u_0-U \ \ \text{and}  \ \ Y^{(2)}=\int_0^t w^{(2)}\,\mathrm{d}s+U_t(0,x)-U_{t},$$
integrating \eqref{WF-X} over $[0,t]$ for $i=1$ and $2$,  and  then subtracting the resulting equations from \eqref{WF-X} for $i=0$ and $1$, respectively, one has  that $(Y^{(1)},Y^{(2)})$ satisfy
\begin{equation}\label{WF-Y}
\big<\rho_0 Y^{(i)}_t, \varphi\big>_{H^{-1}_{\rho_0}\times H^1_{\rho_0}}+\Big<\frac{\rho_0 Y^{(i)}_{x}}{\bar\eta_x^2},\varphi_x\Big>=\big<R^{(i)}_1,\varphi_x\big>+\big<R^{(i)}_2,\varphi\big>,\quad i=1,2,
\end{equation}
where $(R^{(i)}_1,R^{(i)}_2)$ ($i=1,2$) denote some remainders. By choosing  $\varphi=Y^{(i)}$, one  gets  from the energy estimates that $Y^{(i)}=0$, and thus $(w^{(1)},w^{(2)})=(U_t,U_{tt})$ a.e.

\subsubsection{A new cross-derivatives embedding theorem}
The degeneracy of the spatial dissipation
near the moving boundary makes it hard to establish the desired elliptic estimates. In order to overcome this difficulty, motivated by   the energy functional that we choose,   we  provide a novel cross-derivatives embedding theorem  (see Proposition  \ref{prop2.1} in Appendix \ref{subsection2.2}). Specifically, it is shown that under some carefully chosen parameter $s$, 
\begin{equation}\label{corssa}
\underbrace{\absf{\rho_0^s \partial_x^{j+1} U}_2}_{\text{high order term}}  \leq C(s)\big(\underbrace{\absf{\rho_0^s \partial_x^{j+1} U+\kappa\rho_0^{s-1}(\rho_0)_x\partial_x^j U}_2}_{\text{crossing term}}+\underbrace{\absf{\rho_0^s \partial_x^{j} U}_2}_{\text{lower order term}}\!\!\big) \ \ \text{for }j\in \NN,
\end{equation}
which   means that  the  weighted $L^2$ estimate on the higher   order  spatial derivatives of $U$ can be dominated by the $L^2$ estimate  on  the crossing term and the   weighted $L^2$ estimate on the  lower   order  spatial derivatives of $U$  while keeping  weights unchanged. 
It should be pointed out that for the special $\alpha=1$, 
 the original version of \eqref{corssa}, i.e.,  
\begin{equation}\label{corssa1}
\absb{\rho_0^{s+\frac{1}{2}} \partial_x^{j+1} U}_2  \leq C(s)\Big( \absb{\rho_0^{s+\frac{1}{2}} \partial_x^{j+1} U+\kappa\rho_0^{s-\frac{1}{2}}(\rho_0)_x\partial_x^j U}_2 + \absf{\rho_0^s \partial_x^{j} U}_2 \Big) \ \ \text{for }j\in \NN,
\end{equation} 
which was first announced in \cites{coutand1,coutand3}, was applied  to establish the  elliptic estimates for $U$ in \cites{LWX,LWX2}. Compared  with \eqref{corssa},  \eqref{corssa1} leads to the increasing of the power of weights when establishing the higher order elliptic estimate from the lower order one. Consequently, if one starts with  $\sqrt{\rho_0}U_x\in L^\infty([0,T];L^2)$, it follows from  \eqref{corssa1} that the optimal elliptic estimates are $\rho_0^\frac{j}{2}\partial_x^{j} U\in L^\infty([0,T];L^2)$ for $j\geq 2$, and thus it is required that $\rho_0\in H^5$ and $u_0$ stays in one  weighted $H^6$ space in \cite{LWX}. In the present paper,  Proposition  \ref{prop2.1} allows us to  reduce the initial regularity required for local existence and uniqueness of classical solutions  to $\rho_0\in H^3$ and $u_0$ belonging to a  weighted $H^4$ space.

\subsubsection{Some new global weighted estimates for the effective velocity independent of the BD entropy}
As observed in $\eqref{lagrange}_2$ or $\eqref{secondreformulation}_1$, in order to establish the desired global-in-time energy estimates for $U$, the global-in-time lower and upper bounds of $\eta_x$, and the regularity estimates  of $H_x$ or $\eta_x$ are required. For the Cauchy problem of the 1-D shallow water equations with far field vacuum, when the initial BD entropy condition \eqref{BDcondition} holds, in \cite{clz}, one can use the classical BD entropy estimate, i.e., $\|\sqrt{\rho}(u+(\log \rho)_y)(t,\cdot)\|_{L^2(\mathbb{R})}<\infty$ for arbitrarily large time $t$, to provide the global-in-time upper bound of the effective velocity $v=u+ (\log\rho)_y$ in Eulerian coordinates. Then the derivative $\rho_y$ in the momentum equations can be controlled by the quantity $\rho(v-u)$. However, when $\alpha=1$ in \eqref{distance}, \eqref{BDcondition} fails, so this type  of argument cannot be applied to our analysis.
In order to overcome this difficulty, we developed some   new weighted $L^p$-estimates ($2\leq p\leq \infty$) of $V$, especially the weighted $L^\infty$-estimate of $V$, which do not require the initial  BD entropy condition \eqref{BDcondition} or the physical vacuum boundary condition \eqref{PVcondition}, and play a crucial role in establishing the global upper bound of $\eta_x$ and the uniform estimates for $U$. 
 
On the one hand, it follows from the time evolution equation  \eqref{eq:effective2} of $V$ that
\begin{equation}\label{odesudu}
V(t,x)=e^{-\int_0^t 2H\,\mathrm{d}s}V(0,x)+\int_0^t \frac{2\rho_0 U }{\eta_x} e^{-\int_\tau^t 2H\,\mathrm{d}s}\,\mathrm{d}\tau.
\end{equation}
Then multiplying above by $\rho_0^K$ for some $K>0$, one can obtain that for $2\leq p <\infty$,
\begin{equation}\label{38'}
\begin{aligned}
|\rho_0^KV(t)|_p&\leq |\rho_0^KV(0)|_p+2\int_0^t \Big|\frac{\rho_0^{1+K}U}{\eta_x}\Big|_p\,\mathrm{d}\tau.
\end{aligned}
\end{equation}
Actually, the boundedness of $|\rho_0^KV(t)|_p$ is basically determined by that of $|\rho_0^KV(0)|_p$, since the global-in-time weighted $L^p$-estimates of $U$ in \eqref{38'} are somewhat easy to derive.
Consequently,  we can choose suitable  $K$ in \eqref{38'} to get the desired weighted $L^p$ $(2\leq p<\infty)$ estimates of $V$ in Lemma \ref{boundv}, i.e., 
\begin{equation}\label{VVVccc}
\rho_0^{r\alpha} V\in L^\infty([0,T];L^p)\quad \text{for}\quad r>\frac{p-1}{p},\,\,2\leq p<\infty.
\end{equation}

On the other hand, we find that the  Sobolev embedding theorem and \eqref{VVVccc} give 
\begin{equation}\label{6223nnnn}
\int_0^t \absB{\frac{\rho_0^{1+\alpha}U}{\eta_x}}_\infty\,\ds\ <\infty \ \ \text{for all }0\leq t\leq T,  
\end{equation}
which, along with \eqref{odesudu}, yields that (see Lemma \ref{boundv'})
\begin{equation}\label{VVVddd}
 \rho_0^\alpha V \in L^\infty([0,T];L^\infty).
\end{equation}
\subsubsection{Global-in-time lower and upper bounds  of $\eta_x$}
Now we show how to derive the uniform lower and upper bounds for $\eta_x$ when there is no restriction on the size of the initial data,  particularly for the case $\alpha=1$, where the classical BD entropy estimates fail.

The uniform lower bound of $\eta_x$ is motivated from the following several observations:
\begin{itemize}

 \item we can  get the global uniform upper bound of $H$
   only by the fundamental energy estimates \eqref{energy-e} (see Lemma \ref{bound of density});
 \smallskip
\item The $L^q$ ($2\leq q<\infty$) energy estimates of $U$ in Lemma \ref{W.E.I} imply a refined upper bound of $H$, that is, $H\leq C(T)\rho_0$.
\smallskip
\item \eqref{HHH} and the refined global upper bound of $H$ imply
\begin{equation}\label{2,,5}
H(t,x)=\frac{\rho_0(x)}{\eta_x(t,x)}\leq C(T)\rho_0\implies \eta_x(t,x)\geq C(T)^{-1}.
\end{equation}
\end{itemize}
 
To establish the global-in-time upper bound of $\eta_x$, we first use
the weighted logarithmic estimate \eqref{39'} to control the interior
region.

On the other hand, for the global upper bound of $\eta_x$, our first aim is to establish 
\begin{equation}\label{39'}
\sup_{t\in [0,T]}|\rho_0^K\log\eta_x|_\infty\leq C(K,T),
\end{equation}
for some $K>0$ (see Lemma \ref{U.B.J}), which is motivated by the identity 
\begin{equation}\label{008}
\log \eta_x(t,x)=\int_0^t \frac{U_x}{\eta_x}(s,x)\,\mathrm{d}s. 
\end{equation}
Therefore, we only need  to show that  
\begin{equation}\label{41}
\rho_0^K \frac{U_x}{\eta_x}\in L^1([0,T];L^\infty).
\end{equation}

Actually,  \eqref{39'} only ensures that $\eta_x\leq C(\sigma,T)$   in  $[0,T]\times [\sigma,1-\sigma]$ for any $\sigma\in (0,\frac{1}{2})$. To get the upper bound of $\eta_x$ near the vacuum boundary, we analyze equation $\eqref{secondreformulation}_1$ and derive the additional $\rho_0$-decay of the boundary momentum integral, thereby obtaining $H\geq C(T)^{-1}\rho_0$ in a neighborhood whose width
depends only on the initial data, the fixed parameters, and $T$.
Combining the interior and boundary estimates, together with the boundary condition $\eta_x|_{x\in\Gamma}=1$, gives the desired
upper bound of $\eta_x$.  

Turning back to the derivation of \eqref{41}, we develop an unconventional technique:
some $\eta_x$-weighted energy estimates   for $(U,U_x)$.
For convenience, let $K>0$ be a generic parameter. Formally,  to get \eqref{41}, one needs to derive some weighted $L^2$-estimate of $(\frac{U_x}{\eta_x})_x$ by the Sobolev embedding $H^1\hookrightarrow L^\infty$.  If we just  multiply $\eqref{secondreformulation}_1$ by $\rho_0^KU$ and $\rho_0^KU_t$ successively,  and then  integrate these resulting equalities, we shall obtain  
\begin{equation}\label{42}
\begin{aligned}
\rho_0^{K}U,\ \ \rho_0^{K}\frac{U_x}{\eta_x}\in L^\infty([0,T];L^2),\quad \rho_0^{K}U_t, \ \ \frac{\rho_0^{K}}{\eta_x}\Big(\frac{U_x}{\eta_x}\Big)_x\in L^2([0,T];L^2),
\end{aligned}
\end{equation}
where the estimate for  $\frac{\rho_0^{K}}{\eta_x}(\frac{U_x}{\eta_x})_x$ can be obtained from $\eqref{secondreformulation}_1$ and that of $\rho_0^{K}U_t$. However, \eqref{42} cannot  lead  to  \eqref{41}, since  there are "too many $\eta_x$" appearing in the denominators of $\rho_0^K\frac{U_x}{\eta_x}$ and $\frac{\rho_0^K}{\eta_x}(\frac{U_x}{\eta_x})_x$. To address this issue, we multiply $\eqref{secondreformulation}_1$ by $\rho_0^K \eta_x U$ and $\rho_0^K \eta_x U_t$ (introducing the weight $\eta_x$), and then  derive the following types of estimates:
\begin{equation}\label{edt}
\rho_0^{K} \sqrt{\eta_x}U, \ \rho_0^{K}\frac{U_x}{\sqrt{\eta_x}}\in L^\infty([0,T];L^2),\quad 
\rho_0^{K}\sqrt{\eta_x}U_t, \ \frac{\rho_0^{K}}{\sqrt{\eta_x}}\left(\frac{U_x}{\eta_x}\right)_x\in L^2([0,T];L^2).
\end{equation}
The above estimates, together with Lemma \ref{sobolev-embedding}, would lead to \eqref{41}, that is, 
\begin{align}
&\int_0^T\Big|\rho_0^{K}\frac{U_x}{\eta_x}\Big|_{\infty}\,\mathrm{d}t\leq \sqrt{T}\Big(\int_0^T\Big|\rho_0^{K}\frac{U_x}{\eta_x}\Big|_{\infty}^2\,\mathrm{d}t\Big)^\frac{1}{2}\leq C(T)+\int_0^T \bigg|\left(\rho_0^{2K} \frac{U_x^2}{\eta_x^2}\right)_x\bigg|_1\mathrm{d}t\notag\\
\leq& C(T)+2\int_0^T  \bigg|\rho_0^{2K} \Big|\frac{U_x}{\sqrt{\eta_x}}\Big| \Big|\frac{1}{\sqrt{\eta_x}}\left(\frac{U_x}{\eta_x}\right)_x\Big|\bigg|_1 \mathrm{d}t +(\text{Remainders})\label{edt''}\\
\leq& C(T)+\underbrace{2\sup_{t\in[0,T]}\Big|\rho_0^{K} \frac{U_x}{\sqrt{\eta_x}}\Big|_{2}\int_0^T \Big|\frac{\rho_0^{K}}{\sqrt{\eta_x}}\left(\frac{U_x}{\eta_x}\right)_x\Big|_{2}\,\mathrm{d}t}_{\text{controlled by \eqref{edt}}}  +(\text{Remainders})\leq C(T).\notag
\end{align}
The detailed calculations of \eqref{edt} and \eqref{edt''}  can be found in Lemmas \ref{W.E.IV}-\ref{W.E.V} and Corollary \ref{W.E.VI}, respectively.

\subsubsection{Global weighted energy estimates}
For establishing the global weighted energy estimates, first,  using \eqref{V-expression}, we rewrite $\eqref{secondreformulation}_1$   as follows:
\begin{equation*}
\rho_0 U_t -\left(\frac{\rho_0U_x}{\eta_x^2}\right)_x+2\frac{\rho_0^2}{\eta_x}(V-U)=0,
\end{equation*}
which, along with  the global lower and upper bounds of $\eta_x$, and the new weighted $L^p$  estimate for $V$,
yields all the following tangential estimates (see  \S \ref{subsection7.1}) 
\begin{equation}\label{2211}
\sqrt{\rho_0}\partial_t^j U,\,\, \sqrt{\rho_0}\partial_t^k U_x \in L^\infty([0,T];L^2) \ \ \text{for }j=0,1,2 \text{ and }k=0,1.
\end{equation}

Next, for the elliptic estimates, we take $\rho_0^{\left(\frac{3}{2}-\varepsilon_0\right)\alpha} \partial_x^3U$, for example, and the other ones can be obtained similarly.  In \S \ref{subsection7.2}, according to  \eqref{VVVccc}, \eqref{VVVddd}, the lower and upper bounds of $\eta_x$, and  \eqref{2211}, we  first obtain 
$\rho_0^{\frac{1}{2}+\varepsilon}\left(\eta_x^{-1} U_{x}\right)_{xx} \in L^\infty([0,T];L^2)$ for any $\varepsilon>0$ from $\eqref{secondreformulation}_1$. 
Second,  applying $\partial_x^2$  to \eqref{008} leads to  the following  identity:
\begin{equation*}
\partial_x^3\eta=\eta_x^{-1}\eta_{xx}^2+\eta_x \int_0^t\left(\eta_x^{-1} U_{x}\right)_{xx}\,\ds,
\end{equation*}
It follows that 
$\rho_0^{\frac{1}{2}+\varepsilon} \partial_x^3\eta\in L^\infty([0,T];L^2)$.
Thus   $\rho_0^{\frac{1}{2}+\varepsilon}\partial_x^3U \in L^\infty([0,T];L^2)$, which, by carefully choosing $\varepsilon$ and applying Proposition \ref{prop2.1}, yields $\rho_0^{\left(\frac{3}{2}-\varepsilon_0\right)\alpha}\partial_x^3 U\in L^\infty([0,T];L^2)$.

\section{Global existence  and uniqueness for the linearized problems}\label{Section3}

We will give the proofs for  the local  existence and uniqueness results stated in Theorem \ref{theorem3.1} in \S3-\S5. In this section, we first linearize the reformulated problems \eqref{fp''}-\eqref{fp}, and then establish  the global  existence and uniqueness  of  classical solutions to the linearized problems by the Galerkin scheme.
For convenience, denote by $\left<\cdot,\cdot\right>_{X^*\times X}$ the pairing between the space $X$ and its dual space $X^*$, and $\langle\cdot,\cdot\rangle$ the inner product of $L^2$, that is,
\begin{equation*}
\left<F,f\right>_{X^*\times X}:=F(f) \ \ \text{for } F\in X^*,f\in X;\quad  \langle f,g\rangle:=\int fg\,\dx \ \ \text{for }f,g\in L^2.    
\end{equation*}
In particular, if $X\into L^2 \into X^*$, and $F\in L^2$, then $\left<F,f\right>_{X^*\times X}=\langle F,f\rangle$. In addition, we denote by $\left< F,f \right>_{X_t^*(Y^*)\times X_t(Y)}$  the pairing between the vector-valued space $X([0,T];Y)$ and its dual  space $X^*([0,T];Y^*)$.

\subsection{Linearization}\label{Subsection3.1}

In order to solve the nonlinear problems \eqref{fp''}-\eqref{fp}, one needs to consider  the following linearized ones:
\begin{itemize}
\item when $0<\alpha\leq \frac{1}{3}$,
\begin{equation}\label{lp''}
\begin{cases}
\quad\displaystyle \phi_0^{2} U_t-\left(\frac{\phi_0^{2} U_{x}}{\bar\eta_x^2}\right)_x\\[8pt]
\displaystyle=\Big(\frac{1}{\alpha}-2\Big)\frac{\phi_0(\phi_0)_x U_x}{\bar\eta_x^2}-\bigg(\frac{\phi_0^{2+\frac{1}{\alpha}}}{\bar\eta_x^2}\bigg)_x+\Big(2-\frac{1}{\alpha}\Big)\frac{\phi_0^{1+\frac{1}{\alpha}} (\phi_0)_x}{\bar\eta_x^2}&\text{in } (0,T]\times I,\\[8pt]
\quad U =u_0&\text{on }\{t=0\}\times I;
\end{cases}
\end{equation}
\item when $\frac{1}{3}<\alpha\leq 1$,
\begin{equation}\label{lp}
\begin{cases}
\displaystyle \phi_0^{\frac{1}{\alpha}} U_t-\bigg(\frac{\phi_0^{\frac{1}{\alpha}} U_{x}}{\bar\eta_x^2}\bigg)_x+\bigg(\frac{\phi_0^\frac{2}{\alpha}}{\bar\eta_x^2}\bigg)_x=0&\text{in } (0,T]\times I,\\[8pt]
U =u_0&\text{on }\{t=0\}\times I,
\end{cases}
\end{equation}
\end{itemize}
where $\bar \eta $ stands for the flow map corresponding to $\bar U$,
\begin{equation}\label{given-flow}
\bar \eta (t,x)=x+\int_0^t \bar U(s,x)\,\ds,\quad \bar \eta |_{t=0}=\mathrm{id},
\end{equation}
and $\bar U$ is a given function satisfying that $\bar U(0,x)=u_0(x)$ for $x\in I$, and for any $T>0$,
\begin{itemize}
\item if $0<\alpha\leq \frac{1}{3}$,
$$\bar U\in C([0,T];H^3)\cap C^1([0,T];H^1), \ \ \sup_{t\in[0,T]}E(t,\bar U)<\infty, \ \  \bar U\in \mathscr{C}([0,T];E);$$
\end{itemize}
\begin{itemize}
\item if $\frac{1}{3}<\alpha\leq 1$,
 $$ \ \ \ \ 
 \bar U\in C([0,T];W^{3,1})\cap C^1([0,T];W^{1,1}), \ \ \sup_{t\in[0,T]}\widetilde E(t,\bar U)<\infty, \ \ \bar U\in \mathscr{C}([0,T];\widetilde E).$$
\end{itemize}

Moreover, it will be assumed here that there exists a time $\widetilde T>0$ such that  $\frac{1}{2}\leq \bar\eta_x\leq \frac{3}{2}$ on $[0,\widetilde T]\times \bar I$, which  will be shown in \S \ref{subsection4.1}-\S \ref{subsection4.2} for our linearization procedure, and we assume also $T\in (0, \widetilde T]$. The main result in this section on the global  existence and uniqueness  for the linear problems above can be stated in the following lemma.
\begin{Lemma}\label{existence-linearize}
Suppose that \eqref{dist} and \eqref{a1} or \eqref{a1'} hold. Then for any $0<T\leq \widetilde T$,
\begin{enumerate}
\item[$\mathrm{i)}$] if $0<\alpha\leq \frac{1}{3}$, \eqref{lp''} admits a unique classical solution $U$ in $[0,T]\times \bar I$ satisfying 
\begin{equation*}
U\in C([0,T];H^3)\cap C^1([0,T];H^1), \ \ \sup_{t\in [0,T]}E(t,U)<\infty, \ \ U\in \mathscr{C}([0,T];E);    
\end{equation*}
\item[$\mathrm{ii)}$] if $\frac{1}{3}<\alpha\leq 1$, \eqref{lp} admits a unique classical solution $U$ in $[0,T]\times \bar I$ satisfying 
\begin{equation*}
U\in C([0,T];W^{3,1})\cap C^1([0,T];W^{1,1}), \ \  \sup_{t\in [0,T]}\widetilde E(t,U)<\infty, \ \  U\in \mathscr{C}([0,T];\widetilde E).    
\end{equation*}
\end{enumerate}
Moreover, $U$ satisfies the Neumann boundary condition, $U_x(t,x)=0$ on $ [0,T]\times \Gamma$.
\end{Lemma}

\subsection{The Galerkin method: weak solutions to some general degenerate systems}\label{section-galerkin}
Before proving Lemma \ref{existence-linearize}, we first show the global-in-time existence of weak solutions to the general problems \eqref{galerkin-w} and \eqref{galerkin-w'} via the Galerkin method (see  Propositions \ref{prop1}-\ref{prop042}). These results will be extensively used in \S \ref{subsection3.3}-\S\ref{subsection3.4} for the proof of Lemma \ref{existence-linearize}.

\subsubsection{Case $0<\alpha\leq \frac{1}{3}$}\label{subsubsection3.2.1}
In the first case $0<\alpha\leq \frac{1}{3}$, we consider the following initial boundary value problem:
\begin{equation}\label{galerkin-w}
\begin{cases}
\displaystyle \phi_0^2 w_t - \left(\frac{\phi_0^2 w_x}{\bar\eta_x^2}\right)_x = K_0\frac{\phi_0(\phi_0)_xw_x}{\bar\eta_x^2} - \left(\phi_0 P_1\right)_x+ P_2&\text{in } (0,T]\times I,\\
w =w_0&\text{on } \{t=0\}\times I,
\end{cases}
\end{equation}
where $P_1,P_2\in L^2([0,T];L^2)$, $w_0\in L^2_{\phi_0^2}$ and $K_0\geq 0$ is a given constant, and the definition of weak solutions of this problem can be given as follows.
\begin{Definition}\label{def3.1}
A function $w(t,x)$ is said to be a weak solution in $[0,T]\times I$ to \eqref{galerkin-w}, if
\begin{enumerate}
\item[$\mathrm{i)}$] $\phi_0 w\in C([0,T];L^2), \quad \phi_0 w_x\in L^2([0,T];L^2),\quad \phi_0^2 w_t\in L^2([0,T];H^{-1}_{\phi_0^2})$;
\item[$\mathrm{ii)}$]
the following equation holds for all $\varphi\in H^1_{\phi_0^2}$ and a.e. time $0<t\leq T$,
\begin{equation}\label{weak.F.}
\begin{aligned}
 \qquad \left<\phi_0^2 w_t, \varphi\right>_{H^{-1}_{\phi_0^2}\times H^1_{\phi_0^2}}\!+\left<\frac{\phi_0^2 w_{x}}{\bar\eta_x^2},\varphi_x\right>&=\left<K_0\frac{\phi_0(\phi_0)_xw_x}{\bar\eta_x^2}+P_2,\varphi\right>+\left<\phi_0 P_1, \varphi_x\right>;
\end{aligned}
\end{equation}
\item[$\mathrm{iii)}$] $w(0,x)=w_0(x)$ for a.e. $ x\in I$. 
\end{enumerate}
\end{Definition}

The  aim of \S \ref{subsubsection3.2.1} is to establish the following existence theory and the related estimates. 
\begin{Proposition}\label{prop1}
For all $0<T\leq \widetilde T$, there exists a unique weak solution $w$ in $[0,T]\times I$ to the problem \eqref{galerkin-w}, satisfying the following estimate: 
\begin{equation*}
\begin{aligned}
&\sup_{t\in[0,T]} \abs{\phi_0 w}_2^2+\int_0^T \big(\abs{\phi_0 w_x}_2^2+ \|\phi_0^2 w_t\|_{-1,\phi_0^2}^2\,\big)\dt
\leq   C\Big(\abs{\phi_0 w_0}_2^2+ \int_0^T  (\abs{P_1}_2^2+ \abs{P_2}_2^2 \,)\dt\Big).
\end{aligned}
\end{equation*}
\end{Proposition}

\begin{proof}
\underline{\textbf{Step 1: introduction of the Galerkin scheme.}}  First, it follows from Lemma \ref{W-space} that, for given $w_0\in L^2_{\phi_0^2}$, there exists a smooth sequence $\{w_0^\delta\}_{\delta>0}\subset C^\infty(\bar I)$ satisfying
\begin{equation}\label{wdelta-w}
\phi_0 w^\delta_0 \to \phi_0 w_0 \ \ \text{in }L^2 \quad \text{as} \quad \delta \rightarrow 0.
\end{equation}

Second, by solving the eigenvalue problem $-\Delta e +e =\lambda e$ with the Neumann boundary condition (see Chapter 9 of \cite{Sayas}), one can  choose a Hilbert basis $\{e_j\}_{j=1}^\infty$ of $H^1$ with $e_j\in C^\infty(\bar I)$ and $(e_j)_x|_{x\in \Gamma}=0$ ($j\geq 1$), which is orthonormal in $L^2$ and orthogonal in $H^1$.

Next, given any $0<\delta<1$ and $n\in \NN^*$, set 
\begin{equation}\label{U^n}
X^{n,\delta}(t,x):=\sum_{k=1}^n \mu_k^{n,\delta}(t) e_k(x),
\end{equation}
where $\mu_k^{n,\delta}(t)$ are selected by solving the following initial value problem of the ODE system:
\begin{equation}\label{galerkin-n}
\begin{cases}
\quad\displaystyle\big<\phi_0^2 X^{n,\delta}_t, e_j\big>+\bigg<\frac{\phi_0^2 X^{n,\delta}_{x}}{\bar\eta_x^2},(e_j)_x\bigg>\\[10pt]
\displaystyle=K_0\bigg<\frac{\phi_0(\phi_0)_xX^{n,\delta}_x}{\bar\eta_x^2},e_j\bigg>+\left<\phi_0 P_1, (e_j)_x\right>+\left<P_2, e_j\right>\ \ \text{in } (0,T],\\[10pt]
\quad \mu_j^{n,\delta}(0)=\left<w_0^\delta,e_j\right>,\,\,j=1,2,\cdots,n.
\end{cases}
\end{equation}
For simplicity, one can rewrite \eqref{galerkin-n} as  
\begin{equation}\label{mu^n}
\begin{cases}
\displaystyle \AA\cdot\frac{\mathrm{d}}{\dt}\mu^{n,\delta}(t)+\BB(t)\cdot\mu^{n,\delta}(t)=\CC(t) \ \ \text{in } (0,T],\\[10pt]
\mu_j^{n,\delta}(0)=\left<w_0^\delta,e_j\right>,\,\,j=1,2,\cdots,n,
\end{cases}
\end{equation}
where
\begin{equation*}
\begin{split}
\mu^{n,\delta}(t):=&(\mu_j^{n,\delta}(t))_{j=1}^n,\quad \mathbb{A}=\left(\int \phi_0^2 e_k e_j\,\dx\right)_{k,  j=1}^n,\\
\BB(t)=&\left(\int \frac{\phi_0^2 (e_k)_x (e_j)_x}{\bar\eta_x^2}\,\dx-K_0\int \frac{\phi_0(\phi_0)_x (e_k)_xe_j}{\bar\eta_x^2}\,\dx\right)_{k, j=1}^n,\\
\CC(t)=&\left(\int\phi_0 P_1(e_j)_x\,\dx+\int P_2 e_j\,\dx\right)_{j=1}^n.
\end{split}
\end{equation*}

In order to solve \eqref{mu^n}, one needs first to show that the  matrix $\AA$ is  non-singular.
\begin{Lemma}\label{singular}
The vectors $\{\phi_0 e_j\}_{j=1}^n$, $n\in\NN^*$, are linearly independent. In particular, the matrix $\AA$ is non-singular for each $n\in \NN^*$.
\end{Lemma}
\begin{proof}
Since $\{e_j\}_{j=1}^{\infty}$ are linearly independent, so are  $\{\phi_0 e_j\}_{j=1}^\infty$.
Hence its Gram matrix $\AA$ is, of course, non-singular.
\end{proof}

In addition, one can check that $\BB(t)\in W^{2,\infty}(0,T)$, $\CC(t)\in L^2(0,T)$. Thus, based on standard ODEs' theory, one can show the following existence result.
\begin{Lemma}\label{existence-mu}
There exists a small time $0<T_n\leq \widetilde T$ which depends only on $n$, such that \eqref{mu^n} admits a unique solution $\mu_j^{n,\delta}\in AC[0,T_{n}]$, for each $j=1,2,\cdots, n$. Here, $AC$ stands for the space of absolutely continuous functions. Consequently, $X^{n,\delta}(t,x)$ belongs to $AC([0,T_{n}];C^\infty(\bar I))$, and $X^{n,\delta}$ is differentiable a.e. in $t$, for each $n\in \NN^*$ and $0<\delta<1$.
\end{Lemma}
\begin{proof}
 Consider the general recursive integral equations:
\begin{equation}\label{integral-equ}
\begin{aligned}
\mu^{k+1}(t)&=\mu_0+\int_0^t \big(\bar \CC(\tau)-\bar \BB(\tau)\mu^k(\tau)\big)\,\mathrm{d}\tau,\quad k\in \NN,
\end{aligned}
\end{equation}
where $ \mu^0(t)=\mu_0= (w_0^\delta,e_j)$, $\bar \BB(t):=\AA^{-1}\cdot \BB(t)$, $\bar \CC(t):=\AA^{-1}\cdot \CC(t)$. 

Next, set $\hat \mu^k(t)=\mu^{k}(t)-\mu^{k-1}(t)$. It follows from \eqref{integral-equ} that
\begin{equation*}
\hat\mu^{k+1}(t)=-\int_0^t \bar \BB(\tau)\hat\mu^k(\tau)\,\mathrm{d}\tau,  
\end{equation*}
which leads to
\begin{equation}\label{picard-mu}
\sup_{\tau\in[0,t]}|\hat\mu^{k+1}(\tau)|\leq t|\bar \BB|_\infty \sup_{\tau\in[0,t]}|\hat\mu^{k}(\tau)|\leq \frac{1}{2} \sup_{\tau\in[0,t]}|\hat\mu^{k}(\tau)|,
\end{equation}
provided that $t\leq T_n:=\min\{(2|\bar \BB|_\infty+1)^{-1},\widetilde T\}$. Note that, by induction, \eqref{picard-mu} implies that $\sum_{k=1}^{\infty} \sup_{t\in[0,T_n]}|\hat\mu^{k}(t)|<\infty$, and hence $\{\mu^k\}_{k=0}^{\infty}$ is a Cauchy sequence that converges uniformly to some limit $\mu\in L^\infty(0,T_n)$, namely,
\begin{equation*}
\mu^k \to \mu \ \ \text{in }L^\infty \ \ \text{as }k\to \infty.
\end{equation*}

Passing to the limit $k\to \infty$ in \eqref{integral-equ} shows that
\begin{equation}\label{equ39}
\mu(t)=\mu_0+\int_0^t \left(\bar \CC(s)-\bar \BB(s)\mu(s)\right)\,\ds \ \ \text{for all }0\leq t\leq T_n,
\end{equation}
which yields that $\mu(t)\in AC[0,T_n]$ and $\eqref{mu^n}_1$ holds for a.e.  $t\in (0,T_n)$.  The uniqueness and continuity  are direct  consequences of \eqref{equ39}. The proof of Lemma \ref{existence-mu} is completed.
\end{proof}

\underline{\textbf{Step 2: Uniform estimates of $X^{n,\delta}$.}} First, multiplying both sides of \eqref{galerkin-n} by $\mu^{n,\delta}_j(t)$, and then summing $j$ from $1$ to $n$, according to Lemma \ref{hardy-inequality}, one  obtains that
\begin{equation}\label{I11I22}
\begin{aligned}
&\frac{1}{2}\frac{\mathrm{d}}{\dt}\int \phi_0^2|X^{n,\delta}|^2\,\dx+\int \frac{\phi_0^2|X^{n,\delta}_x|^2}{\bar\eta_x^2}\,\dx\\
=&\underline{K_0\int \frac{\phi_0(\phi_0)_x X^{n,\delta}_x X^{n,\delta}}{\bar\eta_x^2}\,\dx}_{:=I_1}+ \int \phi_0 P_1 X^{n,\delta}_x\,\dx+\int P_2 X^{n,\delta}\,\dx\\
\leq & I_1+C\big(\abs{P_1}_2^2+\abs{P_2}_2^2+|X^{n,\delta}|_2^2\big)+\frac{1}{8}|\phi_0 X^{n,\delta}_x|_2^2.\end{aligned}
\end{equation}
For $I_1$, it follows from integration by parts and Lemma \ref{GNinequality} that
\begin{equation}\label{equ311}
\begin{aligned}
I_1
&=\underline{-\frac{K_0}{2}\int \frac{((\phi_0)_x)^2 |X^{n,\delta}|^2}{\bar\eta_x^2}\,\dx}_{\leq 0}-\frac{K_0}{2}\int \left(\frac{(\phi_0)_x}{\bar\eta_x^2}\right)_x\phi_0|X^{n,\delta}|^2\,\dx\\
&\leq C(K_0) \left(\abs{(\phi_0)_{xx}}_\infty+ \abs{(\phi_0)_x}_\infty \abs{\bar\eta_{xx}}_\infty\right)\absb{\phi_0^\frac{1}{2}X^{n,\delta}}_2^2\\
&\leq C(K_0)|\phi_0 X^{n,\delta}|_2^2+\frac{1}{8}|\phi_0 X^{n,\delta}_x|_2^2,
\end{aligned}
\end{equation}
which, along with \eqref{I11I22} and Gr\"onwall's inequality, yields that 
\begin{equation}\label{33125}
\begin{aligned}
&\sup_{t\in[0,T_{n}]} |\phi_0 X^{n,\delta}|_2^2+\int_0^{T_{n}} |\phi_0 X^{n,\delta}_x|_2^2\,\dt
\\
\leq &C(K_0)\Big(|\phi_0 X^{n,\delta}(0)|_2^2+ \int_0^{T_n} \big(\abs{P_1}_2^2+\abs{P_2}_2^2\big)\,\dt\Big),
\end{aligned}
\end{equation}
where $
X^{n,\delta}(0,x)=\sum_{j=1}^n \mu_j^{n,\delta}(0)e_j=\sum_{j=1}^n \big<w_0^\delta,e_j\big> e_j$.

To get the uniform estimate of $\phi_0X^{n,\delta}(0,x)$, on the one hand, it follows from $w_0^\delta\in L^2$ and Lemma \ref{hilbert} (in Appendix \ref{appendix A}) that
\begin{equation*}
\sum_{j=1}^n \big<w_0^\delta,e_j\big>e_j \to w_0^\delta \ \ \text{in }L^2 \ \ \text{for any fixed  } \delta>0 \quad\text{as} \quad n\rightarrow \infty,
\end{equation*}
which yields that 
\begin{equation}\label{wndelta}
\phi_0 X^{n,\delta}(0,x)\to \phi_0  w_0^\delta \ \ \text{in }L^2  \quad\text{as} \quad n\rightarrow \infty.
\end{equation}

We now combine \eqref{wdelta-w} and \eqref{wndelta} by a
diagonal argument. If $w_0=0$, we choose $w_0^\delta=0$ for
every $\delta>0$. We first choose
$\delta$ by \eqref{wdelta-w} and then, with this $\delta$ fixed,
choose $n$ sufficiently large by \eqref{wndelta}. Thus, we can select a diagonal sequence of pairs $(n_j,\delta_j)$ such that $n_j,\delta_j^{-1}\to\infty$ and
\begin{equation}\label{erfen}
\absf{\phi_0X^{n_j,\delta_j}(0)-\phi_0w_0}_2\leq\absf{\phi_0X^{n_j,\delta_j}(0)-\phi_0w_0^{\delta_j}}_2
+\absf{\phi_0w_0^{\delta_j}-\phi_0w_0}_2\to0.
\end{equation}
In the rest of the proof, for brevity, all occurrences of $(X^{n,\delta},\mu^{n,\delta})$ refer to this fixed diagonal sequence $(X^{n_j,\delta_j},\mu^{n_j,\delta_j})$,
and every limit as $n,\delta^{-1}\to\infty$ is taken along
this sequence. If $w_0=0$, then $X^{n,\delta}(0)=0$; if $w_0\neq0$, we discard finitely many pairs
if necessary. Hence, in both cases, we can obtain from \eqref{erfen} that
\begin{equation}\label{initial-converge}
\absf{\phi_0X^{n,\delta}(0)}_2
\leq 2\absf{\phi_0w_0}_2.
\end{equation}

Combining \eqref{33125} and \eqref{initial-converge} along this diagonal sequence, one has for all $0\leq t\leq T_n$,
\begin{equation}\label{3...18}
\begin{aligned}
\absf{\phi_0 X^{n,\delta}(t)}_2^2+\int_0^{t} \absf{\phi_0 X^{n,\delta}_x}_2^2\,\dt
&\leq C(K_0)\Big(\abs{\phi_0 w_0}_2^2+ \int_0^{\widetilde T} \big(\abs{P_1}_2^2+\abs{P_2}_2^2\big)\,\dt\Big).
\end{aligned}
\end{equation}
For each $X^{n,\delta}$ in the fixed diagonal sequence, it follows from \eqref{3...18} that the local solution $\mu^{n,\delta}$ on $[0,T_n]$ in Lemma \ref{existence-mu} can be extended to a global one on $[0,T]$, for all $0<T\leq\widetilde T$. More precisely, assuming contrarily that $T_n<\widetilde T$ is the maximal life span of $\mu^{n,\delta}$, according to \eqref{3...18}, one has
\begin{equation}
\absf{\phi_0X^{n,\delta}(T_n)}_2\leq \limsup_{t\to T_n^-} \absf{\phi_0X^{n,\delta}(t)}_2\leq C\implies X^{n,\delta}(T_n,x)\in L^2_{\phi_0^2}.
\end{equation}
Then $\mu_j^{n,\delta}(T_n)=\left<X^{n,\delta}(T_n),e_j\right>$ can be regarded as a new initial value of  \eqref{mu^n}. Thanks to Lemma \ref{existence-mu}, there exists a small time $T_n'>0$, such that $\mu_j^{n,\delta}$ exists uniquely on $[0,T_n+T_n']$, which contradicts  our assumption. Therefore, \eqref{3...18} holds for any $0<T\leq \widetilde T$, that is,
\begin{equation}\label{3...18'}
\begin{aligned}
&\sup_{t\in[0,T]} \absf{\phi_0 X^{n,\delta}}_2^2+\int_0^{T} \absf{\phi_0 X^{n,\delta}_x}_2^2\,\dt
\leq C(K_0)\Big(\abs{\phi_0 w_0}_2^2+ \!\int_0^T\!\! \big(\abs{P_1}_2^2+\abs{P_2}_2^2\big)\,\dt\Big).
\end{aligned}
\end{equation}

\underline{\textbf{Step 3: Taking the limit along the diagonal
sequence as $n,\delta^{-1}\to\infty$.}} Based on \eqref{3...18'}, by the weak convergence argument,
one may extract a subsequence of the fixed diagonal sequence, still denoted by $X^{n,\delta}$, satisfying
\begin{equation}\label{3..124}
\begin{aligned}
\phi_0 X^{n,\delta} \rightharpoonup X_1\quad  &\text{weakly* in }L^\infty([0,T];L^2),\\
\phi_0 X^{n,\delta}_x \rightharpoonup X_2\quad &\text{weakly in }L^2([0,T];L^2),
\end{aligned}
\end{equation}
for some limits $X_1,X_2$. By Lemma \ref{hardy-inequality}, it holds that $\{X^{n,\delta}\}\subset L^2([0,T];L^2)$ and 
\begin{equation*}
X^{n,\delta}\rightharpoonup w\quad  \text{weakly in }L^2([0,T];L^2),
\end{equation*}
which, by  the definition of weak derivatives, yields that 
$
X_1=\phi_0 w$ and $X_2=\phi_0 w_x$.   
Moreover,  the above weak convergences imply that \eqref{3...18'} also holds for $w$.

We now pass to the limit as $n,\delta^{-1}\to\infty$ along this subsequence in \eqref{galerkin-n}. Setting $\Phi^m(t,x)=\sum_{j=1}^m\xi_j(t)e_j$, where $\xi(t)\in C_c^\infty(0,T)$, we note that, for each fixed $m$, the condition $n\geq m$ holds eventually along this
subsequence. It then follows from \eqref{galerkin-n} that
\begin{equation}\label{equ3320}
\begin{aligned}
&\int_0^T\big<\phi_0^2 X^{n,\delta}_t,\Phi^m\big>\,\dt+\int_0^T\bigg<\frac{\phi_0^2 X^{n,\delta}_x}{\bar\eta_x^2},\Phi_x^m\bigg>\,\dt\\
=&K_0\int_0^T \bigg<\frac{\phi_0(\phi_0)_x X^{n,\delta}_x}{\bar\eta_x^2},\Phi^m\bigg>\,\dt+ \int_0^T \left<\phi_0 P_1,\Phi_x^m\right>\,\dt+\int_0^T \left<P_2,\Phi^m\right>\,\dt.
\end{aligned}
\end{equation}

Next, since $X^{n,\delta}$ and $\Phi^m$ are regular with respect to $t$, then in \eqref{equ3320}, one can transfer  $\partial_t$ from $X^{n,\delta}$ to $\Phi^m$,  and then let $n,\delta^{-1}\to \infty$ along this subsequence, which, combined with  \eqref{3..124}, yields 
\begin{equation}\label{3,,20}
\begin{aligned}
&-\int_0^T\left<\phi_0^2 w, \Phi^m_t\right>\,\dt+\int_0^T\left<\frac{\phi_0^2 w_x}{\bar\eta_x^2},\Phi_x^m\right>\,\dt\\
=&K_0\int_0^T \left<\frac{\phi_0(\phi_0)_x w_x}{\bar\eta_x^2},\Phi^m\right>\,\dt+ \int_0^T \left<\phi_0 P_1,\Phi_x^m\right>\,\dt+\int_0^T \left<P_2,\Phi^m\right>\,\dt.
\end{aligned}
\end{equation}

Since $\phi_0^2w\in L^2([0,T];L^2)$, $\phi_0^2w_t\in H^{-1}([0,T];L^2)\subset H^{-1}([0,T];(H^1)^*)$, 
it follows  from \eqref{3,,20},
Lemma \ref{hardy-inequality} and the definition of  distributional derivatives that
\begin{align}
&\absb{\big<\phi_0^2 w_t, \Phi^m\big>_{H_t^{-1}((H^{1})^*)\times H^1_{0,t}(H^1)}}
=\absB{\int_0^T\left<\phi_0^2 w, \Phi^m_t\right>\,\dt}\notag\\
\leq &\int_0^T\absB{\left<\frac{\phi_0^2 w_x}{\bar\eta_x^2},\Phi_x^m\right>}\,\dt+K_0\int_0^T \absB{\left<\frac{\phi_0(\phi_0)_x w_x}{\bar\eta_x^2},\Phi^m\right>}\,\dt\label{equ321}\\
&+\int_0^T \abs{\left<\phi_0 P_1,\Phi_x^m\right>}\,\dt+\int_0^T \abs{\left<P_2,\Phi^m\right>}\,\dt\notag\\
\leq &C(K_0)\Big(\norm{\phi_0 w_x}_{L^2_t(L^2)}+\sum_{i=1}^2\norm{P_i}_{L^2_t(L^2)}\Big)\norm{\Phi^m}_{L^2_t(H^1_{\phi_0^2})},\notag
\end{align}
for which we can use Lemma \ref{Banach} to extend $\phi_0^2w_t$ from  a functional defined on $H^1_0([0,T];H^1)$ to a functional defined on $L^2([0,T];H^1_{\phi_0^2})$, and obtain that
\begin{equation}\label{H-1}
\normf{\phi_0^2 w_t}_{L^2_t(H^{-1}_{\phi_0^2})}\leq C(K_0) \Big(\norm{\phi_0 w_x}_{L^2_t(L^2)}+\sum_{i=1}^2\norm{P_i}_{L^2_t(L^2)}\Big)\leq C(K_0).
\end{equation}

Here, according to the conditions of Lemma \ref{Banach}, one still has to check that $\left\{\Phi^m\right\}_{m\in\NN^*}$ is dense in $L^2([0,T];H^1_{\phi_0^2})$ to ensure the uniqueness of the extension. Indeed, it suffices to prove the density of  $\mathrm{span}\{e_i\}_{i=1}^{\infty}$ in $H^1_{\phi_0^2}$. First, one deduces from Lemma \ref{W-space} that, for any given $f\in H^1_{\phi_0^2}$, there exists a sequence $\{f^\delta\}_{\delta>0}\subset C^\infty(\bar I)$ satisfying
\begin{equation*}
f^\delta\to f \ \ \text{in }H^1_{\phi_0^2} \ \ \text{as }\delta\to 0.    
\end{equation*}
Next, for every $g\in H^1$, according to Lemma \ref{hilbert},
$\sum_{j=1}^m\langle g,e_j\rangle e_j$ converges to $g$ in $H^1$, and hence in $H^1_{\phi_0^2}$ as $m\to\infty$. Setting $f^{m,\delta}:=\sum_{j=1}^m\langle f^\delta,e_j\rangle e_j$, and choosing first $\delta$ sufficiently small and then $m$ sufficiently large depending on $\delta$, we obtain functions
in $\mathrm{span}\{e_i\}_{i=1}^{\infty}$ converging to $f$ in $H^1_{\phi_0^2}$, which shows the claim.

Analogously, taking  the limit as $m\to \infty$ in \eqref{3,,20}, according to  \eqref{H-1}, one has 
\begin{equation}\label{3..133}
\begin{aligned}
&\int_0^T\left<\phi_0^2 w_t,\Phi\right>_{H^{-1}_{\phi_0^2}\times H^{1}_{\phi_0^2}}\,\dt+\int_0^T\left<\frac{\phi_0^2 w_x}{\bar\eta_x^2},\Phi_x\right>\,\dt\\
=&K_0\int_0^T \left<\frac{\phi_0(\phi_0)_x w_x}{\bar\eta_x^2},\Phi\right>\,\dt+ \int_0^T \left<\phi_0 P_1,\Phi_x\right>\,\dt+\int_0^T \left<P_2,\Phi\right>\,\dt,  
\end{aligned} 
\end{equation}
for all $\Phi\in L^2([0,T];H^1_{\phi_0^2})$. Finally, choosing $\Phi(t,x)=\varphi(x)\in H^1_{\phi_0^2}$ and applying $\partial_t$ to both sides of \eqref{3..133}, one can show that the weak formulation \eqref{weak.F.} holds.

\underline{\textbf{Step 4: Uniqueness and time continuity.}}
Since $w\in L^2([0,T];H^1_{\phi_0^2})$, it follows from \eqref{H-1} and Lemma \ref{Aubin} that
\begin{equation}\label{C-w}
\phi_0 w\in C([0,T];L^2).
\end{equation}

It remains to show that $w(0,x)=w_0$ for a.e. $x\in I$. On the one hand, thanks to \eqref{weak.F.} and \eqref{C-w}, for any $\Phi\in C_c^1([0,T);H^1_{\phi_0^2})$, it holds that
\begin{equation}\label{3...125}
\begin{aligned}
&-\int_0^T \left<\Phi_t,\phi_0^2 w\right>\,\dt+\int_0^T \left<\frac{\phi_0^2w_x}{\bar\eta_x^2},\Phi_x\right>\,\dt-K_0\int_0^T \left<\frac{\phi_0(\phi_0)_x w_x}{\bar\eta_x^2},\Phi\right>\,\dt\\
=&\left<\phi_0^2 w(0),\Phi(0)\right>+\int_0^T \left<\phi_0 P_1,\Phi_x\right>\,\dt+\int_0^T \left<P_2,\Phi\right>\,\dt.
\end{aligned}
\end{equation}
On the other hand, choosing $\Phi^m(t,x)=\sum_{j=1}^m \xi_j(t) e_j$, $\xi(t)\in C^\infty_c[0,T)$ satisfying $\Phi^m\to \Phi$ in $C^1_c([0,T);H^1_{\phi_0^2})$, as $m\to \infty$, one gets from \eqref{equ3320} that for all $m\leq n$,
\begin{equation*}
\begin{aligned}
&-\int_0^T \big<\Phi^m_t,\phi_0^2 X^{n,\delta}\big>\,\dt+\int_0^T \bigg<\frac{\phi_0^2 X^{n,\delta}_x}{\bar\eta_x^2},\Phi^m_x\bigg>\,\dt-K_0\int_0^T \bigg<\frac{\phi_0(\phi_0)_x X^{n,\delta}_x}{\bar\eta_x^2},\Phi^m\bigg>\,\dt\\
=&\big<\phi_0^2 X^{n,\delta} (0),\Phi^m(0)\big>+\int_0^T \left<\phi_0 P_1,\Phi^m_x\right>\,\dt+\int_0^T \left<P_2,\Phi^m\right>\,\dt,
\end{aligned}
\end{equation*}
which, by first letting $n,\delta^{-1}\to\infty$ along the subsequence in \eqref{3..124} for each fixed $m$ and then letting $m\to\infty$, leads to
\begin{equation}\label{3...126}
\begin{aligned}
&-\int_0^T \left<\Phi_t,\phi_0^2 w\right>\,\dt+\int_0^T \left<\frac{\phi_0^2w_x}{\bar\eta_x^2},\Phi_x\right>\,\dt-K_0\int_0^T \left<\frac{\phi_0(\phi_0)_x w_x}{\bar\eta_x^2},\Phi\right>\,\dt\\
=&\left<\phi_0^2 w_0,\Phi(0)\right> +\int_0^T \left<\phi_0 P_1,\Phi_x\right>\,\dt+\int_0^T \left<P_2,\Phi\right>\,\dt.
\end{aligned}
\end{equation}
It follows from  \eqref{3...125}-\eqref{3...126} that  $w(0,x)=w_0$ for a.e. $x\in I$. Finally, setting $w_0=0$, $\varphi=w$ and $P_1=P_2=0$, one gets easily that $w=0$, which implies the uniqueness.

The proof of Proposition \ref{prop1} is completed.
\end{proof}

\begin{Remark}
\eqref{3..133} and \eqref{3...125} are equivalent to  \eqref{weak.F.} (see  Chapter 7 in \cite{evans} for details).
\end{Remark}

\subsubsection{Case $\frac{1}{3}<\alpha\leq 1$}\label{subsubsection3.2.2}
In this case, we  consider the following initial boundary value problem:
\begin{equation}\label{galerkin-w'}
\begin{cases}
\displaystyle \phi_0^\frac{1}{\alpha} w_t - \bigg(\frac{\phi_0^\frac{1}{\alpha} w_x}{\bar\eta_x^2}\bigg)_x = - \big(\phi_0^\frac{1}{2\alpha} P_3\big)_x&\text{in }(0,T]\times I,\\
w =w_0&\text{on }\{t=0\}\times I,
\end{cases}
\end{equation}
where $P_3\in L^2([0,T];L^2)$ and $w_0\in L^2_{\phi_0^{1/\alpha}}$,
and the definition of weak solutions of this problem can be given as follows.

\begin{Definition}
A function $w(t,x)$ is said to be a weak solution in $[0,T]\times I$, to  \eqref{galerkin-w'}, if
\begin{enumerate}
\item[$\mathrm{i)}$] $\phi_0^\frac{1}{2\alpha} w\in C([0,T];L^2), \quad \phi_0^\frac{1}{2\alpha}w_x\in L^2([0,T];L^2),\quad \phi_0^\frac{1}{\alpha}w_t\in L^2([0,T];H^{-1}_{\phi_0^{1/\alpha}})$;
\item[$\mathrm{ii)}$]
the following equation holds for all $\varphi\in H^1_{\phi_0^{1/\alpha}}$ and a.e. time $0<t\leq T$, \begin{equation*}
\begin{aligned}
&\big<\phi_0^\frac{1}{\alpha} w_t, \varphi\big>_{H^{-1}_{\phi_0^{1/\alpha}}\times H^1_{\phi_0^{1/\alpha}}}+\bigg< \frac{\phi_0^\frac{1}{\alpha}w_{x}}{\bar\eta_x^2},\varphi_x\bigg>=\big<\phi_0^\frac{1}{2\alpha}P_3, \varphi_x\big>;
\end{aligned}
\end{equation*}
\item[$\mathrm{iii)}$] $w(0,x)=w_0(x)$  for a.e. $x\in I$. 
\end{enumerate}
\end{Definition}

Observe that one can set $K_0=P_2=0$ and replace $(\phi_0,P_1)\mapsto (\phi_0^\frac{1}{2\alpha},P_3)$ in \eqref{galerkin-w} to obtain \eqref{galerkin-w'}. Compared with $\eqref{galerkin-w}_1$,  $\eqref{galerkin-w'}_1$ takes a simpler form since such  terms as $K_0\bar\eta_x^{-2}\phi_0(\phi_0)_x w_x$ and $P_2$ in $\eqref{galerkin-w}_1$ do not appear in $\eqref{galerkin-w'}_1$. Thus, following the proof of Proposition \ref{prop1} with $K_0=P_2=0$ and $(\phi_0,P_1)$ replaced by $(\phi_0^\frac{1}{2\alpha},P_3)$, 
one can get the following result.

\begin{Proposition}\label{prop042}
There exists a unique weak solution $w$ to  problem \eqref{galerkin-w'}, satisfying the following estimate: for all $0<T\leq \widetilde T$,
\begin{equation*}
\begin{aligned}
&\sup_{t\in[0,T]} \absb{\phi_0^\frac{1}{2\alpha} w}_2^2+\int_0^T \Big(\absb{\phi_0^\frac{1}{2\alpha} w_x}_2^2+\big\|\phi_0^\frac{1}{\alpha} w_t\big\|_{-1,\phi_0^{1/\alpha}}^2\Big)\,\dt
\leq  C\absb{\phi_0^\frac{1}{2\alpha} w_0}_2^2+ C\int_0^T \abs{P_3}_2^2\,\dt.
\end{aligned}
\end{equation*}
\end{Proposition}

\subsection{Proof of Lemma \ref{existence-linearize} when  \texorpdfstring{$0<\alpha\leq \frac{1}{3}$}{}}\label{subsection3.3}

Now, we start to prove Lemma \ref{existence-linearize}. The proof for the case $0<\alpha\leq \frac{1}{3}$ is given by the following several steps.

\begin{proof}
\underline{\textbf{Step 1: Tangential estimate $\phi_0U\in C([0,T];L^2)$.}} Turn back to \eqref{lp''} and let 
\begin{equation*}
\begin{split}
\quad w_0^{(0)}:=&u_0,\ \  K_0:=\frac{1}{\alpha}-2>0,\ \ 
P_1^{(0)}:=\phi_0^\frac{1+\alpha}{\alpha}\bar\eta_x^{-2},\ \  P_2^{(0)}:=\Big(2-\frac{1}{\alpha}\Big)\phi_0^\frac{1+\alpha}{\alpha}(\phi_0)_x\bar\eta_x^{-2}.
\end{split}
\end{equation*}
It is clear that $P_1^{(0)},P_2^{(0)}\in L^2([0,T];L^2)$ and $w_0^{(0)}\in L^2_{\phi_0^2}$. Then it follows from Proposition \ref{prop1} that there exists a unique weak solution $w^{(0)}=U$ satisfying
\begin{equation}\label{3..142}
\phi_0 U\in C([0,T];L^2),\quad \phi_0 U_x\in L^2([0,T];L^2),\quad \phi_0^2 U_t\in L^2([0,T];H^{-1}_{\phi_0^2}).
\end{equation}

\underline{\textbf{Step 2: Tangential estimates $\phi_0U_x,\phi_0U_t\in C([0,T];L^2)$.}}
Applying $\partial_t$ to both sides of $\eqref{lp''}_1$ formally, one has
\begin{equation}\label{33151}
\phi_0^2 U_{tt}-\left(\frac{\phi_0^2 U_{tx}}{\bar\eta_x^2}\right)_x=\Big(\frac{1}{\alpha}-2\Big)\frac{\phi_0(\phi_0)_x U_{tx}}{\bar\eta_x^2}-\big(\phi_0P_1^{(1)}\big)_x+P_2^{(1)},
\end{equation}
where
\begin{equation*}
\begin{aligned}
P_1^{(1)}&:=\frac{2\phi_0 \bar U_x U_{x}}{\bar\eta_x^3}-\frac{2\phi_0^\frac{1+\alpha}{\alpha} \bar U_x}{\bar\eta_x^3},\\
P_2^{(1)}&:=\Big(\frac{1}{\alpha}-2\Big)\bigg(\frac{2\phi_0^\frac{1+\alpha}{\alpha}(\phi_0)_x\bar U_x}{\bar\eta_x^3}-\frac{2\phi_0(\phi_0)_x\bar U_x U_x}{\bar\eta_x^3}\bigg).
\end{aligned}
\end{equation*}

Then regard \eqref{33151} as the equation of $w^{(1)}:=U_t$ and consider the problem 
\begin{equation}\label{33152}
\begin{cases}
\displaystyle \quad\phi_0^2 w^{(1)}_t-\bigg(\frac{\phi_0^2 w^{(1)}_{x}}{\bar\eta_x^2}\bigg)_x\\[8pt]
\displaystyle
=\Big(\frac{1}{\alpha}-2\Big)\frac{\phi_0(\phi_0)_x w^{(1)}_{x}}{\bar\eta_x^2}-\big(\phi_0P_1^{(1)}\big)_x+P_2^{(1)}&\text{in }(0,T]\times I,\\[4pt]
\quad w^{(1)} = U_t(0,x)&\text{on }\{t=0\}\times I.
\end{cases}
\end{equation}
Similarly, it follows from \eqref{3..142} that $w^{(1)}(0,x)\in L^2_{\phi_0^2}$ and $P_1^{(1)},P_2^{(1)}\in L^2([0,T];L^2)$. Consequently, by Proposition \ref{prop1},  the weak solution $w^{(1)}$ to  \eqref{33152} exists uniquely satisfying
\begin{equation}\label{3.142'}
\phi_0 w^{(1)}\in C([0,T];L^2),\quad \phi_0 w^{(1)}_x\in L^2([0,T];L^2),\quad \phi_0^2 w^{(1)}_t\in L^2([0,T];H^{-1}_{\phi_0^2}).
\end{equation}
Now, we check that $w^{(1)}=U_t$. Indeed, since $U$ and $w^{(1)}$ are weak solutions to the problems \eqref{lp''} and \eqref{33152}, respectively, one gets from \eqref{weak.F.} that for all $\varphi\in H^1_{\phi_0^2}$ and a.e. time $0<t\leq T$,
\begin{equation}\label{equ1}
\begin{aligned}
&\left<\phi_0^2 U_t,\varphi\right>_{H^{-1}_{\phi_0^2}\times H^{1}_{\phi_0^2}}+\left<\frac{\phi_0^2 U_x}{\bar\eta_x^2},\varphi_x\right>\\
=&\Big(\frac{1}{\alpha}-2\Big)\left<\frac{\phi_0(\phi_0)_x U_x}{\bar\eta_x^2},\varphi\right>+ \big<\phi_0 P_1^{(0)},\varphi_x\big>+\big<P_2^{(0)},\varphi\big>,
\end{aligned} 
\end{equation}  
and
\begin{equation}\label{equ1'}
\begin{aligned}
&\big<\phi_0^2 w^{(1)}_t,\varphi\big>_{H^{-1}_{\phi_0^2}\times H^{1}_{\phi_0^2}} + \bigg<\frac{\phi_0^2 w^{(1)}_x}{\bar\eta_x^2},\varphi_x\bigg> \\
=&\Big(\frac{1}{\alpha}-2\Big) \bigg<\frac{\phi_0(\phi_0)_x w^{(1)}_x}{\bar\eta_x^2},\varphi\bigg> + \big<\phi_0 P_1^{(1)},\varphi_x\big>+ \big<P_2^{(1)},\varphi\big>.  
\end{aligned} 
\end{equation}
Next, define 
$$W(t,x):=\int_0^t w^{(1)}(s,x)\,\ds+ u_0(x) \ \ \text{and}\ \ Y:=W-U.$$
It suffices to show that $Y=0$ pointwise. To get this, substituting $W$ into \eqref{equ1'} and integrating the resulting equality over $[0,t]$ for $0< t\leq T$, then according to the compatibility condition $\eqref{1..16}_1$, one has
\begin{align*}
&\big<\phi_0^2 W_t,\varphi\big>+\left<\frac{\phi_0^2 W_x}{\bar\eta_x^2},\varphi_x\right>-\Big(\frac{1}{\alpha}-2\Big)\left<\frac{\phi_0(\phi_0)_x W_x}{\bar\eta_x^2},\varphi\right>\\
=&-\int_0^t \left<\frac{2\phi_0^2 \bar U_x Y_x}{\bar\eta_x^3},\varphi_x\right>\,\ds+\Big(\frac{1}{\alpha}-2\Big)\int_0^t \left<\frac{2\phi_0(\phi_0)_x \bar U_x Y_x}{\bar\eta_x^3},\varphi\right>\,\ds\\
&+\big<\phi_0P_1^{(0)},\varphi_x\big>+\big<P_2^{(0)},\varphi\big>, 
\end{align*}
which, together with \eqref{equ1}, leads to
\begin{equation}\label{equ3}
\begin{aligned}
&\left<\phi_0^2 Y_t,\varphi\right>_{H^{-1}_{\phi_0^2}\times H^{1}_{\phi_0^2}}+\left<\frac{\phi_0^2 Y_x}{\bar\eta_x^2},\varphi_x\right>-\Big(\frac{1}{\alpha}-2\Big)\left<\frac{\phi_0(\phi_0)_x Y_x}{\bar\eta_x^2},\varphi\right>\\
=&-\int_0^t \left<\frac{2\phi_0^2 \bar U_x Y_x}{\bar\eta_x^3},\varphi_x\right>\,\ds +\Big(\frac{1}{\alpha}-2\Big)\int_0^t \left<\frac{2\phi_0(\phi_0)_x \bar U_x Y_x}{\bar\eta_x^3},\varphi\right>\,\ds. 
\end{aligned}
\end{equation}
Since $Y\in L^2([0,T];H^1_{\phi_0^2})$, one may set $\varphi=Y$ in \eqref{equ3}. Then it follows from the same calculations in \eqref{equ311}, Lemma \ref{hardy-inequality} and Young's inequality that
\begin{equation*}
\begin{aligned}
\frac{\mathrm{d}}{\dt}\absf{\phi_0 Y}_2^2+\absf{\phi_0 Y_x}_2^2\leq C\absf{\phi_0 Y}_2^2+C(\absf{\bar U_x}_\infty)\int_0^t \big(\absf{\phi_0Y}_2^2+\absf{\phi_0Y_x}_2^2\big)\,\ds,
\end{aligned}
\end{equation*}
which, along with the strong continuity of $Y$ at $t=0$, $Y|_{t=0}=0$ and Gr\"onwall's inequality, implies that for all $0< T\leq \widetilde T$,
\begin{equation*}
\begin{aligned}
&\sup_{t\in[0,T]} \abs{\phi_0 Y}_2^2+\int_0^T \abs{\phi_0 Y_x}_2^2\,\dt\leq C(\absf{\bar U_x}_\infty)Te^{CT}\Big(T\sup_{s\in[0,T]} \abs{\phi_0 Y}_2^2+\int_0^T \abs{\phi_0 Y_x}_2^2\,\dt\Big).
\end{aligned}
\end{equation*}
Note that, by choosing $0<T_0<1$ such that $C(\absf{\bar U_x}_\infty)T_0e^{CT_0}=\frac{1}{2}$ and using Lemma \ref{hardy-inequality}, one can obtain from the estimates above that $Y= 0$ a.e. on $(0,T_0)\times I$. Since $T_0$ depends only on $\alpha$ and $(\phi_0,\bar U)$, one can extend $T_0$ to $\widetilde T$ via the analogous arguments in \eqref{3...18}-\eqref{3...18'}, which yields $Y= 0$ a.e. on $(0,\widetilde T)\times I$.

As a consequence, $U_t$ satisfies \eqref{3.142'}, that is,
\begin{equation}\label{33153}
\phi_0 U_t\in C([0,T];L^2),\quad \phi_0 U_{tx}\in L^2([0,T];L^2),\quad \phi_0^2 U_{tt}\in L^2([0,T];H^{-1}_{\phi_0^2}),
\end{equation}
which, along  with \eqref{3..142}, yields that 
\begin{equation}\label{equ33.38}
\phi_0 U_x\in C([0,T];L^2).
\end{equation}

\underline{\textbf{Step 3: Boundary condition of $U$.}} By \eqref{lp''} and  \eqref{weak.F.}, one has the following lemma.
\begin{Lemma}\label{Lemma-point}
It holds that
\begin{equation}\label{div-Uxx}
\left(\frac{\phi_0^2 U_x}{\bar\eta_x^2}\right)_x\in C([0,T];L^2).
\end{equation}
Furthermore, the equation $\eqref{lp''}_1$ holds for a.e. $(t,x)\in (0,T)\times I$, and $U$ satisfies 
\begin{equation}\label{equ33.41}
\phi_0^2U_x=0 \ \ \text{for }x\in\Gamma.
\end{equation}
\end{Lemma}
\begin{proof}
Indeed,  it follows from \eqref{weak.F.} with $U$ that for all $\varphi\in C_c^\infty$,
\begin{align*}
&\absB{-\left<\frac{\phi_0^2 U_x}{\bar\eta_x^2},\varphi_x\right>} \\
=&\Big|\left<\phi_0^2 U_t,\varphi\right>-\Big(\frac{1}{\alpha}-2\Big)\left<\frac{\phi_0(\phi_0)_x U_x}{\bar\eta_x^2}, \varphi\right>-\big<\big(\phi_0 P_1^{(0)}\big)_x,\varphi\big>+\big<P_2^{(0)},\varphi\big>\Big|\\
\leq &C\big(1+\norm{\phi_0 U_t}_{C_t(L^2)}+\norm{\phi_0 U_x}_{C_t(L^2)}\big)\abs{\varphi}_{2}\leq C\abs{\varphi}_{2},
\end{align*}
which means that $\phi_0^2 \bar\eta_x^{-2} U_x$ admits the weak derivative $\left( \phi_0^2\bar\eta_x^{-2} U_x\right)_x\in L^2$ for a.e. $t\in (0,T)$. In addition, due to the time-independent bound in the right hand side of the above inequality and the time continuity of $(\phi_0 U_x,\phi_0U_t,\bar\eta)$, one obtains \eqref{div-Uxx}. Besides, it follows from \eqref{33153}-\eqref{div-Uxx} that $\eqref{lp''}_1$ holds for a.e. $(t,x)\in (0,T)\times I$.

It remains to prove \eqref{equ33.41}. On the one hand, thanks to the weak formulation \eqref{weak.F.}, it holds that for all $\varphi\in C^\infty(\bar I)$,
\begin{equation}\label{WF'}
\begin{aligned}
\langle \phi_0^2 U_t, \varphi\rangle+\left<\frac{\phi_0^2 U_{x}}{\bar\eta_x^2},\varphi_x\right>
=\Big(\frac{1}{\alpha}-2\Big)\left<\frac{\phi_0(\phi_0)_xU_x}{\bar\eta_x^2},\varphi\right>+\big<\phi_0 P_1^{(0)}, \varphi_x\big>+\big<P_2^{(0)}, \varphi\big>;
\end{aligned}
\end{equation}
on the other hand, it follows from Lemma \ref{sobolev-embedding} and \eqref{div-Uxx} that $\phi_0^2 U_x\in C([0,T]\times \bar I)$. Then  multiplying both sides of $\eqref{lp''}_1$ by such $\varphi$ and integrating the resulting equality over $I$ lead to
\begin{equation}\label{WF''}
\begin{aligned}
&\left<\phi_0^2 U_t, \varphi\right>+\left<\frac{\phi_0^2 U_{x}}{\bar\eta_x^2},\varphi_x\right>-\frac{\phi_0^2U_x}{\bar\eta_x^2}\varphi\Big|_{x=0}^{x=1}\\
=&\Big(\frac{1}{\alpha}-2\Big)\left<\frac{\phi_0(\phi_0)_xU_x}{\bar\eta_x^2},\varphi\right>+\big<\phi_0 P_1^{(0)}, \varphi_x\big>+\big<P_2^{(0)}, \varphi\big>.
\end{aligned}
\end{equation}

Comparing with \eqref{WF'}-\eqref{WF''} and using $\frac{1}{2}\leq \bar\eta_x\leq\frac{3}{2}$ for $(t,x)\in [0,\widetilde T]\times \bar I$, one gets
$\phi_0^2U_x\varphi\big|_{x=0}^{x=1}=0$ for all $\varphi\in C^\infty(\bar I)$,
which yields that $\phi_0^2 U_x|_{\Gamma}=0$. 

The proof of Lemma \ref{Lemma-point} is completed.
\end{proof}

\underline{\textbf{Step 4: Tangential estimates  $\phi_0U_{tx},\phi_0U_{tt}\in C([0,T];L^2)$.}}
We start with the following claim:
\begin{equation}\label{L2Linfty}
U_x\in L^2([0,T];L^\infty).
\end{equation}
Indeed, on the one hand, based on $\eqref{lp''}_1$ and Lemmas \ref{Leibniz}-\ref{qiudao}, one can reformulate $\eqref{lp''}_1$ by multiplying both of its sides by $\phi_0^{\frac{1}{\alpha}-2}$ to deduce that
\begin{equation}\label{0428}
\phi_0^\frac{1}{\alpha} U_t-\bigg(\frac{\phi_0^\frac{1}{\alpha} U_{x}}{\bar\eta_x^2}\bigg)_x+\bigg(\frac{\phi_0^\frac{2}{\alpha}}{\bar\eta_x^2}\bigg)_x=0;
\end{equation}
on the other hand, since $\phi_0^2U_x\in C(\bar I)$, it follows from integrating \eqref{0428} over $[0,x]$, $0<x\leq \frac{1}{2}$, \eqref{33153}, Lemma \ref{hardy-inequality}, $\phi_0\sim d(x)$, $\phi_0^2 U_x|_{x\in\Gamma}=0$ and H\"older's inequality that
\begin{equation}\label{3...136}
\begin{aligned}
\phi_0^\frac{1}{\alpha} U_x(t,x)&= \phi_0^\frac{2}{\alpha}+\bar\eta_x^2 \int_0^x \phi_0^\frac{1}{\alpha} U_t\,\mathrm{d}z\\
\implies \abs{U_x(t,x)}&\leq C x^\frac{1}{\alpha}+ C x^\frac{1}{2}\abs{U_t}_2\\
&\leq C\phi_0^\frac{1}{\alpha}(x)+ C\phi_0^\frac{1}{2}(x)\left(\abs{\phi_0 U_t}_2+\abs{\phi_0 U_{tx}}_2\right).
\end{aligned}
\end{equation}
For $\frac{1}{2}<x\leq 1$, integrating  \eqref{0428} over $[x,1]$ yields that $\eqref{3...136}_2$ still holds. Thus, taking the square of $\eqref{3...136}_2$ and integrating the resulting inequality over $[0,T]$ shows the claim \eqref{L2Linfty}.

Now, we continue to improve the tangential regularities. Applying $\partial_t^2$ to both sides of $\eqref{lp''}_1$ yields formally that
\begin{equation}\label{33154}
\phi_0^2 \partial_t^3 U-\left(\frac{\phi_0^2 \partial_t^2 U_{x}}{\bar\eta_x^2}\right)_x=\Big(\frac{1}{\alpha}-2\Big)\frac{\phi_0(\phi_0)_x \partial_t^2 U_{x}}{\bar\eta_x^2}-\big(\phi_0P_1^{(2)}\big)_x+ P_2^{(2)},
\end{equation}
where
\begin{align}
P_1^{(2)}&:=\frac{2\phi_0 \bar U_{tx} U_{x}}{\bar\eta_x^3}+\frac{4\phi_0 \bar U_x U_{tx}}{\bar\eta_x^3}-\frac{6\phi_0 \bar U_x^2 U_{x}}{\bar\eta_x^4} -\frac{2\phi_0^\frac{1+\alpha}{\alpha} \bar U_{tx}}{\bar\eta_x^3}+\frac{6\phi_0^\frac{1+\alpha}{\alpha} \bar U_x^2}{\bar\eta_x^4},\notag\\
P_2^{(2)}&:=\Big(\frac{1}{\alpha}-2\Big) \bigg(\frac{2\phi_0^\frac{1+\alpha}{\alpha}(\phi_0)_x\bar U_{tx}}{\bar\eta_x^3}-\frac{6\phi_0^\frac{1+\alpha}{\alpha}(\phi_0)_x\bar U_x^2}{\bar\eta_x^4}+\frac{6\phi_0(\phi_0)_x\bar U_x^2 U_x}{\bar\eta_x^4}\bigg.\label{P1P2}\\
&\quad \left.-\frac{2\phi_0(\phi_0)_x\bar U_{tx} U_x}{\bar\eta_x^3}-\frac{4\phi_0(\phi_0)_x\bar U_x U_{tx}}{\bar\eta_x^3}\right).\notag
\end{align}

Then regard \eqref{33154} as the equation of $w^{(2)}:=U_{tt}$ and  consider the  problem 
\begin{equation}\label{33155}
\begin{cases}
\displaystyle \quad \phi_0^2 w^{(2)}_t-\bigg(\frac{\phi_0^2 w^{(2)}_{x}}{\bar\eta_x^2}\bigg)_x\\
\displaystyle=\Big(\frac{1}{\alpha}-2\Big)\frac{\phi_0(\phi_0)_x w^{(2)}_{x}}{\bar\eta_x^2}-\big(\phi_0P_1^{(2)}\big)_x+ P_2^{(2)}&\text{in }(0,T]\times I,\\
\quad w^{(2)} = U_{tt}(0,x)&\text{on }\{t=0\}\times I.
\end{cases}
\end{equation}
One can check from \eqref{33153} and \eqref{L2Linfty} that $P_1^{(2)},P_2^{(2)}\in L^2([0,T];L^2)$, $w^{(2)}(0,x)\in L^2_{\phi_0^2}$. Consequently, it follows from Proposition \ref{prop1} that there exists a unique weak solution $w^{(2)}$ to  \eqref{33155}, and based on the analogous arguments in \textbf{Step 2}, one has $w^{(2)}=U_{tt}$. Thus, 
\begin{equation}\label{33156}
\phi_0 U_{tt}\in C([0,T];L^2),\quad \phi_0 \partial_t^2U_{x}\in L^2([0,T];L^2).
\end{equation}
Certainly, it follows from \eqref{33153} and \eqref{33156} that
\begin{equation}\label{331577}
\phi_0 U_{tx}\in C([0,T];L^2),
\end{equation}
and one can deduce that $\phi_0^2U_{tx}|_{x\in\Gamma}=0$ by the similar discussions in Lemma \ref{Lemma-point}. 

Then it follows from \eqref{3..142}, \eqref{33153}-\eqref{equ33.38} and \eqref{33156}-\eqref{331577} that 
\begin{equation}\label{TETE}
\phi_0 \partial_t^j U\in C([0,T];L^2),\quad j=0,1,2;\quad 
\phi_0 \partial_t^j U_{x}\in C([0,T];L^2),\quad j=0,1.
\end{equation}

\underline{\textbf{Step 5: Elliptic estimate  $\phi_0 U_{xx}\in L^\infty([0,T];L^2)$.}}
First, it follows from \eqref{L2Linfty}-\eqref{3...136} and \eqref{TETE} that
\begin{equation}\label{0434}
U_x\in L^\infty([0,T]\times \bar I).
\end{equation}

Next, due to Lemma \ref{Leibniz}, multiplying both sides of $\eqref{lp''}_1$ by $\bar\eta_x^2\phi_0^{-1}$ yields 
\begin{equation}\label{0435}
\phi_0 U_{xx}+\frac{1}{\alpha} (\phi_0)_xU_x=\phi_0 \bar\eta_x^2 U_t+\frac{2\phi_0 \bar\eta_{xx}U_x}{\bar\eta_x}+\frac{2}{\alpha}\phi_0^{\frac{1}{\alpha}}(\phi_0)_x- \frac{2\phi_0^{\frac{1}{\alpha}+1}\bar\eta_{xx}}{\bar\eta_x},
\end{equation}
for a.e. $(t,x)\in (0,T)\times I$. Then, it follows from \eqref{TETE} that the right hand side of \eqref{0435} can be bounded in the norm of $L^\infty([0,T];L^2)$. Consequently,
\begin{equation}\label{0436}
\phi_0  U_{xx}+\frac{1}{\alpha} (\phi_0)_xU_x\in L^\infty([0,T];L^2),
\end{equation}
which, together with \eqref{0434}, implies that
\begin{equation}\label{0437}
\phi_0  U_{xx} \in L^\infty([0,T];L^2).
\end{equation}

\underline{\textbf{Step 6: Elliptic estimate $\phi_0\partial_t U_{xx}\in L^\infty([0,T];L^2)$.}}
First, we claim that
\begin{equation}\label{0438}
\phi_0^{\iota}U_{tx}\in L^\infty([0,T];L^2) \ \ \text{for all } \iota>0.
\end{equation}
Indeed, for $0<x\leq \frac{1}{2}$, applying $\partial_t$ to both sides of $\eqref{3...136}_1$, one then gets from \eqref{TETE}, Lemma \ref{hardy-inequality} and H\"older's inequality that
\begin{align}
\phi_0^\frac{1}{\alpha} U_{tx}&= 2\bar\eta_x\bar U_x \int_0^x \phi_0^\frac{1}{\alpha} U_t\,\mathrm{d}z+ \bar\eta_x^2\int_0^x \phi_0^\frac{1}{\alpha} U_{tt}\,\mathrm{d}z\notag\\
\implies \absb{\phi_0^\frac{1}{\alpha} U_{tx}(t,x)}&\leq C \Big(\int_0^x z^\frac{1}{\alpha}\abs{U_t}\,\mathrm{d}z+ \int_0^x z^\frac{1}{\alpha}\abs{U_{tt}}\,\mathrm{d}z\Big)\notag\\
&\leq C x^\frac{2-\alpha}{2\alpha}\big(x\abs{U_t}_2 + \abs{\phi_0 U_{tt}}_2\big)\label{00439}\\
&\leq C\phi_0^\frac{2-\alpha}{2\alpha}\big(\phi_0(\abs{\phi_0 U_t}_2+\abs{\phi_0 U_{tx}}_2)+1\big)
\leq C\big(\phi_0^\frac{2+\alpha}{2\alpha}+\phi_0^\frac{2-\alpha}{2\alpha}\big).\notag
\end{align}
Similar arguments can be applied to $[x,1]$, $\frac{1}{2}<x\leq 1$. Then, multiplying both sides of $\eqref{00439}_2$ by $\phi_0^{\iota-\frac{1}{\alpha}}$ and taking the $L^\infty([0,T];L^2)$ norm of the resulting inequality give \eqref{0438}.

Next, since  \eqref{0435} holds pointwise, based on the estimates \eqref{TETE}, the regularities of $(\bar\eta,\bar U)$ and Lemma \ref{qiudao}, one can apply $\partial_t$ to both sides of \eqref{0435} to obtain that 
\begin{equation}\label{utxx-phi0}
\begin{aligned}
\phi_0 \partial_tU_{xx}+\frac{1}{\alpha} (\phi_0)_x U_{tx}
=&\phi_0 \bar\eta_x^2 U_{tt}+2\phi_0 \bar\eta_x \bar U_x U_{t}-\frac{2\phi_0^{1+\frac{1}{\alpha}} \bar U_{xx}}{\bar\eta_x}+\frac{2\phi_0^{1+\frac{1}{\alpha}}\bar\eta_{xx}\bar U_x}{\bar\eta_x^2}\\
&+\frac{2\phi_0 \bar U_{xx} U_x}{\bar\eta_x}-\frac{2\phi_0 \bar\eta_{xx} \bar U_x U_x}{\bar\eta_x^2}+\frac{2\phi_0 \bar\eta_{xx} U_{tx}}{\bar\eta_x}.
\end{aligned}
\end{equation}
Then, it follows from \eqref{TETE}, \eqref{utxx-phi0} and Lemma \ref{hardy-inequality} that 
\begin{equation}\label{0439}
\begin{aligned}
\phi_0  \partial_tU_{xx}+\frac{1}{\alpha} (\phi_0)_x U_{tx} \in L^\infty([0,T];L^2),
\end{aligned}
\end{equation}
which, together with \eqref{0438} and \eqref{0439} multiplied by $\phi_0^\iota$, yields $\phi_0^{1+\iota} \partial_t U_{xx}\in L^\infty([0,T];L^2)$, for all $\iota>0$. Hence, it follows from Proposition \ref{prop2.1} in Appendix \ref{subsection2.2}  that
\begin{equation}\label{0440}
\phi_0 \partial_t U_{xx}\in L^\infty([0,T];L^2).
\end{equation}

\underline{\textbf{Step 7: Elliptic estimate $\phi_0\partial_x^3 U \in L^\infty([0,T];L^2)$.}}
First, it follows from \eqref{TETE}, \eqref{0437} and Lemma \ref{qiudao} that one can apply $\partial_x$ to  \eqref{0435} to obtain
\begin{align}
&\phi_0 \partial_x^3 U +\Big(\frac{1}{\alpha}+1\Big) (\phi_0)_x U_{xx}\notag\\
=&-\frac{1}{\alpha} (\phi_0)_{xx} U_x+\phi_0 \bar\eta_x^2 U_{tx}+2\phi_0 \bar\eta_x\bar\eta_{xx}U_t+ (\phi_0)_x\bar\eta_x^2 U_t\label{0441}\\
&-\frac{2+2\alpha}{\alpha}\frac{\phi_0^{\frac{1}{\alpha}}(\phi_0)_x\bar\eta_{xx}}{\bar\eta_x}-\frac{2\phi_0^{1+\frac{1}{\alpha}}\partial_x^3\bar\eta}{\bar\eta_x}+\frac{2\phi_0^{1+\frac{1}{\alpha}}\bar\eta_{xx}^2}{\bar\eta_x^2}+\frac{2}{\alpha^2}\phi_0^{\frac{1}{\alpha}-1}((\phi_0)_x)^2 \notag\\
&+\frac{2}{\alpha}\phi_0^{\frac{1}{\alpha}}(\phi_0)_{xx} +\frac{2 (\phi_0)_x\bar\eta_{xx}U_x}{\bar\eta_x}+\frac{2\phi_0 \bar\eta_{xx} U_{xx}}{\bar\eta_x}+\frac{2\phi_0 \partial_x^3\bar\eta U_x}{\bar\eta_x}-\frac{2\phi_0 \bar\eta_{xx}^2U_x}{\bar\eta_x^2},\notag
\end{align}
which, along with \eqref{TETE}, \eqref{0437} and Lemma \ref{hardy-inequality}, yields that
\begin{equation}\label{0442}
\phi_0 \partial_x^3U+\Big(\frac{1}{\alpha}+1\Big) (\phi_0)_x U_{xx}\in L^\infty([0,T];L^2).
\end{equation}

Obviously, it follows from $\eqref{0442}\times \phi_0$ and \eqref{0437} that $\phi_0^2\partial_x^3 U\in L^\infty([0,T];L^2)$. Since $2<\frac{1}{2\alpha}+1$, one obtains from Proposition \ref{prop2.1} and \eqref{0437} that
\begin{equation}\label{0443}
\phi_0\partial_x^3 U\in L^\infty([0,T];L^2).
\end{equation}

\underline{\textbf{Step 8: Elliptic estimate $\phi_0\partial_x^4 U \in L^\infty([0,T];L^2)$.}}
 Analogously, based on \eqref{TETE}, \eqref{0437}, \eqref{0443} and Lemma \ref{qiudao}, one can apply $\partial_x$ to both sides of \eqref{0441} to deduce that
\begin{align}
&\phi_0 \partial_x^4 U+\Big(\frac{1}{\alpha}+2\Big) (\phi_0)_x\partial_x^3 U\notag\\
=&\phi_0 \bar\eta_x^2\partial_tU_{xx}+2\left((\phi_0)_x\bar\eta_x+2\phi_0\bar\eta_{xx}\right)\bar\eta_x U_{tx}+2\phi_0 (\bar\eta_{xx}^2+\bar\eta_x\partial_x^3\bar\eta)U_t\notag\\
&+ \left((\phi_0)_{xx}\bar\eta_x+4(\phi_0)_x\bar\eta_{xx}\right)\bar\eta_xU_t -\Big(\frac{2}{\alpha}+1\Big) (\phi_0)_{xx} U_{xx}\notag\\
&-\frac{1}{\alpha} \partial_x^3\phi_0 U_x+2 \Big((\phi_0)_{xx}\bar\eta_{xx}+2(\phi_0)_{x}\partial_x^3\bar\eta -\frac{2(\phi_0)_{x}\bar\eta_{xx}^2}{\bar\eta_x}\Big)\frac{U_x}{\bar\eta_x}\notag\\
&+2\phi_0 \Big(\partial_x^4\bar\eta
-\frac{3\bar\eta_{xx}\partial_x^3\bar\eta}{\bar\eta_x}
+\frac{2\bar\eta_{xx}^3}{\bar\eta_x^2}\Big)\frac{U_x}{\bar\eta_x}\label{0444}\\
&+4 \Big(\phi_0\partial_x^3\bar\eta+(\phi_0)_{x}\bar\eta_{xx}-\frac{\phi_0\bar\eta_{xx}^2}{\bar\eta_x}\Big)\frac{U_{xx}}{\bar\eta_x}+\frac{2\phi_0 \bar\eta_{xx}\partial_x^3U}{\bar\eta_x}\notag\\
&-\frac{2+2\alpha}{\alpha^2}\phi_0^{\frac{1}{\alpha}-1}((\phi_0)_x)^2\frac{\bar\eta_{xx}}{\bar\eta_x}-2\phi_0^{1+\frac{1}{\alpha}}\Big(\frac{\partial_x^4\bar\eta}{\bar\eta_x}-\frac{3\bar\eta_{xx}\partial_x^3\bar\eta}{\bar\eta_x^2}+\frac{2\bar\eta_{xx}^3}{\bar\eta_x^3}\Big)\notag\\
&-\frac{2+2\alpha}{\alpha}\phi_0^{\frac{1}{\alpha}}\Big(\frac{(\phi_0)_{xx}\bar\eta_{xx}}{\bar\eta_x}+\frac{2(\phi_0)_x\partial_x^3\bar\eta}{\bar\eta_x}-\frac{2(\phi_0)_x\bar\eta_{xx}^2}{\bar\eta_x^2}\Big)\notag\\
&+\frac{2(1-\alpha)}{\alpha^3}\phi_0^{\frac{1}{\alpha}-2}((\phi_0)_x)^3+\frac{6}{\alpha^2}\phi_0^{\frac{1}{\alpha}-1}(\phi_0)_x(\phi_0)_{xx}+\frac{2}{\alpha}\phi_0^{\frac{1}{\alpha}}\partial_x^3\phi_0,\notag
\end{align}
which, along with \eqref{TETE}, \eqref{0437}, \eqref{0440}, \eqref{0443}, Lemma \ref{hardy-inequality} and $0<\alpha\leq\frac{1}{3}$, yields that 
\begin{equation}\label{0445}
\phi_0 \partial_x^4 U+\Big(\frac{1}{\alpha}+2\Big) (\phi_0)_x\partial_x^3 U\in L^\infty([0,T];L^2).    
\end{equation}

Similarly, one can get from  $\eqref{0445}\times \phi_0$ and \eqref{0443} that $\phi_0^2\partial_x^4 U\in L^\infty([0,T];L^2)$, which, along with Proposition \ref{prop2.1} and \eqref{0443}, leads to 
\begin{equation}\label{eequ3.66}
\phi_0 \partial_x^4 U\in L^\infty([0,T];L^2).
\end{equation}

In summary, collecting  \eqref{0437}, \eqref{0440}, \eqref{0443} and \eqref{eequ3.66} yields all the elliptic estimates:
\begin{equation}
\phi_0 \partial_x^j U,\,\,\phi_0\partial_t U_{xx}\in L^\infty([0,T];L^2),\quad j=2,3,4.
\end{equation}

\underline{\textbf{Step 9: Time continuity.}} To this end, it suffices to show the time continuity of $\phi_0U_{xx}$. Note that the time continuity of $\phi_0\partial_x^j U$ ($j=3,4$) and $\phi_0\partial_t U_{xx}$ can be proved analogously.

To see this, according to \eqref{TETE}, \eqref{0435}, and the following  regularities of $\bar\eta$,
\begin{equation}\label{regu-eta}
\bar\eta\in C^1([0,T];H^3)\cap C^2([0,T];H^1),\quad \bar\eta_t\in \mathscr{C}([0,T];E),
\end{equation}
one can first show that
\begin{equation}\label{equ357}
\widetilde F(t,x):=\phi_0 U_{xx}+\frac{1}{\alpha}(\phi_0)_x U_x\in C([0,T];L^2).
\end{equation}
Then, it follows from Proposition \ref{prop2.1} that for any $t,t_0\in [0,T]$,
\begin{equation*}
\begin{aligned}
\abs{\phi_0 U_{xx}(t)-\phi_0 U_{xx}(t_0)}_2
\leq  C\absf{\widetilde F(t)-\widetilde F(t_0)}_2+ C\abs{(\phi_0)_{xx}}_\infty \abs{\phi_0 U_{x}(t)-\phi_0 U_{x}(t_0)}_2,
\end{aligned}
\end{equation*}
which, by letting $t\to t_0$, along with \eqref{TETE} and \eqref{equ357}, yields
\begin{equation*}
\abs{\phi_0 U_{xx}(t)-\phi_0 U_{xx}(t_0)}_2\to 0.
\end{equation*}
Then one has $\phi_0 U_{xx}\in C([0,T];L^2)$. 

For $\phi_0\partial_t U_{xx}$ and $\phi_0\partial_x^j U$ ($j=3,4$), it follows from \eqref{0439}, \eqref{0442}, \eqref{0445} and Proposition \ref{prop2.1} that
\begin{equation*}
\begin{aligned}
\phi_0 \partial_tU_{xx}+\frac{1}{\alpha}(\phi_0)_x U_{tx}\in C([0,T];L^2)&\implies \phi_0 \partial_tU_{xx}\in C([0,T];L^2);\\
\phi_0 \partial_x^3U +\Big(\frac{1}{\alpha}+1\Big)(\phi_0)_x U_{xx}\in C([0,T];L^2)&\implies \phi_0 \partial_x^3U \in C([0,T];L^2);\\
\phi_0 \partial_x^4U +\Big(\frac{1}{\alpha}+2\Big)(\phi_0)_x\partial_x^3 U \in C([0,T];L^2)&\implies \phi_0 \partial_x^4 U \in C([0,T];L^2).
\end{aligned}
\end{equation*}

These, together with the tangential estimates \eqref{TETE}, show that $U\in \mathscr{C}([0,T];E)$, which, along with Lemma \ref{hardy-inequality}, yields 
\begin{equation}\label{regu-U}
U\in C([0,T];H^3)\cap C^1([0,T];H^1).
\end{equation}

\underline{\textbf{Step 10: Derivation of Neumann boundary condition.}}
First, multiplying both sides of \eqref{lp''} by  $\phi_0^{-1}$ shows that
\begin{equation}\label{equ37.2}
\phi_0 U_t +\frac{2}{\alpha}\frac{\phi_0^\frac{1}{\alpha}(\phi_0)_x}{\bar\eta_x^2}-\frac{2\phi_0^{1+\frac{1}{\alpha}}\bar\eta_{xx}}{\bar\eta_x^3}=\frac{\phi_0 U_{xx}}{\bar\eta_x^2}+\frac{1}{\alpha}\frac{(\phi_0)_xU_x}{\bar\eta_x^2}-\frac{2\phi_0 \bar\eta_{xx} U_x}{\bar\eta_x^3}.
\end{equation}

Then, according to $\phi_0\in H^3$, \eqref{regu-eta}, \eqref{regu-U} and Lemma \ref{sobolev-embedding}, one has 
\begin{equation}\label{equ37.5}
\phi_0\in C^2(\bar I), \ \ \bar\eta\in C^1([0,T];C^2(\bar I)), \ \ U\in C([0,T];C^2(\bar I))\cap C^1([0,T];C(\bar I)),
\end{equation} 
which, by letting $x\to \Gamma$ in \eqref{equ37.2}, along with $\phi_0|_{x\in\Gamma}=0$ and $\frac{1}{2}\leq \bar\eta_x\leq \frac{3}{2}$, yields that
\begin{equation*}
(\phi_0)_x U_x=0 \ \ \text{for } x\in \Gamma.
\end{equation*}
Since $\phi_0\sim d(x)$ and $(\phi_0)_x|_{\Gamma}\neq 0$, one can obtain from the above equality that $U_x|_{x\in \Gamma}=0$.

Therefore, the proof of Lemma \ref{existence-linearize} $\mathrm{i)}$ is completed. 
\end{proof}

\subsection{Proof of Lemma \ref{existence-linearize} when \texorpdfstring{$\frac{1}{3}<\alpha\leq 1$}{}}\label{subsection3.4}
The proof of Lemma \ref{existence-linearize} when $\frac{1}{3}<\alpha\leq 1$  can be done by following the analogous arguments used for the case $0<\alpha \leq \frac{1}{3}$. 

\begin{proof}
\underline{\textbf{Step 1: Tangential estimates of $U$.}} Compare \eqref{lp} with \eqref{galerkin-w'} and set $P_3$ in \eqref{galerkin-w'} as
$
P_3:= \phi_0^\frac{3}{2\alpha}\bar\eta_x^{-2}$.
Then, by   Proposition \ref{prop042}, the proof of this part can be done by exactly following that of \textbf{Step 1}-\textbf{Step 4} in \S \ref{subsection3.3}. After the repetitive calculation, one can  deduce that  \eqref{lp} admits a unique weak solution $U$ satisfying   the tangential estimates
\begin{equation}\label{TETE1}
\begin{aligned}
&\phi_0^\frac{1}{2\alpha} \partial_t^j U\in C([0,T];L^2),\quad j=0,1,2;\quad \phi_0^\frac{1}{2\alpha} \partial_t^j U_{x}\in C([0,T];L^2),\quad j=0,1.
\end{aligned}
\end{equation}
Moreover, $U$ satisfies the equation $\eqref{lp}_1$  for a.e. $(t,x)\in (0,T)\times I$, and 
\begin{equation}
\phi_0^\frac{1}{\alpha}U_x=\phi_0^\frac{1}{\alpha}U_{tx}=0 \ \ \text{for }x\in \Gamma.
\end{equation}

\underline{\textbf{Step 2: Elliptic estimate  $\phi_0^{\frac{3}{2}-\varepsilon_0}U_{xx}$.}}
First, following the arguments in \eqref{3...136} and integrating $\eqref{lp}_1$ over $[0,x]$ for $x\in \left(0,\frac{1}{2}\right]$ (or $[x,1]$ for $x\in \left(\frac{1}{2},1\right]$), one can deduce from \eqref{TETE1}, Lemma \ref{hardy-inequality} and H\"older's inequality that, for $\iota>0$,
\begin{equation}\label{esti-xx}
\begin{aligned}
\abs{U_x(t,x)}&\leq C \phi_0^\frac{1}{\alpha}+ C\phi_0^{\frac{3\alpha-1}{2\alpha}-\frac{\iota}{2}}\absb{\phi_0^{\frac{1}{2\alpha}-1+\frac{\iota}{2}}U_t}_2\\
&\leq C \phi_0^\frac{1}{\alpha}+ C(\iota)\phi_0^{\frac{3\alpha-1}{2\alpha}-\frac{\iota}{2}}\big(\absb{\phi_0^\frac{1}{2\alpha} U_t}_2+\absb{\phi_0^\frac{1}{2\alpha} U_{tx}}_2\big)
\leq C \big(\phi_0^\frac{1}{\alpha}+ \phi_0^{\frac{3\alpha-1}{2\alpha}-\frac{\iota}{2}}\big).
\end{aligned}
\end{equation}
Multiplying both sides of \eqref{esti-xx} by $\phi_0^{\frac{1}{2\alpha}-2+\iota}$, along with $\phi_0\sim d(x)$, one gets
\begin{equation}\label{Linfty-tx}
\phi_0^{\frac{1}{2\alpha}-2+\iota} U_x\in L^\infty([0,T];L^2) \ \ \text{for all }\iota>0.
\end{equation}

Next, according to Lemma \ref{Leibniz} and multiplying both sides of $\eqref{lp}_1$ by $\bar\eta_x^2\phi_0^{\frac{3}{2}-\frac{1}{\alpha}-\varepsilon_0}$, one has that for a.e. $(t,x)\in (0,T)\times I$, \begin{equation}\label{3...143}
\begin{aligned}
&\phi_0^{\frac{3}{2}-\varepsilon_0} U_{xx}+\frac{1}{\alpha}\phi_0^{\frac{1}{2}-\varepsilon_0}(\phi_0)_xU_x\\
=&\phi_0^{\frac{3}{2}-\varepsilon_0}\bar\eta_x^2 U_t+\frac{2\phi_0^{\frac{3}{2}-\varepsilon_0}\bar\eta_{xx}U_x}{\bar\eta_x}+\frac{2}{\alpha}\phi_0^{\frac{1}{2}+\frac{1}{\alpha}-\varepsilon_0}(\phi_0)_x-\frac{2\phi_0^{\frac{3}{2}+\frac{1}{\alpha}-\varepsilon_0}\bar\eta_{xx}}{\bar\eta_x}.
\end{aligned}
\end{equation}
 Since $\frac{3}{2}-\varepsilon_0\geq \frac{1}{2\alpha}$ ($\frac{3}{2}-\varepsilon_0>\frac{1}{2}$ for $\alpha=1$), it follows from \eqref{TETE1} that 
\begin{equation}\label{3/2Uxx}
\phi_0^{\frac{3}{2}-\varepsilon_0} U_{xx}+\frac{1}{\alpha}\phi_0^{\frac{1}{2}-\varepsilon_0}(\phi_0)_xU_x\in L^\infty([0,T];L^2).
\end{equation}

Therefore, letting $\iota=\frac{5}{2}-\frac{1}{2\alpha}-\varepsilon_0$, then \eqref{Linfty-tx} and \eqref{3/2Uxx} imply that
\begin{equation}\label{3/2Uxx'}
\phi_0^{\frac{3}{2}-\varepsilon_0} U_{xx} \in L^\infty([0,T];L^2).
\end{equation}

\underline{\textbf{Step 3: Elliptic estimate $\phi_0^{\frac{3}{2}-\varepsilon_0}\partial_t U_{xx}$.}}
First, following the arguments in \eqref{00439}, integrating $\eqref{lp}_1$ over $[0,x]$ for $x\in \left(0, \frac{1}{2}\right]$ (or $[x,1]$ for $x\in \left(\frac{1}{2},1\right]$) and then applying $\partial_t$ to the resulting identity, one can get from \eqref{TETE1}, Lemma \ref{hardy-inequality} and H\"older's inequality that
\begin{align}
\absb{\phi_0^\frac{1}{\alpha} U_{tx}(t,x)}&\leq C \int_0^x z^\frac{1}{\alpha}\abs{U_t}\,\mathrm{d}z+ C\int_0^x z^\frac{1}{\alpha}\abs{U_{tt}}\,\mathrm{d}z\notag\\
&\leq C x^{\frac{1+3\alpha}{2\alpha}-\frac{\varsigma}{2}}\absb{\phi_0^{\frac{1}{2\alpha}-1+\frac{\varsigma}{2}}U_t}_2 +C x^\frac{1+\alpha}{2\alpha}\absb{\phi_0^\frac{1}{2\alpha}U_{tt}}_2\label{esti-tx}\\
&\leq C(\varsigma) \phi_0^{\frac{1+3\alpha}{2\alpha}-\frac{\varsigma}{2}}\big(\absb{\phi_0^{\frac{1}{2\alpha}}U_t}_2+\absb{\phi_0^{\frac{1}{2\alpha}}U_{tx}}_2\big)+C \phi_0^\frac{1+\alpha}{2\alpha}\absb{\phi_0^\frac{1}{2\alpha}U_{tt}}_2\notag\\
&\leq  C(\varsigma)\big(\phi_0^{\frac{1+3\alpha}{2\alpha}-\frac{\varsigma}{2}}+\phi_0^\frac{1+\alpha}{2\alpha}\big), \notag
\end{align}
where $\varsigma\in (0,1)$, which multiplied by $\phi_0^{-\frac{1}{2\alpha}-1+\iota}$ ($\iota>0$), along with $\phi_0\sim d(x)$, yields 
\begin{equation}\label{3...148}
\phi_0^{\frac{1}{2\alpha}-1+\iota}U_{tx}\in L^\infty([0,T];L^2),\quad \text{for all }\iota>0.
\end{equation}

Next, applying $\partial_t$ to both sides of \eqref{3...143} and using Lemma \ref{Leibniz}, one has
\begin{equation}\label{equ3.72}
\begin{aligned}
&\phi_0^{\frac{3}{2}-\varepsilon_0} \partial_tU_{xx}+\frac{1}{\alpha}\phi_0^{\frac{1}{2}-\varepsilon_0}(\phi_0)_x U_{tx}\\
= &\phi_0^{\frac{3}{2}-\varepsilon_0}\bar\eta_x^2 U_{tt}+2\phi_0^{\frac{3}{2}-\varepsilon_0}\bar\eta_x \bar U_x U_{t}-\frac{2\phi_0^{\frac{3}{2}+\frac{1}{\alpha}-\varepsilon_0} \bar U_{xx}}{\bar\eta_x}+\frac{2\phi_0^{\frac{3}{2}+\frac{1}{\alpha}-\varepsilon_0}\bar\eta_{xx}\bar U_x}{\bar\eta_x^2}\\
& +\frac{2\phi_0^{\frac{3}{2}-\varepsilon_0}\bar U_{xx} U_x}{\bar\eta_x}-\frac{2\phi_0^{\frac{3}{2}-\varepsilon_0}\bar\eta_{xx} \bar U_x U_x}{\bar\eta_x^2}+\frac{2\phi_0^{\frac{3}{2}-\varepsilon_0} \bar\eta_{xx} U_{tx}}{\bar\eta_x},
\end{aligned}
\end{equation}
which, together with the fact that $\varepsilon_0\in \left(0,\frac{3\alpha-1}{2\alpha}\right]$ (or $\varepsilon_0\in (0,1)$ for $\alpha=1$) and \eqref{TETE1}, implies that 
\begin{equation}\label{3...149}
\begin{aligned}
\phi_0^{\frac{3}{2}-\varepsilon_0} \partial_tU_{xx}+\frac{1}{\alpha}\phi_0^{\frac{1}{2}-\varepsilon_0}(\phi_0)_x U_{tx} \in L^\infty([0,T];L^2).
\end{aligned}
\end{equation}
Hence for $\varepsilon_0\in \left(0,\frac{3\alpha-1}{2\alpha}\right)$ ($\alpha \in \left(\frac{1}{3},1\right]$), setting $\iota:=\frac{3\alpha-1}{2\alpha}-\varepsilon_0$ in  \eqref{3...148} and using \eqref{3...149}, one has 
\begin{equation}\label{3/2Utxx}
\phi_0^{\frac{3}{2}-\varepsilon_0}\partial_t U_{xx}\in L^\infty([0,T];L^2);
\end{equation}
for $\varepsilon_0=\frac{3\alpha-1}{2\alpha}$ ($\alpha\in \left(\frac{1}{3},1\right)$), one can multiply \eqref{3...149} by $\phi_0^\iota$ and deduce from \eqref{3...148} that $\phi_0^{\frac{3}{2}-\varepsilon_0+\iota}\partial_t U_{xx}\in L^\infty([0,T];L^2)$, which, along with Proposition \ref{prop2.1} and \eqref{TETE1}, yields \eqref{3/2Utxx}.

\underline{\textbf{Step 4: Elliptic estimate $\phi_0^{\frac{3}{2}-\varepsilon_0}\partial_x^3 U$.}}
First, multiplying $\eqref{lp}_1$ by $\bar\eta_x^2 \phi_0^{1-\frac{1}{\alpha}}$ and applying $\phi_0^{\frac{1}{2}-\varepsilon_0}\partial_x$ to the resulting equality, which can be justified due to Lemmas \ref{Leibniz}-\ref{qiudao}, one gets that
\begin{align}
&\phi_0^{\frac{3}{2}-\varepsilon_0}\partial_x^3 U +\Big(\frac{1}{\alpha}+1\Big)\phi_0^{\frac{1}{2}-\varepsilon_0}(\phi_0)_x U_{xx}\notag\\
= &-\frac{1}{\alpha}\phi_0^{\frac{1}{2}-\varepsilon_0}(\phi_0)_{xx} U_x+\phi_0^{\frac{3}{2}-\varepsilon_0}\bar\eta_x^2 U_{tx}+2\phi_0^{\frac{3}{2}-\varepsilon_0}\bar\eta_x\bar\eta_{xx}U_t+\phi_0^{\frac{1}{2}-\varepsilon_0}(\phi_0)_x\bar\eta_x^2 U_t \notag\\
&+\frac{2\phi_0^{\frac{1}{2}-\varepsilon_0}(\phi_0)_x\bar\eta_{xx}U_x}{\bar\eta_x}+\frac{2\phi_0^{\frac{3}{2}-\varepsilon_0}\bar\eta_{xx} U_{xx}}{\bar\eta_x}+\frac{2\phi_0^{\frac{3}{2}-\varepsilon_0}\partial_x^3\bar\eta U_x}{\bar\eta_x}-\frac{2\phi_0^{\frac{3}{2}-\varepsilon_0}\bar\eta_{xx}^2U_x}{\bar\eta_x^2}\label{3...151}\\
&-\frac{2+2\alpha}{\alpha}\frac{\phi_0^{\frac{1}{2}+\frac{1}{\alpha}-\varepsilon_0}(\phi_0)_x\bar\eta_{xx}}{\bar\eta_x}-\frac{2\phi_0^{\frac{3}{2}+\frac{1}{\alpha}-\varepsilon_0}\partial_x^3\bar\eta}{\bar\eta_x} +\frac{2\phi_0^{\frac{3}{2}+\frac{1}{\alpha}-\varepsilon_0}\bar\eta_{xx}^2}{\bar\eta_x^2}\notag\\
&+\frac{2}{\alpha^2}\phi_0^{-\frac{1}{2}+\frac{1}{\alpha}-\varepsilon_0}((\phi_0)_x)^2 +\frac{2}{\alpha}\phi_0^{\frac{1}{2}+\frac{1}{\alpha}-\varepsilon_0}(\phi_0)_{xx}.\notag
\end{align}
Then, it follows from \eqref{TETE1}, \eqref{3/2Uxx'} and Lemma \ref{hardy-inequality} that
\begin{equation}\label{3...152}
\phi_0^{\frac{3}{2}-\varepsilon_0}\partial_x^3U+\Big(\frac{1}{\alpha}+1\Big)\phi_0^{\frac{1}{2}-\varepsilon_0}(\phi_0)_x U_{xx}\in L^\infty([0,T];L^2).
\end{equation}

Due to  Proposition \ref{prop2.1}, it remains to verify  $\phi_0^{\frac{1}{2\alpha}+1}\partial_x^3 U\in L^\infty([0,T];L^2)$. Indeed, by \eqref{TETE1}, one can reduce the power of weights in \eqref{3/2Uxx} from $\frac{3}{2}-\varepsilon_0$ to $\frac{1}{2\alpha}$, and obtain that
\begin{equation}\label{3...154}
\phi_0^\frac{1}{2\alpha}U_{xx}+\frac{1}{\alpha}{\phi_0^{\frac{1}{2\alpha}-1}}(\phi_0)_xU_x\in L^\infty([0,T];L^2),
\end{equation}
which, along with \eqref{Linfty-tx}, leads to 
\begin{equation}\label{equ378}
   \phi_0^\frac{1}{2\alpha}U_{xx}\in L^\infty([0,T];L^2). 
\end{equation}
Then, $\eqref{3...152}\times \phi_0^{\frac{1-\alpha}{2\alpha}+\varepsilon_0}$ and \eqref{equ378} imply that  $\phi_0^{\frac{1}{2\alpha}+1}\partial_x^3 U\in L^\infty([0,T];L^2)$. 

Hence, it follows from Proposition \ref{prop2.1} and \eqref{3/2Uxx'} that 
\begin{equation}\label{3/2Uxxx}
\phi_0^{\frac{3}{2}-\varepsilon_0}\partial_x^3 U\in L^\infty([0,T];L^2).
\end{equation}

\underline{\textbf{Step 5: Elliptic estimate $\phi_0^{\frac{3}{2}-\varepsilon_0}\partial_x^4 U$.}}
Analogously, it follows from multiplying $\eqref{lp}_1$ by $\bar\eta_x^2 \phi_0^{1-\frac{1}{\alpha}}$ and applying $\phi_0^{\frac{1}{2}-\varepsilon_0}\partial_x^2$ to the resulting equality (due to Lemmas \ref{Leibniz}-\ref{qiudao}) that
\begin{align}
&\phi_0^{\frac{3}{2}-\varepsilon_0}\partial_x^4 U+\Big(\frac{1}{\alpha}+2\Big)\phi_0^{\frac{1}{2}-\varepsilon_0}(\phi_0)_x\partial_x^3 U\notag\\
=&\phi_0^{\frac{3}{2}-\varepsilon_0}\bar\eta_x^2\partial_tU_{xx}+2\phi_0^{\frac{1}{2}-\varepsilon_0}\left((\phi_0)_x\bar\eta_x+2\phi_0\bar\eta_{xx}\right)\bar\eta_x U_{tx}+2\phi_0^{\frac{3}{2}-\varepsilon_0}\left(\bar\eta_{xx}^2+\bar\eta_x\partial_x^3\bar\eta\right)U_t\notag\\
&+\phi_0^{\frac{1}{2}-\varepsilon_0}\left((\phi_0)_{xx}\bar\eta_x+4(\phi_0)_x\bar\eta_{xx}\right)\bar\eta_xU_t -\Big(\frac{2}{\alpha}+1\Big)\phi_0^{\frac{1}{2}-\varepsilon_0}(\phi_0)_{xx} U_{xx}\notag\\
& -\frac{1}{\alpha}\phi_0^{\frac{1}{2}-\varepsilon_0}\partial_x^3\phi_0 U_x+2\phi_0^{\frac{1}{2}-\varepsilon_0}\Big((\phi_0)_{xx}\bar\eta_{xx}+2(\phi_0)_{x}\partial_x^3\bar\eta -\frac{2(\phi_0)_{x}\bar\eta_{xx}^2}{\bar\eta_x}\Big)\frac{U_x}{\bar\eta_x}\notag\\
&+2\phi_0^{\frac{3}{2}-\varepsilon_0}\Big(\partial_x^4\bar\eta
-\frac{3\bar\eta_{xx}\partial_x^3\bar\eta}{\bar\eta_x}
+\frac{2\bar\eta_{xx}^3}{\bar\eta_x^2}\Big)\frac{U_x}{\bar\eta_x}\label{3-154}\\
&+4\phi_0^{\frac{1}{2}-\varepsilon_0}\Big(\phi_0\partial_x^3\bar\eta+(\phi_0)_{x}\bar\eta_{xx}-\frac{\phi_0\bar\eta_{xx}^2}{\bar\eta_x}\Big)\frac{U_{xx}}{\bar\eta_x}+\frac{2\phi_0^{\frac{3}{2}-\varepsilon_0}\bar\eta_{xx}\partial_x^3U}{\bar\eta_x}\notag\\
&-\frac{2+2\alpha}{\alpha^2}\frac{\phi_0^{-\frac{1}{2}+\frac{1}{\alpha}-\varepsilon_0}((\phi_0)_x)^2\bar\eta_{xx}}{\bar\eta_x}-2\phi_0^{\frac{3}{2}+\frac{1}{\alpha}-\varepsilon_0}\Big(\frac{\partial_x^4\bar\eta}{\bar\eta_x}-\frac{3\bar\eta_{xx}\partial_x^3\bar\eta}{\bar\eta_x^2}+\frac{2\bar\eta_{xx}^3}{\bar\eta_x^3}\Big)\notag\\
& -\frac{2+2\alpha}{\alpha}\phi_0^{\frac{1}{2}+\frac{1}{\alpha}-\varepsilon_0}\Big(\frac{(\phi_0)_{xx}\bar\eta_{xx}}{\bar\eta_x}+\frac{2(\phi_0)_x\partial_x^3\bar\eta}{\bar\eta_x}-\frac{2(\phi_0)_x\bar\eta_{xx}^2}{\bar\eta_x^2}\Big) \notag\\
&+\boxed{\frac{2(1-\alpha)}{\alpha^3}\phi_0^{\frac{1}{\alpha}-\frac{3}{2}-\varepsilon_0}((\phi_0)_x)^3}+\frac{6}{\alpha^2}\phi_0^{\frac{1}{\alpha}-\frac{1}{2}-\varepsilon_0}(\phi_0)_x(\phi_0)_{xx}+\frac{2}{\alpha}\phi_0^{\frac{1}{\alpha}+\frac{1}{2}-\varepsilon_0}\partial_x^3\phi_0.\notag\end{align}
It is worth noting that the framed term in \eqref{3-154} vanishes whenever $\alpha=1$. 

Next, since $\varepsilon_0$ satisfies \eqref{varepsilon0}, it follows from \eqref{TETE1}, \eqref{3/2Uxx'}, \eqref{3/2Utxx}, \eqref{3/2Uxxx}-\eqref{3-154} and Lemma \ref{hardy-inequality} that 
\begin{equation}\label{3-156}
\phi_0^{\frac{3}{2}-\varepsilon_0}\partial_x^4 U+\Big(\frac{1}{\alpha}+2\Big)\phi_0^{\frac{1}{2}-\varepsilon_0}(\phi_0)_x\partial_x^3 U\in L^\infty([0,T];L^2).    
\end{equation}

In order to use Proposition \ref{prop2.1}, it remains  to check that $\phi_0^{\frac{1}{2\alpha}+\frac{3}{2}}\partial_x^4 U\in L^\infty([0,T];L^2)$. To this end, according to \eqref{TETE1}, Lemma \ref{hardy-inequality} and the assumption that $\alpha\leq 1$, one can change the power of weights in \eqref{3...152} from $\frac{3}{2}-\varepsilon_0$ to $\frac{1}{2\alpha}+\frac{1}{2}$, and  obtain that
\begin{equation}\label{equa387}
 \phi_0^{\frac{1}{2\alpha}+\frac{1}{2}}\partial_x^3 U+\Big(\frac{1}{\alpha}+1\Big)\phi_0^{\frac{1}{2\alpha}-\frac{1}{2}}(\phi_0)_xU_{xx}\in L^\infty([0,T];L^2).    
\end{equation}
Since one has already shown that $\phi_0^{\frac{1}{2\alpha}}U_{xx},\,\phi_0^{\frac{1}{2\alpha}+1}\partial_x^3 U\in L^\infty([0,T];L^2)$ in \textbf{Step 4}, it follows from \eqref{equa387} and Proposition \ref{prop2.1} that
\begin{equation}\label{equ3.83}
\phi_0^{\frac{1}{2\alpha}+\frac{1}{2}}\partial_x^3 U\in L^\infty([0,T];L^2).  
\end{equation}

As a consequence, multiplying \eqref{3-156} by $\phi_0^{\frac{1}{2\alpha}+\varepsilon_0}$, one can obtain from \eqref{equ3.83} that $\phi_0^{\frac{1}{2\alpha}+\frac{3}{2}}\partial_x^4 U\in L^\infty([0,T];L^2)$. Then according to Proposition \ref{prop2.1}, \eqref{3/2Uxxx} and \eqref{3-156}, one has 
$
\phi_0^{\frac{3}{2}-\varepsilon_0}\partial_x^4 U\in L^\infty([0,T];L^2)$.

\underline{\textbf{Step 6: Time continuity.}} Following the proof of \textbf{Step 9} in the first case, one can check the continuities of \eqref{3/2Uxx}, \eqref{3...149}, \eqref{3...152} and \eqref{3-156} in turn, and then make use of \eqref{TETE1} and Proposition \ref{prop2.1} to show that
\begin{align*}
\phi_0^{\frac{3}{2}-\varepsilon_0} U_{xx}+\frac{1}{\alpha}\phi_0^{\frac{1}{2}-\varepsilon_0}(\phi_0)_xU_x\in C([0,T];L^2)&\implies \phi_0^{\frac{3}{2}-\varepsilon_0} U_{xx}\in C([0,T];L^2),\\
\phi_0^{\frac{3}{2}-\varepsilon_0} \partial_t U_{xx}+\frac{1}{\alpha}\phi_0^{\frac{1}{2}-\varepsilon_0}(\phi_0)_xU_{tx}\in C([0,T];L^2)&\implies \phi_0^{\frac{3}{2}-\varepsilon_0}\partial_t U_{xx}\in C([0,T];L^2),\\
\phi_0^{\frac{3}{2}-\varepsilon_0}\partial_x^3 U +\Big(\frac{1}{\alpha}+1\Big)\phi_0^{\frac{1}{2}-\varepsilon_0}(\phi_0)_xU_{xx}\in C([0,T];L^2)&\implies \phi_0^{\frac{3}{2}-\varepsilon_0} \partial_x^3 U \in C([0,T];L^2),\\
\phi_0^{\frac{3}{2}-\varepsilon_0}\partial_x^4 U +\Big(\frac{1}{\alpha}+2\Big)\phi_0^{\frac{1}{2}-\varepsilon_0}(\phi_0)_x\partial_x^3 U \in C([0,T];L^2)&\implies \phi_0^{\frac{3}{2}-\varepsilon_0}\partial_x^4 U \in C([0,T];L^2).
\end{align*}
Here, it should be mentioned that, when establishing the continuities of $\phi_0^{\frac{3}{2}-\varepsilon_0}U_{xx}$ and $\phi_0^{\frac{3}{2}-\varepsilon_0}\partial_t U_{xx}$, the parameters $(\varepsilon_0,\alpha)$ do not satisfy the constraints imposed in Proposition \ref{prop2.1} unless $\frac{3}{2}-\varepsilon_0\leq \frac{1}{2\alpha}+\frac{1}{2}$. Fortunately, when $\frac{3}{2}-\varepsilon_0>\frac{1}{2\alpha}+\frac{1}{2}$, the weight $\phi_0^{\frac{3}{2}-\varepsilon_0}$ gains higher-order powers so that we can obtain their continuities directly from \eqref{3...143} and \eqref{equ3.72}; for example, to derive the time-continuity of $\phi_0^{\frac{1}{2}-\varepsilon_0}U_{tx}$ on the left-hand side of \eqref{equ3.72}, we can apply an argument similar to \eqref{esti-tx}--\eqref{3...148} since the power $\frac{1}{2}-\varepsilon_0>\frac{1}{2\alpha}-\frac{1}{2}$, which is larger than $\frac{1}{2\alpha}-1+\iota$ for sufficiently small $\iota>0$ given in \eqref{3...148}.  

Collecting all these estimates, together with \eqref{TETE1}, shows that $U\in \mathscr{C}([0,T];\widetilde E)$, which, along with Lemma \ref{hardy-inequality}, yields $U\in C([0,T];W^{3,1})\cap C^1([0,T];W^{1,1})$.

\underline{\textbf{Step 7: Derivation of the Neumann boundary condition.}} One can deduce from Lemma \ref{sobolev-embedding} that \eqref{equ37.5} holds and then follow the same proof of \textbf{Step 10} in \S \ref{subsection3.3}  to show that $U_x|_{x\in \Gamma}=0$.

Therefore,  the proof of Lemma \ref{existence-linearize} $\mathrm{ii)}$ is completed. 
\end{proof}

\section{Uniform estimates to the linearized problems}\label{Section4}
 
With the help of  Lemma \ref{existence-linearize}, in \S \ref{subsection4.1}, we first consider the case for $0<\alpha\leq \frac{1}{3}$ and give a specific derivation of the uniform estimates on the classical solution $U$ to the linearized problem \eqref{lp''}. The proof for  the case $\frac{1}{3}<\alpha\leq 1$ is basically the same as the one for the case $0<\alpha\leq \frac{1}{3}$, and  we only give a sketch in \S \ref{subsection4.2}. 

\subsection{Uniform estimates for the case \texorpdfstring{$0<\alpha\leq \frac{1}{3}$}{}}\label{subsection4.1}
\begin{Lemma}\label{a-priori-estimates}
 Assume that  $U$ is  the  unique classical solution in $[0,T]\times \bar I$ to the problem \eqref{lp''} obtained in Lemma \ref{existence-linearize}, and the positive constant $c_0$ satisfies
\begin{equation}\label{38}
2+\norm{\phi_0}_{3}+E(0,U)\leq c_0.
\end{equation}
Then there exist a positive time $T_*\in (0,T]$ and constants $c_i$, $i=1,2$, $1<c_0\leq c_1\leq c_2$, which depend only on $c_0$, $\alpha$, $|I|$, $\cC_1$ and  $\cC_2$, such that if for all $0\leq t\leq T_*$,
\begin{equation}\label{39}
\begin{aligned}
\absf{\phi_0 \bar U(t)}_{2}+\absf{\phi_0 \bar U_x(t)}_{2}+\absf{\phi_0 \bar U_t(t)}_{2}+\absf{\phi_0 \bar U_{xx}(t)}_{2}\leq c_1,\\[4pt]
\absf{\phi_0 \bar U_{tx}(t)}_{2}+\absf{\phi_0 \partial^3_x\bar U(t)}_{2}+\absf{\phi_0 \bar U_{tt}(t)}_{2}+\absf{\phi_0 \partial_t\bar U_{xx}(t)}_{2}+\absf{\phi_0 \partial^4_x\bar U(t)}_{2}\leq c_2,
\end{aligned}
\end{equation}
it holds that for all $0\leq t\leq T_*$,
\begin{equation}\label{uniform bounds}
\begin{aligned}
\abs{\phi_0 U(t)}_{2}+\abs{\phi_0 U_x(t)}_{2}+\abs{\phi_0 U_t(t)}_{2}+\abs{\phi_0  U_{xx}(t)}_{2}\leq c_1,\\[4pt]
\abs{\phi_0 U_{tx}(t)}_{2}+\absf{\phi_0 \partial^3_x U(t)}_{2}+\abs{\phi_0 U_{tt}(t)}_{2}+\abs{\phi_0 \partial_t U_{xx}(t)}_{2}+\absf{\phi_0 \partial^4_x U(t)}_{2}\leq c_2.
\end{aligned}
\end{equation}
\end{Lemma}

\subsubsection{Basic estimates}\label{Remark-useful-bounds}
First, according to the assumption \eqref{39} and Lemmas \ref{sobolev-embedding}-\ref{hardy-inequality}, it holds that for all $0\leq t\leq T$,
\begin{equation}\label{useful1}
\begin{split}
\absf{\bar U_x(t)}_2\leq C \sum_{j=1}^2 \absf{\phi_0\partial_x^j\bar U (t)}_2\leq & C c_1,\\
\normf{\bar U_x(t)}_{1,\infty}\leq C \normf{\bar U_{x}(t)}_{2}\leq C \sum_{j=1}^4\absf{\phi_0\partial_x^j\bar U(t)}_{2} \leq& C c_2,\\
\absb{\phi_0^\frac{1}{2}\bar U_{tx}(t)}_{\infty}\leq  C \sum_{j=1}^2\absf{\phi_0\partial_x^j\bar U_t(t)}_{2} \leq & C c_2.
\end{split}
\end{equation}

Next, there exists a positive time $\widetilde T>0$, such that
\begin{equation}\label{useful2}
\frac{1}{2}\leq \bar \eta_x(t,x)\leq \frac{3}{2} \ \ \text{for all } (t,x)\in [0,\widetilde T]\times\bar I.
\end{equation}
Indeed, applying $\partial_x$ to both sides of \eqref{given-flow}, 
one gets from \eqref{useful1} that for $\widetilde T:= (1+2Cc_2)^{-1}$, \begin{equation*}
\begin{aligned}
\abs{\bar\eta_x(t,x) -1}&\leq \int_0^t \absf{\bar U_x}_\infty \,\ds\leq C\widetilde Tc_2\leq \frac{1}{2} \ \ \text{for all } (t,x)\in [0,\widetilde T]\times\bar I.
\end{aligned}
\end{equation*}

Finally, it follows from  \eqref{given-flow}, \eqref{useful1} and Lemma \ref{sobolev-embedding} that for all $0\leq t\leq T$,
\begin{equation}\label{useful3}
\begin{gathered}
\abs{\bar \eta_{xx}(t)}_{\infty}\leq C\norm{\bar \eta_{xx} (t)}_{1}\leq C \int_0^t \normf{\bar U_{xx}}_{1}\,\ds\leq C c_2 t,\\
\absf{\partial_x^3\bar \eta (t)}_{2}\leq C \sum_{j=3}^4\absf{\phi_0\partial_x^j\bar \eta (t)}_{2}\leq C \int_0^t \sum_{j=3}^4\absf{\phi_0\partial_x^j \bar U}_{2}\,\ds\leq Cc_2 t.
\end{gathered}
\end{equation}

\subsubsection{Proof of Lemma \ref{a-priori-estimates}}
The proof will be divided into the following several steps.

\begin{Lemma}\label{c_0-c_1}
Under the same assumptions of Lemma \ref{a-priori-estimates}, it holds that
\begin{equation*}
\absf{\phi_0 U(t)}_{2}+\absf{\phi_0 U_x(t)}_{2}+\absf{\phi_0 U_t(t)}_{2}+\absf{\phi_0 U_{xx}(t)}_{2}\leq C c_0^\frac{2}{\alpha},
\end{equation*}
for all $0\leq t\leq T_1=\min\{\widetilde T, (1+Cc_2)^{-\frac{6\alpha+2}{\alpha}}\}$.
\end{Lemma}

\begin{proof}
\textbf{\underline{Step 1: Estimate of $\phi_0 U$.}}
Multiplying $\eqref{lp''}_1$ by $\bar\eta_x^2U$  and integrating the resulting equality over $I$, then by \eqref{38}, \eqref{useful1}-\eqref{useful3}, H\"older's inequality and Young's inequality, one gets
\begin{align}
&\frac{1}{2}\frac{\mathrm{d}}{\dt}\int \phi_0^2\bar\eta_x^2 U^2\,\dx+\int \phi_0^2 U_x^2\,\dx\notag\\
=&\underline{\Big(\frac{1}{\alpha}-2\Big)\int \phi_0(\phi_0)_x U U_x\,\dx}_{:=\cL_1}+\int \phi_0^2 \bar\eta_x \bar U_x U^2-2\int \frac{\phi_0^2\bar\eta_{xx}U_x U}{\bar\eta_x}\notag\\
&+2\int \frac{\phi_0^{2+\frac{1}{\alpha}}\bar\eta_{xx}U}{\bar\eta_x} \,\dx-\frac{2}{\alpha}\int\phi_0^{1+\frac{1}{\alpha}}(\phi_0)_x U\,\dx\label{eq:cL1-cL5}\\
\leq &\cL_1+C\big(\absf{\bar U_x}_\infty \absf{\phi_0 \bar\eta_x U}_2^2+
 \absf{\bar\eta_{xx}}_\infty \abs{\phi_0 \bar\eta_x U}_2\abs{\phi_0 U_x}_2\big)\notag\\
&+
C\big(\abs{\phi_0}_\infty^{\frac{\alpha+1}{\alpha}}\abs{\bar\eta_{xx}}_\infty\abs{\phi_0 \bar\eta_x U}_2+
\abs{\phi_0}_\infty^{\frac{1}{\alpha}}\abs{(\phi_0)_x}_\infty \abs{\phi_0 \bar\eta_x U}_2\big)\notag\\
\leq& \cL_1+C(c_2+c_2^2t^2)\absf{\phi_0 \bar\eta_x U}_2^2+C(1+c_2^2t^2)c_0^{2+\frac{2}{\alpha}}+\frac{1}{8}\abs{\phi_0 U_x}_2^2.\notag
\end{align}

For $\cL_1$, it follows from integration by parts, $0<\alpha\leq\frac{1}{3}$, \eqref{38}, \eqref{useful2}, Lemma \ref{GNinequality}, H\"older's inequality and Young's inequality that
\begin{align}
\cL_1
&=\underline{-\frac{1-2\alpha}{2\alpha}\int  ((\phi_0)_x)^2U^2\,\dx}_{\leq 0}-\frac{1-2\alpha}{2\alpha}\int \phi_0 (\phi_0)_{xx} U^2\,\dx\label{cL1}\\
&\leq C\big(1+\abs{(\phi_0)_{xx}}_\infty^2\big)\abs{\phi_0 U}_2^2+\frac{1}{8}\abs{\phi_0 U_x}_2^2
\leq Cc_0^2\absf{\phi_0 \bar\eta_x U}_2^2+\frac{1}{8}\abs{\phi_0 U_x}_2^2.\notag
\end{align}


Thus, combining with \eqref{eq:cL1-cL5}-\eqref{cL1}, one can get from Gr\"onwall's inequality, \eqref{useful2} and Lemma \ref{hardy-inequality} that for all $0\leq t\leq T_1:= \min\{\widetilde T, (1+Cc_2)^{-\frac{6\alpha+2}{\alpha}}\}$,
\begin{equation}\label{1}
\abs{\phi_0 U(t)}_{2}^2+\int_0^t \abs{\phi_0 U_x}_{2}^2\,\ds
\leq  C e^{Cc_2^2t}\big(\abs{\phi_0 u_0}_{2}^2+c_0^{2+\frac{2}{\alpha}} c_2^2 t\big)\leq C c_0^2.
\end{equation}

\underline{\textbf{Step 2: Estimate of $\phi_0 U_{x}$.}} Multiplying $\eqref{lp''}_1$ by $\bar\eta_x^2 U_t$ and integrating the resulting equality over $I$, then by  \eqref{38}, \eqref{useful1}-\eqref{useful3}, H\"older's inequality and Young's inequality, one has
\begin{align}
&\frac{1}{2}\frac{\mathrm{d}}{\dt}\int \phi_0^2 U_x^2\,\dx+\int \phi_0^2\bar\eta_x^2 U_t^2\,\dx\notag\\
=&\underline{\Big(\frac{1}{\alpha}-2\Big)\int \phi_0 (\phi_0)_x U_x U_t\,\dx}_{:=\cL_2}-2 \int \frac{\phi_0^2 \bar\eta_{xx}U_x U_t}{\bar\eta_x}\,\dx\notag\\
&+2\int \frac{\phi_0^{2+\frac{1}{\alpha}}\bar\eta_{xx} U_t}{\bar\eta_x}\,\dx- \frac{2}{\alpha}\int \phi_0^{1+\frac{1}{\alpha}}(\phi_0)_x U_t\label{eq:cL6-cL9}\\
\leq &\cL_2+ C \big(\abs{\bar\eta_{xx}}_\infty \abs{\phi_0  U_x}_2+
\abs{\phi_0}_\infty^\frac{\alpha+1}{\alpha}\abs{\bar\eta_{xx}}_\infty+
C\abs{\phi_0}_\infty^{\frac{1}{\alpha}}\abs{(\phi_0)_x}_\infty\big) \abs{\phi_0 U_t}_2\notag\\
\leq &\cL_2+Cc_2^2t^2\abs{\phi_0 U_x}_2^2+C(1+c_2^2t^2)c_0^{2+\frac{2}{\alpha}}+\frac{1}{8}\abs{\phi_0 \bar\eta_x U_t}_2^2.\notag
\end{align}

For $\cL_2$, it follows from \eqref{38}, \eqref{useful2}, Lemma \ref{GNinequality}, H\"older's inequality and Young's inequality that
\begin{equation}\label{cL6}
\begin{aligned}
\cL_2&=\Big(\frac{1}{\alpha}-2\Big)\int \phi_0 (\phi_0)_x U_x U_t\,\dx
\leq C \abs{(\phi_0)_x}_\infty\absb{\phi_0^\frac{1}{2}U_x}_2\absb{\phi_0^\frac{1}{2}U_t}_2\\
&\leq C c_0\big(\abs{\phi_0 U_x}_2+\abs{\phi_0 U_x}_2^\frac{1}{2}\abs{\phi_0 U_{xx}}_2^\frac{1}{2}\big)\big(\abs{\phi_0 U_t}_2 +\abs{\phi_0 U_t}_2^\frac{1}{2}\abs{\phi_0 U_{tx}}_2^\frac{1}{2}\big)\\
&\leq C(\varepsilon_1,\varepsilon_2)c_0^2\abs{\phi_0 U_x}_2^2+\frac{1}{8}\abs{\phi_0 \bar\eta_x U_t}_2^2+\varepsilon_1\abs{\phi_0 U_{xx}}_2^2+\varepsilon_2\abs{\phi_0 U_{tx}}_2^2,
\end{aligned}
\end{equation}
where $\varepsilon_1,\varepsilon_2\in (0,1)$ are arbitrarily small constants that will be determined later.


Hence, it follows from \eqref{eq:cL6-cL9}-\eqref{cL6} that for all $\varepsilon_1,\varepsilon_2\in(0,1)$, 
\begin{equation}\label{step1}
\begin{aligned}
\frac{\mathrm{d}}{\dt}\abs{\phi_0 U_x}_{2}^2+\abs{\phi_0 U_t}_{2}^2\leq & C(\varepsilon_1,\varepsilon_2)c_2^2\abs{\phi_0 U_x}_{2}^2+\varepsilon_1\abs{\phi_0 U_{xx}}_2^2+\varepsilon_2\abs{\phi_0 U_{tx}}_2^2+Cc_2^{4+\frac{2}{\alpha}}.
\end{aligned}
\end{equation}

For $\phi_0 U_{xx}$, it follows from \eqref{0435}, \eqref{38} and \eqref{useful2}-\eqref{useful3} that
\begin{equation}\label{equ3.102}
\begin{aligned}
&\big|\phi_0 U_{xx} + \frac{1}{\alpha} (\phi_0)_x U_x\big|_2\\
\leq &C\big(\abs{\phi_0 U_t}_2 + \abs{\bar\eta_{xx}}_\infty\abs{\phi_0 U_x}_2+\abs{\phi_0}_\infty^\frac{1}{\alpha}\norm{\phi_0}_1+\abs{\phi_0}_\infty^\frac{\alpha+1}{\alpha}\abs{\bar\eta_{xx}}_\infty\big)\\
\leq &C\big(\abs{\phi_0 U_t}_2 + c_2 t\abs{\phi_0 U_x}_2\big)+ C\big(c_0^{1+\frac{1}{\alpha}}+c_0^{1+\frac{1}{\alpha}}c_2t\big),
\end{aligned}
\end{equation}
which, along with Proposition \ref{prop2.1}, leads to
\begin{equation}\label{eq3329}
\begin{aligned}
\abs{\phi_0 U_{xx}}_2&\leq C \big(\big|\phi_0 U_{xx}+ \frac{1}{\alpha} (\phi_0)_x U_x\big|_2+ \abs{(\phi_0)_{xx}}_\infty \abs{\phi_0 U_x}_2\big)\\
&\leq C\big(\abs{\phi_0 U_t}_2 + (c_0+c_2t)\abs{\phi_0 U_x}_2+c_0^{1+\frac{1}{\alpha}}(1+c_2t)\big).
\end{aligned}
\end{equation}

Substituting \eqref{eq3329} into \eqref{step1} and choosing $\varepsilon_1$ suitably small, one has for all $\varepsilon\in(0,1)$,
\begin{equation}\label{step11}
\frac{\mathrm{d}}{\dt}\abs{\phi_0 U_x}_{2}^2+\abs{\phi_0 U_t}_{2}^2\leq C(\varepsilon)c_2^2\abs{\phi_0 U_x}_{2}^2+\varepsilon\abs{\phi_0 U_{tx}}_2^2+Cc_2^{4+\frac{2}{\alpha}}.
\end{equation}

\underline{\textbf{Step 3: Estimates of $\phi_0 U_t$ and $\phi_0 U_{xx}$.}}
Multiplying $\eqref{lp''}_1$ by $\bar\eta_x^2$ first, applying $U_{t}\partial_t$ to the resulting equality and finally integrating it over $I$, then according to \eqref{38}, \eqref{useful1}-\eqref{useful3}, H\"older's inequality and Young's inequality, one has
\begin{align}
&\frac{1}{2}\frac{\mathrm{d}}{\dt}\int \phi_0^2 \bar\eta_x^2 U_t^2\,\dx+\int \phi_0^2  U_{tx}^2\,\dx\notag\\
=&\underline{\Big(\frac{1}{\alpha}-2\Big)\int \phi_0 (\phi_0)_x U_{tx}U_t\,\dx}_{:=\cL_3} -\int \phi_0^2\bar\eta_x \bar U_x U_t^2\,\dx\notag\\
&+ 2\int\frac{\phi_0^{2+\frac{1}{\alpha}}\bar U_{xx} U_t}{\bar\eta_x}\,\dx- 2\int \frac{\phi_0^{2+\frac{1}{\alpha}}\bar\eta_{xx}\bar U_{x} U_t}{\bar\eta_x^2} -2\int \frac{\phi_0^{2}\bar U_{xx} U_x U_t}{\bar\eta_x}\,\dx\notag\\
&-2\int\frac{\phi_0^{2}\bar\eta_{xx} U_{tx}U_t}{\bar\eta_x}\,\dx +2\int \frac{\phi_0^{2}\bar\eta_{xx} \bar U_x U_x U_t}{\bar\eta_x^2}\,\dx\label{2.34}\\
\leq &\cL_3+C \big(\absf{\bar U_x}_\infty \abs{\phi_0 U_t}_2+
\abs{\phi_0}_\infty^\frac{\alpha+1}{\alpha} \absf{\bar U_{xx}}_\infty + \abs{\phi_0}_\infty^\frac{\alpha+1}{\alpha}\absf{\bar U_x}_\infty \abs{\bar\eta_{xx}}_\infty\big) \abs{\phi_0 U_t}_2\notag\\
&+C\big(\absf{\bar U_{xx}}_\infty\abs{\phi_0 U_x}_2+
\abs{\bar\eta_{xx}}_\infty\abs{\phi_0 U_{tx}}_2+
\abs{\bar\eta_{xx}}_\infty\absf{\bar U_x}_\infty \abs{\phi_0 U_x}_2\big)\abs{\phi_0 U_t}_2\notag\\
\leq & \cL_3+ Cc_2^4\big(\abs{\phi_0 U_x}_2^2+\abs{\phi_0 \bar\eta_x U_t}_2^2+c_0^{2+\frac{2}{\alpha}}\big)+\frac{1}{8}\abs{\phi_0 U_{tx}}_2^2.\notag
\end{align}

For $\cL_{3}$,  in a similar way as for $\cL_1$, it follows from integration by parts, \eqref{38}, \eqref{useful2}, Lemma \ref{GNinequality}, H\"older's inequality and Young's inequality that
\begin{equation}\label{cL10}
\begin{aligned}
\cL_{3}
&=\underline{-\frac{1-2\alpha}{2\alpha}\int ((\phi_0)_x)^2 U_t^2\,\dx}_{\leq 0}-\frac{1-2\alpha}{2\alpha}\int \phi_0(\phi_0)_{xx}U_t^2\,\dx \\
&\leq -\frac{1-2\alpha}{2\alpha}\int \phi_0 (\phi_0)_{xx}U_t^2\,\dx \\
&\leq C\abs{(\phi_0)_{xx}}_\infty^2\abs{\phi_0 U_t}_2^2+\frac{1}{8}\abs{\phi_0 U_{tx}}_2^2
\leq Cc_0^2\abs{\phi_0\bar\eta_x U_t}_2^2+\frac{1}{8}\abs{\phi_0 U_{tx}}_2^2.
\end{aligned}
\end{equation}

Thus, it follows from \eqref{2.34}-\eqref{cL10} that
\begin{equation}\label{3}
\frac{\mathrm{d}}{\dt}\abs{\phi_0 \bar\eta_x U_t}_{2}^2+\abs{\phi_0 U_{tx}}_{2}^2\leq Cc_2^4\big(\abs{\phi_0 U_x}_2^2+\abs{\phi_0 \bar\eta_x U_t}_2^2+c_0^{2+\frac{2}{\alpha}}\big).
\end{equation}

Consequently, choosing $\varepsilon$ in \eqref{step11} sufficiently small, it follows from  \eqref{step11}, \eqref{3}, the Gr\"onwall inequality and \eqref{useful2} that for all $0\leq t\leq T_1$,
\begin{equation}\label{3'}
\begin{aligned}
\abs{\phi_0 U_x(t)}_{2}^2+\abs{\phi_0 U_t(t)}_{2}^2&\leq C e^{Cc_2^4t}\big(\abs{\phi_0 (u_0)_x}_{2}^2+\abs{\phi_0 U_t(0)}_{2}^2+c_2^{6+\frac{2}{\alpha}}t\big)\\
&\leq C e^{Cc_2^4t}\big(c_0^2+c_2^{6+\frac{2}{\alpha}}t\big)\leq C c_0^{2}.
\end{aligned}
\end{equation}

The estimate of $\phi_0 U_{xx}$ can be deduced from \eqref{eq3329} and \eqref{3'}, that is,
\begin{equation}\label{4''}
\abs{\phi_0 U_{xx}(t)}_2\leq C\big(c_0+c_0^2+c_0^{1+\frac{1}{\alpha}}\big)\leq C c_0^{\frac{2}{\alpha}} \ \ \text{for all }0\leq t\leq T_1.
\end{equation}

Collecting estimates \eqref{1}, \eqref{3'} and \eqref{4''} completes the proof of Lemma \ref{c_0-c_1}.
\end{proof}

\begin{Lemma}\label{c_1-c_2}
Under the same assumptions of Lemma \ref{a-priori-estimates}, it holds that
\begin{equation*}
\begin{aligned}
\absf{\phi_0 U_{tx}(t)}_2+\absf{\phi_0 \partial_x^3 U(t)}_2+\abs{\phi_0 U_{tt} (t)}_2+\abs{\phi_0\partial_t U_{xx}(t)}_2+\absf{\phi_0 \partial_x^4 U(t)}_2\leq Cc_1^{3+\frac{2}{\alpha}},
\end{aligned}
\end{equation*}
for all $0\leq t\leq T_2=\min\{\widetilde T, (1+Cc_2)^{-\frac{6\alpha+6}{\alpha}}\}$.
\end{Lemma}

\begin{proof}
\underline{\textbf{Step 1: Estimates of $\phi_0 U_{tx}$.}} Multiplying $\eqref{lp''}_1$ by $\bar\eta_x^2$, applying $U_{tt}\partial_t$ to the resulting equality and then integrating it over $I$, one can get from \eqref{38}, \eqref{useful1}-\eqref{useful3}, Lemma \ref{c_0-c_1}, H\"older's inequality, Young's inequality that for all $0\leq t\leq T_1$,
\begin{align}
&\frac{1}{2}\frac{\mathrm{d}}{\dt}\int \phi_0^{2} U_{tx}^2\,\dx +\int \phi_0^{2} \bar\eta_x^2 U_{tt}^2\,\dx\notag\\
=&\underline{\Big(\frac{1}{\alpha}-2\Big) \int \phi_0 (\phi_0)_x U_{tx}U_{tt}\,\dx}_{:=\cL_4} -2\int \phi_0^{2} \bar\eta_x \bar U_x U_t U_{tt}\,\dx\notag\\
&+2\int \frac{\phi_0^{2+\frac{1}{\alpha}} \bar U_{xx} U_{tt}}{\bar\eta_x}\,\dx-2\int \frac{\phi_0^{2+\frac{1}{\alpha}} \bar\eta_{xx} \bar U_x U_{tt}}{\bar\eta_x^2}\,\dx-2\int \frac{\phi_0^{2} \bar U_{xx} U_x U_{tt}}{\bar\eta_x}\,\dx\notag\\
& -2\int \frac{\phi_0^{2}\bar\eta_{xx} U_{tx} U_{tt}}{\bar\eta_x}\,\dx+2\int\frac{\phi_0^{2} \bar\eta_{xx} \bar U_x U_x U_{tt}}{\bar\eta_x^2}\,\dx\label{eq:cL17-cL23}\\
\leq & \cL_4+C\big(\absf{\bar U_x}_\infty \abs{\phi_0 U_t}_2+ \abs{\phi_0}_\infty^\frac{\alpha+1}{\alpha}\absf{\bar U_{xx}}_\infty+\abs{\phi_0}_\infty^\frac{\alpha+1}{\alpha}\absf{\bar U_x}_\infty\abs{\bar\eta_{xx}}_\infty\big)\abs{\phi_0 U_{tt}}_2\notag\\
&+C\big(\absf{\bar U_{xx}}_\infty \abs{\phi_0 U_x}_2 +\abs{\bar\eta_{xx}}_\infty \abs{\phi_0 U_{tx}}_2+\absf{\bar U_{x}}_\infty\abs{\bar\eta_{xx}}_\infty \abs{\phi_0 U_x}_2\big)\abs{\phi_0 U_{tt}}_2\notag\\
\leq &\cL_4 + Cc_2^2\big(\abs{\phi_0 U_{tx}}_2^2+ Cc_2^{4+\frac{4}{\alpha}}\big)+\frac{1}{8}\abs{\phi_0 \bar\eta_x U_{tt}}_2^2.\notag
\end{align}

For $\cL_{4}$, it follows from \eqref{38}, \eqref{useful2}, Lemma \ref{GNinequality}, H\"older's inequality and Young's inequality that for all $\varepsilon_1,\varepsilon_2\in(0,1)$,
\begin{align}
\cL_{4}&=\Big(\frac{1}{\alpha}-2\Big) \int \phi_0 (\phi_0)_xU_{tx}U_{tt}\,\dx\leq C \abs{(\phi_0)_x}_{\infty}\absb{\phi_0^\frac{1}{2} U_{tx}}_2\absb{\phi_0^\frac{1}{2} U_{tt}}_2\notag\\
&\leq C c_0\big(\abs{\phi_0 U_{tx}}_2+\abs{\phi_0 U_{tx}}_2^\frac{1}{2}\abs{\phi_0 \partial_tU_{xx}}_2^\frac{1}{2}\big)\big(\abs{\phi_0 U_{tt}}_2 +\abs{\phi_0 U_{tt}}_2^\frac{1}{2}\absf{\phi_0 \partial_t^2 U_{x}}_2^\frac{1}{2}\big)\label{cL4}\\
&\leq C(\varepsilon_1,\varepsilon_2)c_0^2\abs{\phi_0 U_{tx}}_2^2+\frac{1}{8}\abs{\phi_0 \bar\eta_x U_{tt}}_2^2 +\varepsilon_1\abs{\phi_0 \partial_t U_{xx}}_2^2+\varepsilon_2\absf{\phi_0 \partial_t^2 U_{x}}_2^2.\notag
\end{align}

Therefore, it follows from \eqref{eq:cL17-cL23}-\eqref{cL4} and  \eqref{useful2} that for all $\varepsilon_1,\varepsilon_2\in(0,1)$,
\begin{equation}\label{2.51}
\begin{aligned}
\frac{\mathrm{d}}{\dt}\abs{\phi_0 U_{tx}}^2_2+\abs{\phi_0 U_{tt}}_{2}^2&\leq C(\varepsilon_1,\varepsilon_2)c_2^2\abs{\phi_0 U_{tx}}^2_2  +\varepsilon_1\abs{\phi_0 \partial_t U_{xx}}_2^2\\
&\quad+\varepsilon_2\abs{\phi_0 \partial_t^2 U_{x}}_2^2+Cc_2^{4+\frac{4}{\alpha}}.
\end{aligned}
\end{equation}

To estimate $\phi_0 \partial_t U_{xx}$, one can take the $L^2$-norm of both sides of \eqref{utxx-phi0} and use \eqref{38}-\eqref{39}, \eqref{useful1}-\eqref{useful3}, and Lemmas \ref{c_0-c_1} and \ref{sobolev-embedding}-\ref{hardy-inequality} to conclude that for all $0\leq t\leq T_1$,
\begin{equation*}
\begin{aligned}
&\absb{\phi_0 \partial_tU_{xx}+\frac{1}{\alpha} (\phi_0)_x U_{tx}}_2\\
\leq &C\big(\abs{\phi_0 U_{tt}}_2+\absf{\bar U_x}_\infty \abs{\phi_0 U_{t}}_2+\abs{\phi_0}_\infty^\frac{1}{\alpha}\absf{\phi_0 \bar U_{xx}}_2+\abs{\phi_0}_\infty^\frac{1}{\alpha} \abs{\bar\eta_{xx}}_\infty \absf{\phi_0 \bar U_x}_2\big)\\
&+C\left(\absf{\bar U_{xx}}_\infty \abs{\phi_0 U_x}_2+\absf{\bar U_x}_\infty \abs{\bar \eta_{xx}}_\infty \abs{\phi_0  U_x}_2 +\abs{\bar\eta_{xx}}_\infty \abs{\phi_0 U_{tx}}_2\right)\\
\leq & C\big(\abs{\phi_0 U_{tt}}_2+c_2 \abs{\phi_0 U_{tx}}_2+c_0^\frac{2}{\alpha}c_2^2\big),
\end{aligned}
\end{equation*}
which, together with Proposition \ref{prop2.1}, leads to
\begin{equation}\label{U_txx}
\begin{aligned}
\abs{\phi_0 \partial_tU_{xx}}_2&\leq   C\Big(\absb{\phi_0 \partial_tU_{xx}+\frac{1}{\alpha}\phi_0 (\phi_0)_x U_{tx}}_2+\abs{(\phi_0)_{xx}}_\infty \abs{\phi_0  U_{tx}}_2\Big)\\
&\leq C\big( \abs{\phi_0 U_{tt}}_2+c_2 \abs{\phi_0 U_{tx}}_2+c_0^\frac{2}{\alpha}c_2^2\big).
\end{aligned}
\end{equation}

Hence, plugging \eqref{U_txx} into \eqref{2.51} and choosing $\varepsilon_1$ sufficiently small yield that
\begin{equation}\label{Utx-theta}
\begin{aligned}
\frac{\mathrm{d}}{\dt}\abs{\phi_0 U_{tx}}^2_2+\abs{\phi_0 U_{tt}}_{2}^2\leq C(\varepsilon)c_2^2\abs{\phi_0 U_{tx}}^2_2+\varepsilon\abs{\phi_0 \partial_t^2 U_{x}}_2^2+Cc_2^{4+\frac{4}{\alpha}},
\end{aligned}
\end{equation}
for all $\varepsilon\in (0,1)$ and all $0\leq t\leq T_1$.

\underline{\textbf{Step 2: Estimate of $\phi_0 U_{tt}$.}} Since one has already shown that $U_{tt}$ is the weak solution to the problem \eqref{33155}, then according to Proposition \ref{prop1}, $U_{tt}$ satisfies the weak formulation \eqref{weak.F.}. Thanks to \eqref{33156}, the test function $\varphi$ in \eqref{weak.F.} can be replaced by $U_{tt}$; then, according to H\"older's inequality and Young's inequality, one deduces from Lemma \ref{hardy-inequality} that 
\begin{equation}\label{2.58}
\begin{aligned}
&\frac{1}{2}\frac{\mathrm{d}}{\dt}\int \phi_0^{2} U_{tt}^2\,\dx+\int \frac{\phi_0^{2} (\partial_t^2 U_x)^2}{\bar\eta_x^2}\,\dx\\
=&\underline{\Big(\frac{1}{\alpha}-2\Big)\int \frac{\phi_0 (\phi_0)_x \partial_t^2 U_x U_{tt}}{\bar\eta_x^2}\,\dx}_{:=\cL_5}+\int \phi_0 P_1^{(2)}\partial_t^2 U_{x}\,\dx+\int  P_2^{(2)} U_{tt}\,\dx\\
\leq & \cL_5+ C\big(\absf{P_1^{(2)}}_2^2+\absf{P_2^{(2)}}_2^2+\abs{\phi_0 U_{tt}}_2^2\big)+\frac{1}{8} \absf{\phi_0\partial_t^2 U_x}_2^2,
\end{aligned}
\end{equation}
where $(P_1^{(2)},P_2^{(2)})$ is defined as in \eqref{P1P2}.

For $\cL_{5}$,  it follows from \eqref{38}, \eqref{useful2}, Lemma \ref{GNinequality} and Young's inequality that
\begin{align}
\cL_{5}
&=\underline{-\frac{1-2\alpha}{2\alpha}\int \frac{((\phi_0)_x)^2 U^2_{tt}}{\bar\eta_x^2}\,\dx}_{\leq 0}-\frac{1-2\alpha}{2\alpha}\int \phi_0  \left(\frac{(\phi_0)_{x}}{\bar\eta_x^2}\right)_x U^2_{tt}\,\dx\notag\\
&\leq C\big(\abs{(\phi_0)_{xx}}_\infty^2+\abs{(\phi_0)_{x}}_\infty^2\abs{\bar\eta_{xx}}_\infty^2\big) \abs{\phi_0 U_{tt}}_2^2+\frac{1}{8}\abs{\phi_0 \partial_t^2 U_x}_2^2\label{cL24}\\
&\leq Cc_2^4\abs{\phi_0 U_{tt}}_2^2+\frac{1}{8}\abs{\phi_0 \partial_t^2 U_x}_2^2.\notag
\end{align}

For  the $L^2$-norms of $(P_1^{(2)},P_2^{(2)})$, it follows from \eqref{38}-\eqref{39}, \eqref{useful1}-\eqref{useful2}, Lemmas \ref{c_0-c_1} and \ref{sobolev-embedding}-\ref{hardy-inequality} that
\begin{align}
\absb{P_1^{(2)}}_2&\leq C\big(\abs{\phi_0}_\infty^\frac{1}{2}\absb{\phi_0^\frac{1}{2} \bar U_{tx}}_\infty \abs{U_{x}}_2+\absf{\bar U_x}_\infty \abs{\phi_0 U_{tx}}_2+\absf{\bar U_x}_\infty^2 \abs{\phi_0U_{x}}_2\big)\notag\\
&\quad +C\big(\abs{\phi_0}_\infty^\frac{1}{\alpha}\absf{\phi_0 \bar U_{tx}}_2+\abs{\phi_0}_\infty^\frac{1+\alpha}{\alpha}\absf{\bar U_x}_\infty^2\big)\\
&\leq Cc_2^2\big(\abs{\phi_0 U_{x}}_2+\abs{\phi_0 U_{xx}}_2+\abs{\phi_0 U_{tx}}_2+c_2^{1+\frac{1}{\alpha}}\big)\leq  C\big(c_2^2\abs{\phi_0 U_{tx}}_2+c_2^{3+\frac{2}{\alpha}}\big),\notag
\end{align}
\begin{align}
\absb{P_2^{(2)}}_2&\leq C\abs{(\phi_0)_x}_\infty\big(\abs{\phi_0}_\infty^\frac{1}{\alpha} \abs{\phi_0\bar U_{tx}}_2+\abs{\phi_0}_\infty^\frac{1+\alpha}{\alpha} \abs{\bar U_{x}}_\infty^2+\abs{\bar U_{x}}_\infty^2\abs{\phi_0 U_x}_2\big)\notag\\
&\quad +C\abs{(\phi_0)_x}_\infty\big(\abs{\phi_0}_\infty^\frac{1}{2}\absb{\phi_0^\frac{1}{2} \bar U_{tx}}_\infty \abs{U_{x}}_2+\abs{\bar U_x}_\infty \abs{\phi_0U_{tx}}_2\big)\label{cL29}\\
&\leq Cc_2^3\big(\abs{\phi_0 U_{x}}_2+\abs{\phi_0 U_{xx}}_2+\abs{\phi_0 U_{tx}}_2+c_2^{\frac{2}{\alpha}}\big)\leq  C\big(c_2^3\abs{\phi_0 U_{tx}}_2+c_2^{3+\frac{2}{\alpha}}\big).\notag
\end{align}

Thus, it follows from \eqref{useful2} and \eqref{2.58}-\eqref{cL29} that
\begin{equation*}
\frac{\mathrm{d}}{\dt}\abs{\phi_0 U_{tt}}_{2}^2+\absf{\phi_0 \partial_t^2U_x}_{2}^2\leq C\big(c_2^6\abs{\phi_0 U_{tx}}_2^2+\abs{\phi_0 U_{tt}}_2^2+c_2^{6+\frac{4}{\alpha}}\big),
\end{equation*}
which, along with \eqref{useful2}, \eqref{Utx-theta} with $\varepsilon$ suitably small, and Gr\"onwall inequality, yields that
\begin{equation}\label{2.66}
\begin{aligned}
\abs{\phi_0 U_{tx}(t)}_{2}^2+\abs{\phi_0 U_{tt}(t)}_{2}^2
&\leq C e^{Cc_2^6t}\big(\abs{\phi_0 U_{tx}(0)}_2^2+\abs{\phi_0 U_{tt}(0)}_2^2+c_2^{6+\frac{4}{\alpha}}t\big)\\
&\leq C e^{Cc_2^6t}\big(c_0^2+c_2^{6+\frac{4}{\alpha}}t\big)\leq C c_0^2,
\end{aligned}
\end{equation}
for all $0\leq t\leq T_2:=\min\{\widetilde T, (1+Cc_2)^{-\frac{6\alpha+6}{\alpha}}\}$.

\underline{\textbf{Step 3: Estimates of $\phi_0 \partial_x^3 U$.}} Taking the $L^2$-norm of both sides of \eqref{0441}, one can get from \eqref{38}, \eqref{useful1}-\eqref{useful3}, \eqref{2.66}, Lemmas \ref{c_0-c_1} and \ref{hardy-inequality} that for all $0\leq t\leq T_2$,
\begin{align}
&\absb{\phi_0 \partial_x^3 U +\Big(\frac{1}{\alpha}+1\Big)  (\phi_0)_x U_{xx}}_2\notag\\
\leq& C\left(\abs{(\phi_0)_{xx}}_\infty\abs{U_x}_2+\abs{\phi_0 U_{tx}}_2+\abs{\bar\eta_{xx}}_\infty \abs{\phi_0 U_t}_2\right)\notag\\
& +C\big(\abs{(\phi_0)_x}_\infty \abs{U_t}_2+ \abs{\phi_0}_\infty^\frac{1}{\alpha}\abs{(\phi_0)_x}_\infty \abs{\bar\eta_{xx}}_\infty+ \abs{\phi_0}_\infty^\frac{1}{\alpha}\absf{\phi_0 \partial_x^3\bar\eta}_2\big)\notag\\
& + C\big(\abs{\phi_0}_\infty^\frac{\alpha+1}{\alpha}\abs{\bar\eta_{xx}}_\infty^2+\abs{\phi_0}_\infty^\frac{1-\alpha}{\alpha}\abs{(\phi_0)_x}_\infty^2+\abs{\phi_0}_\infty^\frac{1}{\alpha}\abs{(\phi_0)_{xx}}_2\big)\notag\\
& + C\left(\abs{(\phi_0)_x}_\infty\abs{\bar\eta_{xx}}_\infty\abs{U_x}_2+\abs{\bar\eta_{xx}}_\infty\abs{\phi_0 U_{xx}}_2\right)\label{3369}\\
& + C\big(\absf{\partial_x^3\bar\eta}_2\abs{\phi_0U_x}_\infty+\abs{\bar\eta_{xx}}_\infty^2\abs{\phi_0 U_{x}}_2\big)\notag\\
\leq &Cc_0\left(\abs{\phi_0 U_x}_2+\abs{\phi_0 U_{xx}}_2\right)+Cc_0\left(\abs{\phi_0 U_t}_2+\abs{\phi_0 U_{tx}}_2\right)+C\big(c_0^{1+\frac{1}{\alpha}}+c_0^{1+\frac{1}{\alpha}}c_2^2t^2\big)\notag\\
\leq &C\big(c_0^{1+\frac{2}{\alpha}}+c_0^{2}+c_0^{1+\frac{1}{\alpha}}\big) \leq Cc_0^{1+\frac{2}{\alpha}},\notag
\end{align}
which, along with Proposition \ref{prop2.1} and Lemma \ref{c_0-c_1}, leads to
\begin{equation}\label{3order}
\begin{aligned}
\absf{\phi_0 \partial_x^3 U}_2&\leq C \Big(\absb{\phi_0 \partial_x^3 U+\Big(\frac{1}{\alpha}+1\Big)(\phi_0)_x U_{xx}}_2+ \abs{(\phi_0)_{xx}}_\infty \abs{\phi_0 U_{xx}}_2\Big)\leq Cc_0^{1+\frac{2}{\alpha}}.
\end{aligned}
\end{equation}

\underline{\textbf{Step 4: Estimate of $\phi_0 \partial_tU_{xx}$.}} First, it follows from \eqref{2.66}, \eqref{3order}, Lemmas \ref{c_0-c_1} and \ref{sobolev-embedding}-\ref{hardy-inequality} that for all $t\in [0, T_2]$,
\begin{equation}\label{est-Ux}
\begin{aligned}
&\abs{\phi_0 U_t(t)}_\infty\leq C\sum_{j=0}^1\absb{\phi_0^\frac{3}{2} \partial_x^j U_{t}(t)}_2 \leq C\abs{\phi_0}_\infty^\frac{1}{2}\sum_{j=0}^1 \absf{\phi_0 \partial_x^j U_{t}(t)}_2 \leq C c_0^{1+\frac{2}{\alpha}},\\
&\abs{U_x(t)}_\infty \leq C \abs{U_{xx}(t)}_2\leq C\sum_{j=2}^3 \absf{\phi_0 \partial_x^j U(t)}_2 \leq C c_0^{1+\frac{2}{\alpha}}.
\end{aligned}
\end{equation}

Next, due to \eqref{utxx-phi0}, it follows from \eqref{38}-\eqref{39}, \eqref{useful1}-\eqref{useful3}, \eqref{2.66}, \eqref{est-Ux}, Lemmas \ref{c_0-c_1} and \ref{sobolev-embedding}-\ref{hardy-inequality} that for all $t\in [0,T_2]$,
\begin{equation}\label{equ3.128}
\begin{aligned}
&\absb{\phi_0 \partial_tU_{xx}+\frac{1}{\alpha} (\phi_0)_x U_{tx}}_2\\
\leq &C\big(\abs{\phi_0 U_{tt}}_2+\absf{\bar U_x}_2 \abs{\phi_0 U_{t}}_\infty+\abs{\phi_0}_\infty^\frac{1}{\alpha}\absf{\phi_0 \bar U_{xx}}_2+\abs{\phi_0}_\infty^\frac{1}{\alpha} \abs{\bar\eta_{xx}}_\infty \absf{\phi_0 \bar U_x}_2\big)\\
&+C\left(\absf{\phi_0 \bar U_{xx}}_2 \abs{U_x}_\infty+\absf{\bar U_x}_\infty \abs{\bar \eta_{xx}}_\infty \abs{\phi_0 U_x}_2 +\abs{\bar\eta_{xx}}_\infty \abs{\phi_0 U_{tx}}_2\right)\\
\leq  &C\big(c_0 +c_0^{1+\frac{2}{\alpha}}c_1+c_2^{2+\frac{2}{\alpha}}t\big)\leq C c_1^{2+\frac{2}{\alpha}},
\end{aligned}
\end{equation}
which, along with Proposition \ref{prop2.1} and \eqref{2.66}, leads to
\begin{equation}\label{U_txx'}
\begin{aligned}
\abs{\phi_0 \partial_tU_{xx}}_2&\leq C\Big( \absb{\phi_0 \partial_tU_{xx}+\frac{1}{\alpha} (\phi_0)_x U_{tx}}_2+\abs{(\phi_0)_{xx}}_\infty\abs{\phi_0 U_{tx}}_2\Big)\leq C c_1^{2+\frac{2}{\alpha}}.
\end{aligned}
\end{equation}

\underline{\textbf{Step 5: Estimate of $\phi_0\partial_x^4 U$.}} According to \eqref{0444}, it holds that
\begin{align}
&\phi_0 \partial_x^4 U+\Big(\frac{1}{\alpha}+2\Big) (\phi_0)_x\partial_x^3 U\notag\\
=&\underline{\phi_0 \bar\eta_x^2\partial_tU_{xx}+2\left((\phi_0)_x\bar\eta_x+2\phi_0\bar\eta_{xx}\right)\bar\eta_x U_{tx}+2\phi_0 \left(\bar\eta_{xx}^2+\bar\eta_x\partial_x^3\bar\eta\right)U_t}\notag\\
&\underline{+ \left((\phi_0)_{xx}\bar\eta_x+4(\phi_0)_x\bar\eta_{xx}\right)\bar\eta_xU_t}_{:=\cL_{6}}\underline{-\Big(\frac{2}{\alpha}+1\Big) (\phi_0)_{xx} U_{xx}-\frac{1}{\alpha} \partial_x^3\phi_0 U_x}_{:=\cL_{7}}\notag\\
&\underline{+2 \left((\phi_0)_{xx}\bar\eta_{xx}+2(\phi_0)_{x}\partial_x^3\bar\eta -\frac{2(\phi_0)_{x}\bar\eta_{xx}^2}{\bar\eta_x}\right)\frac{U_x}{\bar\eta_x}}_{:=\cL_{8}}\notag\\
&+\underline{2\phi_0 \left(\partial_x^4\bar\eta
-\frac{3\bar\eta_{xx}\partial_x^3\bar\eta}{\bar\eta_x}
+\frac{2\bar\eta_{xx}^3}{\bar\eta_x^2}\right)\frac{U_x}{\bar\eta_x}}_{:=\cL_{9}}\label{cL32}\\
&+\underline{4 \left(\phi_0\partial_x^3\bar\eta+(\phi_0)_{x}\bar\eta_{xx}-\frac{\phi_0\bar\eta_{xx}^2}{\bar\eta_x}\right)\frac{U_{xx}}{\bar\eta_x}+\frac{2\phi_0 \bar\eta_{xx}\partial_x^3U}{\bar\eta_x}}_{:=\cL_{10}}\notag\\
&\underline{-\frac{2+2\alpha}{\alpha^2}\frac{\phi_0^{\frac{1}{\alpha}-1}((\phi_0)_x)^2\bar\eta_{xx}}{\bar\eta_x}-2\phi_0^{1+\frac{1}{\alpha}}\left(\frac{\partial_x^4\bar\eta}{\bar\eta_x}-\frac{3\bar\eta_{xx}\partial_x^3\bar\eta}{\bar\eta_x^2}+\frac{2\bar\eta_{xx}^3}{\bar\eta_x^3}\right)}_{:=\cL_{11}}\notag\\
& \underline{-\frac{2+2\alpha}{\alpha}\phi_0^{\frac{1}{\alpha}}\left(\frac{(\phi_0)_{xx}\bar\eta_{xx}}{\bar\eta_x}+\frac{2(\phi_0)_x\partial_x^3\bar\eta}{\bar\eta_x}-\frac{2(\phi_0)_x\bar\eta_{xx}^2}{\bar\eta_x^2}\right)}_{:=\cL_{12}}\notag\\
&+\underline{\frac{2(1-\alpha)}{\alpha^3}\phi_0^{\frac{1}{\alpha}-2}((\phi_0)_x)^3+\frac{6}{\alpha^2}\phi_0^{\frac{1}{\alpha}-1}(\phi_0)_x(\phi_0)_{xx}+\frac{2}{\alpha}\phi_0^{\frac{1}{\alpha}}\partial_x^3\phi_0}_{:=\cL_{13}}.\notag
\end{align}

It follows from $0<\alpha\leq \frac{1}{3}$, \eqref{38}-\eqref{39}, \eqref{useful1}-\eqref{useful3}, \eqref{2.66}, \eqref{3order}-\eqref{est-Ux}, \eqref{U_txx'}, Lemmas \ref{c_0-c_1} and \ref{sobolev-embedding}-\ref{hardy-inequality} that for all $t\in [0,T_2]$,  
\begin{align}
\abs{\cL_{6}}_2&\leq C\left( \abs{\phi_0 \partial_t U_{xx}}_2+\abs{(\phi_0)_x}_\infty\abs{U_{tx}}_2+\abs{\bar\eta_{xx}}_\infty \abs{\phi_0U_{tx}}_2\right)\notag\\
&\quad +C\left(\abs{(\phi_0)_{xx}}_\infty+\abs{(\phi_0)_{x}}_\infty\abs{\bar\eta_{xx}}_\infty\right)\abs{U_{t}}_2\notag\\
&\quad + C\big(\abs{\bar\eta_{xx}}_\infty^2 \abs{\phi_0 U_t}_2 + \absf{\partial_x^3\bar\eta}_2 \abs{\phi_0U_t}_\infty\big)\notag\\
&\leq C\big(\norm{\phi_0}_3+\abs{\bar\eta_{xx}}_\infty+\norm{\phi_0}_3^2+\abs{\bar\eta_{xx}}_\infty^2+\absf{\partial_x^3\bar\eta}_2\big)\left(\abs{\phi_0U_{t}}_2+\abs{\phi_0U_{tx}}_2\right) \notag\\
&\quad +C(1+\norm{\phi_0}_3)\abs{\phi_0\partial_t U_{xx}}_2\notag\\
&\leq C(c_0^2+c_2^4t)c_0^{1+\frac{2}{\alpha}}+Cc_0c_1^{2+\frac{2}{\alpha}}\leq C c_1^{3+\frac{2}{\alpha}}, \notag\\
\abs{\cL_{7}}_2&\leq C \norm{\phi_0}_3(\abs{U_x}_\infty+\abs{\phi_0U_{xx}}_2+\absf{\phi_0\partial_x^3 U}_2)\leq C c_0^{2+\frac{2}{\alpha}},\notag\\
\abs{\cL_{8}}_2&\leq C\big(\abs{\bar\eta_{xx}}_\infty +\abs{\bar\eta_{xx}}_\infty^2\big)\norm{\phi_0}_3\abs{U_x}_2+C\abs{(\phi_0)_x}_\infty\absf{\partial_x^3\bar\eta}_2\abs{U_x}_\infty\notag\\
&\leq C c_0c_2^2t(\absf{\phi_0 U_x}_2+\absf{\phi_0 U_{xx}}_2) +C c_0^{2+\frac{2}{\alpha}}c_2t\leq C c_0^{2+\frac{2}{\alpha}}c_2^2t\leq C c_0,\notag\\
\abs{\cL_{9}}_2&\leq C \big(\absf{\phi_0\partial_x^4 \bar\eta}_2+\abs{\bar\eta_{xx}}_\infty\abs{\phi_0\partial_x^3\bar\eta}_2+\abs{\bar\eta_{xx}}_\infty^3\big)\abs{U_x}_\infty\notag\leq C c_0^{1+\frac{2}{\alpha}}c_2^3t\leq C c_0, \notag\\
\abs{\cL_{10}}_2&\leq C\big(\absf{\partial_x^3\bar\eta}_2+\abs{\bar\eta_{xx}}_\infty^2\big)\abs{\phi_0 U_{xx}}_\infty + C\abs{(\phi_0)_x}_\infty \abs{\bar\eta_{xx}}_\infty^2 \abs{U_{xx}}_2\label{cL37}\\
&\quad +C\abs{\bar\eta_{xx}}_\infty \absf{\phi_0\partial_x^3 U}_2
\leq C c_0^{2+\frac{2}{\alpha}}c_2^2t\leq C c_0,\notag \\
\abs{\cL_{11}}_2&\leq C \abs{\phi_0}_\infty^\frac{1-\alpha}{\alpha}\norm{\phi_0}_3\abs{\bar\eta_{xx}}_\infty+C\abs{\phi_0}_\infty^\frac{1}{\alpha}\big(\absf{\phi_0 \partial_x^4\bar\eta}_2+\abs{\bar\eta_{xx}}_\infty\absf{\phi_0 \partial_x^3\bar\eta}_2\big)\notag\\
&\quad +C\abs{\phi_0}_\infty^\frac{1}{\alpha}\big(\abs{\bar\eta_{xx}}_\infty^2\abs{\phi_0 \bar\eta_{xx}}_2\big)\leq C c_0^\frac{1}{\alpha}c_2^3t\leq Cc_0,\notag\\
\abs{\cL_{12}}_2&\leq C \abs{\phi_0}_\infty^\frac{1}{\alpha}\norm{\phi_0}_3\big(\abs{\bar\eta_{xx}}_\infty+\absf{\partial_x^3\bar\eta}_2+\abs{\bar\eta_{xx}}_\infty^2\big)
\leq C c_0^{1+\frac{1}{\alpha}}c_2^2t\leq C c_0,\notag\\
\abs{\cL_{13}}_2
&\leq C\big(\abs{\phi_0}_{\infty}^{\frac{1-2\alpha}{\alpha}}\abs{(\phi_0)_x}_\infty^3+\abs{\phi_0}_\infty^\frac{1-\alpha}{\alpha}\abs{(\phi_0)_x}_\infty\abs{(\phi_0)_{xx}}_\infty\big)\notag\\
&\quad +C\abs{\phi_0}_\infty^\frac{1}{\alpha}\norm{\phi_0}_3
\leq C c_0^{1+\frac{1}{\alpha}}.\notag    
\end{align}

Therefore, substituting \eqref{cL37} into \eqref{cL32} yields
\begin{equation*}
\absb{\phi_0 \partial_x^4 U+\Big(\frac{1}{\alpha}+2\Big) (\phi_0)_x\partial_x^3 U}_2\leq C c_1^{3+\frac{2}{\alpha}}.
\end{equation*}
It follows from  Proposition \ref{prop2.1} and \eqref{3order} that 
\begin{equation*}
\begin{aligned}
\absf{\phi_0 \partial_x^4 U}_2&\leq C\Big(\absb{\phi_0 \partial_x^4 U+\Big(\frac{1}{\alpha}+2\Big) (\phi_0)_x\partial_x^3 U}_2+\abs{(\phi_0)_{xx}}_\infty \absf{\phi_0 \partial_x^3 U}_2\Big)
\leq  Cc_1^{3+\frac{2}{\alpha}},
\end{aligned}
\end{equation*}
which, along with  \eqref{2.66}, \eqref{3order} and \eqref{U_txx'}, completes the proof of Lemma \ref{c_1-c_2}.
\end{proof}

Now, we can turn back to prove Lemma \ref{a-priori-estimates}.
\begin{proof}
Collecting the bounds in Lemmas \ref{c_0-c_1}-\ref{c_1-c_2}, one has
\begin{equation*}
\begin{aligned}
\abs{\phi_0 U(t)}_{2}+\abs{\phi_0U_x(t)}_{2}+
\abs{\phi_0U_t(t)}_{2}+\abs{\phi_0 U_{xx}(t)}_{2}&\leq Cc_0^{\frac{2}{\alpha}},\\
\abs{\phi_0U_{tx}(t)}_{2}+\absf{\phi_0\partial^3_x U(t)}_{2} +\abs{\phi_0U_{tt}(t)}_{2}+\abs{\phi_0\partial_t U_{xx}(t)}_{2}+\absf{\phi_0\partial^4_x U(t)}_{2}&\leq Cc_1^{3+\frac{2}{\alpha}}.
\end{aligned}
\end{equation*}
Thus, defining the constants $c_1,c_2$ and the time $T$ as
\begin{equation}\label{3.84}
\begin{aligned}
c_1&=Cc_0^\frac{2}{\alpha},\ \  
c_2=Cc_1^{3+\frac{2}{\alpha}}=C^{4+\frac{2}{\alpha}}c_0^\frac{6\alpha+4}{\alpha^2},\ \
T=\min\{T_1,T_2\}\leq (1+Cc_2)^{-\frac{6\alpha+6}{\alpha}},
\end{aligned}
\end{equation}
one can arrive at the uniform estimates \eqref{uniform bounds}.
The proof of Lemma \ref{a-priori-estimates} is completed.
\end{proof}

\subsection{Uniform estimates for the case \texorpdfstring{$\frac{1}{3}<\alpha\leq 1$}{}}\label{subsection4.2}

Analogous to Lemma \ref{a-priori-estimates}, one can establish the following uniform bounds.
\begin{Lemma}\label{a-priori-estimates2}
Consider \eqref{lp}. Assume that  $U$ is  the  unique classical solution in $[0,T]\times \bar I$ to the problem \eqref{lp''} obtained in Lemma \ref{existence-linearize}, and the positive constant $c_0$ satisfies 
\begin{equation}\label{38''}
2+\norm{\phi_0}_{3}+\widetilde E(0,U)\leq c_0.
\end{equation}
Then there exist a positive time $T_*\in (0,T]$ and constants $c_i$, $i=1,2$, $1<c_0\leq c_1\leq c_2$, which depend only on $c_0$, $\alpha$, $\varepsilon_0$, $|I|$, $\cC_1$ and  $\cC_2$, such that if for all $0\leq t\leq T_*$,
\begin{equation}\label{39''}
\begin{aligned}
\absb{\phi_0^\frac{1}{2\alpha} \bar U(t)}_{2}+\absb{\phi_0^\frac{1}{2\alpha} \bar U_x(t)}_{2}+\absb{\phi_0^\frac{1}{2\alpha} \bar U_t(t)}_{2}+\absb{\phi_0^{\frac{3}{2}-\varepsilon_0} \bar U_{xx}(t)}_{2}\leq c_1,\\[4pt]
\absb{\phi_0^\frac{1}{2\alpha} \bar U_{tx}(t)}_{2}+\absb{\phi_0^{\frac{3}{2}-\varepsilon_0} \partial^3_x\bar U(t)}_{2}\leq c_2,\\[4pt]
\absb{\phi_0^\frac{1}{2\alpha} \bar U_{tt}(t)}_{2}+\absb{\phi_0^{\frac{3}{2}-\varepsilon_0} \partial_t\bar U_{xx}(t)}_{2}+\absb{\phi_0^{\frac{3}{2}-\varepsilon_0} \partial^4_x\bar U(t)}_{2}^2\leq c_2,
\end{aligned}
\end{equation}
it holds that for all $0\leq t\leq T_*$,\begin{equation}\label{uniform bounds''}
\begin{aligned}
\absb{\phi_0^\frac{1}{2\alpha} U(t)}_{2}+\absb{\phi_0^\frac{1}{2\alpha} U_x(t)}_{2}+\absb{\phi_0^\frac{1}{2\alpha} U_t(t)}_{2}+\absb{\phi_0^{\frac{3}{2}-\varepsilon_0}  U_{xx}(t)}_{2}\leq c_1,\\[4pt]
\absb{\phi_0^\frac{1}{2\alpha} U_{tx}(t)}_{2}+\absb{\phi_0^{\frac{3}{2}-\varepsilon_0} \partial^3_x U(t)}_{2}\leq c_2,\\[4pt]
\absb{\phi_0^\frac{1}{2\alpha} U_{tt}(t)}_{2}+\absb{\phi_0^{\frac{3}{2}-\varepsilon_0} \partial_t U_{xx}(t)}_{2}+\absb{\phi_0^{\frac{3}{2}-\varepsilon_0} \partial^4_x U(t)}_{2}\leq c_2.
\end{aligned}
\end{equation}
\end{Lemma}

\begin{proof}
We only  sketch the proof here. Denote by $\mathscr{P}(\cdot)$ the generic polynomial functions with the form $\mathscr{P}(s)=\sum_{j=0}^n p_j s^j$, $n\in \NN^*$, $0\leq p_j\in \RR$.

\underline{\textbf{Step 1: Some useful estimates.}} First, it is convenient to give some useful estimates as in \eqref{useful1}-\eqref{useful3}. According to  Lemmas \ref{sobolev-embedding}-\ref{hardy-inequality}, one has that  for all $0\leq t\leq T$, \begin{equation}\label{useful1'}
\begin{split}
\absb{\phi_0^{\frac{1}{2}-\varepsilon_0}\bar U_{x}(t)}_2\leq C \sum_{j=1}^2 \absb{\phi_0^{\frac{3}{2}-\varepsilon_0} \partial_x^j \bar U(t)}_2\leq & \mathscr{P}(c_1),\\
\normf{\bar U_x(t)}_{1,\infty} \leq C \normf{\bar U_x(t)}_{2,1}\leq C \sum_{j=1}^4 \absb{\phi_0^{\frac{3}{2}-\varepsilon_0} \partial_x^j \bar U(t)}_2\leq & \mathscr{P}(c_2),\\
\absf{\phi_0\partial_x^3 \bar U(t)}_\infty\leq C\sum_{j=3}^4 \absb{\phi_0^\frac{3}{2}\partial_x^j \bar U(t)}_2\leq & \mathscr{P}(c_2).
\end{split}
\end{equation}

Next, one can find a small constant $\widetilde T:= (1+\mathscr{P}(c_2))^{-1}$ such that
\begin{equation}\label{useful2'}
\frac{1}{2}\leq \bar\eta_x(t,x)\leq \frac{3}{2},\quad \text{for all }(t,x)\in [0,\widetilde T]\times \bar I.
\end{equation}

Finally, it holds that for all $0\leq t\leq T$, \begin{equation}\label{useful3'}
\begin{split}
\abs{\bar\eta_{xx}(t)}_\infty \leq C\norm{\bar\eta_{xx}(t)}_{1,1}\leq C \int_0^t\normf{\bar U_{xx}}_{1,1}\leq &\mathscr{P}(c_2)t,\\
\absb{\phi_0^{\frac{1}{2}-\varepsilon_0} \bar\eta_{xx}(t)}_2+\absb{\phi_0^{\frac{1}{2}-\varepsilon_0}\partial_x^3\bar\eta(t)}_2 \leq C\sum_{j=2}^4 \absb{\phi_0^{\frac{3}{2}-\varepsilon_0} \partial_x^j U(t)}_2\leq &\mathscr{P}(c_2)t.
\end{split}
\end{equation}

\underline{\textbf{Step 2: Analogy of Lemma \ref{c_0-c_1}.}}
First, following the proof in \textbf{Step 1} of Lemma \ref{c_0-c_1}, one obtains that
\begin{equation}\label{equ3.140}
\absb{\phi_0^\frac{1}{2\alpha} U(t)}_2^2+\int_0^t \absb{\phi_0^\frac{1}{2\alpha}U_x}_2^2\,\ds \leq \mathscr{P}(c_0),
\end{equation}
for all $0\leq t\leq T_1:=\min\{\widetilde T, (1+\mathscr{P}(c_2))^{-M}\}$, for some large $M>0$.

Next, based on the proof in \textbf{Step 2} of Lemma \ref{c_0-c_1}, multiplying \eqref{lp} by $U_t$ and integrating the resulting equality over $I$,  one has
\begin{equation}\label{equ3.141}
\begin{aligned}
&\frac{1}{2}\frac{\mathrm{d}}{\dt}\int \frac{\phi_0^\frac{1}{\alpha}U_x^2}{\bar\eta_x^2}\,\dx+ \int\phi_0^\frac{1}{\alpha} U_t^2\,\dx \\
= &-\int \frac{\phi_0^\frac{1}{\alpha}\bar U_x U_x^2}{\bar\eta_x^3}\,\dx-\frac{2}{\alpha}\int \frac{\phi_0^{\frac{2}{\alpha}-1}(\phi_0)_x U_t}{\bar\eta_x^2}\,\dx+2\int \frac{\phi_0^\frac{2}{\alpha}\bar\eta_{xx} U_t}{\bar\eta_x^3}\,\dx.
\end{aligned}
\end{equation}
Note that, comparing \eqref{equ3.141} with \eqref{eq:cL6-cL9}, there is no crossing term $\left(\frac{1}{\alpha}-2\right)\int \phi_0 (\phi_0)_x U_x U_t\,\dx$
in \eqref{equ3.141}. Thus, according to \eqref{useful1'}-\eqref{useful2'}, Lemma \ref{hardy-inequality}, H\"older's inequality, Young's inequality and Gr\"onwall's inequality, one has
\begin{equation}\label{equ3.142}
\absb{\phi_0^\frac{1}{2\alpha}U_x(t)}_2^2+\int_0^t \absb{\phi_0^\frac{1}{2\alpha}U_t}_2^2\,\ds \leq \mathscr{P}(c_0) \ \ \text{for all }0\leq t\leq T_1.
\end{equation}

Similarly, applying $U_t\partial_t$ to both sides of \eqref{lp} and integrating the resulting equality over $I$ yield 
\begin{equation*}
\frac{1}{2}\frac{\mathrm{d}}{\dt}\int \phi_0^\frac{1}{\alpha} U_t^2\,\dx +\int \frac{\phi_0^\frac{1}{\alpha} U_{tx}^2}{\bar\eta_x^2}\,\dx= \int \frac{2\phi_0^\frac{1}{\alpha}\bar U_x U_x U_{tx}}{\bar \eta_x^3}\,\dx - \int \frac{2\phi_0^\frac{2}{\alpha} \bar U_x U_{tx}}{\bar \eta_x^3}\,\dx,
\end{equation*}
which, by the analogous calculations, along with \eqref{equ3.140} and \eqref{equ3.142}, implies that
\begin{equation}\label{equ3.143}
\absb{\phi_0^\frac{1}{2\alpha}U_t(t)}_2^2+\int_0^t \absb{\phi_0^\frac{1}{2\alpha}U_{tx}}_2^2\,\ds \leq\mathscr{P}(c_0) \ \ \text{for all }0\leq t\leq T_1.
\end{equation}

To get the estimate of $\phi_0^{\frac{3}{2}-\varepsilon_0}U_{xx}$, recall \eqref{3...143},
\begin{equation*}
\begin{aligned}
&\phi_0^{\frac{3}{2}-\varepsilon_0} U_{xx}+\frac{1}{\alpha}\phi_0^{\frac{1}{2}-\varepsilon_0}(\phi_0)_xU_x\\
=&\phi_0^{\frac{3}{2}-\varepsilon_0}\bar\eta_x^2 U_t+\frac{2\phi_0^{\frac{3}{2}-\varepsilon_0}\bar\eta_{xx}U_x}{\bar\eta_x}+\frac{2}{\alpha}\phi_0^{\frac{1}{2}+\frac{1}{\alpha}-\varepsilon_0}(\phi_0)_x- \frac{2\phi_0^{\frac{3}{2}+\frac{1}{\alpha}-\varepsilon_0}\bar\eta_{xx}}{\bar\eta_x}.
\end{aligned}
\end{equation*}
One can take the $L^2$-norm of the above equality to get that
\begin{equation}\label{44500}
\begin{aligned}
\absB{\phi_0^{\frac{3}{2}-\varepsilon_0} U_{xx}+\frac{1}{\alpha}\phi_0^{\frac{1}{2}-\varepsilon_0}(\phi_0)_xU_x}_2
\leq &C\abs{\phi_0}_\infty^{\frac{3}{2}-\frac{1}{2\alpha}-\varepsilon_0}\big(\absb{\phi_0^\frac{1}{2\alpha}U_t}_2+\absb{\bar\eta_{xx}}_\infty\absb{\phi_0^\frac{1}{2\alpha}U_x}_2\big)\\
&+C\big(\abs{\phi_0}_\infty^{\frac{3}{2}+\frac{1}{\alpha}-\varepsilon_0}\abs{\bar\eta_{xx}}_\infty+ \abs{\phi_0}_\infty^{\frac{1}{2}+\frac{1}{\alpha}-\varepsilon_0}\abs{(\phi_0)_x}_\infty\big)\\
\leq &\mathscr{P}(c_0)\left(1+\mathscr{P}(c_2)t\right)\leq \mathscr{P}(c_0).
\end{aligned}
\end{equation}
Next, it follows from the arguments for \eqref{3...136} and the estimate \eqref{equ3.143} that
\begin{equation*}
\begin{aligned}
\abs{U_x(t,x)}&\leq C \phi_0^\frac{1}{\alpha}+ C\phi_0^\frac{\alpha-1}{2\alpha}\absb{\phi_0^{\frac{1}{2\alpha}}U_t}_2\leq \mathscr{P}(c_0)\big(\phi_0^\frac{1}{\alpha}+ \phi_0^\frac{\alpha-1}{2\alpha}\big),
\end{aligned}
\end{equation*} 
which, along with $\phi_0\sim d(x)$, implies that for all $\iota>0$,
\begin{equation}\label{4453}
\absb{\phi_0^{\frac{1}{2\alpha}-1+\iota}U_x(t)}_2 \leq C(\iota)\mathscr{P}(c_0)\ \ \text{for all }0\leq t\leq T_1.
\end{equation}
Hence, if $0<\varepsilon_0<\frac{3\alpha-1}{2\alpha}$, one can set $\iota=\frac{3\alpha-1}{2\alpha}-\varepsilon_0$ in \eqref{4453}, and get from \eqref{44500} that 
\begin{equation}\label{445555}
\absb{\phi_0^{\frac{3}{2}-\varepsilon_0} U_{xx}(t)}\leq \mathscr{P}(c_0) \ \ \text{for all }0\leq t\leq T_1; 
\end{equation}
while, if $\varepsilon_0=\frac{3\alpha-1}{2\alpha}$, \eqref{44500} can be reduced to 
\begin{equation}\label{44500'}
\begin{aligned}
\absB{\phi_0^\frac{1}{2\alpha} U_{xx}+\frac{1}{\alpha}\phi_0^{\frac{1}{2\alpha}-1}(\phi_0)_xU_x}_2\leq \mathscr{P}(c_0),
\end{aligned}
\end{equation}
and then one can deduce from Proposition \ref{prop2.1} and \eqref{equ3.142} that 
\begin{equation}
\absb{\phi_0^\frac{1}{2\alpha} U_{xx}(t)}_2\leq \mathscr{P}(c_0)\ \ \text{for all }0\leq t\leq T_1.
\end{equation}

In conclusion, it holds that for all $0\leq t\leq T_1$,
\begin{equation}\label{equu3145}
\absb{\phi_0^\frac{1}{2\alpha}U(t)}_2+\absb{\phi_0^\frac{1}{2\alpha}U_x(t)}_2+\absb{\phi_0^\frac{1}{2\alpha} U_t(t)}_2+\absb{\phi_0^{\frac{3}{2}-\varepsilon_0} U_{xx}(t)}_2\leq \mathscr{P}(c_0).
\end{equation}

\underline{\textbf{Step 3: Analogy of Lemma \ref{c_1-c_2}.}}  

\underline{\textbf{Step 3.1: Tangential estimates.}}  As in \textbf{Steps 1}-\textbf{2} of the proof for Lemma \ref{c_1-c_2}, one can apply $U_{tt}\partial_t$ to both sides of \eqref{lp}  to get the 3rd order estimates, and use the weak formulation for the equation of $U_{tt}$ to get the 4th order estimates, that is, 
\begin{equation}\label{equ3.147}
\begin{aligned}
&\frac{1}{2}\frac{\mathrm{d}}{\dt}\int\frac{\phi_0^\frac{1}{\alpha}U_{tx}^2}{\bar\eta_x^2}\,\dx+ \int \phi_0^\frac{1}{\alpha}U_{tt}^2\,\dx\\
=&-\int\frac{\phi_0^\frac{1}{\alpha}\bar U_x U_{tx}^2}{\bar \eta_{x}^3}\,\dx- \int \bigg(\frac{2\phi_0^\frac{1}{\alpha}\bar U_x U_x}{\bar \eta_x^3}\bigg)_x U_{tt}\,\dx+ \int \bigg(\frac{2\phi_0^\frac{2}{\alpha}\bar U_x}{\bar \eta_x^3}\bigg)_x U_{tt}\,\dx,
\end{aligned}
\end{equation}
and 
\begin{equation}
\begin{aligned}
&\frac{1}{2}\frac{\mathrm{d}}{\dt}\int \phi_0^\frac{1}{\alpha}U_{tt}^2\,\dx+\int \frac{\phi_0^\frac{1}{\alpha}(\partial_t^2 U_x)^2}{\bar \eta_x^2}\,\dx\\
=&-\int \frac{2\phi_0^\frac{2}{\alpha} \bar U_{tx}}{\bar\eta_x^3}\partial_t^2U_x\,\dx+\int \frac{6\phi_0^\frac{2}{\alpha} \bar U_x^2}{\bar \eta_x^4}\partial_t^2U_x\,\dx+\int \frac{2\phi_0^\frac{1}{\alpha}\bar U_{tx} U_{x}}{\bar\eta_x^3}\partial_t^2U_x\,\dx\\
&+\int \frac{4\phi_0^\frac{1}{\alpha}\bar U_x U_{tx}}{\bar\eta_x^3}\partial_t^2U_x\,\dx-\int \frac{6\phi_0^\frac{1}{\alpha} \bar U_x^2 U_x}{\bar \eta_x^4}\partial_t^2U_x\,\dx.
\end{aligned}
\end{equation}
Again, different from \eqref{eq:cL17-cL23}, there is no crossing term
$\left(\frac{1}{\alpha}-2\right) \int \phi_0 (\phi_0)_x U_{tx}U_{tt}\,\dx$
in \eqref{equ3.147}. Hence, it follows from similar computations as in \textbf{Step 1} and \textbf{Step 2} of Lemma \ref{c_1-c_2} that for all $0\leq t\leq T_2:=\min\{\widetilde T, (1+\mathscr{P}(c_2))^{-M'}\}$, for some large $M'>0$, 
\begin{equation}\label{equ3.148}
\absb{\phi_0^\frac{1}{2\alpha}U_{tx}(t)}_2+\absb{\phi_0^\frac{1}{2\alpha}U_{tt}(t)}_2\leq \mathscr{P}(c_0).
\end{equation}

\underline{\textbf{Step 3.2: Estimate of $\phi_0^{\frac{3}{2}-\varepsilon_0}\partial_x^3 U$.}}
According to \eqref{3...151}, one has
\begin{align}
&\phi_0^{\frac{3}{2}-\varepsilon_0}\partial_x^3 U +\Big(\frac{1}{\alpha}+1\Big)\phi_0^{\frac{1}{2}-\varepsilon_0}(\phi_0)_x U_{xx}\notag\\
=&\underline{-\frac{1}{\alpha}\phi_0^{\frac{1}{2}-\varepsilon_0}(\phi_0)_{xx} U_x+\phi_0^{\frac{3}{2}-\varepsilon_0}\bar\eta_x^2 U_{tx}+2\phi_0^{\frac{3}{2}-\varepsilon_0}\bar\eta_x\bar\eta_{xx}U_t}_{:=\cL_{14}}\notag\\
&\underline{+\phi_0^{\frac{1}{2}-\varepsilon_0}(\phi_0)_x\bar\eta_x^2 U_t +\frac{2\phi_0^{\frac{1}{2}-\varepsilon_0}(\phi_0)_x\bar\eta_{xx}U_x}{\bar\eta_x}+\frac{2\phi_0^{\frac{3}{2}-\varepsilon_0}\bar\eta_{xx} U_{xx}}{\bar\eta_x}}_{:=\cL_{15}}\label{5049}\\
& \underline{+\frac{2\phi_0^{\frac{3}{2}-\varepsilon_0}\partial_x^3\bar\eta U_x}{\bar\eta_x}-\frac{2\phi_0^{\frac{3}{2}-\varepsilon_0}\bar\eta_{xx}^2U_x}{\bar\eta_x^2}-\frac{2+2\alpha}{\alpha}\frac{\phi_0^{\frac{1}{2}+\frac{1}{\alpha}-\varepsilon_0}(\phi_0)_x\bar\eta_{xx}}{\bar\eta_x}}_{:=\cL_{16}}\notag\\
& \underline{-\frac{2\phi_0^{\frac{3}{2}+\frac{1}{\alpha}-\varepsilon_0}\partial_x^3\bar\eta}{\bar\eta_x}+\frac{2\phi_0^{\frac{3}{2}+\frac{1}{\alpha}-\varepsilon_0}\bar\eta_{xx}^2}{\bar\eta_x^2}+\frac{2}{\alpha^2}\phi_0^{-\frac{1}{2}+\frac{1}{\alpha}-\varepsilon_0}((\phi_0)_x)^2 +\frac{2}{\alpha}\phi_0^{\frac{1}{2}+\frac{1}{\alpha}-\varepsilon_0}(\phi_0)_{xx}}_{:=\cL_{17}}.\notag
\end{align}

Taking the $L^2$-norm of both sides of the above equality, one gets from \eqref{varepsilon0}, \eqref{useful1'}-\eqref{useful3'}, \eqref{equu3145}, \eqref{equ3.148} and Lemma \ref{hardy-inequality} that for all $0\leq t\leq T_2$,
\begin{align}
\abs{\cL_{14}}_2&\leq C \abs{(\phi_0)_{xx}}_\infty \absb{\phi_0^{\frac{1}{2}-\varepsilon_0}U_x}_2+C\abs{\phi_0}_\infty^{\frac{3}{2}-\frac{1}{2\alpha}-\varepsilon_0}\big(\absb{\phi_0^\frac{1}{2\alpha}U_{tx}}_2+\abs{\bar\eta_{xx}}_\infty\absb{\phi_0^\frac{1}{2\alpha}U_t}_2\big)\notag\\
&\leq \mathscr{P}(c_0)(1+\mathscr{P}(c_2)t)\bigg(\sum_{j=1}^2\absb{\phi_0^{\frac{3}{2}-\varepsilon_0}\partial_x^j U}_2+\sum_{j=0}^1\absb{\phi_0^\frac{1}{2\alpha} \partial_x^j U_t}_2\bigg)\leq \mathscr{P}(c_0),\notag\\
\abs{\cL_{15}}_2&\leq C\abs{(\phi_0)_x}_\infty\big(\absb{\phi_0^{\frac{1}{2}-\varepsilon_0}U_t}_2 +\abs{(\phi_0)_x}_\infty\abs{\bar\eta_{xx}}_\infty\absb{\phi_0^{\frac{1}{2}-\varepsilon_0}U_x}_2\big)\notag\\
&\quad +C\abs{\bar\eta_{xx}}_\infty\absb{\phi_0^{\frac{3}{2}-\varepsilon_0} U_{xx}}_2\notag\\
&\leq \mathscr{P}(c_0)(1+\mathscr{P}(c_2)t)\bigg(\sum_{j=1}^2\absb{\phi_0^{\frac{3}{2}-\varepsilon_0}\partial_x^j U}_2+\sum_{j=0}^1\absb{\phi_0^\frac{1}{2\alpha} \partial_x^j U_t}_2\bigg)\leq  \mathscr{P}(c_0),\notag\\
\abs{\cL_{16}}_2&\leq C\big(\absb{\phi_0^{\frac{1}{2}-\varepsilon_0}\partial_x^3\bar\eta}_2\abs{\phi_0 U_x}_\infty+\abs{\bar\eta_{xx}}_\infty^2\absb{\phi_0^{\frac{3}{2}-\varepsilon_0} U_x}_2\big)\label{equ3.1500}\\
&\quad +C\abs{\phi_0}_\infty^{\frac{1}{2}+\frac{1}{\alpha}-\varepsilon_0}\abs{(\phi_0)_x}_\infty \abs{\bar\eta_{xx}}_\infty\notag\\
&\leq \mathscr{P}(c_0)(1+\mathscr{P}(c_2)t)\bigg(\sum_{j=1}^2\absb{\phi_0^{\frac{3}{2}-\varepsilon_0}\partial_x^j U}_2+1\bigg) \leq \mathscr{P}(c_0),\notag\\
\abs{\cL_{17}}_2&\leq C\big( \abs{\phi_0}_\infty^\frac{1}{\alpha}\absb{\phi_0^{\frac{3}{2}-\varepsilon_0}\partial_x^3\bar\eta}_2+\abs{\phi_0}_\infty^\frac{1}{\alpha} \abs{\bar\eta_{xx}}_\infty \absb{\phi_0^{\frac{3}{2}-\varepsilon_0}\bar\eta_{xx}}_2\big)\notag\\
&\quad +C\big(\absb{\phi_0^{-\frac{1}{2}+\frac{1}{\alpha}-\varepsilon_0}}_2\abs{(\phi_0)_x}_\infty^2+ \abs{\phi_0}_\infty^{\frac{1}{2}+\frac{1}{\alpha}-\varepsilon_0}\abs{(\phi_0)_{xx}}_\infty\big)\notag\\
&\leq \mathscr{P}(c_0)(1+\mathscr{P}(c_2)t)\leq \mathscr{P}(c_0).\notag
\end{align}

Then, it follows from \eqref{5049}-\eqref{equ3.1500} that
\begin{equation*}
\absB{\phi_0^{\frac{3}{2}-\varepsilon_0}\partial_x^3 U +\Big(\frac{1}{\alpha}+1\Big)\phi_0^{\frac{1}{2}-\varepsilon_0}(\phi_0)_x U_{xx}}_2\leq \mathscr{P}(c_0),
\end{equation*}
which, together with \eqref{equu3145} and Proposition \ref{prop2.1}, leads to
\begin{equation}\label{equ3.151}
\absb{\phi_0^{\frac{3}{2}-\varepsilon_0}\partial_x^3 U(t)}_2\leq \mathscr{P}(c_0) \ \ \text{for all }0\leq t\leq T_2.
\end{equation}

\underline{\textbf{Step 3.3: Estimate of $\phi_0^{\frac{3}{2}-\varepsilon_0}\partial_t U_{xx}$.}}
First, according to $\frac{1}{3}<\alpha\leq 1$, \eqref{equu3145}, \eqref{equ3.148}, \eqref{equ3.151} and Lemmas \ref{sobolev-embedding}-\ref{hardy-inequality}, one gets that for all $0\leq t\leq T_2$,
\begin{equation}\label{equ3.152}
\begin{aligned}
\abs{\phi_0 U_t(t)}_\infty &\leq C \sum_{j=0}^1 \absb{\phi_0^\frac{3}{2} \partial_x^j U_t(t)}_2\\
&\leq C\big(\abs{\phi_0}_\infty^{\frac{3}{2}-\frac{1}{2\alpha}}\absb{\phi_0^{\frac{1}{2\alpha}} U_{t}(t)}_2+\abs{\phi_0}_\infty^{\varepsilon_0}\absb{\phi_0^{\frac{3}{2}-\varepsilon_0} U_{tx}(t)}_2\big)\leq  \mathscr{P}(c_0),\\
\abs{U_x(t)}_\infty &\leq C \norm{U_x(t)}_{1,1}\leq C \sum_{j=1}^3\absb{\phi_0^{\frac{3}{2}-\varepsilon_0}\partial_x^j U(t)}_2\leq \mathscr{P}(c_0).
\end{aligned}
\end{equation}

Then it follows from \eqref{equ3.72}, \eqref{39''}, \eqref{useful1'}-\eqref{useful3'}, \eqref{equu3145}, \eqref{equ3.148}, \eqref{equ3.152} and Lemma \ref{hardy-inequality}  that for all $0\leq t\leq T_2$,
\begin{align}
&\absB{\phi_0^{\frac{3}{2}-\varepsilon_0} \partial_tU_{xx}+\frac{1}{\alpha}\phi_0^{\frac{1}{2}-\varepsilon_0}(\phi_0)_x U_{tx}}_2\notag\\
\leq  &C\abs{\phi_0}_\infty^{\frac{3}{2}-\frac{1}{2\alpha}-\varepsilon_0}\absb{\phi_0^\frac{1}{2\alpha}U_{tt}}_2 +C\absb{\phi_0^{\frac{1}{2}-\varepsilon_0}\bar U_x}_2\abs{\phi_0U_{t}}_\infty+C\abs{\phi_0}_\infty^\frac{1}{\alpha}\absb{\phi_0^{\frac{3}{2}-\varepsilon_0} \bar U_{xx}}_2\notag\\
&+C\abs{\phi_0}_\infty^\frac{1}{\alpha}\abs{\bar\eta_{xx}}_\infty\absb{\phi_0^{\frac{3}{2}-\varepsilon_0} \bar U_{x}}_2+C\absb{\phi_0^{\frac{3}{2}-\varepsilon_0}\bar U_{xx}}_2\abs{U_x}_\infty\label{equ3.153}\\
&+C\abs{\bar\eta_{xx}}_\infty \absb{\phi_0^{\frac{3}{2}-\varepsilon_0}\bar U_x}_2 \abs{U_x}_\infty+C\abs{\phi_0}_\infty^{\frac{3}{2}-\frac{1}{2\alpha}-\varepsilon_0}\abs{\bar\eta_{xx}}_\infty\absb{\phi_0^\frac{1}{2\alpha}U_{tx}}_2\notag\\
\leq &\mathscr{P}(c_0)(\mathscr{P}(c_1)+\mathscr{P}(c_2)t) \leq \mathscr{P}(c_1)\notag.
\end{align}
Next, as for  \eqref{esti-tx}, one gets from \eqref{equ3.143}, \eqref{equ3.148} and Lemma \ref{hardy-inequality} that
\begin{equation*}
|U_{tx}(t,x)|\leq C\phi_0^\frac{\alpha-1}{2\alpha}\absf{\phi_0\bar U_{x}}_\infty\big(\absb{\phi_0^\frac{1}{2\alpha}U_{t}}_2+\absb{\phi_0^\frac{1}{2\alpha}U_{tx}}_2\big)+C\phi_0^\frac{\alpha-1}{2\alpha}\absb{\phi_0^\frac{1}{2\alpha}U_{tt}}_2\leq \mathscr{P}(c_1)\phi_0^\frac{\alpha-1}{2\alpha},
\end{equation*}
which, along with $\phi_0\sim d(x)$, implies that for all $\iota>0$,
\begin{equation}\label{4453'}
\absb{\phi_0^{\frac{1}{2\alpha}-1+\iota}U_{tx}(t)}_2 \leq C(\iota)\mathscr{P}(c_1)\ \ \text{for all }0\leq t\leq T_2.
\end{equation}
Hence, if $0<\varepsilon_0<\frac{3\alpha-1}{2\alpha}$, one can set $\iota=\frac{3\alpha-1}{2\alpha}-\varepsilon_0$ in \eqref{4453'}, and get from \eqref{equ3.153} that 
\begin{equation}\label{equ3.154}
\absb{\phi_0^{\frac{3}{2}-\varepsilon_0} \partial_tU_{xx}(t)}\leq \mathscr{P}(c_1) \ \ \text{for all }0\leq t\leq T_2; 
\end{equation}
while, if $\varepsilon_0=\frac{3\alpha-1}{2\alpha}$, \eqref{equ3.153} can be reduced to 
\begin{equation}\label{44500''}
\begin{aligned}
\absB{\phi_0^\frac{1}{2\alpha}\partial_t U_{xx}+\frac{1}{\alpha}\phi_0^{\frac{1}{2\alpha}-1}(\phi_0)_xU_{tx}}_2\leq \mathscr{P}(c_1),
\end{aligned}
\end{equation}
and then one can deduce from Proposition \ref{prop2.1} and \eqref{equ3.148} that 
\begin{equation}
\absb{\phi_0^\frac{1}{2\alpha}\partial_t U_{xx}(t)}_2\leq \mathscr{P}(c_1)\ \ \text{for all }0\leq t\leq T_2.
\end{equation}

\underline{\textbf{Step 3.4: Estimate of $\phi_0^{\frac{3}{2}-\varepsilon_0}\partial_x^4 U$.}}
According to \eqref{3-154}, one has
\begin{align}
&\phi_0^{\frac{3}{2}-\varepsilon_0}\partial_x^4 U+\Big(\frac{1}{\alpha}+2\Big)\phi_0^{\frac{1}{2}-\varepsilon_0}(\phi_0)_x\partial_x^3 U\notag\\
=&\underline{\phi_0^{\frac{3}{2}-\varepsilon_0}\bar\eta_x^2\partial_tU_{xx}+2\phi_0^{\frac{1}{2}-\varepsilon_0}\left((\phi_0)_x\bar\eta_x+2\phi_0\bar\eta_{xx}\right)\bar\eta_x U_{tx}+2\phi_0^{\frac{3}{2}-\varepsilon_0}\left(\bar\eta_{xx}^2+\bar\eta_x\partial_x^3\bar\eta\right)U_t}_{:=\cL_{18}}\notag\\
&\underline{+\phi_0^{\frac{1}{2}-\varepsilon_0}\left((\phi_0)_{xx}\bar\eta_x+4(\phi_0)_x\bar\eta_{xx}\right)\bar\eta_xU_t -\Big(\frac{2}{\alpha}+1\Big)\phi_0^{\frac{1}{2}-\varepsilon_0}(\phi_0)_{xx} U_{xx}}_{:=\cL_{19}}\notag\\
&\underline{-\frac{1}{\alpha}\phi_0^{\frac{1}{2}-\varepsilon_0}\partial_x^3\phi_0 U_x+2\phi_0^{\frac{1}{2}-\varepsilon_0}\Big((\phi_0)_{xx}\bar\eta_{xx}+2(\phi_0)_{x}\partial_x^3\bar\eta -\frac{2(\phi_0)_{x}\bar\eta_{xx}^2}{\bar\eta_x}\Big)\frac{U_x}{\bar\eta_x}}_{:=\cL_{20}}\notag\\
&\underline{+2\phi_0^{\frac{3}{2}-\varepsilon_0}\Big(\partial_x^4\bar\eta
-\frac{3\bar\eta_{xx}\partial_x^3\bar\eta}{\bar\eta_x}
+\frac{2\bar\eta_{xx}^3}{\bar\eta_x^2}\Big)\frac{U_x}{\bar\eta_x}}_{:=\cL_{21}}\label{4068}\\
&\underline{+4\phi_0^{\frac{1}{2}-\varepsilon_0}\Big(\phi_0\partial_x^3\bar\eta+(\phi_0)_{x}\bar\eta_{xx}-\frac{\phi_0\bar\eta_{xx}^2}{\bar\eta_x}\Big)\frac{U_{xx}}{\bar\eta_x}+\frac{2\phi_0^{\frac{3}{2}-\varepsilon_0}\bar\eta_{xx}\partial_x^3U}{\bar\eta_x}}_{:=\cL_{22}}\notag\\
&\underline{-\frac{2+2\alpha}{\alpha^2}\frac{\phi_0^{-\frac{1}{2}+\frac{1}{\alpha}-\varepsilon_0}((\phi_0)_x)^2\bar\eta_{xx}}{\bar\eta_x}-2\phi_0^{\frac{3}{2}+\frac{1}{\alpha}-\varepsilon_0}\Big(\frac{\partial_x^4\bar\eta}{\bar\eta_x}-\frac{3\bar\eta_{xx}\partial_x^3\bar\eta}{\bar\eta_x^2}+\frac{2\bar\eta_{xx}^3}{\bar\eta_x^3}\Big)}_{:=\cL_{23}}\notag\\
&\underline{-\frac{2+2\alpha}{\alpha}\phi_0^{\frac{1}{2}+\frac{1}{\alpha}-\varepsilon_0}\Big(\frac{(\phi_0)_{xx}\bar\eta_{xx}}{\bar\eta_x}+\frac{2(\phi_0)_x\partial_x^3\bar\eta}{\bar\eta_x}-\frac{2(\phi_0)_x\bar\eta_{xx}^2}{\bar\eta_x^2}\Big)}_{:=\cL_{24}}\notag\\
&\underline{+\frac{2(1-\alpha)}{\alpha^3}\phi_0^{\frac{1}{\alpha}-\frac{3}{2}-\varepsilon_0}((\phi_0)_x)^3+\frac{6}{\alpha^2}\phi_0^{\frac{1}{\alpha}-\frac{1}{2}-\varepsilon_0}(\phi_0)_x(\phi_0)_{xx}+\frac{2}{\alpha}\phi_0^{\frac{1}{\alpha}+\frac{1}{2}-\varepsilon_0}\partial_x^3\phi_0}_{:=\cL_{25}}.\notag
\end{align}

It follows from \eqref{varepsilon0}, \eqref{useful1'}-\eqref{useful3'}, \eqref{equu3145}, \eqref{equ3.148}, \eqref{equ3.151}-\eqref{equ3.152}, \eqref{equ3.154} and Lemma \ref{hardy-inequality} that for all $0\leq t\leq T_2$,
\begin{align}
\abs{\cL_{18}}_2&\leq C \absb{\phi_0^{\frac{3}{2}-\varepsilon_0} \partial_t U_{xx}}_2+ C\left(\abs{(\phi_0)_x}_\infty+ \abs{\phi_0}_\infty \abs{\bar\eta_{xx}}_\infty\right)\absb{\phi_0^{\frac{1}{2}-\varepsilon_0}U_{tx}}_2\notag\\
&\quad + C\Big(\abs{\phi_0}_\infty^{\frac{1}{2}-\varepsilon_0}\abs{\bar\eta_{xx}}_\infty^2+\absb{\phi_0^{\frac{1}{2}-\varepsilon_0}\partial_x^3\bar\eta}_2\Big)\abs{\phi_0U_t}_\infty \notag\\
&\leq \mathscr{P}(c_1)+\mathscr{P}(c_0)\left(1+\mathscr{P}(c_2)t\right)\leq \mathscr{P}(c_1),\notag\\
\abs{\cL_{19}}_2&\leq C\left(\abs{(\phi_0)_{xx}}_\infty+\abs{(\phi_0)_x}_\infty \abs{\bar\eta_{xx}}_\infty \right) \absb{\phi_0^{\frac{1}{2}-\varepsilon_0}U_t}_2+ C\abs{(\phi_0)_{xx}}_\infty \absb{\phi_0^{\frac{1}{2}-\varepsilon_0}U_{xx}}_2\notag\\
&\leq \mathscr{P}(c_0)\left(1+\mathscr{P}(c_2)t\right)\leq \mathscr{P}(c_1),\notag\\
\abs{\cL_{20}}_2&\leq C\abs{\partial_x^3\phi_0}_2 \absb{\phi_0^{\frac{1}{2}-\varepsilon_0}U_x}_\infty + C\norm{\phi_0}_3\big(\abs{\bar\eta_{xx}}_\infty+\abs{\bar\eta_{xx}}_\infty^2\big) \absb{\phi_0^{\frac{1}{2}-\varepsilon_0} U_x}_2 \notag\\
&\quad + C\abs{(\phi_0)_x}_\infty \absb{\phi_0^{\frac{1}{2}-\varepsilon_0} \partial_x^3 \bar\eta}_2\abs{U_x}_\infty\notag\\
&\leq \mathscr{P}(c_0)\left(1+\mathscr{P}(c_2)t\right)\leq \mathscr{P}(c_1),\notag\\
\abs{\cL_{21}}_2&\leq C \Big(\absb{\phi_0^{\frac{3}{2}-\varepsilon_0}\partial_x^4\bar\eta}_2+\abs{\bar\eta_{xx}}_\infty\absb{\phi_0^{\frac{3}{2}-\varepsilon_0}\partial_x^3\bar\eta}_2+\abs{\bar\eta_{xx}}_\infty^2\absb{\phi_0^{\frac{3}{2}-\varepsilon_0} \bar\eta_{xx}}_2\Big) \abs{U_x}_\infty \label{equ3.156}\\
&\leq \mathscr{P}(c_0)\mathscr{P}(c_2)t\leq \mathscr{P}(c_1),\notag\\
\abs{\cL_{22}}_2&\leq C \big(\absf{\phi_0\partial_x^3 \bar\eta}_\infty + \abs{(\phi_0)_x}_\infty \abs{\bar\eta_{xx}}_\infty + \abs{\phi_0}_\infty\abs{\bar\eta_{xx}}_\infty^2\big)\absb{\phi_0^{\frac{1}{2}-\varepsilon_0} U_{xx}}_2\notag\\
&\quad +C\abs{\bar\eta_{xx}}_\infty \absb{\phi_0^{\frac{3}{2}-\varepsilon_0}\partial_x^3 U}_2 
\leq \mathscr{P}(c_0)\mathscr{P}(c_2)t\leq \mathscr{P}(c_1),\notag\\
\abs{\cL_{23}}_2&\leq C \absb{\phi_0^{\frac{1}{\alpha}-\frac{1}{2}-\varepsilon_0}}_2\abs{(\phi_0)_x}_\infty^2 \abs{\bar\eta_{xx}}_\infty+C\abs{\phi_0}_\infty^\frac{1}{\alpha}\Big(\absb{\phi_0^{\frac{3}{2}-\varepsilon_0}\partial_x^4 \bar\eta}_2+\abs{\bar\eta_{xx}}_\infty \absb{\phi_0^{\frac{3}{2}-\varepsilon_0}\partial_x^3\bar\eta}_2\Big)\notag\\
&\quad +C\abs{\phi_0}_\infty^\frac{1}{\alpha}\abs{\bar\eta_{xx}}_\infty^2\absb{\phi_0^{\frac{3}{2}-\varepsilon_0} \bar\eta_{xx}}_2\leq \mathscr{P}(c_0)\mathscr{P}(c_2)t\leq \mathscr{P}(c_1),\notag\\
\abs{\cL_{24}}_2&\leq C\abs{\phi_0}_\infty^\frac{1}{\alpha}\norm{\phi_0}_3\Big(\absb{\phi_0^{\frac{1}{2}-\varepsilon_0}\bar\eta_{xx}}_2+\absb{\phi_0^{\frac{1}{2}-\varepsilon_0}\partial_x^3\bar\eta}_2+\abs{\bar\eta_{xx}}_\infty\absb{\phi_0^{\frac{1}{2}-\varepsilon_0}\bar\eta_{xx}}_2\Big)\notag\\
&\leq C\mathscr{P}(c_0)\mathscr{P}(c_2)t\leq C\mathscr{P}(c_1),\notag\\
\abs{\cL_{25}}_2&\leq C\underline{\absb{\phi_0^{\frac{1}{\alpha}-\frac{3}{2}-\varepsilon_0}}_2\abs{(\phi_0)_x}_\infty^3}_{(=0,\text{ if }\alpha=1)}+C\absb{\phi_0^{\frac{1}{\alpha}-\frac{1}{2}-\varepsilon_0}}_2\abs{(\phi_0)_x}_\infty\abs{(\phi_0)_{xx}}_\infty\notag\\
&\quad+C\abs{\phi_0}_\infty^{\frac{1}{\alpha}+\frac{1}{2}-\varepsilon_0}\absf{\partial_x^3\phi_0}_2
\leq \mathscr{P}(c_0).\notag
\end{align}

Then, it follows from \eqref{4068}-\eqref{equ3.156} that
\begin{equation*}
\absB{\phi_0^{\frac{3}{2}-\varepsilon_0}\partial_x^4 U +\Big(\frac{1}{\alpha}+2\Big)\phi_0^{\frac{1}{2}-\varepsilon_0}(\phi_0)_x\partial_x^3 U}_2\leq \mathscr{P}(c_0),
\end{equation*}
which, together with \eqref{equ3.151} and Proposition \ref{prop2.1}, implies that
\begin{equation}\label{equ3.157}
\absb{\phi_0^{\frac{3}{2}-\varepsilon_0}\partial_x^4 U(t)}_2\leq \mathscr{P}(c_1) \ \ \text{for all }0\leq t\leq T_2.
\end{equation}

Collecting all estimates \eqref{equ3.148}, \eqref{equ3.151}, \eqref{equ3.154} and \eqref{equ3.157} yields that for all $0\leq t\leq T_2$,
\begin{equation}\label{equ3.158}
\begin{aligned}
\absb{\phi_0^\frac{1}{2\alpha} U_{tx}(t)}_{2}+\absb{\phi_0^{\frac{3}{2}-\varepsilon_0} \partial^3_x U(t)}_{2}\leq \mathscr{P}(c_1),\\
\absb{\phi_0^\frac{1}{2\alpha} U_{tt}(t)}_{2}+\absb{\phi_0^{\frac{3}{2}-\varepsilon_0} \partial_t U_{xx}(t)}_{2}+\absb{\phi_0^{\frac{3}{2}-\varepsilon_0} \partial^4_x U(t)}_{2}\leq \mathscr{P}(c_1).
\end{aligned}
\end{equation}

\underline{\textbf{Step 4: Choices of $c_1$, $c_2$ and $T$.}} 
Defining the constants $c_1,c_2$ and the time $T$ as
\begin{equation}
c_1=\sqrt{\mathscr{P}(c_0)},\quad c_2=\sqrt{\mathscr{P}(c_1)},\quad T=\min\{T_1,T_2\},
\end{equation}
according to \eqref{equu3145} and \eqref{equ3.158}, one then obtain the uniform estimates \eqref{uniform bounds''}.

The proof of Lemma \ref{a-priori-estimates2} is completed.
\end{proof}

\section{Local existence  and uniqueness for the  nonlinear problem}\label{Section5}

\S \ref{Section5} is devoted to proving the local existence  of the unique classical solution for  problem \eqref{fp''}-\eqref{fp}, i.e., Theorem \ref{theorem3.1}, and Theorem \ref{theorem1.3} follows as a consequence. For simplicity, we only prove the case of $0<\alpha\leq \frac{1}{3}$, and the case of $\frac{1}{3}<\alpha\leq 1$ can be treated analogously.

\subsection{Proof of Theorem \ref{theorem3.1}}
This will be divided into the following several steps:

\underline{\textbf{Step 1: Construction of the iterative sequence.}}
We use the same notations as in  Lemma \ref{a-priori-estimates}. Set 
\begin{equation*}
U^0(t,x):= u_0,\quad \eta^0(t,x)=x+tu_0.
\end{equation*}
Then, for given $c_0$ as in Lemma \ref{a-priori-estimates} and $c_i$, $i=1,2$, defined by \eqref{3.84}, there exists a small positive time $T^\prime\leq T:=\min\{\widetilde T, (1+Cc_2)^{-\frac{6\alpha+6}{\alpha}}\}$, such that  for all $t\in[0,T']$,
\begin{equation}\label{395}
\begin{aligned}
\absf{\phi_0 U^0(t)}_{2}+\absf{\phi_0 U_x^0(t)}_{2}+
\absf{\phi_0 U_t^0(t)}_{2}+\absf{\phi_0 U_{xx}^0(t)}_{2}\leq c_1,\quad \absf{\eta_x^0(t)-1}_\infty\leq \frac{1}{2},\\
\absf{\phi_0 U^0_{tx}(t)}_{2}+\absf{\phi_0 \partial^3_xU^0(t)}_{2}+\absf{\phi_0 U_{tt}^0(t)}_{2}+\absf{\phi_0 \partial_t  U_{xx}^0(t)}_{2}+\absf{\phi_0 \partial^4_xU^0(t)}_{2}\leq c_2.
\end{aligned}
\end{equation}

Next, let $(\bar U,\bar \eta )=(U^0,\eta^0)$ in \eqref{lp''} be the first generation of solutions. According to Lemma \ref{existence-linearize}, there exists a unique classical solution $(U^1,\eta^1)$ to the problem \eqref{lp''}. Clearly, one can deduce from Lemma \ref{uniform bounds} that $U^1$ satisfies \eqref{395}, which implies that $|\eta_x^1-1|_\infty\leq \frac{1}{2}$ on $t\in [0,T']$ for the same $T'$. 

Then, the approximate sequence $(U^{k+1},\eta^{k+1})$, $k\geq 1$, can be constructed as follows: given $(U^k,\eta^k)$, define $(U^{k+1},  \eta^{k+1})$ by solving the following problem,
\begin{equation}\label{396}
\begin{cases}
\quad\displaystyle\phi_0^2 U_{t}^{k+1}-\left(\frac{\phi_0^2 U_{x}^{k+1}}{(\eta^k_x)^2}\right)_x\\
\displaystyle=\Big(\frac{1}{\alpha}-2\Big)\frac{\phi_0 (\phi_0)_x U_x^{k+1}}{(\eta^k_x)^2}- \bigg(\frac{\phi_0^{2+\frac{1}{\alpha}}}{(\eta^k_x)^2}\bigg)_x+\Big(2-\frac{1}{\alpha}\Big)\frac{\phi_0^{1+\frac{1}{\alpha}}(\phi_0)_x}{(\eta^k_x)^2}&\text{in }(0,T]\times I,\\
\quad  \eta^{k+1}_t = U^{k+1} &\text{in }(0,T]\times I,\\[4pt]
\quad (U^{k+1},\eta^{k+1}) =(u_0,\mathrm{id})&\text{on }\{t=0\}\times I.
\end{cases}
\end{equation}
It follows from Lemma \ref{a-priori-estimates} that one can successfully obtain an iterative solution sequence $(U^k,\eta^k)$ satisfying \eqref{395}, that is, for all $t\in[0,T']$ and all $k\geq 0$,
\begin{equation}\label{395k}
\begin{aligned}
\absf{\phi_0 U^k(t)}_{2}+\absf{\phi_0 U_x^k(t)}_{2}+
\absf{\phi_0 U_t^k(t)}_{2}+\absf{\phi_0 U_{xx}^k(t)}_{2}\leq c_1,\quad\absf{\eta_x^k(t)-1}_\infty\leq \frac{1}{2},\\
\absf{\phi_0 U^k_{tx}(t)}_{2}+\absf{\phi_0 \partial^3_xU^k(t)}_{2}+\absf{\phi_0 U_{tt}^k(t)}_{2}+\absf{\phi_0 \partial_t  U_{xx}^k(t)}_{2}+\absf{\phi_0 \partial^4_xU^k(t)}_{2}\leq c_2.
\end{aligned}
\end{equation}

\underline{\textbf{Step 2: Convergence of $(U^k,\eta^k)$.}}
Set
\begin{equation}\label{picarddifference}
\wU^{k+1}:=U^{k+1}-U^k,\quad \weta^{k+1}:=\eta^{k+1}-\eta^k=\int_0^t \wU^{k+1}(s,x)\,\ds,
\end{equation}
and introduce the following basic energy function:
\begin{equation*}
\begin{aligned}
\wE^k(t)&:= \sup_{s\in[0,t]}\absf{\phi_0 \wU^{k}}_2^2+\int_0^t\absf{\phi_0 \wU^{k}_{x}}_2^2\,\ds.   
\end{aligned}
\end{equation*}

It follows from \eqref{396} that
\begin{equation}\label{398}
\begin{cases}
\displaystyle \phi_0^2 \wU_t^{k+1}\!\!-\bigg(\frac{\phi_0^2 \wU_{x}^{k+1}}{(\eta_x^k)^2}\bigg)_x\!\! \!=\!\Big(\frac{1}{\alpha}-2\Big)\frac{\phi_0 (\phi_0)_x \wU_x^{k+1}}{(\eta_x^k)^2} +(\phi_0 \cR_1^k)_x+ \cR_2^k&\text{in }(0,T']\times I,\\
 \weta^{k+1}_t= \wU^{k+1},&\text{in }(0,T']\times I,\\[4pt]
(\wU^{k+1},\weta^{k+1})=(0,0)&\text{on }\{t=0\}\times I,
\end{cases}
\end{equation}
where
\begin{equation*}
\begin{aligned}
&\cR_1^k=  \big(\phi_0^{1+\frac{1}{\alpha}}-\phi_0U_{x}^{k}\big)\frac{(\eta_x^k+\eta_x^{k-1})\weta_x^k}{(\eta_x^k \eta_x^{k-1})^2},\\
&\cR_2^k=\Big(\frac{1}{\alpha}-2\Big) \phi_0(\phi_0)_x\big(\phi_0^\frac{1}{\alpha}-U_x^{k}\big) \frac{(\eta_x^k+\eta_x^{k-1})\weta_x^k}{(\eta_x^k \eta_x^{k-1})^2}.
\end{aligned}
\end{equation*}
Then, one can conclude from \eqref{395k}-\eqref{picarddifference} and $\weta_{tx}^k=\wU^k_x$ that 
\begin{equation}\label{basic}
\begin{aligned}
\absf{\cR_1^k(t)}_2^2+\absf{\cR_2^k(t)}_2^2
\leq Ct\int_0^t\absf{\phi_0 \wU_x^k}_2^2\,\ds.
\end{aligned}
\end{equation}

Next, multiplying $\eqref{398}_1$ by $\wU^{k+1}$ and integrating the resulting equality over $I$, then following the proofs of \eqref{eq:cL1-cL5}-\eqref{cL1}, one can get from \eqref{basic}, Lemma \ref{hardy-inequality}, H\"older's inequality and Young's inequality that
\begin{equation*}
\begin{aligned}
\frac{\mathrm{d}}{\dt}\absf{\phi_0 \wU^{k+1}}_2^2+\absf{\phi_0 \wU_x^{k+1}}_2^2&\leq C\absf{\phi_0 \wU^{k+1}}_2^2+\absf{\cR_1^k(t)}_2^2+\absf{\cR_2^k(t)}_2^2\\
&\leq C\absf{\phi_0 \wU^{k+1}}_2^2+Ct\int_0^t\absf{\phi_0 \wU_x^k}_2^2\,\ds,
\end{aligned}
\end{equation*}
which, along with Gr\"onwall's inequality, leads to  
\begin{equation}\label{3104}
\wE^{k+1}(t)\leq C t^2e^{Ct}\wE^k(t) \ \  \text{for all }0\leq t\leq T^\prime \text{ and all }k\geq 1.
\end{equation}

Choosing $t=T_*$ such that $CT_*^2 e^{CT_*}\leq \frac{1}{2}$ and $T_*\leq T'$, then one can get 
\begin{equation*}
\wE^{k+1}(T_*)\leq \frac{1}{2} \wE^k(T_*) \ \ \text{for all } k\geq 1,
\end{equation*}
which yields
\begin{equation}\label{total-bound}
\sum_{k=1}^{\infty} \wE^{k}(T_*) \leq \bigg(\sum_{k=0}^{\infty}\frac{1}{2^k}\bigg) \wE^{1}(T_*)\leq C(c_0).
\end{equation}

Then, it follows from \eqref{total-bound} that 
\begin{equation*}
\begin{aligned}
\phi_0 U^{m}-\phi_0 U^{n}\to  0\quad \text{in } C([0,T_*];L^2),\quad   \phi_0  U^{m}_x-\phi_0 U^{n}_x \to 0\quad \text{in } L^2([0,T_*];L^2),
\end{aligned}
\end{equation*}
as $m,n\to \infty$, which implies that $\{U^k\}_{k\geq 0}$ converges to a unique limit $U$ as $k\to\infty$ in the following sense:
\begin{equation}\label{3112}
\begin{aligned}
\phi_0 U^k\to \phi_0 U \quad \text{in } C([0,T_*];L^2),\quad \phi_0 U^k_x\to \phi_0 U_x \quad \text{in } L^2([0,T_*];L^2).
\end{aligned}
\end{equation}
This, together with \eqref{395k}, \eqref{398}, Lemmas \ref{GNinequality'}-\ref{GNinequality} and \ref{hardy-inequality},  shows that
\begin{equation}\label{3126}
\begin{aligned}
U^k\to U \quad &\text{in } C([0,T_*];H^{s_1}),\quad \text{for all }s_1\in [0,3),\\
U_t^k\to U_t \quad &\text{in } L^2([0,T_*];H^{s_2}),\ \ \text{for all }s_2\in [0,1).
\end{aligned}
\end{equation}
Since $\{U^k\}_{k\geq 0}\subset C([0,T_*];H^3)$ by Lemma \ref{existence-linearize}, one gets  from $\eqref{3126}$ that $U\in C([0,T_*];H^{s_1})$, $U_t \in L^2([0,T_*];H^{s_2})$. Then, letting $k\to\infty$ in $\eqref{396}_1$-$\eqref{396}_2$, one gets that $\eqref{fp''}_1$ holds for a.e. $(t,x)\in (0,T_*)\times I$ and $\eqref{fp''}_2$ holds  continuously. 

Next, letting $k\to\infty$ in $\eqref{396}_1$, one deduces from $\eqref{fp''}_1$ and \eqref{3126} that  
\begin{equation}\label{equ5.15}
\begin{aligned}
\phi_0^2 U_t^k
\to &\left(\frac{\phi_0^{2} U_{x}}{\eta_x^2}\right)_x+\Big(\frac{1}{\alpha}-2\Big)\frac{\phi_0(\phi_0)_x U_x}{\eta_x^2}-\bigg(\frac{\phi_0^{2+\frac{1}{\alpha}}}{\eta_x^2}\bigg)_x\\
&\quad +\Big(2-\frac{1}{\alpha}\Big)\frac{\phi_0^{1+\frac{1}{\alpha}} (\phi_0)_x}{\eta_x^2}\quad \text{in }C([0,T]\times \bar I).  
\end{aligned} 
\end{equation}
This, along with $\eqref{3126}_2$ and the uniqueness of the limits, implies that
$$\phi_0^2U_t^k\to \phi_0^2 U_t\quad \text{in }C([0,T]\times \bar I),\quad \phi_0^2 U_t\in C([0,T]\times \bar I).$$
Hence, $\eqref{fp''}_1$ holds  continuously.

Moreover, it follows from the lower semi-continuity of weak convergence that \eqref{395k} still holds for $U$, that is, $\sup_{t\in[0,T_*]}E(t,U)<\infty$, which, along with Lemma \ref{hardy-inequality}, leads to 
\begin{equation}\label{rrggg}
\begin{aligned}
\sup_{t\in[0,T_*]}\absf{\partial_x^j U}_2&\leq C \sup_{t\in[0,T_*]}\left(\absf{\phi_0\partial_x^j U}_2+\absf{\phi_0\partial_x^{j+1} U}_2\right)\leq  C\sup_{t\in[0,T_*]} E(t,U)<\infty;\\
\sup_{t\in[0,T_*]}\absf{\partial_x^k U_t}_2&\leq C\sup_{t\in[0,T_*]} \big(\absf{\phi_0\partial_x^k U_t}_2+\absf{\phi_0\partial_x^{k+1} U_t}_2\big)\leq C \sup_{t\in[0,T_*]} E(t,U)<\infty,
\end{aligned}
\end{equation}
for $j=0,1,2,3$ and $k=0,1$, that is, $U\in  L^\infty([0,T_*];H^3)\cap  W^{1,\infty}([0,T_*];H^1)$.
 
The proof of the existence is completed.

\underline{\textbf{Step 3: Uniqueness, time continuity and \eqref{AN111}.}}
Suppose that there exist two solutions $U_1$ and $U_2$ on $[0,T_*]\times \bar I$. Define
\begin{equation*}
\begin{split}
\eta_i(t,x):=&x+\int_0^t U_i(s,x)\,\ds,\quad \weta:=\eta_2-\eta_1,\quad \wU:=U_2-U_1,\\
\wE(t):=&\sup_{s\in[0,t]}\absf{\phi_0 \wU}_2^2+\int_0^t\absf{\phi_0 \wU_{x}}_2^2\,\ds.   
\end{split}
\end{equation*}

Then, by \eqref{fp''}, $(\wU,\weta)$ solves the following problem
\begin{equation}\label{fp'}
\begin{cases}
\displaystyle \phi_0^2 \wU_t-\bigg(\frac{\phi_0^2 \wU_{x}}{(\eta_2)_x^2}\bigg)_x =\Big(\frac{1}{\alpha}-2\Big)\frac{\phi_0 (\phi_0)_x \wU_x}{(\eta_2)_x^2} +(\phi_0 \widetilde\cR_1)_x+ \widetilde\cR_2&\text{in }(0,T']\times I,\\
\weta_t=\wU&\text{in }(0,T']\times I,\\[4pt]
(\wU,\weta)=(0,0)&\text{on }\{t=0\}\times I,
\end{cases}
\end{equation}
where
\begin{equation*}
\begin{aligned}
&\widetilde\cR_1=  \big(\phi_0^{1+\frac{1}{\alpha}}-\phi_0(U_1)_{x}\big)\frac{((\eta_1)_x+(\eta_2)_x)\weta_x}{((\eta_1)_x(\eta_2)_x)^2},\\
&\widetilde\cR_2=\Big(\frac{1}{\alpha}-2\Big) \phi_0(\phi_0)_x\big(\phi_0^\frac{1}{\alpha}-(U_1)_x\big) \frac{((\eta_1)_x+(\eta_2)_x)\weta_x}{((\eta_1)_x(\eta_2)_x)^2}.
\end{aligned}
\end{equation*}
Clearly, $\widetilde\cR_i$ $(i=1,2)$ satisfies \eqref{basic} with $\wU^{k}$ replaced by $\wU$, that is,
\begin{equation}\label{basic'}
\begin{aligned}
\absf{\widetilde\cR_1(t)}_2^2+\absf{\widetilde\cR_2(t)}_2^2
\leq Ct\int_0^t\absf{\phi_0 \wU_x}_2^2\,\ds.
\end{aligned}
\end{equation}

Hence, following the proof of \textbf{Step 2}, multiplying $\eqref{fp'}_1$ by $\wU$ and integrating the resulting equality over $I$, one can get from  \eqref{basic'} that 
\begin{equation*}
\begin{aligned}
\frac{\mathrm{d}}{\dt}\absf{\phi_0 \wU}_2^2+\absf{\phi_0 \wU_x}_2^2 \leq C\absf{\phi_0 \wU}_2^2+Ct\int_0^t\absf{\phi_0 \wU_x}_2^2\,\ds,
\end{aligned}
\end{equation*}
which, along with Gr\"onwall's inequality and the definition of $T_*$, leads to $\wE(t) \equiv 0$ for all $0\leq t\leq T_*$.
Therefore, it follows from Lemma \ref{hardy-inequality} that $U_1\equiv U_2$.

The time continuity can be shown by following the proofs in Lemma \ref{existence-linearize}, which is omitted here for simplicity. Finally, \eqref{AN111} is a direct consequence of Lemmas \ref{sobolev-embedding}-\ref{hardy-inequality}. Therefore, one completes the proof of Theorem \ref{theorem3.1}.

\begin{Remark}
For the case $\frac{1}{3}<\alpha\leq 1$, the corresponding estimates in \eqref{rrggg} become
\begin{equation}\label{rrrggg}
\begin{aligned}
\sup_{t\in[0,T_*]}\absf{\partial_x^j U}_1&\leq C\sup_{t\in[0,T_*]}\big(\absb{\phi_0^{\frac{3}{2}-\varepsilon_0}\partial_x^j U}_2+\absb{\phi_0^{\frac{3}{2}-\varepsilon_0}\partial_x^{j+1} U}_2\big)\leq  C\sup_{t\in[0,T_*]}\widetilde E(t,U);\\
\sup_{t\in[0,T_*]}\absf{\partial_x^k U_t}_1&\leq C\sup_{t\in[0,T_*]}\big(\absb{\phi_0^{\frac{3}{2}-\varepsilon_0}\partial_x^k U_t}_2+\absb{\phi_0^{\frac{3}{2}-\varepsilon_0}\partial_x^{k+1} U_t}_2\big)\leq C \sup_{t\in[0,T_*]}\widetilde E(t,U),
\end{aligned}
\end{equation}
for $j=0,1,2,3$ and $k=0,1$, that is, $U\in  L^\infty([0,T_*];W^{3,1})\cap  W^{1,\infty}([0,T_*];W^{1,1})$.
\end{Remark}

\subsection{Proof of Theorem \ref{theorem1.3}}\label{proof1.3}
Define $(\rho(t,y),u(t,y))$ by \eqref{equ1.28} in Appendix \ref{section-CT}. Then it follows from Theorem \ref{theorem3.1} and \eqref{uty} that $(\rho(t,y),u(t,y),\Gamma(t))$ becomes the unique classical solution in $\overline{\II(T_*)}$ to the \textbf{VFBP} \eqref{shallow} satisfying \eqref{eulerregularity1} or \eqref{eulerregularity2}, which completes the proof of Theorem \ref{theorem1.3}.

\subsection{Local  existence and uniqueness  with a general initial flow map} 
To carry out the continuation argument in \S\ref{Section8}, we allow a general initial flow map. More precisely, let $\tau\geq0$ be fixed, and let $u_*=u_*(x)$ and $\eta_*=\eta_*(x)$ be given functions. Consider a flow map $\eta=\eta(t,x)$ satisfying
\begin{equation}
\eta(t,x) =\eta_*(x)+\int_\tau^t U(s,x)\,\mathrm{d}s\quad \text{for $t\geq \tau$},
\end{equation}
where $\eta_*$ satisfies the following regularity properties:
\begin{equation}\label{eta00}
c_*\leq (\eta_*)_x\leq c^* \text{ with $0<c_*\leq c^*<\infty$},\quad\big|\rho_0^\frac{1}{2}(\eta_*)_x\big|_2 +\sum_{k=2}^4\big|\rho_0^{(\frac{3}{2}-\varepsilon_0)\alpha}\partial_x^k\eta_*\big|_2 <\infty.
\end{equation}
Retaining the original Lagrangian labels, the conserved reference density remains $\rho_0=H\eta_x$. For $T>0$, we therefore consider the following version of \eqref{secondreformulation} on $(\tau,\tau+T]\times I$:
\begin{equation}\label{secondreformulationz}
\begin{cases}
\displaystyle\rho_0 U_t+ \left(\frac{\rho_0^2}{\eta_x^2}\right)_x-\left(\frac{\rho_0 U_x}{\eta_x^2}\right)_x=0&\text{in}  \ \ (\tau,\tau+T]\times I,\\[2pt]
\displaystyle \eta_t= U&\text{in} \ \  (\tau,\tau+T]\times I,\\[2pt]
(U,\eta) =(u_*,\eta_*)&\text{on}  \ \  \{t=\tau\}\times I.
\end{cases}
\end{equation}

We aim to prove the following theorem.
\begin{Theorem}\label{theorem5.1}
Let \eqref{distance} hold for $\frac{1}{3}<\alpha \leq 1$. Assume $(\rho_0,u_*,\eta_*)$ satisfies \eqref{eta00} and 
\begin{equation}\label{a1z}
2+\|\rho_0^\alpha\|_{3}+\big|\rho_0^\frac{1}{2}(\eta_*)_x\big|_2^2 +\sum_{k=2}^4\big|\rho_0^{(\frac{3}{2}-\varepsilon_0)\alpha}\partial_x^k\eta_*\big|_2^2+\widetilde E(\tau,U)\leq \mathfrak{M}\quad \text{for some $\mathfrak{M}>0$},
\end{equation}
then there exists a time $T_*>0$, which depends only on $(c_*,c^*,\cC_1,\cC_2,\alpha,\varepsilon_0,|I|,\mathfrak{M})$ and is independent of $\tau$, such that problem \eqref{secondreformulationz} with $T=T_*$ admits a unique classical solution $(U,\eta)$ in $[\tau,\tau+T_*]\times \bar I $ satisfying
\begin{equation}\label{b111z}
\begin{gathered}
U\in \mathscr{C}([\tau,\tau+T_*];\widetilde E),\quad \eta\in \mathscr{C}^1([\tau,\tau+T_*];\widetilde E);\\
\frac{c_*}{2}\leq \eta_x(t,x)\leq 2c^* \ \ \text{for all } (t,x)\in [\tau,\tau+T_*]\times\bar I.
\end{gathered}
\end{equation}
Moreover,  \eqref{N111}-\eqref{AN111} hold on $(\tau,\tau+T_*]$ and \eqref{continuous'} holds on $[\tau,\tau+T_*]\times I$. 
\end{Theorem}

\begin{Remark}
We note that problem \eqref{secondreformulationz} is used only for the continuation argument in \S\ref{Section8} and is not equivalent to \eqref{lagrange}. The temporal derivatives at $t=\tau$ entering $\widetilde E(\tau,U)$ are determined recursively from \eqref{secondreformulationz}, which are similar to \eqref{1..16} in Appendix \ref{AppB}:
\begin{equation}\label{1..16z}
\begin{aligned}
U_t(\tau,x)&=\frac{1}{\rho_0}\Big( \frac{\rho_0(u_*)_{x}}{(\eta_*)_x^2}-\frac{\rho_0^{2}}{(\eta_*)_x^2}\Big)_x, \\
U_{tx}(\tau,x)&=\Big(\frac{1}{\rho_0}\big(\frac{\rho_0(u_*)_{x}}{(\eta_*)_x^2}-\frac{\rho_0^{2}}{(\eta_*)_x^2}\big)_x\Big)_x,\\
U_{txx}(\tau,x)&=\Big(\frac{1}{\rho_0}\big(\frac{\rho_0(u_*)_{x}}{(\eta_*)_x^2}-\frac{\rho_0^{2}}{(\eta_*)_x^2}\big)_x\Big)_{xx}, \\
U_{tt}(\tau,x)&=\rho_0^{-1}\Big(\frac{\rho_0 U_{tx}(\tau,x)}{(\eta_*)_x^2}-2\frac{\rho_0 (u_*)_x^2}{(\eta_*)_x^3}+2\frac{\rho_0^2(u_*)_x}{(\eta_*)_x^3}\Big)_x. 
\end{aligned}
\end{equation}

Following the same argument as in Appendix \ref{AppB}, we can show that \eqref{1..16z} implies the compatibility condition $(u_*)_x|_{x\in\Gamma}=0$. Indeed,  $\eqref{1..16z}_1$ implies 
\begin{equation}\label{fulubianjiez}
U_t(\tau,x)=\frac{1}{(\eta_*)_x^2}\Big((u_*)_{xx}+\frac{1}{\alpha}\frac{(\rho_0^\alpha)_x}{\rho_0^{\alpha}}(u_*)_x -\frac{2}{\alpha}\rho_0^{1-\alpha}(\rho_0^\alpha)_x\Big)-\frac{2\big((u_*)_x-\rho_0 \big)(\eta_*)_{xx}}{(\eta_*)_x^3}.
\end{equation}
Then,  taking the $L^\infty$ norm,  \eqref{distance}, \eqref{a1z}, and Lemmas \ref{sobolev-embedding}-\ref{hardy-inequality} give 
\begin{equation*}
\begin{aligned}
\absf{\rho_0^{-\alpha}(\rho_0^\alpha)_x(u_*)_x}_\infty&\leq C\big( \abs{U_t(\tau)}_\infty+\norm{(u_*)_{x}}_{1,\infty}+\abs{\rho_0^\alpha}_\infty^{\frac{1-\alpha}{\alpha}} \|\rho_0^\alpha\|_{1,\infty}\big)\big(1+|(\eta_*)_{xx}|_\infty\big)\\
&\leq C  \widetilde E(\tau,U) +C\leq C,
\end{aligned}
\end{equation*}
where $C$ depends only on $(c_*,c^*,\cC_1,\cC_2,\alpha,\mathfrak{M})$. Since $\rho_0^\alpha\sim d(x)$, we obtain $(u_*)_x|_{x\in\Gamma}=0$.
\end{Remark}

\begin{proof}[Proof of Theorem \ref{theorem5.1}]
We adapt the construction in \S\ref{Section3}-\S\ref{Section5}. Throughout the proof, $C$ denotes a
generic positive constant, and $\mathscr P$ denotes a generic polynomial
with positive coefficients. Both $C$ and the coefficients of
$\mathscr P$ depend only on
$(c_*,c^*,\cC_1,\cC_2,\alpha,\varepsilon_0,|I|,\mathfrak M)$ and are independent of $\tau$. Recall that
$\phi_0=\rho_0^\alpha$. The bound in \eqref{a1z} can be written as
\begin{equation}\label{xin-e0}
2+\|\phi_0\|_{3}+\big|\phi_0^\frac{1}{2\alpha}(\eta_*)_x\big|_2^2 +\sum_{k=2}^4\big|\phi_0^{\frac{3}{2}-\varepsilon_0}\partial_x^k\eta_*\big|_2^2+\widetilde E(\tau,U)\leq \mathfrak{M}.
\end{equation}

\underline{\textbf{Step 1: Estimates for a general initial flow map.}}
The temporal derivatives at $t=\tau$ are defined recursively by \eqref{1..16z}.
Thus, \eqref{a1z} controls all the initial terms occurring in the proof of
Lemma \ref{a-priori-estimates2}. In particular, the compatibility
calculation following \eqref{fulubianjiez} gives
\begin{equation*}
\absf{(u_*)_x(x)}\leq C\rho_0^\alpha(x)\leq Cd(x)
\quad\text{for all }x\in\bar I,
\qquad (u_*)_x\big|_{\Gamma}=0.
\end{equation*}

We first derive the coefficient estimates that replace
\eqref{useful2'}--\eqref{useful3'}. Since
$0<\varepsilon_0<1$ by \eqref{varepsilon0},  \eqref{xin-e0}, together with
Lemmas \ref{sobolev-embedding}--\ref{hardy-inequality},
implies
\begin{equation}\label{general-initial-map-estimates}
\absf{(\eta_*)_{xx}}_\infty
+\absf{\phi_0\partial_x^3\eta_*}_\infty
+\sum_{j=2}^3
\absb{\phi_0^{\frac12-\varepsilon_0}\partial_x^j\eta_*}_2 +\sum_{j=2}^4
\absb{\phi_0^{\frac32-\varepsilon_0}\partial_x^j\eta_*}_2
\leq C.
\end{equation}

Let $\bar U$ satisfy the bounds in \eqref{39''}, with $[0,T]$
replaced by $[\tau,\tau+T]$, and set
\begin{equation}\label{general-given-flow}
\bar\eta(t,x)=\eta_*(x)+\int_\tau^t\bar U(s,x)\,\ds.
\end{equation}
The argument for \eqref{useful1'} gives
\begin{equation*}
\normf{\bar U_x(t)}_{1,\infty}
+\absf{\phi_0\partial_x^3\bar U(t)}_\infty
\leq\mathscr P(c_2).
\end{equation*}
Combining \eqref{general-initial-map-estimates}--\eqref{general-given-flow},
we obtain
\begin{equation}\label{general-flow-map-bounds}
\begin{aligned}
\absf{\bar\eta_x(t)-(\eta_*)_x}_\infty
&\leq\mathscr P(c_2)(t-\tau),\\
\absf{\bar\eta_{xx}(t)}_\infty
+\absf{\phi_0\partial_x^3\bar\eta(t)}_\infty
+\sum_{j=2}^3
\absb{\phi_0^{\frac12-\varepsilon_0}\partial_x^j\bar\eta(t)}_2
&\leq C+\mathscr P(c_2)(t-\tau),\\
\sum_{j=2}^4
\absb{\phi_0^{\frac32-\varepsilon_0}\partial_x^j\bar\eta(t)}_2
&\leq C+\mathscr P(c_2)(t-\tau).
\end{aligned}
\end{equation}
We choose $T>0$ sufficiently small so that
\begin{equation*}
\mathscr P(c_2)T\leq\frac{c_*}{2}.
\end{equation*}
It then follows from \eqref{general-flow-map-bounds} and  $c_*\leq (\eta_*)_x\leq c^*$ that
\begin{equation}\label{general-flow-Jacobian-bounds}
\frac{c_*}{2}\leq\bar\eta_x(t,x)\leq
c^*+\frac{c_*}{2}\leq2c^*
\quad\text{for all }(t,x)\in[\tau,\tau+T]\times\bar I.
\end{equation}
Thus, \eqref{general-flow-map-bounds}--\eqref{general-flow-Jacobian-bounds}
are precisely the general-initial-map counterparts of
\eqref{useful2'}--\eqref{useful3'}.  

\underline{\textbf{Step 2: Construction and uniform estimates.}}
Take
\begin{equation*}
U^0(t,x)=u_*(x),\qquad
\eta^0(t,x)=\eta_*(x)+(t-\tau)u_*(x).
\end{equation*}
For $k\geq0$, once $(U^k,\eta^k)$ is known, we define
$(U^{k+1},\eta^{k+1})$ by
\begin{equation}\label{general-flow-iteration}
\begin{cases}
\displaystyle
\phi_0^{\frac1\alpha}U_t^{k+1}
-\Big(\dfrac{\phi_0^{\frac1\alpha}U_x^{k+1}}
{(\eta_x^k)^2}\Big)_x
+\Big(\dfrac{\phi_0^{\frac2\alpha}}
{(\eta_x^k)^2}\Big)_x=0
&\text{in }(\tau,\tau+T]\times I,\\
\displaystyle
\eta^{k+1}(t,x)=\eta_*(x)
+\int_\tau^tU^{k+1}(s,x)\,\ds
&\text{in }[\tau,\tau+T]\times I,\\[4pt]
(U^{k+1},\eta^{k+1})=(u_*,\eta_*)
&\text{on }\{t=\tau\}\times I.
\end{cases}
\end{equation}

First, we find that the proof of Lemma \ref{existence-linearize} ii), based on Proposition \ref{prop042}, applies to $\eqref{general-flow-iteration}_1$. Indeed, the Galerkin argument uses only the positive lower and upper bounds for $\eta_x^k$ and the weighted bounds for its spatial derivatives, rather than the identity $\eta^k(\tau,x)=x$. These coefficient bounds follow from
\eqref{general-flow-map-bounds}--\eqref{general-flow-Jacobian-bounds}.
The initial values of $U_t^{k+1}$, $U_{tx}^{k+1}$,
$U_{txx}^{k+1}$, and $U_{tt}^{k+1}$ are the quantities determined by
the recursive formulas in \eqref{1..16z}.

We next apply the proof of Lemma \ref{a-priori-estimates2}. The  identities \eqref{3...143}, \eqref{equ3.72}, \eqref{3...151}, and
\eqref{3-154} remain unchanged, since $\bar\eta_t=\bar U$. In the estimates following
\eqref{44500}, \eqref{5049}, \eqref{equ3.153}, and \eqref{4068},
we use \eqref{general-flow-map-bounds} in place of
\eqref{useful3'}. Hence, every factor
$\mathscr P(c_2)t$ arising from a higher-order derivative of the flow map
is replaced by
\begin{equation*}
C+\mathscr P(c_2)(t-\tau).
\end{equation*}
All the time integrals are now taken over $[\tau,t]$, so every small factor $t$ in the original proof is replaced by $t-\tau$. The additional constant is controlled by \eqref{xin-e0} and can be incorporated into $\mathscr P(c_0)$. Therefore, the arguments leading to \eqref{equu3145}, \eqref{equ3.148}, \eqref{equ3.151}, \eqref{equ3.154}, and \eqref{equ3.157}--\eqref{equ3.158} remain valid, with constants depending only on $(c_*,c^*,\cC_1,\cC_2,\alpha,\varepsilon_0,|I|,\mathfrak M)$.  

Taking $c_0=\mathfrak M$ and choosing $c_1$, $c_2$, and $T$ successively
as in \textbf{Step 4} of the proof of Lemma \ref{a-priori-estimates2}, we
obtain, after decreasing $T$ if necessary,
\begin{equation}\label{general-flow-uniform-iteration-bounds}
\begin{aligned}
&\frac{c_*}{2}\leq\eta_x^k(t,x)\leq2c^*
\quad\text{for all }(t,x)\in[\tau,\tau+T]\times\bar I,\\
&\sup_{k\geq0}\sup_{t\in[\tau,\tau+T]}
\widetilde E(t,U^k)\leq C,\qquad \sup_{k\geq0}\sup_{t\in[\tau,\tau+T]}
\sum_{j=2}^4
\absb{\phi_0^{\frac32-\varepsilon_0}
\partial_x^j\eta^k(t)}_2\leq C.
\end{aligned}
\end{equation}
In particular, $T$ depends only on
$(c_*,c^*,\cC_1,\cC_2,\alpha,\varepsilon_0,|I|,\mathfrak M)$ and is independent of $\tau$.

\underline{\textbf{Step 3: Convergence of the approximate solutions.}}
Set
\begin{equation*}
\wU^{k+1}=U^{k+1}-U^k,\qquad
\weta^{k+1}=\eta^{k+1}-\eta^k
=\int_\tau^t\wU^{k+1}(s,x)\,\ds,
\end{equation*}
and define
\begin{equation*}
\wE^k(t)=
\sup_{s\in[\tau,t]}
\absb{\rho_0^{\frac12}\wU^k(s)}_2^2
+\int_\tau^t
\absb{\rho_0^{\frac12}\wU_x^k(s)}_2^2\,\ds.
\end{equation*}

For $k\geq1$, subtracting two consecutive equations in
\eqref{general-flow-iteration} gives
\begin{equation}\label{general-flow-difference-equation}
\rho_0\wU_t^{k+1}
-\Big(\frac{\rho_0\wU_x^{k+1}}{(\eta_x^k)^2}\Big)_x
=\big(\rho_0^{\frac12}\cR_1^k\big)_x,
\end{equation}
where
\begin{equation*}
\cR_1^k=\rho_0^{\frac12}(\rho_0-U_x^k)
\frac{(\eta_x^k+\eta_x^{k-1})\weta_x^k}
{(\eta_x^k\eta_x^{k-1})^2}.
\end{equation*}

Notice that
\begin{equation*}
\absb{\rho_0^{\frac12}\weta_x^k(t)}_2^2
\leq(t-\tau)\int_\tau^t
\absb{\rho_0^{\frac12}\wU_x^k(s)}_2^2\,\ds.
\end{equation*}
Combining this estimate with
\eqref{general-flow-uniform-iteration-bounds} yields
\begin{equation}\label{general-flow-remainder-estimate}
\absf{\cR_1^k(t)}_2^2
\leq C(t-\tau)\int_\tau^t
\absb{\rho_0^{\frac12}\wU_x^k(s)}_2^2\,\ds.
\end{equation}

Multiplying \eqref{general-flow-difference-equation} by
$\wU^{k+1}$, integrating over $I$, and using
\eqref{general-flow-Jacobian-bounds} and
\eqref{general-flow-remainder-estimate}, we obtain
\begin{equation*}
\frac{\mathrm{d}}{\dt}
\absb{\rho_0^{\frac12}\wU^{k+1}}_2^2
+c\absb{\rho_0^{\frac12}\wU_x^{k+1}}_2^2
\leq C\absb{\rho_0^{\frac12}\wU^{k+1}}_2^2
+C\absf{\cR_1^k}_2^2.
\end{equation*}
It follows from Gr\"onwall's inequality that
\begin{equation}\label{general-flow-contraction}
\wE^{k+1}(t)
\leq C(t-\tau)^2e^{C(t-\tau)}\wE^k(t)
\quad\text{for all }t\in[\tau,\tau+T].
\end{equation}
We decrease $T$ once more so that
\begin{equation*}
CT^2e^{CT}\leq\frac12.
\end{equation*}
It follows that
\begin{equation*}
\wE^{k+1}(\tau+T)\leq\frac12\wE^k(\tau+T),
\qquad
\sum_{k=1}^{\infty}\big(\wE^k(\tau+T)\big)^{\frac12}<\infty,
\end{equation*}
which is similar to \eqref{total-bound}.

Finally, following an argument similar to the rest of \textbf{Step 2} in the proof of Theorem \ref{theorem3.1}, we can obtain the existence of the solution $(U,\eta)$ to problem \eqref{secondreformulationz}.

\underline{\textbf{Step 4: Uniqueness, time continuity,  and \eqref{N111}--\eqref{AN111}.}} The uniqueness and the time continuity follow by the same argument as in \textbf{Step 3} of the proof of Theorem \ref{theorem3.1}. We only sketch the proof of \eqref{N111}--\eqref{AN111} in this case.

By $\eqref{secondreformulationz}_1$, for
$(t,x)\in(\tau,\tau+T]\times I$,
\begin{equation*}
U_t=\frac{1}{\eta_x^2}
\Big(
U_{xx}+\frac{1}{\alpha}\rho_0^{-\alpha}(\rho_0^\alpha)_xU_x
-\frac{2}{\alpha}\rho_0^{1-\alpha}(\rho_0^\alpha)_x
\Big)
-\frac{2\eta_{xx}(U_x-\rho_0)}{\eta_x^3}.
\end{equation*}
Equivalently,
\begin{equation*}
\rho_0^{-\alpha}(\rho_0^\alpha)_xU_x
=\alpha\eta_x^2U_t-\alpha U_{xx}
+2\rho_0^{1-\alpha}(\rho_0^\alpha)_x
+\frac{2\alpha\eta_{xx}}{\eta_x}(U_x-\rho_0).
\end{equation*}
Using the calculation similar to \eqref{rrrggg}, together with
Lemma \ref{sobolev-embedding} and \eqref{general-flow-map-bounds}--\eqref{general-flow-Jacobian-bounds}, every term on the right-hand side of the above is bounded uniformly on
$(\tau,\tau+T]\times I$. Hence,
\begin{equation*}
\absf{\rho_0^{-\alpha}(\rho_0^\alpha)_xU_x}_\infty\leq C.
\end{equation*}
Since $d(x)\sim \rho_0^\alpha\in H^3(I)\subset C^1(\bar I)$, this implies
\begin{equation*}
\absf{U_x(t,x)}\leq Cd(x)
\quad\text{for all }(t,x)\in(\tau,\tau+T]\times\bar I,
\qquad U_x\big|_{(\tau,\tau+T]\times\Gamma}=0.
\end{equation*}
Together with the compatibility estimate for $(u_*)_x$ obtained in
\textbf{Step 1}, this also holds at $t=\tau$. Consequently,
\eqref{N111}--\eqref{AN111} hold with $(0,T_*]$ replaced by
$(\tau,\tau+T]$.

Finally, taking $T_*=T$, we obtain all the conclusions of Theorem \ref{theorem5.1}.  
\end{proof}

\section{Global-in-time boundedness  of the effective velocity and \texorpdfstring{$\eta_x$}{}}\label{Section6}

We will give the proof for the global existence and uniqueness results  stated in Theorem \ref{Theorem1.1} in \S6-\S8. According to Theorem \ref{theorem3.1} ii), there exists a unique local-in-time classical solution $U$ in $[0,T_*]\times \bar I$ to the problem \ef{secondreformulation} for some positive time $T_*$, which satisfies \eqref{b111'} and the homogeneous Neumann boundary condition, $U_x(t,x)|_{\Gamma}=0$ for $t\in  [0,T_*]$. Let $\overline T_*$ denote the maximal existence time of the above
classical solution. Hereinafter, $0<T<\overline T_*$ is arbitrary.
The constants $C(\cdots,T)$ in the following a priori estimates depend only on $(\rho_0,u_0,\varepsilon_0,\alpha,|I|,T)$ and can be chosen
increasing in $T$ and finite for every finite $T$.  


\subsection{The lower bound of $\eta_x$}\label{subsection6.1}
First, the so-called effective velocity is defined as follows.
\begin{Definition}
Let $U$, $H$, $\eta$, $\rho_0$, $\alpha$ be defined as in \S \ref{section1}. $V$ is said to be the effective velocity if
\begin{equation}\label{V-expression}
V=U+\frac{H_x}{\rho_0}=U+\frac{1}{\alpha}\frac{(\rho_0^\alpha)_x}{\rho_0^\alpha\eta_x}-\frac{\eta_{xx}}{\eta_x^2}.
\end{equation}
\end{Definition}

Next, we give the  fundamental energy estimates and BD entropy estimates.
\begin{Lemma}\label{BD-entropy}
For any $T>0$, it holds that for all $0\leq t\leq T$,
\begin{itemize}
\item fundamental energy estimates $(0<\alpha\leq 1)$
\begin{equation}\label{energy-e}
\int \Big(\rho_0 U^2+\frac{\rho_0^2}{\eta_x}\Big)\,\dx+\int_0^t\int \frac{\rho_0U_x^2}{\eta_x^2}\,\dx\ds\leq C;
\end{equation}
\item BD entropy estimates $(0<\alpha<1)$
\begin{equation}\label{BD-e}
\int \Big(\rho_0\Big|U+\frac{H_x}{\rho_0}\Big|^2+\frac{\rho_0^2}{\eta_x}\Big)\,\dx+\int_0^t\int \frac{H_x^2}{\eta_x}\,\dx\ds\leq C.
\end{equation}
\end{itemize}
\end{Lemma}
\begin{proof}
For simplicity, we only give the proof of the BD entropy estimates. It follows from the equation $\eqref{lagrange}_1$ by applying $\partial_x$ that
\begin{equation*}
H_{tx}+\left(H\frac{U_x}{\eta_x}\right)_x=0,
\end{equation*}
which, together with $\eqref{lagrange}_2$ and \eqref{V-expression}, leads to
\begin{equation}\label{B3}
\rho_0 V_t+(H^2)_x=0.
\end{equation}
Thus, multiplying \eqref{B3} by $V$ and integrating the resulting equality yield that
\begin{equation}
\frac{\mathrm{d}}{\dt}\int \Big(\frac{1}{2}\rho_0 V^2+\eta_x H^2\Big)\,\dx+2\int \frac{H_x^2}{\eta_x}\,\dx=0.
\end{equation}
Integrating above over $[0,T]$ gives the desired conclusion.
\end{proof}

\begin{Remark}
It follows from \eqref{distance} that  $\rho_0$ satisfies the initial requirement of the  BD entropy estimate when $0<\alpha<1$, i.e.,
\begin{equation*}
\begin{aligned}
\int \rho_0(\log\rho_0)_x^2\,\dx&=\frac{1}{\alpha^2}\int \rho_0^{1-2\alpha}(\rho_0^\alpha)_x^2\,\dx\leq C\int d(x)^{\frac{1}{\alpha}-2}\,\dx\leq C.
\end{aligned}
\end{equation*}
\end{Remark}

Clearly, by \eqref{V-expression} and \eqref{B3}, one can deduce the following corollary.
\begin{Corollary}
The effective velocity $V$ satisfies the following equation:
\begin{equation}\label{eq:effective2}
 V_t+ 2H(V-U)=0.
\end{equation}
\end{Corollary}

Based on the above discussions,  one can derive the upper bound of depth.
\begin{Lemma}\label{bound of density}
For any $T>0$ and $0<\alpha\leq 1$, it holds that
\begin{equation*}
\abs{H(t)}_\infty\leq C \ \ \text{for all }0\leq t\leq T.
\end{equation*}
\end{Lemma}
\begin{proof}
Integrating $\eqref{secondreformulation}_1$ over $(0,x)$ for $x\in I$ shows that
\begin{equation*}
\frac{\mathrm{d}}{\dt}\int_0^x \rho_0U\,\mathrm{d}z-\frac{\rho_0 U_x}{\eta_x^2}+\frac{\rho_0^2}{\eta_x^2}=0,
\end{equation*}
which, along with $\eqref{lagrange}_1$, yields 
\begin{equation}\label{zuoban}
\frac{\mathrm{d}}{\dt}\left(\int_0^x \rho_0U\,\mathrm{d}z+H\right)+\frac{\rho_0^2}{\eta_x^2}=0.
\end{equation}

Then, integrating above over $[0,t]$ gives
\begin{equation}\label{eq:411}
\int_0^x \rho_0U(t,z)\,\mathrm{d}z+H(t,x)\leq \int_0^x \rho_0 u_0\,\mathrm{d}z+\rho_0(x).
\end{equation}

Finally, it follows from \eqref{energy-e} in Lemma \ref{BD-entropy} that
\begin{equation*}
\abs{\int_0^x \rho_0U(t,z)\,\mathrm{d}z}\leq \abs{\rho_0^\alpha}_\infty^\frac{1}{2\alpha}\big|\rho_0^\frac{1}{2}U\big|_2\leq C \ \ \text{for all } 0\leq t\leq T,
\end{equation*}
which, along  with \eqref{eq:411}, yields the boundedness of $H$.

Thus, the proof of Lemma \ref{bound of density} is completed.
\end{proof}

Next, we can obtain the $L^q$ ($2\leq q<\infty$) estimates of the velocity.
\begin{Lemma}\label{W.E.I}
For any $T>0$, $0<\alpha\leq 1$ and  $0\leq p<\infty$, it holds that
\begin{equation*}
\big|\rho_0^\frac{1}{p+2}U(t)\big|_{p+2}^{p+2}+\int_0^t \Big|\rho_0^{\frac{1}{2}}\frac{|U|^{\frac{p}{2}}U_x}{\eta_x}\Big|_{2}^2\,\ds\leq C(p,T) \ \ \text{for all }0\leq t\leq T.\end{equation*}
\end{Lemma}
\begin{proof}
Multiplying $\eqref{secondreformulation}_1$ by $\abs{U}^p U$ $(0\leq p<\infty)$ yields
\begin{equation*}
\begin{aligned}
&\frac{1}{p+2} \left(\rho_0\abs{U}^{p+2}\right)_t+\left(\frac{\rho_0^2\abs{U}^pU}{\eta_x^2}-\frac{\rho_0 \abs{U}^pUU_x}{\eta_x^2}\right)_x\\
=&(p+1)\left(\frac{\rho_0^2}{\eta_x^2}-\frac{\rho_0U_x}{\eta_x^2}\right)\abs{U}^pU_x,
\end{aligned}
\end{equation*}
which, along with Lemma \ref{bound of density}, $H=\frac{\rho_0}{\eta_x}$, the H\"older and Young inequalities, implies 
\begin{align*}
&\frac{1}{p+2}\frac{\mathrm{d}}{\dt}\big|\rho_0^{\frac{1}{p+2}}U\big|_{p+2}^{p+2}+(p+1)\Big|\rho_0^{\frac{1}{2}}\frac{\abs{U}^{\frac{p}{2}}U_x}{\eta_x}\Big|_{2}^2\\
=&(p+1)\int \frac{\rho_0^2\abs{U}^pU_x}{\eta_x^2}\,\dx=(p+1)\int H^2\abs{U}^pU_x \,\dx\\
\leq &C(p)\int  \rho_0 H^2\abs{U}^p \,\dx+\frac{p+1}{8}\Big|\rho_0^{\frac{1}{2}}\frac{\abs{U}^{\frac{p}{2}}U_x}{\eta_x}\Big|_{2}^2\\
\leq& C(p)\abs{\rho_0}_{\infty}^\frac{2}{p+2}|H|_\infty^2\big|\rho_0^{\frac{1}{p+2}}U\big|_{p+2}^{p}+\frac{p+1}{8}\Big|\rho_0^{\frac{1}{2}}\frac{\abs{U}^{\frac{p}{2}}U_x}{\eta_x}\Big|_{2}^2\\
\leq& C(p) +\big|\rho_0^{\frac{1}{p+2}}U\big|_{p+2}^{p+2}+\frac{p+1}{8}\Big|\rho_0^{\frac{1}{2}}\frac{\abs{U}^{\frac{p}{2}}U_x}{\eta_x}\Big|_{2}^2.
\end{align*}

Then, it follows from the Gr\"onwall inequality that
\begin{equation}\label{equ67}
\big|\rho_0^\frac{1}{p+2}U(t)\big|_{p+2}^{p+2}+\int_0^t \Big|\rho_0^{\frac{1}{2}}\frac{|U|^{\frac{p}{2}}U_x}{\eta_x}\Big|_{2}^2\,\ds\leq C(p,T)\big(\big|\rho_0^\frac{1}{p+2}u_0\big|_{p+2}^{p+2}+1\big),
\end{equation}
for all $0\leq t\leq T$. For the initial data, one can get from Lemmas \ref{sobolev-embedding}-\ref{hardy-inequality} that
\begin{equation*}
\begin{aligned}
\big|\rho_0^{\frac{1}{p+2}}u_0\big|^{p+2}_{p+2}&\leq \abs{\rho_0}_\infty\abs{u_0}_\infty^{p+2}\leq C(p),
\end{aligned}
\end{equation*}
which, along with \eqref{equ67}, yields the desired estimate in this lemma.

The proof of Lemma \ref{W.E.I} is completed.
\end{proof}

Now, one can get the following lower bound of $\eta_x$.
\begin{Lemma}\label{L.B.J}
For any $T>0$ and $0<\alpha\leq 1$, it holds that
\begin{equation*}
\inf_{(t,x)\in [0,T]\times \bar I} \eta_x(t,x)\geq C^{-1}(T) >0.
\end{equation*}
\end{Lemma}
\begin{proof}
First, define 
\begin{equation*}
F(t,x)=\begin{cases}
\displaystyle\int_0^x\rho_0(z)U(t,z)\,\mathrm{d}z,
&0\leq x\le \frac12,\\[1.2ex]
\displaystyle-\int_x^1\rho_0(z)U(t,z)\,\mathrm{d}z,
&\frac12<x\leq 1.
\end{cases}
\end{equation*}
Then integrating $\eqref{secondreformulation}_1$ over $(0,x)$ for $0\leq x\leq \frac{1}{2}$ and $(x,1)$ for $\frac{1}{2}<x\leq 1$, respectively, and using $\eqref{lagrange}_1$, we can obtain an equation similar to \eqref{zuoban}:
\begin{equation}\label{youban}
H_t+H^2=-F_t.  
\end{equation}

Next, it follows from Lemma \ref{W.E.I}, $\rho_0^\alpha\sim d(x)$, and the H\"older inequality that, for any $q\geq 2$ and $(t,x)\in [0,T]\times [0,\frac{1}{2}]$,
\begin{equation}\label{q-renyi}
|F(t,x)| \le
\Big(\int_{0}^x\rho_0\,\mathrm{d}z\Big)^{1-\frac1q}
\Big(\int_I \rho_0|U|^q\,\mathrm{d}z\Big)^{\frac1q} \le C(q,T) 
\rho_0^{(1+ \alpha)(1-\frac1q)}.
\end{equation}
Hence, taking $q=\frac{1+\alpha}{\alpha}$ ($q\geq 2$ since $0<\alpha\leq 1$) in the above yields
\begin{equation}\label{abb}
|F(t,x)| \le C(T)\rho_0\qquad\text{for $(t,x)\in [0,T]\times [0,\frac{1}{2}]$}.
\end{equation}
Similarly, applying the above argument, we can also show that \eqref{abb} holds for $(t,x)\in [0,T]\times [\frac{1}{2},1]$. Therefore, we arrive at
\begin{equation}\label{abb2}
|F(t,x)| \le C(T)\rho_0\qquad\text{for $(t,x)\in [0,T]\times \bar I$}.
\end{equation}

Finally, integrating \eqref{youban} with respect to $t$, we obtain from \eqref{abb2} that, for any $(t,x)\in [0,T]\times \bar I$,
\begin{equation*}
H(t,x)
=\rho_0(x)+F(0,x)-F(t,x)
-\underline{\int_0^tH(s,x)^2\,\mathrm{d}s}_{\,\geq 0}\leq C(T)\rho_0,
\end{equation*}
which, along with equation $H=\frac{\rho_0}{\eta_x}$, yields the desired uniform lower bound of $\eta_x$.
\end{proof}

\subsection{\texorpdfstring{$L^p$}{} $(2\leq p\leq \infty)$ estimates of the effective velocity}\label{subsection6.2}
Now, one can derive the $L^p$ ($2\leq p<\infty$) estimates of $V$.

\begin{Lemma}\label{boundv}
For any $T>0$, $0<\alpha\leq 1$, $r>\frac{p-1}{p}$ and $2\leq p<\infty$, it holds that
\begin{equation*}
\abs{\rho_0^{r\alpha} V(t)}_p\leq C(r,p,T)\ \ \text{for all }0\leq t\leq T.
\end{equation*}
\end{Lemma}
\begin{proof}
Multiplying \eqref{eq:effective2} by $\rho_0^{r\alpha}$ ($r>\frac{p-1}{p}$, $2\leq p<\infty$) yields
\begin{equation*}
\left(\rho_0^{r\alpha} V(t,x)\right)_t+2H\left(\rho_0^{r\alpha} V(t,x)\right)=\frac{2\rho_0^{r\alpha+1} U(t,x)}{\eta_x(t,x)}.
\end{equation*}

Then, one can solve the above ODE to deduce that
\begin{equation}\label{solutionv}
\rho_0^{r\alpha} V(t,x)=\rho_0^{r\alpha} V(0,x)e^{-\int_0^t 2H(s,x)\,\ds}+ \int_0^t \frac{2\rho_0^{r\alpha+1} U(\tau,x)}{\eta_x(\tau,x)}e^{-\int_\tau^t 2H(s,x)\,\ds}\,\mathrm{d}\tau.
\end{equation}
Taking the $L^p$-norm ($2\leq p<\infty$) of both sides of \eqref{solutionv}, one gets from $H\geq 0$,  Lemmas \ref{W.E.I}-\ref{L.B.J} and the Minkowski integral inequality that 
\begin{equation}\label{4431}
\begin{aligned}
\abs{\rho_0^{r\alpha} V(t)}_p&\leq \abs{\rho_0^{r\alpha} V(0)}_p+2\int_0^t \Big|\frac{\rho_0^{r\alpha+1}U}{\eta_x}\Big|_p\,\ds\\
&\leq  \abs{\rho_0^{r\alpha}V(0)}_p+C(T)\abs{\rho_0^\alpha}_\infty^{r+\frac{1}{\alpha}-\frac{1}{\alpha p}}\int_0^t \big|\rho_0^{\frac{1}{p}}U\big|_p\,\ds\\
&\leq  \abs{\rho_0^{r\alpha} V(0)}_p+C(r,p,T).
\end{aligned}
\end{equation}
For the initial data, since $(r-1)p>-1$, one has
\begin{equation*}
\begin{aligned}
\abs{\rho_0^{r\alpha} V(0)}_p&=\absB{\rho_0^{r\alpha}\left(u_0+\frac{1}{\alpha}\frac{(\rho_0^\alpha)_x}{\rho_0^\alpha}\right)}_p\\
&\leq C\big( \abs{\rho_0^\alpha}_\infty^r \abs{u_0}_p+ \big|\rho_0^{(r-1)\alpha}\big|_p\abs{(\rho_0^\alpha)_x}_\infty\big)\\
&\leq C(r) \widetilde E(0,U)+C\left(\int d(x)^{(r-1)p}\,\dx\right)^\frac{1}{p}
\leq C(r,p),
\end{aligned}
\end{equation*}
which, along with \eqref{4431}, yields the desired estimate in this lemma.

The proof of Lemma \ref{boundv} is completed.
\end{proof}

Finally, one can obtain the weighted  $L^\infty$-boundedness of $V$.  
\begin{Lemma}\label{boundv'}
For any $T>0$ and $0<\alpha\leq 1$, it holds that
\begin{equation*}
\abs{\rho_0^\alpha V(t)}_\infty\leq C(T)\ \ \text{for all }0\leq t\leq T.
\end{equation*}
\end{Lemma}
\begin{proof}

According to Lemmas \ref{L.B.J} and \ref{boundv}, Lemma \ref{sobolev-embedding} and H\"older's inequality, one has
\begin{align*}
\absB{\frac{\rho_0^{1+\alpha}U}{\eta_x}}_\infty&=\absf{\rho_0^\alpha H U}_\infty\leq C\big(\absf{\rho_0^\alpha H U}_1+\abs{(\rho_0^\alpha)_x HU}_1+\abs{\rho_0^\alpha H_xU}_1+\abs{\rho_0^\alpha HU_x}_1\big)\\
&\leq C\Big(\absB{\frac{\rho_0^{\alpha+1}U}{\eta_x}}_1+\absB{\frac{\rho_0(\rho_0^\alpha)_x U}{\eta_x}}_1+\abs{\rho_0^{\alpha+1}(V-U)U}_1+\absB{\frac{\rho_0^{\alpha+1}U_x}{\eta_x}}_1\Big)\\
&\leq C(T)\Big(\abs{\rho_0^\alpha}_\infty^{1+\frac{1}{2\alpha}}+\abs{(\rho_0^\alpha)_x}_\infty\abs{\rho_0^\alpha}_\infty^{\frac{1}{2\alpha}}+\abs{\rho_0^\alpha}_\infty^\frac{1}{2}\big|\rho_0^\frac{\alpha+1}{2} V\big|_2\Big)\absb{\rho_0^\frac{1}{2} U}_2\\
&\quad +C\Big(|\rho_0^\alpha|_\infty|\rho_0^\frac{1}{2}U|_2^2+\abs{\rho_0^\alpha}_\infty^{1+\frac{1}{2\alpha}}\absB{\rho_0^\frac{1}{2}\frac{U_x}{\eta_x}}_2\Big)\leq C(T)\Big(1+\absB{\rho_0^\frac{1}{2}\frac{U_x}{\eta_x}}_2\Big).
\end{align*}
Integrating above over $[0,t]$, along with \eqref{energy-e}, implies that
\begin{equation}\label{6223}
\int_0^t \absB{\frac{\rho_0^{1+\alpha}U}{\eta_x}}_\infty\,\ds\leq C(T) \ \ \text{for all }0\leq t\leq T.    
\end{equation}

Next, set $r=1$ in \eqref{solutionv}. Then it holds that
\begin{equation*}
\rho_0^\alpha V(t,x)=\rho_0^\alpha V(0,x)e^{-\int_0^t 2H(s,x)\,\ds}+ \int_0^t \frac{2\rho_0^{1+\alpha} U(\tau,x)}{\eta_x(\tau,x)}e^{-\int_\tau^t 2H(s,x)\,\ds}\,\mathrm{d}\tau.
\end{equation*}
Taking the $L^\infty$-norm of both sides of the above equality, then one can get from \eqref{6223} that
\begin{equation}\label{equ4431}
\abs{\rho_0^\alpha V(t)}_\infty\leq \abs{\rho_0^\alpha V(0)}_\infty+2\int_0^t \Big|\frac{\rho_0^{1+\alpha}U}{\eta_x}\Big|_\infty\,\ds \leq \abs{\rho_0^\alpha V(0)}_\infty+C(T).
\end{equation}

For the initial data, it holds that
\begin{equation*}
\begin{aligned}
\abs{\rho_0^\alpha V(0)}_\infty&= \Big|\rho_0^\alpha\left(u_0+\frac{1}{\alpha}\frac{(\rho_0^\alpha)_x}{\rho_0^\alpha}\right)\Big|_\infty
\leq C\left( \abs{\rho_0^\alpha}_\infty \abs{u_0}_\infty+ \abs{(\rho_0^\alpha)_x}_\infty\right)\\
&\leq C \widetilde E(0,U)+ C\leq C.
\end{aligned}
\end{equation*}

The proof of Lemma \ref{boundv'} is completed.
\end{proof}

\subsection{The upper bound of $\eta_x$}\label{subsection6.4}
This subsection is devoted to showing the  upper bound of $\eta_x$.
First, one  needs to refine the power of weights in Lemma \ref{W.E.I}.  

\begin{Lemma}\label{L.P.E}
For any $T>0$, $0<\alpha\leq 1$, $\beta>-\alpha$ and $0\leq p<\infty$, it holds that
\begin{equation*}
\begin{gathered}
\sup_{t\in[0,T]}\int \rho_0^{\beta} \abs{U}^p(t,x)\,\dx \leq C(\beta,p,T).
\end{gathered}
\end{equation*}
\end{Lemma}
\begin{proof}
According to Lemma \ref{W.E.I}, $\rho_0^\alpha\in H^3$ and $\rho_0\sim d(x)^\frac{1}{\alpha}$, it suffices to prove the lemma for the case when $-\alpha<\beta\leq 1$ and $0<p<\infty$. The proof is divided into the following two steps.

\underline{\textbf{Step 1: $\beta>0$.}}
For every $\beta>0$, $0<p<\infty$, and any given $\varepsilon$ such that $0<\varepsilon<\min\left\{1,\beta,\frac{p}{2}\right\}$, it follows from Lemma \ref{W.E.I} and H\"older's inequality that for all $0\leq t\leq T$,
\begin{equation}\label{equ68}
\begin{aligned}
\int \rho_0^\beta \abs{U}^p\,\dx
&\leq \left(\int \rho_0^\frac{\beta-\varepsilon}{1-\varepsilon}\,\dx\right)^{1-\varepsilon} \left(\int \rho_0 \abs{U}^\frac{p}{\varepsilon}\,\dx\right)^\varepsilon\\
&\leq \abs{\rho_0^\alpha}_\infty^\frac{\beta-\varepsilon}{\alpha} \left(\int \rho_0 \abs{U}^\frac{p}{\varepsilon}\,\dx\right)^\varepsilon
\leq C(\beta,p,T).
\end{aligned}
\end{equation}

\underline{\textbf{Step 2: $-\alpha<\beta\leq 0$.}} 
Now, suppose that $\varepsilon,r$ are fixed constants depending only on $\alpha,\beta$ such that
$$0<\varepsilon<\alpha- \abs{\beta}\quad \text{and}\quad 1< r< \frac{\alpha}{\abs{\beta}+\varepsilon}.$$
Then, it follows from $\rho_0^\alpha\sim d(x)$, \eqref{equ68} and H\"older's inequality that
\begin{equation*}
\begin{aligned}
\int \rho_0^{- \abs{\beta}} \abs{U}^p\,\dx
&\leq \left(\int \rho_0^{-(\abs{\beta}+\varepsilon)r}\,\dx\right)^\frac{1}{r}\left(\int \rho_0^{\frac{\varepsilon r}{r-1}}\abs{U}^\frac{pr}{r-1}\,\dx\right)^\frac{r-1}{r}\\
&\leq C(\beta,p,T) \left(\int d(x)^{-\frac{(\abs{\beta}+\varepsilon)r}{\alpha}}\,\dx\right)^\frac{1}{r}
\leq C(\beta,p,T),
\end{aligned}
\end{equation*}
for all $0\leq t\leq T$. Therefore, the proof of Lemma \ref{L.P.E} is completed.
\end{proof}

Next, the following lemma deals with the first order estimate of the effective velocity.
\begin{Lemma}\label{estimates-V_x}
For any $T>0$, $0<\alpha\leq 1$ and $\beta>3\alpha$, it holds that
\begin{equation*}
\big|\rho_0^{\frac{\beta}{2}}V_x(t)\big|_2\leq C(\beta,T) \ \ \text{for all }0\leq t\leq T.
\end{equation*}
\end{Lemma}
\begin{proof}
Applying $\rho_0^\beta V_x \partial_x$ to both sides of  \eqref{eq:effective2} and integrating the resulting equality over $I$, along with \eqref{V-expression}, one gets
\begin{equation*}
\begin{aligned}
\frac{1}{2}\frac{\mathrm{d}}{\dt}\int \rho_0^\beta V_x^2\,\dx&=2\int \rho_0^\beta H (U_x- V_x) V_x\,\dx-2\int \rho_0^{\beta} H_x (V-U) V_x\,\dx\\
&=2\int \rho_0^{\beta+1}\frac{U_x}{\eta_x} V_x\,\dx-2\int \rho_0^\beta H V_x^2\,\dx-2\int \rho_0^{\beta+1} (V-U)^2 V_x\,\dx.
\end{aligned}
\end{equation*}
Then it follows from $\beta>3\alpha$, Lemmas \ref{bound of density}, \ref{boundv} and \ref{L.P.E}, H\"older's inequality and Young's inequality that
\begin{equation*}
\begin{aligned}
\frac{\mathrm{d}}{\dt}\int \rho_0^\beta V_x^2\,\dx&\leq C\Big(\abs{\rho_0^\alpha}_\infty^{\frac{\beta+1}{2\alpha}}\Big|\rho_0^\frac{1}{2}\frac{U_x}{\eta_x}\Big|_2 \big|\rho_0^\frac{\beta}{2} V_x\big|_2+\abs{H}_\infty \big|\rho_0^\frac{\beta}{2}V_x\big|_2^2\Big)\\
&\quad+C\Big(\big|\rho_0^\frac{\beta+2}{4} V\big|_4^2 \big|\rho_0^\frac{\beta}{2} V_x\big|_2+\big|\rho_0^\frac{\beta+2}{4}U\big|_4^2 \big|\rho_0^\frac{\beta}{2} V_x\big|_2\Big)\\
&\leq C(\beta,T)\Big(1+\big|\rho_0^\frac{\beta}{2}V_x\big|_2^2+\Big|\rho_0^\frac{1}{2}\frac{U_x}{\eta_x}\Big|_2^2\Big),
\end{aligned}
\end{equation*}
which, along with  Gr\"onwall's inequality and  Lemma \ref{BD-entropy}, yields that 
\begin{equation}\label{V_x}
 \big|\rho_0^\frac{\beta}{2}V_x(t)\big|_2\leq C(\beta,T)\big(\big|\rho_0^\frac{\beta}{2}V_x(0)\big|_2+1\big) \ \ \text{for }0\leq t\leq T.
\end{equation}
For the  initial data, since $\beta>3\alpha$, $\rho_0\sim d(x)^\frac{1}{\alpha}$ and
\begin{equation*}
V_x(0,x)=\frac{1}{\alpha}\frac{(\rho_0^\alpha)_{xx}}{\rho_0^\alpha}-\frac{1}{\alpha}\frac{(\rho_0^\alpha)_x^2}{\rho_0^{2\alpha}}+(u_0)_x,
\end{equation*}
it follows that
\begin{equation*}
\begin{aligned}
\int \rho_0^{\beta} (V_x(0,x))^2\,\dx&\leq C\int \big(\rho_0^{\beta-2\alpha}(\rho_0^\alpha)_{xx}^2+ \rho_0^{\beta-4\alpha}(\rho_0^\alpha)_x^4+\rho_0^\beta(u_0)_x^2\big)\,\dx\\
&\leq C\big(\abs{(\rho_0^\alpha)_{xx}}_\infty^2\big|\rho_0^{\beta-2\alpha}\big|_1+\abs{(\rho_0^\alpha)_{x}}_\infty^4\big|\rho_0^{\beta-4\alpha}\big|_1+\abs{\rho_0^\alpha}_\infty^\frac{\beta}{\alpha}\norm{u_0}_1^2\big)\\
&\leq C \big(\big|d^{\frac{\beta}{\alpha}-2}\big|_1+\big|d^{\frac{\beta}{\alpha}-4}\big|_1\big)+C(\beta)\leq C(\beta),
\end{aligned}
\end{equation*}
which, along with \eqref{V_x}, yields the desired estimate in this lemma.

The proof of Lemma \ref{estimates-V_x} is completed.
\end{proof}

With the help of Lemmas \ref{L.P.E}-\ref{estimates-V_x}, one can improve the order of $\eta_x$ in Lemma \ref{W.E.I}. 
\begin{Lemma}\label{W.E.IV}
For any $T>0$, $0<\alpha\leq 1$, $0\leq p<\infty$, it holds that
\begin{equation*}
\big|\rho_0^\frac{1+\iota}{p+2} \eta_x^\frac{1}{p+2} U(t)\big|_{p+2}^{p+2}+\int_0^t\Big|\rho_0^\frac{1+\iota}{2}\frac{\abs{U}^\frac{p}{2} U_x}{\sqrt{\eta_x}}\Big|_2^2\,\ds\leq C(\iota, p, T),
\end{equation*}
for all $0\leq t\leq T$, where $\iota=0$  if $0<\alpha<1$ and $\iota>0$ if $\alpha=1$.
\end{Lemma}
\begin{proof}
Let $\iota=0$  if $0<\alpha<1$ and $\iota>0$ if $\alpha=1$. Multiplying $\eqref{secondreformulation}_1$ by $\rho_0^{\iota}\eta_x\abs{U}^p U$ and integrating the resulting equality over $I$ give that
\begin{equation}\label{cJ}
\begin{aligned}
&\frac{1}{p+2}\frac{\mathrm{d}}{\dt}\int \rho_0^{1+\iota} \eta_x\abs{U}^{p+2}\,\dx+(p+1)\int  \frac{\rho_0^{1+\iota}\abs{U}^p U_x^2}{\eta_x}\,\dx\\
=&\frac{1}{p+2}\int \rho_0^{1+\iota}  \abs{U}^{p+2}U_x\,\dx-\frac{\iota}{\alpha}\int \frac{\rho_0^{1+\iota-\alpha}(\rho_0^\alpha)_x\abs{U}^pUU_x}{\eta_x}\,\dx \\
&-\int\frac{\rho_0^{1+\iota} \eta_{xx}}{\eta_x^2}  \abs{U}^p UU_x\,\dx-2\int \rho_0^{1+\iota}H_x \abs{U}^p U\,\dx:=\sum_{i=1}^4 \cJ_i.
\end{aligned}
\end{equation}

For $\cJ_1$, one deduces from integration by parts and Lemma \ref{L.P.E} that
\begin{align}
|\cJ_1|&=\frac{1}{p+2}\Big|\int\rho_0^{1+\iota}  \abs{U}^{p+2}U_x \,\dx\Big|=\frac{1}{(p+2)(p+3)}\Big|\int \rho_0^{1+\iota}  \big(U \abs{U}^{p+2}\big)_x\,\dx\Big|\notag\\
&=\frac{1+\iota}{\alpha(p+2)(p+3)}\Big|\int \rho_0^{1+\iota-\alpha}(\rho_0^\alpha)_x U \abs{U}^{p+2}\,\dx\Big|\label{cJ1}\\
&\leq C(\iota,p) \abs{(\rho_0^\alpha)_x}_\infty\big|\rho_0^\frac{1+\iota-\alpha}{p+3}U\big|_{p+3}^{p+3}
\leq C(\iota,p,T).\notag
\end{align}

For $\cJ_2$, it follows from Lemmas \ref{L.B.J}, \ref{L.P.E}, H\"older's inequality and Young's inequality that
\begin{equation}\label{JJ2}
\begin{aligned}
\cJ_2&=-\frac{\iota}{\alpha}\int \frac{\rho_0^{1+\iota-\alpha}(\rho_0^\alpha)_x\abs{U}^pUU_x}{\eta_x}\,\dx\\
&\leq C(\iota)\abs{(\rho_0^\alpha)_x}_\infty\Big|\rho_0^\frac{1+\iota}{2}\frac{\abs{U}^\frac{p}{2} U_x}{\sqrt{\eta_x}}\Big|_2\big|\rho_0^{\frac{1+\iota-2\alpha}{p+2}} \eta_x^{-\frac{1}{p+2}} U\big|_{p+2}^\frac{p+2}{2}\\
&\leq C(\iota,p,T)+\frac{1}{8}\Big|\rho_0^\frac{1+\iota}{2}\frac{\abs{U}^\frac{p}{2} U_x}{\sqrt{\eta_x}}\Big|_2^2.
\end{aligned}
\end{equation}

For $\cJ_3$, it follows from \eqref{V-expression}  that
\begin{equation}\label{cJ3}
\begin{aligned}
\cJ_3&=-\int\frac{\rho_0^{1+\iota} \eta_{xx}}{\eta_x^2}  \abs{U}^p UU_x\,\dx\\
&= \int \rho_0^{1+\iota}(V-U) \abs{U}^p U U_x \,\dx- \frac{1}{\alpha}\int \frac{\rho_0^{1+\iota-\alpha}(\rho_0^\alpha)_x\abs{U}^p UU_x}{\eta_x}\,\dx\\
&=\int \rho_0^{1+\iota} V |U|^{p}UU_x\,\dx- \frac{1}{\alpha}\int \frac{\rho_0^{1+\iota-\alpha}(\rho_0^\alpha)_x\abs{U}^p UU_x}{\eta_x}\,\dx-(p+2)\cJ_1\\
&:=\cJ_{31}+\cJ_{32}-(p+2)\cJ_1.
\end{aligned}
\end{equation}

For $\cJ_{31}$, letting $0<\varepsilon< 1-\alpha+\iota$, one gets from integration by parts, Lemmas \ref{boundv} and \ref{L.P.E}-\ref{estimates-V_x}, H\"older's inequality and Young's inequality that
\begin{equation}\label{cJ31}
\begin{aligned}
\cJ_{31}&=\int \rho_0^{1+\iota} V |U|^{p}UU_x\,\dx=\frac{1}{p+2}\int \rho_0^{1+\iota} V \left(|U|^{p+2}\right)_x\,\dx\\
&=- \frac{1}{p+2}\int \rho_0^{1+\iota}  V_x \abs{U}^{p+2}\,\dx- \frac{1+\iota}{\alpha(p+2)}\int \rho_0^{1+\iota-\alpha}(\rho_0^\alpha)_x V \abs{U}^{p+2}\,\dx\\
&\leq C(\iota,p) \big(\big|\rho_0^{\frac{3\alpha}{2}+\varepsilon} V_x\big|_2+ \abs{(\rho_0^\alpha)_x}_\infty\big|\rho_0^{\frac{\alpha}{2}+\varepsilon} V\big|_2\big)\big|\rho_0^\frac{2-3\alpha+2\iota-2\varepsilon}{2p+4}U\big|_{2p+4}^{p+2}\\
&\leq C(\iota,p,T),
\end{aligned}
\end{equation}
where one has used the fact that $2-3\alpha+2\iota-2\varepsilon>-\alpha$.

To handle $\cJ_{32}$, one can obtain from the same calculations of $\cJ_2$ that
\begin{equation}\label{cJ32'}
\cJ_{32}\leq C(\iota,p,T)+\frac{1}{8}\Big|\rho_0^\frac{1+\iota}{2}\frac{\abs{U}^\frac{p}{2} U_x}{\sqrt{\eta_x}}\Big|_2^2.
\end{equation}

Substituting \eqref{cJ1}, \eqref{cJ31}-\eqref{cJ32'} into \eqref{cJ3} yields
\begin{equation}\label{cJJ3'}
\cJ_{3}\leq C(p)|\cJ_1|+C(\iota,p,T)+\frac{1}{8}\Big|\rho_0^\frac{1+\iota}{2}\frac{\abs{U}^\frac{p}{2} U_x}{\sqrt{\eta_x}}\Big|_2^2\leq C(\iota,p,T)+\frac{1}{8}\Big|\rho_0^\frac{1+\iota}{2}\frac{\abs{U}^\frac{p}{2} U_x}{\sqrt{\eta_x}}\Big|_2^2.    
\end{equation}

Finally, we treat $\cJ_4$. For any fixed $\varepsilon\in (0,1)$, it follows from Lemmas \ref{L.B.J}, \ref{boundv} and \ref{L.P.E}, \eqref{V-expression} and H\"older's inequality that
\begin{equation}\label{cJ2}
\begin{aligned}
\cJ_4&=-2\int \rho_0^{1+\iota} H_x\abs{U}^p U\,\dx
=-2\int \rho_0^{2+\iota}(V-U)\abs{U}^p U\,\dx\\
&\leq 2\big|\rho_0^\frac{2+\iota}{p+2} U\big|_{p+2}^{p+2}+2\big|\rho_0^{\frac{\alpha}{2}+\varepsilon} V\big|_2\big|\rho_0^{\frac{4+2\iota-\alpha-2\varepsilon}{2p+2}}U\big|_{2p+2}^{p+1}
\leq C(\iota,p,T),
\end{aligned}
\end{equation}
which, along  with \eqref{cJ}-\eqref{JJ2}, \eqref{cJJ3'}-\eqref{cJ2} and Gr\"onwall's inequality, yields  the desired estimates.

The proof of Lemma \ref{W.E.IV} is completed.
\end{proof}

The following auxiliary lemma relates to the weighted $L^2$-estimates of $U_x$.\begin{Lemma}\label{W.E.V}
For any $T>0$, $0<\alpha\leq 1$, it holds that
\begin{equation*}
\begin{aligned}
\Big|\rho_0^{\alpha+\frac{1+\iota}{2}} \frac{U_x}{\sqrt{\eta_x}}(t)\Big|_{2}^2&+\int_0^t\Big|\rho_0^{\alpha+\frac{1+\iota}{2}}\sqrt{\eta_x}U_t\Big|_{2}^2\,\ds\\
&+\int_0^t\Big|\rho_0^{\alpha+\frac{1+\iota}{2}}\frac{1}{\sqrt{\eta_x}}\left(\frac{U_x}{\eta_x}\right)_x\Big|_{2}^2\,\ds \leq C(\iota,T),
\end{aligned}
\end{equation*}
for all $0\leq t\leq T$, where $\iota=0$  if $0<\alpha<1$ and $\iota>0$ if $\alpha=1$.
\end{Lemma}

\begin{proof}
\underline{\textbf{Step 1: Tangential estimates.}}
Let $\iota=0$  if $0<\alpha<1$ and $\iota>0$ if $\alpha=1$. Multiplying $\eqref{secondreformulation}_1$ by $\rho_0^{2\alpha+\iota}\eta_x U_t$ and integrating the resulting equality over $I$, one gets
\begin{equation}\label{cJ'}
\begin{aligned}
&\frac{1}{2}\frac{\mathrm{d}}{\dt}\int \frac{\rho_0^{2\alpha+1+\iota}U_x^2}{\eta_x}\,\dx+\int \rho_0^{2\alpha+1+\iota}\eta_x U_t^2\,\dx\\
=&-\frac{1}{2}\int \frac{\rho_0^{2\alpha+1+\iota} U_x^3}{\eta_x^2}\,\dx- 2\int \rho_0^{2\alpha+1+\iota}H_x U_t\,\dx\\
&\underline{-\frac{2\alpha+\iota}{\alpha}\int \frac{\rho_0^{\alpha+1+\iota}(\rho_0^\alpha)_x U_xU_t}{\eta_x}\,\dx-\int \frac{\rho_0^{2\alpha+1+\iota}\eta_{xx}U_xU_t}{\eta_x^2}\,\dx}_{:=\cJ_7}:= \sum_{i=5}^{7} \cJ_i. 
\end{aligned}
\end{equation}

For $\cJ_5$, one can deduce from integration by parts, $U_x|_{x\in\Gamma}=0$, Lemma \ref{L.B.J}, H\"older's inequality and Young's inequality that for all $0<\varepsilon<1$, 
\begin{align}
\cJ_5&=-\frac{1}{2}\int \frac{\rho_0^{2\alpha+1+\iota} U_x^3}{\eta_x^2}\,\dx\notag\\
&=-\left.\frac{1}{2}\frac{\rho_0^{2\alpha+1+\iota}UU_x^2}{\eta_x^2}\right|_{x=0}^{x=1}+\frac{2\alpha+1+\iota}{2\alpha}\int \frac{\rho_0^{\alpha+1+\iota}(\rho_0^\alpha)_x UU_x^2}{\eta_x^2}\,\dx\notag\\
&\quad +\int \rho_0^{2\alpha+1+\iota} \frac{UU_x}{\eta_x}\left(\frac{U_x}{\eta_x}\right)_x\,\dx\notag\\
&=\frac{2\alpha+1+\iota}{2\alpha}\int \frac{\rho_0^{\alpha+1+\iota}(\rho_0^\alpha)_x UU_x^2}{\eta_x^2}\,\dx
+\int \rho_0^{2\alpha+1+\iota} \frac{UU_x}{\eta_x}\left(\frac{U_x}{\eta_x}\right)_x\,\dx\label{cJ5}\\
&\leq C(\iota,T) \abs{\rho_0^\alpha}_\infty\abs{(\rho_0^\alpha)_x}_\infty\Big|\rho_0^\frac{1+\iota}{2}\frac{U_x}{\sqrt{\eta_x}}\Big|_2 \Big|\rho_0^\frac{1+\iota}{2}\frac{UU_x}{\sqrt{\eta_x}}\Big|_2\notag\\
&\quad +C(\iota)\abs{\rho_0^\alpha}_\infty\Big|\rho_0^\frac{1+\iota}{2}\frac{U U_x}{\sqrt{\eta_x}}\Big|_2\Big|\rho_0^{\alpha+\frac{1+\iota}{2}}\frac{1}{\sqrt{\eta_x}}\left(\frac{U_x}{\eta_x}\right)_x\Big|_2\notag\\
&\leq C(\iota,\varepsilon,T)\sum_{j=0}^1\Big|\rho_0^\frac{1+\iota}{2}\frac{\abs{U}^jU_x}{\sqrt{\eta_x}}\Big|_2^2+\varepsilon\Big|\rho_0^{\alpha+\frac{1+\iota}{2}}\frac{1}{\sqrt{\eta_x}}\left(\frac{U_x}{\eta_x}\right)_x\Big|_2^2.\notag
\end{align}

For $\cJ_6$, it follows from \eqref{V-expression}, Lemmas \ref{bound of density}-\ref{L.B.J}, \ref{boundv} and \ref{L.P.E}, H\"older's inequality and Young's inequality that
\begin{equation}\label{cJ6}
\begin{aligned}
\cJ_6&=- 2\int \rho_0^{2\alpha+1+\iota}H_x U_t\,\dx
=- 2\int \rho_0^{2\alpha+2+\iota}(V-U) U_t\,\dx\\
&\leq 2  \big(\big|\rho_0^{\alpha+\frac{3+\iota}{2}}V\big|_2+\big|\rho_0^{\alpha+\frac{3+\iota}{2}}U\big|_2\big)\big|\rho_0^{\alpha+\frac{1+\iota}{2}}U_t\big|_2\\
&\leq C(\iota,T)+\frac{1}{8}\big|\rho_0^{\alpha+\frac{1+\iota}{2}}\sqrt{\eta_x}U_t\big|_2^2.
\end{aligned}
\end{equation}

At last, for $\cJ_7$, it follows from  \eqref{V-expression}, Lemmas \ref{L.B.J} and  \ref{boundv'}, H\"older's inequality and Young's inequality  that 
\begin{align}
\cJ_7&=-\frac{2\alpha+\iota}{\alpha}\int \frac{\rho_0^{\alpha+1+\iota}(\rho_0^\alpha)_x U_xU_t}{\eta_x}\,\dx-\int \frac{\rho_0^{2\alpha+1+\iota}\eta_{xx}U_xU_t}{\eta_x^2}\,\dx\notag\\
&=-\frac{2\alpha+1+\iota}{\alpha}\int \frac{\rho_0^{\alpha+1+\iota}(\rho_0^\alpha)_x U_xU_t}{\eta_x}\,\dx+\int \rho_0^{2\alpha+1+\iota}(V-U)U_xU_t\,\dx\notag\\
&\leq C(\iota,T) \left(\abs{(\rho_0^\alpha)_x}_\infty+\abs{\rho_0^\alpha V}_\infty\right)\Big|\rho_0^\frac{1+\iota}{2}\frac{U_x}{\sqrt{\eta_x}}\Big|_2 \big|\rho_0^{\alpha+\frac{1+\iota}{2}}\sqrt{\eta_x}U_t\big|_2\label{cJ7}\\
&\quad +C(\iota)\Big|\rho_0^\frac{1+\iota}{2}\frac{U U_x}{\sqrt{\eta_x}}\Big|_2 \big|\rho_0^{\alpha+\frac{1+\iota}{2}}\sqrt{\eta_x}U_t\big|_2\notag\\
&\leq C(\iota,T)\sum_{j=0}^1\Big|\rho_0^\frac{1+\iota}{2}\frac{\abs{U}^jU_x}{\sqrt{\eta_x}}\Big|_2^2+\frac{1}{8}\big|\rho_0^{\alpha+\frac{1+\iota}{2}}\sqrt{\eta_x}U_t\big|_2^2.\notag
\end{align}

Thus, it follows from \eqref{cJ'}-\eqref{cJ7} that for all $0<\varepsilon<1$, 
\begin{equation}\label{equ4.24}
\begin{aligned}
&\frac{\mathrm{d}}{\dt}\Big|\rho_0^{\alpha+\frac{1+\iota}{2}} \frac{U_x}{\sqrt{\eta_x}}\Big|_{2}^2+\big|\rho_0^{\alpha+\frac{1+\iota}{2}}\sqrt{\eta_x}U_t\big|_{2}^2\\
\leq &C(\iota,\varepsilon,T)\Big(1+\sum_{j=0}^1\Big|\rho_0^\frac{1+\iota}{2}\frac{\abs{U}^jU_x}{\sqrt{\eta_x}}\Big|_2^2\Big)+\varepsilon\Big|\rho_0^{\alpha+\frac{1+\iota}{2}}\frac{1}{\sqrt{\eta_x}}\left(\frac{U_x}{\eta_x}\right)_x\Big|_2^2.
\end{aligned}
\end{equation}

\underline{\textbf{Step 2: Elliptic estimates.}}
For the elliptic term, one can multiply both sides  of  $\eqref{secondreformulation}_1$ by $\rho_0^{\alpha+\frac{\iota-1}{2}}\sqrt{\eta_x}$ and apply  \eqref{V-expression} to the resulting equality to get that
\begin{equation*}
\begin{aligned}
&\rho_0^{\alpha+\frac{1+\iota}{2}}\frac{1}{\sqrt{\eta_x}}\left(\frac{U_x}{\eta_x}\right)_x
=\rho_0^{\alpha+\frac{1+\iota}{2}} \sqrt{\eta_x}U_t - \rho_0^{\alpha+\frac{1+\iota}{2}}(V-U)\frac{U_x}{\sqrt{\eta_x}}+2\rho_0^{\alpha+\frac{3+\iota}{2}}\frac{1}{\sqrt{\eta_x}}(V-U).
\end{aligned}
\end{equation*}
Then taking the $L^2$-norm of both sides of the above equality, one gets from Lemmas \ref{L.B.J} and  \ref{boundv'}-\ref{L.P.E} that
\begin{equation}\label{equ4.25}
\begin{aligned}
&\Big|\rho_0^{\alpha+\frac{1+\iota}{2}}\frac{1}{\sqrt{\eta_x}}\left(\frac{U_x}{\eta_x}\right)_x\Big|_{2}\\
\leq &\big|\rho_0^{\alpha+\frac{1+\iota}{2}} \sqrt{\eta_x}U_t\big|_{2} + \abs{\rho_0^\alpha V}_\infty\absB{\rho_0^\frac{1+\iota}{2}\frac{U_x}{\sqrt{{\eta_x}}}}_2+\abs{\rho_0^\alpha}_\infty\absB{\rho_0^\frac{1+\iota}{2}\frac{UU_x}{\sqrt{{\eta_x}}}}_2 \\
&+C(T)\big(\abs{\rho_0^\alpha}_\infty^\frac{3+\iota}{2\alpha}\abs{\rho_0^\alpha V}_\infty+\absb{\rho_0^{\alpha+\frac{3+\iota}{2}}U}_2\big) \\
\leq& \absb{\rho_0^{\alpha+\frac{1+\iota}{2}} \sqrt{\eta_x}U_t}_{2} + C(T)\sum_{j=0}^1\absB{\rho_0^\frac{1+\iota}{2}\frac{\abs{U}^jU_x}{\sqrt{{\eta_x}}}}_2+C(\iota,T).
\end{aligned}
\end{equation}

Plugging \eqref{equ4.25} into \eqref{equ4.24} and then choosing $\varepsilon$ sufficiently small, one gets that
\begin{equation*}
\frac{\mathrm{d}}{\dt}\absB{\rho_0^{\alpha+\frac{1+\iota}{2}} \frac{U_x}{\sqrt{\eta_x}}}_{2}^2+\absb{\rho_0^{\alpha+\frac{1+\iota}{2}}\sqrt{\eta_x}U_t}_{2}^2\leq C(\iota,T)\Big(1+\sum_{j=0}^1\absB{\rho_0^\frac{1+\iota}{2}\frac{\abs{U}^jU_x}{\sqrt{\eta_x}}}_2^2\Big).
\end{equation*}
Integrating above over $[0,t]$, one can obtain from  Lemma \ref{W.E.IV} that
\begin{equation}\label{equ4.26}
\begin{aligned}
\absB{\rho_0^{\alpha+\frac{1+\iota}{2}} \frac{U_x}{\sqrt{\eta_x}}(t)}_{2}^2\!+\!\!\int_0^t\absb{\rho_0^{\alpha+\frac{1+\iota}{2}}\sqrt{\eta_x}U_t}_{2}^2 \ds\!\leq \!C(\iota,T)\big(\absb{\rho_0^{\alpha+\frac{1+\iota}{2}}(u_0)_x}_2^2+1\big) \!\leq \!C(\iota,T),
\end{aligned}
\end{equation}
where one has used the fact that 
\begin{equation*}
\absb{\rho_0^{\alpha+\frac{1+\iota}{2}}(u_0)_x}_2^2\leq \abs{\rho_0}_\infty^{2\alpha+ \iota} \absb{\rho_0^\frac{1}{2}(u_0)_x}_2^2\leq C(\iota) \widetilde E(0,U)\leq C(\iota).
\end{equation*}

Finally, it follows from \eqref{equ4.25}-\eqref{equ4.26} and Lemma \ref{W.E.IV} that for all $0\leq t\leq T$,
\begin{equation*}
\int_0^t\absB{\rho_0^{\alpha+\frac{1+\iota}{2}}\frac{1}{\sqrt{\eta_x}}\left(\frac{U_x}{\eta_x}\right)_x}_{2}^2\,\ds\leq C(\iota,T).  
\end{equation*}

The proof of Lemma \ref{W.E.V} is completed.
\end{proof}

Consequently,  the following corollary holds.
\begin{Corollary}\label{W.E.VI}
For any $T>0$ and $0<\alpha\leq 1$, it holds that
\begin{equation*}
\int_0^t \absB{\rho_0^{\frac{\alpha+1+\iota}{2}} \frac{U_x}{\eta_x}}_{\infty}^2\,\ds \leq C(\iota, T) \ \ \text{for all }0\leq t\leq T,
\end{equation*}
where $\iota=0$  if $0<\alpha<1$ and $\iota>0$ if $\alpha=1$.
\end{Corollary}

\begin{proof}
Let $\iota=0$  if $0<\alpha<1$ and $\iota>0$ if $\alpha=1$. According to Lemmas \ref{L.B.J} and \ref{sobolev-embedding}, $U_x|_{x\in\Gamma}=0$ and H\"older's inequality, one has
\begin{align*}
\absB{\rho_0^{\frac{\alpha+1+\iota}{2}}\frac{U_x}{\eta_x}}_{\infty}^2&\leq \int \absB{\left(\rho_0^{\alpha+1+\iota} \frac{U_x^2}{\eta_x^2}\right)_x}\,\dx\\
&\leq 2\int \rho_0^{\alpha+1+\iota} \absB{\frac{U_x}{\sqrt{\eta_x}}} \absB{\frac{1}{\sqrt{\eta_x}}\left(\frac{U_x}{\eta_x}\right)_x}\,\dx  +C(\iota)\int \rho_0^{1+\iota} \abs{(\rho_0^\alpha)_x}\frac{U_x^2}{\eta_x^2}\,\dx\\
&\leq 2 \absB{\rho_0^{\frac{1+\iota}{2}} \frac{U_x}{\sqrt{\eta_x}}}_{2}\absB{\rho_0^{\alpha+\frac{1+\iota}{2}}\frac{1}{\sqrt{\eta_x}}\left(\frac{U_x}{\eta_x}\right)_x}_{2} +C(\iota,T)\abs{(\rho_0^\alpha)_x}_\infty \absB{\rho_0^{\frac{1+\iota}{2}} \frac{U_x}{\sqrt{\eta_x}}}_{2}^2\\
&\leq C \absB{\rho_0^{\alpha+\frac{1+\iota}{2}}\frac{1}{\sqrt{\eta_x}}\left(\frac{U_x}{\eta_x}\right)_x}_{2}^2+C(\iota,T)\absB{\rho_0^{\frac{1+\iota}{2}} \frac{U_x}{\sqrt{\eta_x}}}_{2}^2.
\end{align*}
Thus, integrating above over $[0,t]$ and using Lemmas \ref{W.E.IV} and  \ref{W.E.V} lead to the desired conclusion.

The proof of Corollary \ref{W.E.VI} is completed.
\end{proof}

Now, one can establish the upper bound of $\eta_x$.
\begin{Lemma}\label{U.B.J}
For every $T>0$, it holds that
\begin{equation}
\sup_{(t,x)\in[0,T]\times\bar I} \eta_x(t,x) \leq C(T).
\end{equation}
\end{Lemma}
\begin{proof}
\underline{\textbf{Step 1.}} We first establish the upper bound of $\eta_x$ in the interior of $I$. Since $\eta_{tx}=U_x$ and $\eta_x(0,x)=1$, we have
\begin{equation*}
\log\eta_x(t,x)=\int_0^t\frac{U_x}{\eta_x}(s,x)\,\mathrm{d}s.
\end{equation*}
Then the H\"older inequality, together with Corollary \ref{W.E.VI}, implies  
\begin{equation}\label{bdbd}
\rho_0^\kappa|\log\eta_x(t,x)|
\leq t^{1/2}
\Bigl(\int_0^t
\Big|\rho_0^\kappa\frac{U_x}{\eta_x}(s)
\Big|_{\infty}^2\,\mathrm{d}s\Bigr)^{\frac{1}{2}}
\leq C(T) 
\end{equation}  
for all $(t,x)\in [0,T]\times \bar I$, where
\begin{equation*}
\kappa=
\begin{cases}
(\alpha+1)/2&\text{if }0<\alpha<1,\\
3/2&\text{if }\alpha=1.
\end{cases}
\end{equation*}

Since $\rho_0^\alpha \sim d(x)$, it follows from \eqref{bdbd} and Lemma \ref{L.B.J} that, for every fixed $b\in (0,\frac{1}{2})$,
\begin{equation}\label{eq:eta-upper-interior}
\eta_x(t,x)\leq C(b,T) 
\qquad\text{for }(t,x)\in[0,T]\times[b,1-b].
\end{equation}

\smallskip
\noindent\underline{\textbf{Step 2.}} For the uniform upper bound of $\eta_x$ near the vacuum boundary, repeating the same calculation \eqref{q-renyi}-\eqref{abb2} in Lemma \ref{L.B.J}, we can obtain that, for any $q\geq 2$,
\begin{equation}\label{eq:eta-upper-flux-bound}
|F(t,x)|
\leq C(q,T)\rho_0(x)^{(1+\alpha)(1-\frac{1}{q})}\qquad\text{for all $(t,x)\in [0,T]\times \bar I$}.
\end{equation}
Here, $F(t,x)$ is defined as in Lemma \ref{L.B.J}. 

Next, it follows from Lemma \ref{bound of density} and equation \eqref{youban} that 
\begin{equation*}
H_t+CH\geq H_t+H^2=-F_t,
\end{equation*}
where the constant $C$ is given as in Lemma \ref{bound of density}. Solving the above inequality, we deduce from $H(0,x)=\rho_0(x)$ that
\begin{equation*}
H(t,x)\geq e^{-Ct}\rho_0(x)
-\int_0^t e^{-C(t-s)}F_t(s,x)\,\mathrm{d}s.
\end{equation*}
Integration by parts in time therefore yields
\begin{equation}\label{ode-ine}
H(t,x)\geq e^{-Ct}\rho_0(x)-F(t,x)+e^{-Ct}F(0,x)
+C\int_0^t e^{-C(t-s)}F(s,x)\,\mathrm{d}s.
\end{equation}
Since
\begin{equation*}
C\int_0^t e^{-C(t-s)}\,\mathrm{d}s=1-e^{-Ct}\leq 1,\qquad |e^{-Ct}F(0,x)|\leq \sup_{0\leq t\leq T}|F(t,x)|,
\end{equation*}
we obtain from \eqref{ode-ine} that
\begin{equation*}
H(t,x)\geq e^{-CT}\rho_0(x)
-3\sup_{0\leq t\leq T}|F(t,x)|.
\end{equation*}
Combining this inequality with \eqref{eq:eta-upper-flux-bound} and $q=1+\frac{1+\alpha}{\alpha}$, we obtain from $\rho_0^\alpha\sim d(x)$ that
\begin{equation*}
H(t,x)\geq\rho_0(x)\big(e^{-CT}-3C(T) d(x)^{\frac{\alpha}{1+2\alpha}}\big)
\qquad\text{for }(t,x)\in[0,T]\times \bar I.
\end{equation*}

Now, set 
\begin{equation*}
b_T=\min\Big\{\frac16,
\Big(\frac{e^{-CT}}{6C(T)}\Big)^{\frac{1+2\alpha}{\alpha}}\Big\}.
\end{equation*}
Then $b_T>0$ depends only on $(\alpha,\varepsilon_0,|I|,\rho_0,u_0,T)$ and, for $t\in [0,T]$ and $0\leq d(x)\leq b_T$ (that is, $x\in [0,b_T]\cup [1-b_T,1]$), we have
\begin{equation*}
H(t,x)\geq\frac12e^{-CT}\rho_0(x).
\end{equation*}

Clearly, using equation $H=\frac{\rho_0}{\eta_x}$, we conclude that
\begin{equation}\label{eq:eta-upper-boundary}
\eta_x(t,x)\leq2e^{CT}=:C(T)
\qquad\text{for }t\in [0,T],\quad x\in [0,b_T]\cup [1-b_T,1].
\end{equation}

Finally, taking $b=b_T$ in \eqref{eq:eta-upper-interior}, combined with \eqref{eq:eta-upper-boundary}, yields the desired estimate.
\end{proof}

\section{Global-in-time weighted energy estimates on the velocity}\label{Section7}

In this section, we aim to establish the global-in-time weighted estimates of $U$. The tangential estimates are derived in \S\ref{subsection7.1}, and the elliptic estimates are derived in \S\ref{subsection7.2}-\S \ref{subsection7.3}. 

\subsection{Tangential estimates on the velocity}\label{subsection7.1}
We first give a lemma to build a bridge between the spatial derivatives and the temporal ones.

\begin{Lemma}\label{U_x-U_t}
For any $T>0$, $0<\alpha\leq 1$ and $ 0\leq\iota<3\alpha+1$, it holds that
\begin{equation*}
\begin{aligned}
\abs{U_x(t,x)}&\leq C \rho_0(x)+C(T)\rho_0(x)^\frac{\alpha-1}{2}\absb{\rho_0^\frac{1}{2} U_t(t)}_2;\\
\abs{U_x(t,x)}&\leq C \rho_0(x)+C(\iota,T)\rho_0(x)^\frac{3\alpha-1-\iota}{2}\absb{\rho_0^{\frac{1+\iota}{2}-\alpha} U_t(t)}_2;\\
\abs{U_{tx}(t,x)}&\leq C\rho_0^2(x)+ C(\iota,T)\big(\rho_0(x)^{3\alpha-1-\iota}\absb{\rho_0^{\frac{1+\iota}{2}-\alpha} U_t(t)}_2^2+\rho_0(x)^\frac{\alpha-1}{2}\absb{\rho_0^\frac{1}{2}U_{tt}(t)}_2\big),
\end{aligned}
\end{equation*}
for all $(t,x)\in [0,T]\times \bar I$.
\end{Lemma}
\begin{proof}
Integrating $\eqref{secondreformulation}_1$ over $[0,x]$ for $0\leq x\leq \frac{1}{2}$ yields
\begin{equation}\label{4.4.36}
U_x(t,x)= \rho_0(x)+ \eta_x^2 \rho_0^{-1}\int_0^x \rho_0(z) U_t(t,z)\,\mathrm{d}z.
\end{equation}
Then it follows from $\rho_0\sim d(x)^\frac{1}{\alpha}$, Lemma \ref{U.B.J} and H\"older's inequality that
\begin{equation}\label{equ72}
\begin{aligned}
\abs{U_x(t,x)}&\leq C x^\frac{1}{\alpha}+ C(T) x^{-\frac{1}{\alpha}}\left(\int_0^x z^\frac{1}{\alpha}\,\mathrm{d}z\right)^\frac{1}{2}\left(\int\rho_0 U_t^2\,\mathrm{d}z\right)^\frac{1}{2}\\
&\leq C d(x)^\frac{1}{\alpha}+ C(T) d(x)^\frac{\alpha-1}{2\alpha}\absb{\rho_0^\frac{1}{2}U_t(t)}_2.
\end{aligned}
\end{equation}
One can perform the same calculation for $\frac{1}{2}< x\leq 1$ by integrating $\eqref{secondreformulation}_1$ over $[x,1]$ to get \eqref{equ72}. 

Similarly, it also holds that
\begin{equation}\label{equ73}
\abs{U_x(t,x)}\leq C d(x)^\frac{1}{\alpha}+ C(\iota,T) d(x)^\frac{3\alpha-1-\iota}{2\alpha}\absb{\rho_0^{\frac{1+\iota}{2}-\alpha} U_t(t)}_2,
\end{equation}
for all $(t,x)\in [0,T]\times \bar I$, where $0\leq \iota<3\alpha+1$.

Finally, differentiating \eqref{4.4.36} with respect to  $t$ gives
\begin{equation*}
U_{tx}(t,x)=2\eta_x U_x \rho_0^{-1} \int_0^x \rho_0(z) U_t(t,z)\,\mathrm{d}z+\eta_x^2\rho_0^{-1}\int_0^x \rho_0(z) U_{tt}(t,z)\,\mathrm{d}z,
\end{equation*}
which, together with an analogous computation, \eqref{equ73} and Young's inequality, leads to
\begin{align*}
\abs{U_{tx}(t,x)}&\leq C(\iota,T) \abs{U_x(t,x)} d(x)^{\frac{3\alpha-1-\iota}{2\alpha}}\big|\rho_0^{\frac{1+\iota}{2}-\alpha} U_t(t)\big|_2+C(T) d(x)^{\frac{\alpha-1}{2\alpha}}\big|\rho_0^\frac{1}{2}U_{tt}(t)\big|_2\\
&\leq Cd(x)^{\frac{2}{\alpha}}+ C(\iota,T) d(x)^{\frac{3\alpha-1-\iota}{\alpha}}\big|\rho_0^{\frac{1+\iota}{2}-\alpha} U_t(t)\big|_2^2+C(T)d(x)^{  \frac{\alpha-1}{2\alpha}  }   \big|\rho_0^\frac{1}{2}U_{tt}(t)\big|_2.
\end{align*}

The proof of Lemma \ref{U_x-U_t} is completed.
\end{proof}

Now, we are going to derive the tangential estimates in the following three lemmas.

\begin{Lemma}\label{TTE1}
For any $T>0$, $0<\alpha\leq 1$ and $2\leq p<\infty$, it holds that
\begin{equation*}
\absb{\rho_0^\frac{\beta}{p}U(t)}_p^p+\int_0^t\absb{\rho_0^\frac{\beta}{2}\abs{U}^\frac{p-2}{2} U_x}_2^2\,\ds\leq C(\beta,p,T),
\end{equation*}
for all $0\leq t\leq T$, where $\beta>\alpha$ if $0<\alpha<1$ and $\beta=1$ if $\alpha=1$. 
\end{Lemma}

\begin{proof}
One notes that, due to  Lemmas \ref{W.E.I} and \ref{U.B.J}, it remains only to show the case for $\beta>\alpha$, $0<\alpha<1$. To this end, using \eqref{V-expression}, one can rewrite $\eqref{lagrange}_2$   as follows:
\begin{equation}\label{equ4.30}
\rho_0 U_t -\left(\frac{\rho_0U_x}{\eta_x^2}\right)_x+2\frac{\rho_0^2}{\eta_x}(V-U)=0.
\end{equation}
Multiply \eqref{equ4.30} by $\rho_0^{\beta-1}\abs{U}^{p-2}U$ and integrate the resulting equality over $I$ to get
\begin{equation}\label{GG1-GG3}
\begin{aligned}
&\frac{1}{p}\frac{\mathrm{d}}{\dt}\int \rho_0^{\beta}\abs{U}^p\,\dx+(p-1)\int \frac{\rho_0^{\beta} \abs{U}^{p-2} U_x^2}{\eta_x^2}\,\dx\\
=&\frac{1-\beta}{\alpha}\int \frac{\rho_0^{\beta-\alpha}(\rho_0^\alpha)_x\abs{U}^{p-2} UU_x}{\eta_x^2}\,\dx-2\int \frac{\rho_0^{\beta+1}}{\eta_x}(V-U)\abs{U}^{p-2}U\,\dx
:=\sum_{i=1}^2\cG_i.
\end{aligned}
\end{equation}
Then  it follows from $\beta>\alpha$, Lemmas  \ref{L.B.J} and  \ref{boundv'}-\ref{L.P.E}, H\"older's inequality and Young's inequality that
\begin{align}
\cG_1&=\frac{1-\beta}{\alpha}\int \frac{\rho_0^{\beta-\alpha}(\rho_0^\alpha)_x\abs{U}^{p-2} UU_x}{\eta_x^2}\,\dx\notag\\
&\leq C(\beta,T) \abs{(\rho_0^\alpha)_x}_\infty\absb{\rho_0^\frac{\beta-2\alpha}{p}\eta_x^{-\frac{2}{p}}U}_{p}^\frac{p}{2}\absB{\frac{\rho_0^\frac{\beta}{2}\abs{U}^\frac{p-2}{2}U_x}{\eta_x}}_{2}\notag\\
&\leq C(\beta,p,T)+\frac{p-1}{8}\absB{\frac{\rho_0^\frac{\beta}{2}\abs{U}^\frac{p-2}{2}U_x}{\eta_x}}_{2}^2,\label{cG1-cG2}\\
\cG_2&=-2\int \frac{\rho_0^{\beta+1}}{\eta_x}(V-U)\abs{U}^{p-2}U\,\dx\notag\\
&\leq C(T)\big(\abs{\rho_0^\alpha V}_{\infty}\absb{\rho_0^\frac{\beta+1-\alpha}{p-1}U}_{p-1}^{p-1}+2\absb{\rho_0^\frac{\beta+1}{p}U}_{p}^{p}\big)\leq C(\beta,p,T),\notag
\end{align}
which, along  with \eqref{GG1-GG3}, Lemma \ref{U.B.J} and Gr\"onwall's inequality, yields the conclusion.

The proof of Lemma \ref{TTE1} is completed.
\end{proof}

\begin{Lemma}\label{TE2}
For any $T>0$ and $0<\alpha\leq 1$, it holds that
\begin{equation*}
\absb{\rho_0^\frac{1}{2} U_x(t)}_2^2+\int_0^t \absb{\rho_0^\frac{1}{2} U_t}_2^2 \,\ds\leq C(T) \ \ \text{for all }0\leq t\leq T.\end{equation*}
\end{Lemma}

\begin{proof}
\noindent\underline{\textbf{Step 1.}} 
First, multiplying \eqref{equ4.30} by $\eta_x\rho_0^\frac{\varepsilon-1}{2} (\varepsilon>0)$, along with \eqref{V-expression},  one  gets that 
\begin{equation}\label{Uxx-q}
\rho_0^\frac{1+\varepsilon}{2} \left(\frac{U_x}{\eta_x}\right)_x=\rho_0^\frac{1+\varepsilon}{2}\eta_x U_t +2 \rho_0^{\frac{3+\varepsilon}{2}}(V-U) - \rho_0^\frac{1+\varepsilon}{2} (V-U)  U_x,
\end{equation}
which, along with Lemmas \ref{boundv}, \ref{L.P.E},  \ref{U.B.J}  and  \ref{U_x-U_t}, leads to
\begin{equation}\label{4466}
\begin{aligned}
\absB{\rho_0^\frac{1+\varepsilon}{2} \left(\frac{U_x}{\eta_x}\right)_x}_2&\leq C(T)\abs{\rho_0^\alpha}_\infty^\frac{\varepsilon}{2\alpha}\absb{\rho_0^\frac{1}{2} U_t}_2 + C\absf{\rho_0^\alpha}_\infty^{\frac{3-\alpha}{2\alpha}}\absb{\rho_0^{\frac{\alpha+\varepsilon}{2}}V}_2+C\absb{\rho_0^{\frac{3+\varepsilon}{2}}U}_2\\
&\quad+C\big(\absb{\rho_0^\frac{\alpha+\varepsilon}{2}V}_2+\absb{\rho_0^\frac{\alpha+\varepsilon}{2}U}_2\big)\absb{\rho_0^{\frac{1-\alpha}{2}}U_x}_\infty\\
&\leq C(\varepsilon,T)\big(1+\absb{\rho_0^\frac{1}{2} U_t}_2\big).
\end{aligned}
\end{equation}

\underline{\textbf{Step 2.}} Multiplying \eqref{equ4.30} by $U_t$ and integrating the resulting equality over $I$ yield
\begin{equation}\label{G3-G4}
\frac{1}{2}\frac{\mathrm{d}}{\dt}\int \frac{\rho_0 U_x^2}{\eta_x^2}\,\dx + \int \rho_0 U_t^2\,\dx=-\int \frac{\rho_0 U_x^3}{\eta_x^3}\,\dx-2\int \frac{\rho_0^2 (V-U) U_t}{\eta_x}\,\dx:=\sum_{i=3}^4 \cG_i.
\end{equation}

For $\cG_3$, if $0<\alpha<1$, it follows from integration by parts, \eqref{V-expression}, \eqref{4466},  Lemmas \ref{L.B.J}, \ref{boundv'} and  \ref{hardy-inequality}, H\"older's inequality and Young's inequality that for any $\varepsilon>0$,
\begin{align}
\cG_3&=-\int \frac{\rho_0 U_x^3}{\eta_x^3}\,\dx
=\left.-\frac{\rho_0 U U_x^2}{\eta_x^3}\right|_{x=0}^{x=1}+\int H_x\frac{UU_x^2}{\eta_x^2}\,\dx+2\int \frac{\rho_0UU_x}{\eta_x^2}\left(\frac{U_x}{\eta_x}\right)_x\,\dx\notag\\
&=\int \rho_0(V-U)\frac{UU_x^2}{\eta_x^2}\,\dx+2\int \frac{\rho_0UU_x}{\eta_x^2}\left(\frac{U_x}{\eta_x}\right)_x\,\dx\notag\\
&\leq C(T) \abs{\rho_0^\alpha V}_\infty\absB{\rho_0^{\frac{1+\varepsilon}{2}-\alpha}\frac{U_x}{\eta_x}}_2\absb{\rho_0^{\frac{1-\varepsilon}{2}}UU_x}_2+C(T)\abs{\rho_0^\alpha}_\infty^\frac{\varepsilon}{\alpha}\absb{\rho_0^\frac{1-\varepsilon}{2}UU_x}_2^2\label{G3}\\
&\quad + C(T)\absb{\rho_0^{\frac{1-\varepsilon}{2}}UU_x}_2\absB{\rho_0^{\frac{1+\varepsilon}{2}}\left(\frac{U_x}{\eta_x}\right)_x}_2\notag\\
&\leq C(\varepsilon,T)\Big(1+\sum_{j=0}^1\absb{\rho_0^{\frac{1-\varepsilon}{2}}|U|^jU_x}_2^2\Big)+\frac{1}{8}\absb{\rho_0^{\frac{1}{2}}U_t}_2^2;\notag
\end{align}
while, if $\alpha=1$, it follows from Lemmas \ref{L.B.J} and \ref{U_x-U_t}, and Young's inequality that
\begin{equation}\label{G3'}
\begin{aligned}
\cG_3&=-\int \frac{\rho_0 U_x^3}{\eta_x^3}\,\dx
\leq C(T)\abs{U_x}_\infty \absB{\rho_0^\frac{1}{2}\frac{U_x}{\eta_x}}_2^2\\
&\leq  C(T)\big(1+\absb{\rho_0^\frac{1}{2}U_t}_2\big)\absB{\rho_0^\frac{1}{2}\frac{U_x}{\eta_x}}_2^2\\
&\leq C(T)\big(1+\absb{\rho_0^\frac{1}{2}U_x}_2^2\big)\absB{\rho_0^\frac{1}{2}\frac{U_x}{\eta_x}}_2^2+\frac{1}{8}\absb{\rho_0^\frac{1}{2}U_t}_2^2.
\end{aligned}
\end{equation}

For $\cG_4$, it follows from Lemmas \ref{L.B.J} and \ref{boundv'}-\ref{L.P.E}, H\"older's inequality and Young's inequality that
\begin{equation}\label{G4}
\begin{aligned}
\cG_4&=-2\int \frac{\rho_0^2(V-U)U_t}{\eta_x}\,\dx\\
&\leq C(T)\big(\abs{\rho_0^\alpha}_\infty^{\frac{3}{2\alpha}-1} \abs{\rho_0^\alpha V}_\infty+\absb{\rho_0^\frac{3}{2}U}_2\big)\absb{\rho_0^\frac{1}{2}U_t}_2
\leq C(T)+\frac{1}{8}\absb{\rho_0^\frac{1}{2}U_t}_2^2.
\end{aligned}
\end{equation}

Substituting \eqref{G3}-\eqref{G4} into \eqref{G3-G4} and choosing $0<\varepsilon<1-\alpha$ in \eqref{G3} when $0<\alpha<1$, one can get from the Gr\"onwall inequality, Lemmas \ref{U.B.J} and \ref{TTE1} that
\begin{equation*}
\absb{\rho_0^\frac{1}{2}U_x(t)}_2^2+\int_0^t\absb{\rho_0^\frac{1}{2}U_t}_2^2\,\ds\leq C(T) \ \ \text{for all }0\leq t\leq T.
\end{equation*}

The proof of Lemma \ref{TE2} is completed.
\end{proof}

\begin{Lemma}\label{TE3}
For any $T>0$ and $\frac{1}{3}<\alpha\leq 1$, it holds that for all $0\leq t\leq T$, \begin{equation*}
\sum_{j=1}^2\absb{\rho_0^\frac{1}{2}\partial_t^jU(t)}_{2}^2+\absb{\rho_0^\frac{1}{2} U_{tx}(t)}_{2}^2+\int_0^t\absb{\rho_0^\frac{1}{2} \partial_t^2 U_x }_{2}^2\,\ds\leq C(T).
\end{equation*}

\end{Lemma}

\begin{proof}
\noindent\underline{\textbf{Step 1: Estimate of $\rho_0^\frac{1}{2}U_t$.}}
First, Lemma \ref{U_x-U_t} implies that
\begin{equation}\label{equ4.39}
\begin{aligned}
\absb{\rho_0^\frac{1}{4}U_x(t)}_4&\leq C\abs{\rho_0^\alpha}_\infty^\frac{5}{4\alpha}+ C(T)\left(\int \rho_0^{2\alpha-1}\,\dx\right)^\frac{1}{4} \absb{\rho_0^\frac{1}{2}U_t(t)}_2\\
&\leq C(T)\big(1+\absb{\rho_0^\frac{1}{2}U_t(t)}_2\big),
\end{aligned}
\end{equation}
where one has used the facts that  $\alpha>\frac{1}{3}$ and $\rho_0\sim d(x)^\frac{1}{\alpha}$.

Next, applying $\partial_t$ to both sides of $\eqref{secondreformulation}_1$ shows that
\begin{equation}\label{partial_t}
\rho_0 U_{tt}-\left(\frac{\rho_0 U_{tx}}{\eta_x^2}\right)_x = \left(\frac{2\rho_0^2 U_x}{\eta_x^3}-\frac{2\rho_0 U_x^2}{\eta_x^3}\right)_x.
\end{equation}
Multiplying \eqref{partial_t} by $U_t$ and integrating the resulting equality over $I$ lead to
\begin{equation}\label{461}
\frac{1}{2}\frac{\mathrm{d}}{\dt}\int \rho_0 U_t^2\,\dx +\int \frac{\rho_0 U_{tx}^2}{\eta_x^2}\,\dx= \int \frac{2\rho_0 U_x^2 U_{tx}}{\eta_x^3}\,\dx - \int \frac{2\rho_0^2 U_x  U_{tx}}{\eta_x^3}\,\dx:=\sum_{i=5}^{6} \cG_i.
\end{equation}

For $\cG_{5}$-$\cG_6$, since $\alpha>\frac{1}{3}$, it follows from Lemmas \ref{L.B.J} and \ref{TE2}, \eqref{equ4.39}, H\"older's inequality and Young's inequality that
\begin{align}
\cG_{5}&=\int \frac{2\rho_0 U_x^2 U_{tx}}{\eta_x^3}\,\dx
\leq C(T)\absb{\rho_0^\frac{1}{4} U_x}_4^2 \absB{\rho_0^\frac{1}{2}\frac{U_{tx}}{\eta_x}}_2\notag\\
&\leq C(T) \big(1+\absb{\rho_0^\frac{1}{2} U_t}_2^2\big)\absB{\rho_0^\frac{1}{2}\frac{U_{tx}}{\eta_x}}_2\label{G5-G6}\notag\\
&\leq C(T)\big(1+\absb{\rho_0^\frac{1}{2} U_t}_2^4 \big)+\frac{1}{8}\absB{\rho_0^\frac{1}{2} \frac{U_{tx}}{\eta_x}}_2^2,\\
\cG_{6}& = - \int \frac{2\rho_0^2 U_x  U_{tx}}{\eta_x^3}\,\dx\notag\\
&\leq C(T) \absb{\rho_0^\frac{3}{2} U_x}_2 \absB{\rho_0^\frac{1}{2}\frac{U_{tx}}{\eta_x}}_2\leq C(T) +\frac{1}{8}\absB{\rho_0^\frac{1}{2}\frac{U_{tx}}{\eta_x}}_2^2.\notag
\end{align}

Thus, it follows from  \eqref{461}-\eqref{G5-G6},  Gr\"onwall's inequality, and  Lemmas \ref{L.B.J} and  \ref{TE2} that
\begin{equation}\label{4463}
\absb{\rho_0^\frac{1}{2}U_t(t)}_{2}^2+\int_0^t\absb{\rho_0^\frac{1}{2} U_{tx}}_{2}^2\,\ds\leq C(T) \ \ \text{for all }0\leq t\leq T,
\end{equation}
which, along with  \eqref{4466}, implies  that for all $\varepsilon>0$,
\begin{equation}\label{equ4.44}
\absB{\rho_0^\frac{1+\varepsilon}{2} \left(\frac{U_x}{\eta_x}\right)_x(t)}_2\leq C(\varepsilon,T)\big(1+\absb{\rho_0^\frac{1}{2} U_t(t)}_2\big)\leq C(\varepsilon,T).
\end{equation}

\underline{\textbf{Step 2: Estimate of $\rho_0^\frac{1}{2}U_{tx}$.}} First, since $\frac{1}{3}<\alpha\leq 1$, it follows  from \eqref{4463}, Lemmas \ref{U_x-U_t} and \ref{hardy-inequality}, and  $\rho_0\sim d(x)^\frac{1}{\alpha}$ that
\begin{equation}\label{U_infty-U_tx}
\begin{aligned}
\absb{\rho_0^\frac{1-3\alpha+\iota}{2} U_x(t)}_\infty&\leq C\abs{\rho_0^\alpha}_\infty^\frac{3-3\alpha+\iota}{2\alpha}+C(\iota,T)\absb{\rho_0^{\frac{1+\iota}{2}-\alpha}U_t(t)}_2\\
&\leq C(\iota) +C(\iota,T)\abs{\rho_0^\alpha}_\infty^\frac{\iota}{2\alpha}  \sum_{j=0}^1\absb{\rho_0^{\frac{1}{2}}\partial_x^j U_{t}(t)}_2 \\
&\leq C(\iota,T)\big(1+ \absb{\rho_0^{\frac{1}{2}}U_{tx}}_2\big),
\end{aligned}
\end{equation}
where $\iota$ satisfies  $0\leq \iota<3\alpha+1$ if $\frac{1}{3}<\alpha<1$ and $0<\iota<4$ if $\alpha=1$. However, since $\rho_0^\alpha\in L^\infty$, \eqref{U_infty-U_tx} actually holds for all $\iota\geq 0$ when $\frac{1}{3}<\alpha<1$ and for all $\iota>0$ when $\alpha=1$.

Next, multiplying \eqref{partial_t} by $U_{tt}$ and then integrating the resulting equality over $I$ yield
\begin{equation}\label{480}
\begin{aligned}
&\frac{1}{2}\frac{\mathrm{d}}{\dt}\int \rho_0 \frac{U_{tx}^2}{\eta_x^2}\,\dx+ \int \rho_0 U_{tt}^2\,\dx\\
=&-\int \rho_0 \frac{U_x U_{tx}^2}{\eta_{x}^3}\,\dx- \int \left(\frac{2\rho_0 U_x^2}{\eta_x^3}\right)_x U_{tt}\,\dx+ \int \left(\frac{2\rho_0^2 U_x}{\eta_x^3}\right)_x U_{tt}\,\dx
:=\sum_{i=7}^{9} \cG_i.
\end{aligned}
\end{equation}

For $\cG_{7}$-$\cG_{9}$, since $\frac{1}{3}<\alpha\leq 1$, choosing $\iota=0$ if $\frac{1}{3}<\alpha<1$ and $\iota=1$ if $\alpha=1$, one gets from \eqref{V-expression}, Lemmas \ref{L.B.J},  \ref{boundv}-\ref{L.P.E}, \ref{TE2}, and \ref{hardy-inequality}, \eqref{4463}-\eqref{U_infty-U_tx}, H\"older's inequality and Young's inequality that
\begin{align}
\cG_{7}&=-\int \rho_0 \frac{U_x U_{tx}^2}{\eta_{x}^3}\,\dx
\leq C(T) \abs{\rho_0^\alpha}_\infty^{\frac{3\alpha-1-\iota}{2\alpha}} \absb{\rho_0^\frac{1-3\alpha+\iota}{2}U_x}_\infty \absb{\rho_0^\frac{1}{2} U_{tx}}_2^2\notag\\
&\leq C(T) \big(1+\absb{\rho_0^\frac{1}{2}U_{tx}}_2^2\big)\absb{\rho_0^\frac{1}{2} U_{tx}}_2^2,\notag\\
\cG_{8}&=- \int \left(\frac{2\rho_0 U_x^2}{\eta_x^3}\right)_x U_{tt}\,\dx\notag\\
&=-2\int H_x \frac{U_x^2}{\eta_x^2}  U_{tt}\,\dx- 4\int H\frac{U_x}{\eta_x}\left(\frac{U_x}{\eta_x}\right)_x U_{tt}\,\dx\notag\\
&=-2\int \rho_0(V-U) \frac{U_x^2}{\eta_x^2}  U_{tt}\,\dx- 4\int \frac{\rho_0 U_x}{\eta_x^2}\left(\frac{U_x}{\eta_x}\right)_x U_{tt}\,\dx\notag\\
&\leq C(T)\big(\absb{\rho_0^{3\alpha-\frac{1}{2}-\iota}V}_2+\absb{\rho_0^{3\alpha-\frac{1}{2}-\iota}U}_2\big)\absb{\rho_0^\frac{1-3\alpha+\iota}{2}U_x}_\infty^2\absb{\rho_0^\frac{1}{2} U_{tt}}_2\label{G7-G9}\\
&\quad+ C(T)\absb{\rho_0^\frac{1-3\alpha+\iota}{2}U_x}_\infty\absB{\rho_0^\frac{3\alpha-\iota}{2}\left(\frac{U_x}{\eta_x}\right)_x}_2\absb{\rho_0^\frac{1}{2} U_{tt}}_2\notag\\
&\leq C(T)\big(1+\absb{\rho_0^\frac{1-3\alpha+\iota}{2}U_x}_\infty^2\big)\absb{\rho_0^\frac{1-3\alpha+\iota}{2}U_x}_\infty^2+\frac{1}{8}\absb{\rho_0^\frac{1}{2} U_{tt}}_2^2\notag\\
&\leq C(T)\big(1+\absb{\rho_0^\frac{1}{2}U_{tx}}_2^4\big)+\frac{1}{8}\absb{\rho_0^\frac{1}{2} U_{tt}}_2^2,\notag\\
\cG_{9}&= \int \left(\frac{2\rho_0^2 U_x}{\eta_x^3}\right)_x U_{tt}\,\dx
=4 \int H H_x\frac{U_x}{\eta_x} U_{tt}\,\dx+2\int H^2\left(\frac{U_x}{\eta_x}\right)_x U_{tt}\,\dx\notag\\
&=4 \int \rho_0^2 (V-U) \frac{U_x}{\eta_x^2} U_{tt}\,\dx+2\int \frac{\rho_0^2}{\eta_x^2}\left(\frac{U_x}{\eta_x}\right)_x U_{tt}\,\dx\notag\\
&\leq C(T) \left(\abs{\rho_0^\alpha V}_\infty+\abs{\rho_0^\alpha U}_\infty\right)\absB{\rho_0^{\frac{3}{2}-\alpha}\frac{U_x}{\eta_x}}_2\absb{\rho_0^\frac{1}{2} U_{tt}}_2+C(T)\absB{\rho_0^\frac{3}{2}\left(\frac{U_x}{\eta_x}\right)_x}_2\absb{\rho_0^\frac{1}{2} U_{tt}}_2\notag\\
&\leq C(T) \Big(\absB{\rho_0^{\frac{3}{2}}\frac{U_x}{\eta_x}}_2+\absB{\rho_0^\frac{3}{2}\left(\frac{U_x}{\eta_x}\right)_x}_2\Big)\absb{\rho_0^\frac{1}{2} U_{tt}}_2
\leq C(T)+\frac{1}{8}\absb{\rho_0^\frac{1}{2} U_{tt}}_2^2.\notag
\end{align}

Then it follows from  \eqref{480}-\eqref{G7-G9} and Lemma  \ref{U.B.J} that
\begin{equation*}
\frac{\mathrm{d}}{\dt}\absB{\rho_0^\frac{1}{2}\frac{U_{tx}}{\eta_x}}_2^2+\absb{\rho_0^\frac{1}{2}U_{tt}}_2^2\leq C(T)+C(T)\big(1+\absb{\rho_0^\frac{1}{2}U_{tx}}_2^2\big) \absB{\rho_0^\frac{1}{2}\frac{U_{tx}}{\eta_x}}_2^2,
\end{equation*}
which, along with the Gr\"onwall inequality, Lemma \ref{U.B.J} and \eqref{4463}, yields that 
\begin{equation}\label{4455}
\absb{\rho_0^\frac{1}{2} U_{tx}(t)}_2^2+\int_0^t \absb{\rho_0^\frac{1}{2} U_{tt}}_2^2\,\ds\leq C(T) \ \ \text{for all }0\leq t\leq T.
\end{equation}

\underline{\textbf{Step 3: Estimate of $\rho_0^\frac{1}{2}U_{tt}$.}}
First, it follows from \eqref{U_infty-U_tx} and \eqref{4455} that
\begin{equation}\label{eq450}
\absb{\rho_0^\frac{1-3\alpha+\iota}{2}U_x(t)}_\infty\leq C(\iota,T)\big(1+\absb{\rho_0^\frac{1}{2}U_{tx}(t)}_2\big)\leq C(\iota,T),
\end{equation}
for all $0\leq t\leq T$, where $\iota\geq 0$ if $\frac{1}{3}<\alpha<1$ and $\iota>0$ if $\alpha=1$.

Next, one can formally apply $\partial_t$ to both sides of \eqref{partial_t} to obtain that
\begin{equation}\label{partial_tt}
\rho_0 \partial_t^3U -\left(\frac{\rho_0 \partial_t^2U_{x}}{\eta_x^2}\right)_x= \left(\frac{2\rho_0^2 U_{tx}}{\eta_x^3}-\frac{6\rho_0^2 U_x^2}{\eta_x^4}-\frac{6\rho_0 U_x U_{tx}}{\eta_x^3}+\frac{6\rho_0 U_x^3}{\eta_x^4}\right)_x.
\end{equation}
Multiplying \eqref{partial_tt} by $U_{tt}$ and integrating the resulting equality over $I$ imply that
\begin{equation}\label{498}
\begin{aligned}
&\frac{1}{2}\frac{\mathrm{d}}{\dt}\int \rho_0U_{tt}^2\,\dx+\int \rho_0 \frac{(\partial_t^2 U_x)^2}{\eta_x^2}\,\dx\\
=&-\int \frac{2\rho_0^2 U_{tx}}{\eta_x^3}\partial_t^2U_x\,\dx+\int \frac{6\rho_0^2 U_x^2}{\eta_x^4}\partial_t^2U_x\,\dx+\int \frac{6\rho_0 U_x U_{tx}}{\eta_x^3}\partial_t^2U_x\,\dx\\
&-\int \frac{6\rho_0 U_x^3}{\eta_x^4}\partial_t^2U_x\,\dx:=\sum_{i=10}^{13}\cG_i.
\end{aligned}
\end{equation}

It should be noted that the above energy equality can be verified by the standard functional method. Indeed, using the notations given at the beginning of \S \ref{Section3}, for a.e. $t\in (0,T)$ and all test functions $\varphi\in C^\infty(\bar I)$, one can get from \eqref{partial_t} that
\begin{equation*}
\langle\rho_0 U_{tt},\varphi\rangle=\left<-\frac{\rho_0 U_{tx}}{\eta_x^2}-\frac{2\rho_0^2 U_x}{\eta_x^3}+\frac{2\rho_0 U_x^2}{\eta_x^3},\varphi_x\right>,
\end{equation*}
which, together with the facts that $U$ satisfies \eqref{b111'} and $\partial_t^2 U_x \in L^2([0,T];L^2_{\rho_0})$, leads to
\begin{equation*}
\begin{aligned}
\frac{\mathrm{d}}{\dt}\langle\rho_0 U_{tt},\varphi\rangle&=\left<-\frac{\rho_0 \partial_t^2U_{x}}{\eta_x^2}-\frac{2\rho_0^2 U_{tx}}{\eta_x^3}+\frac{6\rho_0^2 U_x^2}{\eta_x^4}+\frac{6\rho_0 U_x U_{tx}}{\eta_x^3}-\frac{6\rho_0 U_x^3}{\eta_x^4},\varphi_x\right>\\
&\leq \big(C(T)\absb{\rho_0^\frac{1}{2}\partial_t^2U_x(t)}_2+A_1(t)\big)\norm{\varphi}_{1,\rho_0}\leq A_2(t)\norm{\varphi}_{1,\rho_0},
\end{aligned}
\end{equation*}
for some positive functions $A_1(t),A_2(t)\in L^2(0,T)$. Thus, it follows from Lemma \ref{W-space} and the Lemma 1.1 on page 250 of \cite{temam} that $(\rho_0 U_{tt})_t =\rho_0\partial_t^3 U\in L^2([0,T];H^{-1}_{\rho_0})$ and
\begin{equation*}
\frac{\mathrm{d}}{\dt}\langle\rho_0  U_{tt}, U_{tt}\rangle=2\langle\rho_0\partial_t^3 U, U_{tt}\rangle_{H^{-1}_{\rho_0}\times H^1_{\rho_0}}.
\end{equation*}
Consequently, one can deduce \eqref{498} from the above formula.

For $\cG_{10}$-$\cG_{13}$, due to $\frac{1}{3}<\alpha\leq 1$, setting $\iota=0$ if $\frac{1}{3}<\alpha<1$ and $\iota=1$ if $\alpha=1$, then one gets from  Lemmas \ref{L.B.J} and \ref{TE2}, \eqref{4455}-\eqref{eq450}, H\"older's inequality and Young's inequality that 
\begin{align}
\cG_{10}&=-\int \frac{2\rho_0^2 U_{tx}}{\eta_x^3}\partial_t^2U_x\,\dx\notag\\
&\leq C(T) \abs{\rho_0^\alpha}_\infty^\frac{1}{\alpha}\absb{\rho_0^\frac{1}{2}U_{tx}}_2\absb{\rho_0^\frac{1}{2}\partial_t^2 U_{x}}_2\leq C(T)+\frac{1}{8}\absb{\rho_0^\frac{1}{2}\eta_x^{-1}\partial_t^2 U_{x}}_2^2,\notag\\
\cG_{11}&=\int \frac{6\rho_0^2 U_x^2}{\eta_x^4}\partial_t^2U_x\,\dx\notag\\
&\leq C(T) \absb{\rho_0^\frac{1-3\alpha+\iota}{2} U_{x}}_\infty\absb{\rho_0^\frac{3\alpha+2-\iota}{2} U_{x}}_2\absb{\rho_0^\frac{1}{2}\partial_t^2 U_{x}}_2\leq C(T)+\frac{1}{8}\absb{\rho_0^\frac{1}{2}\eta_x^{-1}\partial_t^2 U_{x}}_2^2,\notag\\
\cG_{12}&=\int \frac{6\rho_0 U_x U_{tx}}{\eta_x^3}\partial_t^2U_x\,\dx\label{G13}\\
&\leq C(T) \absb{\rho_0^\frac{1-3\alpha+\iota}{2} U_x}_\infty\absb{\rho_0^\frac{3\alpha-\iota}{2}U_{tx}}_2\absb{\rho_0^\frac{1}{2}\partial_t^2 U_{x}}_2\leq C(T)+\frac{1}{8}\absb{\rho_0^\frac{1}{2}\eta_x^{-1}\partial_t^2 U_{x}}_2^2,\notag\\
\cG_{13}&=-\int \frac{6\rho_0 U_x^3}{\eta_x^4}\partial_t^2U_x\,\dx\notag\\
&\leq C(T)\absb{\rho_0^\frac{1-3\alpha+\iota}{2} U_{x}}_\infty^2\absb{\rho_0^{3\alpha-\iota-\frac{1}{2}} U_{x}}_2\absb{\rho_0^\frac{1}{2}\partial_t^2 U_{x}}_2\leq C(T)+\frac{1}{8}\absb{\rho_0^\frac{1}{2}\eta_x^{-1}\partial_t^2 U_{x}}_2^2.\notag
\end{align}

Thus, according to  \eqref{498}-\eqref{G13} and Gr\"onwall's inequality, one can get
\begin{equation*}
\absb{\rho_0^\frac{1}{2} U_{tt}(t)}_{2}^2+\int_0^t\absb{\rho_0^\frac{1}{2}\partial_t^2 U_{x}}_{2}^2\,\ds\leq C(T)\ \ \text{for all }0\leq t\leq T.\end{equation*}

The proof of Lemma \ref{TE3} is completed.
\end{proof}

\subsection{The second and third order elliptic estimates on the velocity}\label{subsection7.2}

We first prove the following estimates.
\begin{Lemma}\label{D.V.E}
For any $T>0$, $\frac{1}{3}<\alpha\leq 1$ and $\varepsilon>0$, it holds that for all $0\leq t\leq T$,
\begin{equation*}
\absB{\rho_0^{\frac{1}{2}+\varepsilon-\alpha}\left(\frac{U_{x}}{\eta_x}\right)_x(t)}_{2}+\absB{\rho_0^{\frac{1}{2}+\varepsilon}\left(\frac{U_{x}}{\eta_x}\right)_{xx}(t)}_{2}\leq C(\varepsilon,T).
\end{equation*}
  
\end{Lemma}

\begin{proof}
\underline{\textbf{Step 1: Estimate of $\rho_0^{\frac{1}{2}+\varepsilon-\alpha}\left(U_{x}\eta_x^{-1}\right)_x$.}} 
Multiplying \eqref{Uxx-q} by $\rho_0^{\frac{\varepsilon}{2}-\alpha}$ ($\varepsilon>0$) gives 
\begin{equation}
\rho_0^{\frac{1}{2}+\varepsilon-\alpha} \left(\frac{U_x}{\eta_x}\right)_x=\rho_0^{\frac{1}{2}+\varepsilon-\alpha}\eta_x U_t +2 \rho_0^{\frac{3}{2}+\varepsilon-\alpha}(V-U) - \rho_0^{\frac{1}{2}+\varepsilon-\alpha} (V-U)  U_x,
\end{equation}
which, along with  \eqref{eq450}, Lemmas \ref{boundv}, \ref{L.P.E}, \ref{U.B.J},  \ref{TE3} and  \ref{hardy-inequality}, yields that for all $0\leq t\leq T$, 
\begin{equation}\label{4465}
\begin{aligned}
\absB{\rho_0^{\frac{1}{2}+\varepsilon-\alpha} \left(\frac{U_x}{\eta_x}\right)_x}_2&\leq C(T)\absb{\rho_0^{\frac{1}{2}+\varepsilon-\alpha} U_t}_2+ 2\absb{\rho_0^{\frac{3}{2}+\varepsilon-\alpha} V}_2+ 2\absb{\rho_0^{\frac{3}{2}+\varepsilon-\alpha} U}_2\\
&\quad +  \big(\absb{\rho_0^{\frac{\alpha+\varepsilon}{2}} V}_2+\absb{\rho_0^{\frac{\alpha+\varepsilon}{2}} U}_2\big) \absb{\rho_0^{\frac{1-3\alpha+\varepsilon}{2}} U_x}_\infty\\
&\leq C(T)\absb{\rho_0^{\frac{1}{2}+\varepsilon-\alpha} U_t}_2+C(\varepsilon,T)\\
&\leq C(T)\big(\absb{\rho_0^{\frac{1}{2}+\varepsilon} U_t}_2+\absb{\rho_0^{\frac{1}{2}+\varepsilon} U_{tx}}_2\big)+C(\varepsilon,T)
\leq C(\varepsilon,T).
\end{aligned}
\end{equation}

\underline{\textbf{Step 2: Estimate of $\rho_0^{\frac{1}{2}+\varepsilon}\left(U_{x}\eta_x^{-1}\right)_{xx}$.}}  
Multiplying $\eqref{lagrange}_2$ by $\frac{1}{H}$ and then applying  $\rho_0^{\frac{1}{2}+\varepsilon}\partial_x$ ($\varepsilon>0$) to the resulting equality, one gets
\begin{equation}\label{4.4.63}
\begin{aligned}
\rho_0^{\frac{1}{2}+\varepsilon} \left(\frac{U_x}{\eta_x}\right)_{xx}&=\rho_0^{\frac{1}{2}+\varepsilon}(\eta_x U_t)_x+ 2\rho_0^{\frac{1}{2}+\varepsilon}H_{xx}\\
&\quad-\rho_0^{\frac{1}{2}+\varepsilon}\frac{H_x}{H}\left(\frac{U_x}{\eta_x}\right)_x-\rho_0^{\frac{1}{2}+\varepsilon}\left(\frac{H_x}{H}\right)_x \frac{U_x}{\eta_x}:=\sum_{i=14}^{17} \cG_i.
\end{aligned}
\end{equation}

For $\cG_{14}$, since $\frac{1}{3}<\alpha\leq 1$, according to  \eqref{V-expression}, Lemmas \ref{boundv'},  \ref{U.B.J},  \ref{TE3} and  \ref{hardy-inequality}, one has
\begin{align}
\abs{\cG_{14}}_2&=\absb{\rho_0^{\frac{1}{2}+\varepsilon}(\eta_x U_t)_x}_2=\absB{\rho_0^{\frac{1}{2}+\varepsilon} \eta_{xx} U_t+\eta_x\rho_0^{\frac{1}{2}+\varepsilon} U_{tx}}_2\notag\\
&=\absB{\rho_0^{\frac{1}{2}+\varepsilon} \eta_{x}^2(U-V) U_t+\frac{1}{\alpha} \eta_{x}\rho_0^{\frac{1}{2}+\varepsilon-\alpha}(\rho_0^\alpha)_x U_t+\eta_x\rho_0^{\frac{1}{2}+\varepsilon}U_{tx}}_2\\
&\leq C(T) \left(\abs{\rho_0^\alpha V}_\infty+\abs{\rho_0^\alpha U}_\infty+\abs{(\rho_0^\alpha)_x}_\infty\right)\absb{\rho_0^{\frac{1}{2}+\varepsilon-\alpha}U_t}_2+C(T)\absb{\rho_0^{\frac{1}{2}+\varepsilon} U_{tx}}_2\notag\\
&\leq C(\varepsilon,T)\big(\absb{\rho_0^\frac{3\alpha}{2} U}_2+\absb{\rho_0^\frac{3\alpha}{2} U_x}_2+1\big)\big(\absb{\rho_0^{\frac{1}{2}}U_t}_2+\absb{\rho_0^\frac{1}{2} U_{tx}}_2\big)
\leq C(\varepsilon,T).\notag
\end{align}

For $\cG_{15}$, it follows from \eqref{V-expression}, Lemmas \ref{boundv} and \ref{L.P.E}-\ref{estimates-V_x} and  \ref{TE2} that
\begin{equation}
\begin{aligned}
\abs{\cG_{15}}_2&=\absb{2\rho_0^{\frac{1}{2}+\varepsilon}H_{xx}}_2=\absB{\frac{2}{\alpha}\rho_0^{\frac{3}{2}+\varepsilon-\alpha}(\rho_0^\alpha)_x(V-U)+2\rho_0^{\frac{3}{2}+\varepsilon}(V_x-U_x)}_2\\
&\leq C \abs{(\rho_0^\alpha)_x}_\infty\big(\absb{\rho_0^{\frac{3}{2}+\varepsilon-\alpha} V}_2+\absb{\rho_0^{\frac{3}{2}+\varepsilon-\alpha} U}_2\big)+2\big(\absb{\rho_0^{\frac{3}{2}+\varepsilon}V_x}_2+\absb{\rho_0^{\frac{3}{2}+\varepsilon}U_x}_2\big)\\
&\leq C(\varepsilon,T).
\end{aligned}
\end{equation}

For $\cG_{16}$, it follows from $\frac{1}{3}<\alpha\leq 1$, \eqref{V-expression}, \eqref{4465}, Lemmas \ref{boundv'}-\ref{L.P.E}, \ref{U.B.J},  \ref{TE2} and  \ref{hardy-inequality} that
\begin{align}
\abs{\cG_{16}}_2&=\absB{\rho_0^{\frac{1}{2}+\varepsilon}\frac{H_x}{H}\left(\frac{U_x}{\eta_x}\right)_x}_2=\absB{\eta_x\rho_0^{\frac{1}{2}+\varepsilon} (V-U)\left(\frac{U_x}{\eta_x}\right)_x}_2\notag\\
&\leq C(T)\left(\abs{\rho_0^\alpha V}_\infty+\abs{\rho_0^\alpha U}_\infty\right)\absB{\rho_0^{\frac{1}{2}+\varepsilon-\alpha}\left(\frac{U_x}{\eta_x}\right)_x}_2\\
&\leq C(\varepsilon,T)\left(1+\abs{\rho_0^\alpha U}_\infty\right)\notag\\
&\leq C(\varepsilon,T)\big(1+\absb{\rho_0^{\frac{3\alpha}{2}} U}_2+\absb{\rho_0^{\frac{3\alpha}{2}} U_x}_2\big)
\leq C(\varepsilon,T).\notag
\end{align}

At last, for $\cG_{17}$, it follows from $\frac{1}{3}<\alpha\leq 1$,  \eqref{V-expression}, Lemmas \ref{boundv}-\ref{estimates-V_x}, \ref{U.B.J},  \ref{TE2}, \ref{hardy-inequality} and \eqref{eq450}   that 
\begin{align}
\abs{\cG_{17}}_2&=\absB{\rho_0^{\frac{1}{2}+\varepsilon}\left(\frac{H_x}{H}\right)_x \frac{U_x}{\eta_x}}_2\notag\\
&=\absB{\rho_0^{\frac{1}{2}+\varepsilon}(V_x-U_x) U_x+\frac{1}{\alpha}\rho_0^{\frac{1}{2}+\varepsilon-\alpha}(\rho_0^\alpha)_x (V-U) U_x-\eta_x\rho_0^{\frac{1}{2}+\varepsilon}(V-U)^2U_x}_2\notag\\
&\leq C \big(\absb{\rho_0^{\frac{3\alpha+\varepsilon}{2}}V_x}_2+\absb{\rho_0^{\frac{3\alpha+\varepsilon}{2}}U_x}_2\big)\absb{\rho_0^\frac{1-3\alpha+\varepsilon}{2} U_x}_\infty\notag\\
&\quad+C\abs{(\rho_0^\alpha)_x}_\infty\big(\absb{\rho_0^\frac{\alpha+\varepsilon}{2}V}_2+\absb{\rho_0^\frac{\alpha+\varepsilon}{2}U}_2\big)\absb{\rho_0^\frac{1-3\alpha+\varepsilon}{2} U_x}_\infty\label{G17}\\
&\quad+ C(T)\left(\abs{\rho_0^\alpha V}_\infty+\abs{\rho_0^\alpha U}_\infty\right)\big(\absb{\rho_0^{\frac{\alpha+\varepsilon}{2}}V}_2+\absb{\rho_0^{\frac{\alpha+\varepsilon}{2}}U}_2\big)\absb{\rho_0^\frac{1-3\alpha+\varepsilon}{2} U_x}_\infty\notag\\
&\leq C(\varepsilon,T)(1+\abs{\rho_0^\alpha U}_\infty)\notag\\
&\leq C(\varepsilon,T)\big(1+\absb{\rho_0^{\frac{3\alpha}{2}} U}_2+\absb{\rho_0^{\frac{3\alpha}{2}} U_x}_2\big)
\leq C(\varepsilon,T).\notag
\end{align}

Therefore, collecting \eqref{4.4.63}-\eqref{G17} leads to
\begin{equation}\label{equ4.63}
\absB{\rho_0^{\frac{1}{2}+\varepsilon}\left(\frac{U_x}{\eta_x}\right)_{xx}(t)}_2\leq C(\varepsilon,T),
\end{equation}
for all $0\leq t\leq T$, $\frac{1}{3}<\alpha\leq 1$ and  $\varepsilon>0$. The proof of Lemma \ref{D.V.E} is completed.
\end{proof}

Now  one can  improve the regularities of $\eta_x$ with the help of Lemma \ref{D.V.E}.
\begin{Corollary}\label{H.E.J}
For any $T>0$, $\frac{1}{3}<\alpha\leq 1$ and $\varepsilon>0$, it holds that for all $0\leq t\leq T$,
\begin{equation*}
\absb{\rho_0^{\frac{1}{2}+\varepsilon-\alpha} \eta_{xx}(t)}_{2}+\absb{\rho_0^{\frac{1-\alpha}{2}+\varepsilon} \eta_{xx}(t)}_{\infty}+\absb{\rho_0^{\frac{1}{2}+\varepsilon}\partial_x^3 \eta(t)}_{2}\leq C(\varepsilon,T).
\end{equation*}
\end{Corollary}

\begin{proof}
Note that 
\begin{equation}\label{equ4.68}
\begin{gathered}
\eta_{xx}=\eta_x \int_0^t\left(\frac{U_x}{\eta_x}\right)_x\,\ds,\quad 
\partial_x^3\eta=\frac{\eta_{xx}^2}{\eta_x}+\eta_x\int_0^t \left(\frac{U_x}{\eta_x}\right)_{xx}\,\ds.
\end{gathered}
\end{equation}

Multiplying the first identity in \eqref{equ4.68} by $\rho_0^{\frac{1}{2}+\varepsilon-\alpha}$ and $\rho_0^{\frac{1-\alpha}{2}+\varepsilon}$, respectively, then according to  Lemmas \ref{U.B.J}, \ref{D.V.E} and  \ref{hardy-inequality}, one can get 
\begin{align*}
\absb{\rho_0^{\frac{1}{2}+\varepsilon-\alpha}\eta_{xx}}_2&\leq C(T) \int_0^t\absB{\rho_0^{\frac{1}{2}+\varepsilon-\alpha}\left(\frac{U_x}{\eta_x}\right)_{x}}_2\,\ds\leq C(\varepsilon,T),\\
\absb{\rho_0^{\frac{1-\alpha}{2}+\varepsilon}\eta_{xx}}_\infty&\leq C(T) \int_0^t\absB{\rho_0^{\frac{1-\alpha}{2}+\varepsilon}\left(\frac{U_x}{\eta_x}\right)_{x}}_\infty\,\ds\\
&\leq C(T) \int_0^t \sum_{j=1}^2 \absB{\rho_0^{\frac{1}{2}+\varepsilon}\partial_x^j \left(\frac{U_x}{\eta_x}\right)}_2 \,\ds
\leq C(\varepsilon,T).
\end{align*}

Consequently, multiplying the second identity in \eqref{equ4.68} by $\rho_0^{\frac{1}{2}+\varepsilon}$, one gets from the above estimates, Lemmas \ref{L.B.J}, \ref{U.B.J} and  \ref{D.V.E} that 
\begin{align*}
\absb{\rho_0^{\frac{1}{2}+\varepsilon}\partial_x^3\eta}_2&\leq C(T)\absb{\rho_0^\frac{1+\varepsilon-\alpha}{2}\eta_{xx}}_\infty \absb{\rho_0^\frac{\alpha}{2}\eta_{xx}}_2 + C(T)\int_0^t \absB{\rho_0^{\frac{1}{2}+\varepsilon}\left(\frac{U_x}{\eta_x}\right)_{xx}}_2\,\ds
\leq C(\varepsilon,T).
\end{align*}

Therefore, the proof of Corollary \ref{H.E.J} is completed.
\end{proof}

Using Lemma \ref{D.V.E} and Corollary \ref{H.E.J}, one can deduce the following elliptic estimates.
\begin{Lemma}\label{E.E.I}
For any $T>0$, $\frac{1}{3}<\alpha\leq 1$ and $\varepsilon>0$, it holds that
\begin{equation*}
\absb{\rho_0^{\frac{1}{2}+\varepsilon-\alpha} U_{xx}(t)}_{2}+\absb{\rho_0^{\frac{1}{2}+\varepsilon}\partial_x^3 U(t)}_{2}\leq C(\varepsilon,T)\ \ \text{for all }0\leq t\leq T.\end{equation*}
\end{Lemma}
\begin{proof}
\underline{\textbf{Step 1: Estimate of $\rho_0^{\frac{1}{2}+\varepsilon-\alpha} U_{xx}$.}} 
Note that
\begin{equation*}
U_{xx}=\eta_x\left(\frac{U_x}{\eta_x}\right)_x+\frac{U_x\eta_{xx}}{\eta_x}.
\end{equation*}
Then, for all $\varepsilon>0$, $\frac{1}{3}<\alpha\leq 1$, it follows from \eqref{eq450}, Lemmas \ref{L.B.J},  \ref{U.B.J} and  \ref{D.V.E},  and Corollary \ref{H.E.J} that
\begin{equation}\label{equ4.65}
\begin{aligned}
\absb{\rho_0^{\frac{1}{2}+\varepsilon-\alpha} U_{xx}}_2&=\absB{\rho_0^{\frac{1}{2}+\varepsilon-\alpha} \eta_x\left(\frac{U_{x}}{\eta_x}\right)_x+\rho_0^{\frac{1}{2}+\varepsilon-\alpha} \frac{U_{x}\eta_{xx}}{\eta_x}}_2\\
&\leq C(T)\absB{\rho_0^{\frac{1}{2}+\varepsilon-\alpha} \left(\frac{U_{x}}{\eta_x}\right)_x}_2+ C(T)\absb{\rho_0^{\frac{\alpha+\varepsilon}{2}}\eta_{xx}}_2\absb{\rho_0^\frac{1-3\alpha+\varepsilon}{2}U_x}_\infty\\
&\leq C(\varepsilon,T)\big(1+\absb{\rho_0^\frac{1-3\alpha+\varepsilon}{2}U_x}_\infty\big)
\leq C(\varepsilon,T),
\end{aligned}
\end{equation}
where one has used the fact that  $\frac{\alpha}{2}>\frac{1}{2}-\alpha$ if $\alpha>\frac{1}{3}$.

\underline{\textbf{Step 2: Estimate of $\rho_0^{\frac{1}{2}+\varepsilon}\partial_x^3 U$.}}  
Note that
\begin{equation*}
\partial_x^3 U=\eta_x\left(\frac{U_x}{\eta_x}\right)_{xx}+2\eta_{xx}\left(\frac{U_x}{\eta_x}\right)_{x}+\frac{U_x\partial_x^3\eta}{\eta_x}.
\end{equation*}
Then for all $\varepsilon>0$, $\frac{1}{3}<\alpha\leq 1$, it follows from \eqref{eq450}, Lemmas \ref{L.B.J},  \ref{U.B.J},  \ref{D.V.E} and \ref{hardy-inequality}, and  Corollary \ref{H.E.J}  that
\begin{align}
\absb{\rho_0^{\frac{1}{2}+\varepsilon} \partial_x^3 U}_2&=\absB{\rho_0^{\frac{1}{2}+\varepsilon} \eta_x\left(\frac{U_{x}}{\eta_x}\right)_{xx}+2\rho_0^{\frac{1}{2}+\varepsilon} \eta_{xx}\left(\frac{U_x}{\eta_x}\right)_{x}+\rho_0^{\frac{1}{2}+\varepsilon}\frac{U_x\partial_x^3\eta}{\eta_x}}_2\notag\\
&\leq C(T)\absB{\rho_0^{\frac{1}{2}+\varepsilon} \left(\frac{U_{x}}{\eta_x}\right)_{xx}}_2+C\absb{\rho_0^{\frac{1-\alpha+\varepsilon}{2}}\eta_{xx}}_\infty\absB{\rho_0^{\frac{\alpha+\varepsilon}{2}}\left(\frac{U_{x}}{\eta_x}\right)_x}_2\notag\\
&\quad + C(T)\absb{\rho_0^\frac{3\alpha+\varepsilon}{2} \partial_x^3 \eta}_2\absb{\rho_0^\frac{1-3\alpha+\varepsilon}{2} U_x}_\infty\\
&\leq C(\varepsilon,T)\Big(1+\absB{\rho_0^{\frac{3\alpha+\varepsilon}{2}}\left(\frac{U_{x}}{\eta_x}\right)_x}_2+\absB{\rho_0^{\frac{3\alpha+\varepsilon}{2}}\left(\frac{U_{x}}{\eta_x}\right)_{xx}}_2\Big)\notag\\
&\leq C(\varepsilon,T).\notag
\end{align}

The proof of Lemma \ref{E.E.I} is completed.
\end{proof}

\begin{Lemma}\label{E.E.II}
For any $T>0$ and $\frac{1}{3}<\alpha\leq 1$, it holds that
\begin{equation*}
\absb{\rho_0^{\left(\frac{3}{2}-\varepsilon_0\right)\alpha}U_{xx}(t)}_2+\absb{\rho_0^{\left(\frac{3}{2}-\varepsilon_0\right)\alpha}\partial_x^3 U(t)}_2\leq C(T) \ \ \text{for all }0\leq t\leq T, 
\end{equation*}
where $\varepsilon_0$ is defined as in \eqref{varepsilon0}.
\end{Lemma}

\begin{proof}
\noindent\underline{\textbf{Step 1: Estimate of $\rho_0^{\left(\frac{3}{2}-\varepsilon_0\right)\alpha}U_{xx}$.}}
This is a direct consequence of Lemma \ref{E.E.I}. Indeed, define $\varepsilon_0$ as in \eqref{varepsilon0} and set $\varepsilon>0$ as
\begin{equation*}
\varepsilon=\Big(\frac{5}{2}-\varepsilon_0\Big)\alpha-\frac{1}{2}>0.    
\end{equation*}
Then one has $\frac{1}{2}+\varepsilon-\alpha=\left(\frac{3}{2}-\varepsilon_0\right)\alpha$, and hence according to Lemma \ref{E.E.I}, it holds that
\begin{equation}\label{equ739}
\absb{\rho_0^{\left(\frac{3}{2}-\varepsilon_0\right)\alpha}U_{xx}(t)}_2\leq C(T) \ \ \text{for all }0\leq t\leq T.
\end{equation}

\underline{\textbf{Step 2: Estimate of $\rho_0^{\left(\frac{3}{2}-\varepsilon_0\right)\alpha}\partial_x^3 U$.}}
If $0<\varepsilon_0<\frac{3\alpha-1}{2\alpha}$, $0<\alpha\leq 1$, one can set $\varepsilon>0$ in Lemma \ref{E.E.I} as
\begin{equation*}
\varepsilon:=\Big(\frac{3\alpha-1}{2\alpha}-\varepsilon_0\Big)\alpha>0.
\end{equation*}
Then the conclusion holds automatically. Thus, it suffices to show the case when $\varepsilon_0=\frac{3\alpha-1}{2\alpha}$ and establish the following estimate:
\begin{equation}\label{critial-case}
\big|\rho_0^\frac{1}{2}\partial_x^3 U(t)\big|_2\leq C(T) \ \ \text{for all }0\leq t\leq T.
\end{equation}
Actually, for  the parameter $\varepsilon_0$ defined in \eqref{varepsilon0},  if $\varepsilon_0=\frac{3\alpha-1}{2\alpha}$, then it holds  that $\frac{3\alpha-1}{2\alpha}<\frac{1}{\alpha}-1$, i.e., $\frac{1}{3}<\alpha<\frac{3}{5}$; otherwise, one has  $\varepsilon_0<\frac{1}{\alpha}-1\leq \frac{3\alpha-1}{2\alpha}$, and then the desired estimate holds automatically as discussed above. Thus, it suffices to show \eqref{critial-case} under $\frac{1}{3}<\alpha<\frac{3}{5}$.

To get \eqref{critial-case} under $\frac{1}{3}<\alpha<\frac{3}{5}$, one may follow the proofs in \eqref{3...151}. Multiplying $\eqref{secondreformulation}_1$ by $\eta_x^2 \rho_0^{\alpha-1}$ and  applying $\rho_0^{\frac{1}{2}-\alpha}\partial_x$ to the resulting equality, one gets
\begin{align}
&\rho_0^\frac{1}{2}\partial_x^3 U +\Big(\frac{1}{\alpha}+1\Big)\rho_0^{\frac{1}{2}-\alpha}(\rho_0^\alpha)_x U_{xx}\notag\\
=&\underline{-\frac{1}{\alpha}\rho_0^{\frac{1}{2}-\alpha}(\rho_0^\alpha)_{xx} U_x+\rho_0^\frac{1}{2}\eta_x^2 U_{tx}+2\rho_0^\frac{1}{2}\eta_x\eta_{xx}U_t}_{:=\cG_{18}}\notag\\
&\underline{+\rho_0^{\frac{1}{2}-\alpha}(\rho_0^\alpha)_x\eta_x^2 U_t +\frac{2\rho_0^{\frac{1}{2}-\alpha}(\rho_0^\alpha)_x\eta_{xx}U_x}{\eta_x}+\frac{2\rho_0^\frac{1}{2}\eta_{xx} U_{xx}}{\eta_x}}_{:=\cG_{19}}\label{G18-G21}\\
&\underline{+\frac{2\rho_0^\frac{1}{2}\partial_x^3\eta U_x}{\eta_x}-\frac{2\rho_0^\frac{1}{2}\eta_{xx}^2U_x}{\eta_x^2}-\frac{2+2\alpha}{\alpha}\frac{\rho_0^{\frac{3}{2}-\alpha}(\rho_0^\alpha)_x\eta_{xx}}{\eta_x}}_{:=\cG_{20}}\notag\\
&\underline{-\frac{2\rho_0^{\frac{3}{2}}\partial_x^3\eta}{\eta_x}+\frac{2\rho_0^{\frac{3}{2}}\eta_{xx}^2}{\eta_x^2}+\frac{2}{\alpha^2}\rho_0^{\frac{3}{2}-2\alpha}(\rho_0^\alpha)_x^2 +\frac{2}{\alpha}\rho_0^{\frac{3}{2}-\alpha}(\rho_0^\alpha)_{xx}}_{:=\cG_{21}}.\notag
\end{align}

For $\cG_{18}$-$\cG_{19}$, it follows from \eqref{eq450}, Lemmas \ref{L.B.J}, \ref{U.B.J}, \ref{TE2}-\ref{TE3}, \ref{E.E.I} and \ref{hardy-inequality}, and  Corollary \ref{H.E.J} that
\begin{align}
\abs{\cG_{18}}_2&=\absB{-\frac{1}{\alpha}\rho_0^{\frac{1}{2}-\alpha}(\rho_0^\alpha)_{xx} U_x+\rho_0^\frac{1}{2}\eta_x^2 U_{tx}+2\rho_0^\frac{1}{2}\eta_x\eta_{xx}U_t}_2\notag\\
&\leq C \abs{(\rho_0^\alpha)_{xx}}_\infty \absb{\rho_0^{\frac{1}{2}-\alpha} U_x}_2+ C(T)\absb{\rho_0^\frac{1}{2}U_{tx}}_2+C(T)\absb{\rho_0^{\frac{1-\alpha}{2}+\varepsilon}\eta_{xx}}_\infty \absb{\rho_0^{\frac{\alpha}{2}-\varepsilon}U_t}_2\notag\\
&\leq  C\sum_{j=1}^2\absb{\rho_0^{\frac{1}{2}}\partial_x^j U}_2 + C(T)\bigg(1+ \sum_{j=0}^1\absb{\rho_0^{\frac{3\alpha}{2}-\varepsilon}\partial_x^j U_t}_2\bigg)\leq C(T),\notag\\
\abs{\cG_{19}}_2&=\absB{\rho_0^{\frac{1}{2}-\alpha}(\rho_0^\alpha)_x\eta_x^2 U_t +\frac{2\rho_0^{\frac{1}{2}-\alpha}(\rho_0^\alpha)_x\eta_{xx}U_x}{\eta_x}+\frac{2\rho_0^\frac{1}{2}\eta_{xx} U_{xx}}{\eta_x}}_2\label{G18-G19}\\
&\leq C(T)\abs{(\rho_0^\alpha)_{x}}_\infty \absb{\rho_0^{\frac{1}{2}-\alpha} U_t}_2+ C(T) \abs{(\rho_0^\alpha)_x}_\infty\sum_{j=0}^1\absb{\rho_0^{\frac{1}{2}}\partial_x^j(\eta_{xx}U_x)}_2\notag\\
&\quad+C(T) \absb{\rho_0^{\frac{1-\alpha}{2}+\varepsilon}\eta_{xx}}_\infty \absb{\rho_0^{\frac{\alpha}{2}-\varepsilon}U_{xx}}_2\notag\\
&\leq C(T)\bigg(\sum_{j=0}^1\absb{\rho_0^{\frac{1}{2}} \partial_x^j U_t}_2+\sum_{j=1}^3 \absb{\rho_0^{\frac{3\alpha}{2}-\varepsilon}\partial_x^j U}_2+1\bigg)\leq C(T),\notag
\end{align}
where $\varepsilon$ shall be chosen such that $0<\varepsilon< \frac{3\alpha-1}{2}$.

For $\cG_{20}$-$\cG_{21}$, setting $\varepsilon$ as above, one gets from \eqref{eq450}, Lemma \ref{L.B.J}, Corollary \ref{H.E.J} and  $\frac{1}{3}<\alpha<\frac{3}{5}$ that
\begin{align}
\abs{\cG_{20}}_2&=\absB{\frac{2\rho_0^\frac{1}{2}\partial_x^3\eta U_x}{\eta_x}-\frac{2\rho_0^\frac{1}{2}\eta_{xx}^2U_x}{\eta_x^2}-\frac{2+2\alpha}{\alpha}\frac{\rho_0^{\frac{3}{2}-\alpha}(\rho_0^\alpha)_x\eta_{xx}}{\eta_x}}_2\notag\\
&\leq C(T) \big(\abs{U_x}_\infty \absb{\rho_0^{\frac{1}{2}}\partial_x^3\eta}_2+ \absb{\rho_0^{\frac{1-\alpha}{2}+\varepsilon}\eta_{xx}}_\infty\absb{\rho_0^{\frac{1}{2}-\alpha+\varepsilon}\eta_{xx}}_2 \absb{\rho_0^{\frac{3\alpha-1}{2}-2\varepsilon}U_x}_\infty\big)\notag\\
&\quad+C(T)\abs{(\rho_0^\alpha)_x}_\infty\absb{\rho_0^{\frac{3}{2}-\alpha}\eta_{xx}}_2\notag\\
&\leq C(T)\Big(1+ \int_0^t\absb{\rho_0^{\frac{1}{2}}\partial_x^3 U}_2\,\ds\Big),\label{G20-G21}\\
\abs{\cG_{21}}_2&=\absB{-\frac{2\rho_0^{\frac{3}{2}}\partial_x^3\eta}{\eta_x}+\frac{2\rho_0^{\frac{3}{2}}\eta_{xx}^2}{\eta_x^2}+\frac{2}{\alpha^2}\rho_0^{\frac{3}{2}-2\alpha}(\rho_0^\alpha)_x^2 +\frac{2}{\alpha}\rho_0^{\frac{3}{2}-\alpha}(\rho_0^\alpha)_{xx}}_2\notag\\
&\leq C(T) \big(\abs{\rho_0^\alpha}_\infty^\frac{1}{\alpha} \absb{\rho_0^{\frac{1}{2}}\partial_x^3\eta}_2+ \abs{\rho_0^\alpha}_\infty^\frac{1}{2\alpha}\absb{\rho_0^{\frac{1}{2}}\eta_{xx}}_\infty^2\big)\notag\\
&\quad+C(T) \big(\abs{(\rho_0^{\alpha})_x}_\infty^2\abs{\rho_0^\alpha}_\infty^{\frac{3-4\alpha}{2\alpha}} + \abs{(\rho_0^\alpha)_{xx}}_\infty \abs{\rho_0^\alpha}_\infty^{\frac{3-2\alpha}{2\alpha}}\big)\notag\\
&\leq C(T)\Big(1+ \int_0^t\absb{\rho_0^{\frac{1}{2}}\partial_x^3 U}_2\,\ds\Big).\notag
\end{align}

Collecting \eqref{G18-G21}-\eqref{G20-G21} shows that for any $\frac{1}{3}<\alpha<\frac{3}{5}$, 
\begin{equation*}
\absB{\rho_0^\frac{1}{2}\partial_x^3 U +\Big(\frac{1}{\alpha}+1\Big)\rho_0^{\frac{1}{2}-\alpha}(\rho_0^\alpha)_x U_{xx}}_2\leq C(T)\Big(1+ \int_0^t\absb{\rho_0^{\frac{1}{2}}\partial_x^3 U}_2\,\ds\Big),
\end{equation*}
for all $0\leq t\leq T$, which, along with Proposition \ref{prop2.1} and \eqref{equ739}, yields  that
\begin{equation}\label{equ7.43}
\begin{aligned}
\absb{\rho_0^\frac{1}{2}\partial_x^3 U(t)}_2&\leq C\Big(\absB{\rho_0^\frac{1}{2}\partial_x^3 U +\Big(\frac{1}{\alpha}+1\Big)\rho_0^{\frac{1}{2}-\alpha}(\rho_0^\alpha)_x U_{xx}}_2+\abs{(\rho_0^\alpha)_{xx}}_\infty \absb{\rho_0^\frac{1}{2}U_{xx}}_2\Big) \\
&\leq C(T)\Big(1+ \int_0^t\absb{\rho_0^{\frac{1}{2}}\partial_x^3 U}_2\,\ds\Big).
\end{aligned}
\end{equation}

Finally, it follows from \eqref{equ7.43}  and Gr\"onwall's inequality that \eqref{critial-case} holds.

The proof of Lemma \ref{E.E.II} is completed.
\end{proof}

It follows from Lemmas \ref{E.E.II} and  \ref{hardy-inequality} that the following corollary holds.
\begin{Corollary}\label{H.E.J.II}
For any $T>0$ and $\frac{1}{3}<\alpha\leq 1$, it holds that for all $0\leq t\leq T$,
\begin{equation*}
\absb{\rho_0^{\left(\frac{1}{2}-\varepsilon_0\right)\alpha} \eta_{xx}(t)}_{2}+\absb{\rho_0^{\left(1-\varepsilon_0\right)\alpha} \eta_{xx}(t)}_{\infty}+\absb{\rho_0^{\left(\frac{3}{2}-\varepsilon_0\right)\alpha}\partial_x^3 \eta(t)}_{2}\leq C(T),
\end{equation*}
where $\varepsilon_0$ is defined as in \eqref{varepsilon0}.
\end{Corollary}
\begin{proof}
Note that Lemma \ref{E.E.II} yields
\begin{equation}\label{equ7745}
\begin{aligned}
\absb{\rho_0^{\left(\frac{3}{2}-\varepsilon_0\right)\alpha} \eta_{xx}(t)}_{2}&\leq \int_0^t\absb{\rho_0^{\left(\frac{3}{2}-\varepsilon_0\right)\alpha} U_{xx}}_{2}\,\ds\leq C(T), \\
\absb{\rho_0^{\left(\frac{3}{2}-\varepsilon_0\right)\alpha}\partial_x^3 \eta (t)}_{2}&\leq \int_0^t\absb{\rho_0^{\left(\frac{3}{2}-\varepsilon_0\right)\alpha} \partial_x^3 U}_{2}\,\ds\leq C(T).
\end{aligned}
\end{equation}

In addition, \eqref{equ7745} and Lemma \ref{hardy-inequality} imply that
\begin{equation*}
\absb{\rho_0^{\left(\frac{1}{2}-\varepsilon_0\right)\alpha} \eta_{xx}}_{2}+\absb{\rho_0^{\left(1-\varepsilon_0\right)\alpha} \eta_{xx}}_{\infty}\leq C\sum_{j=2}^3\absb{\rho_0^{\left(\frac{3}{2}-\varepsilon_0\right)\alpha}\partial_x^j \eta}_{2}\leq C(T).
\end{equation*}

The proof of Corollary \ref{H.E.J.II} is completed.
\end{proof}

\subsection{The fourth order elliptic estimates on the velocity}\label{subsection7.3}

\begin{Lemma}\label{E.E.III}
For any $T>0$ and $\frac{1}{3}<\alpha\leq 1$, it holds that
\begin{equation*}
\absb{\rho_0^{\left(\frac{3}{2}-\varepsilon_0\right)\alpha}\partial_tU_{xx}(t)}_2\leq C(T) \ \  \text{for all }0\leq t\leq T,
\end{equation*}
where $\varepsilon_0$ is defined as in \eqref{varepsilon0}.
\end{Lemma}
\begin{proof}
Applying $\eta_x^2\rho_0^{\left(\frac{3}{2}-\varepsilon_0\right)\alpha-1}\partial_t$ to both sides of $\eqref{secondreformulation}_1$, along with \eqref{V-expression}, yields
\begin{equation}\label{Utxx-q}
\begin{aligned}
&\rho_0^{\left(\frac{3}{2}-\varepsilon_0\right)\alpha}\partial_t U_{xx}+\frac{1}{\alpha}\rho_0^{\left(\frac{1}{2}-\varepsilon_0\right)\alpha}(\rho_0^\alpha)_xU_{tx}\\
=&\eta_x^2\rho_0^{\left(\frac{3}{2}-\varepsilon_0\right)\alpha}U_{tt}- 4\rho_0^{\left(\frac{3}{2}-\varepsilon_0\right)\alpha+1}(V-U) U_x- 2\rho_0^{\left(\frac{3}{2}-\varepsilon_0\right)\alpha+1}\left(\frac{ U_x}{\eta_x}\right)_x\\
&+2\rho_0^{\left(\frac{3}{2}-\varepsilon_0\right)\alpha}(V-U) U_x^2 +4\rho_0^{\left(\frac{3}{2}-\varepsilon_0\right)\alpha}U_x\left(\frac{U_x}{\eta_x}\right)_x+2\rho_0^{\left(\frac{3}{2}-\varepsilon_0\right)\alpha}\frac{\eta_{xx}U_{tx}}{\eta_x}
:= \sum_{i=22}^{27} \cG_i.
\end{aligned}
\end{equation}

First, for $\cG_{22}$-$\cG_{26}$, it follows from \eqref{eq450}, Lemmas \ref{boundv}-\ref{L.P.E}, \ref{U.B.J}, \ref{U_x-U_t}-\ref{D.V.E} and \ref{hardy-inequality},  and Corollary \ref{H.E.J.II} that 
\begin{align}
\abs{\cG_{22}}_2&\leq C(T)\absb{\rho_0^{\left(\frac{3}{2}-\varepsilon_0\right)\alpha} U_{tt}}_2\leq C(T),\notag\\
\abs{\cG_{23}}_2&\leq C\left(\abs{\rho_0^\alpha V}_\infty+\abs{\rho_0^\alpha U}_\infty\right)\absb{\rho_0^{\left(\frac{1}{2}-\varepsilon_0\right)\alpha+1} U_x}_2\notag\\
&\leq C\bigg(\abs{\rho_0^\alpha V}_\infty+\sum_{j=0}^1\absb{\rho_0^\frac{3\alpha}{2} \partial_x^j U}_2\bigg)\absb{\rho_0^{\left(\frac{1}{2}-\varepsilon_0\right)\alpha+1} U_x}_2\leq C(T),\notag\\
\abs{\cG_{24}}_2&\leq C\abs{\rho_0^\alpha}_\infty^\frac{1}{\alpha}\absB{\rho_0^{\left(\frac{3}{2}-\varepsilon_0\right)\alpha}\left(\frac{U_x}{\eta_x}\right)_x}_2\leq C(T),\label{G27}\\
\abs{\cG_{25}}_2&\leq C\left(\abs{\rho_0^\alpha V}_\infty+\abs{\rho_0^\alpha U}_\infty\right)\abs{U_x}_\infty\absb{\rho_0^{\left(\frac{1}{2}-\varepsilon_0\right)\alpha} U_x}_2\notag\\
&\leq C\bigg(\abs{\rho_0^\alpha V}_\infty+\sum_{j=0}^1\absb{\rho_0^\frac{3\alpha}{2} \partial_x^j U}_2\bigg)\abs{U_x}_\infty\sum_{j=1}^2\absb{\rho_0^{\left(\frac{3}{2}-\varepsilon_0\right)\alpha} \partial_x^j U}_2\leq C(T),\notag\\
\abs{\cG_{26}}_2&\leq C\abs{U_x}_\infty\absB{\rho_0^{\left(\frac{3}{2}-\varepsilon_0\right)\alpha}\left(\frac{U_x}{\eta_x}\right)_x}_2\leq C(T).\notag
\end{align}

Next, for $\cG_{27}$, since $\alpha>\frac{1}{3}$, according to Lemmas \ref{U_x-U_t}, \ref{TE3} and \ref{hardy-inequality}, choosing $0<\iota\leq \frac{5\alpha-1}{2}$ in Lemma \ref{U_x-U_t}, one gets
\begin{equation}\label{equ745}
\absb{\rho_0^\frac{1-\alpha}{2} U_{tx}}_\infty\leq C(T)\bigg(1+\sum_{j=0}^1\absb{\rho_0^\frac{1}{2}\partial_x^j U_t}_2^2+\absb{\rho_0^\frac{1}{2} U_{tt}}_2\bigg)\leq C(T).
\end{equation}
Thus, it follows from Lemma  \ref{L.B.J}, Corollary \ref{H.E.J.II}, \eqref{equ745} and $\alpha>\frac{1-\alpha}{2}$ that
\begin{equation}\label{G277}
\abs{\cG_{27}}_2=\absB{2\rho_0^{\left(\frac{3}{2}-\varepsilon_0\right)\alpha}\frac{\eta_{xx}U_{tx}}{\eta_x}}_2\leq C\absb{\rho_0^{\left(\frac{1}{2}-\varepsilon_0\right)\alpha} \eta_{xx}}_{2}\abs{\rho_0^\alpha U_{tx}}_\infty\leq C(T).
\end{equation}

Then, according to  \eqref{Utxx-q}-\eqref{G27} and \eqref{G277}, one has
\begin{equation}\label{abbove}
\absB{\rho_0^{\left(\frac{3}{2}-\varepsilon_0\right)\alpha}\partial_t U_{xx}+\frac{1}{\alpha}\rho_0^{\left(\frac{1}{2}-\varepsilon_0\right)\alpha}(\rho_0^\alpha)_xU_{tx}}_2\leq C(T).
\end{equation}
Moreover,  if $0<\iota<3\alpha+1$ and $\frac{1}{3}<\alpha\leq 1$,  it follows from     Lemmas \ref{U_x-U_t}, \ref{TE3} and \ref{hardy-inequality} that 
\begin{equation}\label{above}
\absb{\rho_0^{\frac{1}{2}-\alpha+\iota}U_{tx}}_2\leq C(\iota,T)\Big(1+\sum_{j=0}^1\absb{\rho_0^\frac{1}{2}\partial_x^j U_t}_2^2+\absb{\rho_0^\frac{1}{2}U_{tt}}_2\Big)\leq C(\iota,T).
\end{equation}
Since $\rho_0\in L^\infty$, \eqref{above} actually holds for all $\iota> 0$ and $\frac{1}{3}<\alpha\leq 1$. Hence,  if $0<\varepsilon_0<\frac{3\alpha-1}{2\alpha}$, one can set $\iota=\frac{3\alpha-1}{2}-\varepsilon_0\alpha$ in \eqref{above}, and get from \eqref{abbove} that
\begin{equation}
\absb{\rho_0^{\left(\frac{3}{2}-\varepsilon_0\right)\alpha}\partial_t U_{xx}}_2\leq C(T);
\end{equation}
while, if $\varepsilon_0=\frac{3\alpha-1}{2\alpha}$, \eqref{abbove} can be reduced to
\begin{equation}\label{abbove'}
\absB{\rho_0^{\frac{1}{2}}\partial_t U_{xx}+\frac{1}{\alpha}\rho_0^{\frac{1}{2}-\alpha}(\rho_0^\alpha)_xU_{tx}}_2\leq C(T),
\end{equation}
which, along with Proposition \ref{prop2.1} and Lemma \ref{TE3}, implies that
\begin{equation}\label{equ7.47}
\begin{aligned}
\absb{\rho_0^{\frac{1}{2}}\partial_t U_{xx}}_2&\leq C\absB{\rho_0^{\frac{1}{2}}\partial_t U_{xx}+\frac{1}{\alpha}\rho_0^{\frac{1}{2}-\alpha}(\rho_0^\alpha)_xU_{tx}}_2+C\abs{(\rho_0^\alpha)_{xx}}_\infty \absb{\rho_0^{\frac{1}{2}}U_{tx}}_2
\leq C(T).
\end{aligned}
\end{equation}

The proof of Lemma \ref{E.E.III} is completed.
\end{proof}

\begin{Lemma}\label{E.E.IV}
For any $T>0$ and $\frac{1}{3}<\alpha\leq 1$, it holds that
\begin{equation*}
\absb{\rho_0^{\left(\frac{3}{2}-\varepsilon_0\right)\alpha}\partial_x^4 U(t)}_2\leq C(T) \ \ \text{for all }0\leq t\leq T,
\end{equation*}
where $\varepsilon_0$ is defined as in \eqref{varepsilon0}.
\end{Lemma}

\begin{proof}
Similar to the derivation of \eqref{3-154}, one can replace  $\bar\eta$ in \eqref{3-154} with  $\eta$ to obtain that
\begin{align}
&\rho_0^{\left(\frac{3}{2}-\varepsilon_0\right)\alpha}\partial_x^4 U+\Big(\frac{1}{\alpha}+2\Big)\rho_0^{\left(\frac{1}{2}-\varepsilon_0\right)\alpha}(\rho_0^\alpha)_x\partial_x^3 U\notag\\
=&\underline{\rho_0^{\left(\frac{3}{2}-\varepsilon_0\right)\alpha} \eta_x^2\partial_tU_{xx}+2\rho_0^{\left(\frac{1}{2}-\varepsilon_0\right)\alpha}\left((\rho_0^\alpha)_x\eta_x+2\rho_0^\alpha \eta_{xx}\right) \eta_x U_{tx}+2\rho_0^{\left(\frac{3}{2}-\varepsilon_0\right)\alpha}\left(\eta_{xx}^2+\eta_x\partial_x^3\eta\right)U_t}_{:=\cG_{28}}\notag\\
&\underline{+\rho_0^{\left(\frac{1}{2}-\varepsilon_0\right)\alpha}\left((\rho_0^\alpha)_{xx}\eta_x+4(\rho_0^\alpha)_x\eta_{xx}\right)\eta_xU_t -\Big(\frac{2}{\alpha}+1\Big)\rho_0^{\left(\frac{1}{2}-\varepsilon_0\right)\alpha}(\rho_0^\alpha)_{xx} U_{xx}}_{:=\cG_{29}}\notag\\
& \underline{-\frac{1}{\alpha}\rho_0^{\left(\frac{1}{2}-\varepsilon_0\right)\alpha}\partial_x^3\rho_0^\alpha U_x+2\rho_0^{\left(\frac{1}{2}-\varepsilon_0\right)\alpha}\Big((\rho_0^\alpha)_{xx}\eta_{xx}+2(\rho_0^\alpha)_x\partial_x^3\eta -\frac{2(\rho_0^\alpha)_x\eta_{xx}^2}{\eta_x}\Big)\frac{U_x}{\eta_x}}_{:=\cG_{30}}\notag\\
&\underline{+2\rho_0^{\left(\frac{3}{2}-\varepsilon_0\right)\alpha}\Big(\partial_x^4\eta
-\frac{3\eta_{xx}\partial_x^3\eta}{\eta_x}
+\frac{2\eta_{xx}^3}{\eta_x^2}\Big)\frac{U_x}{\eta_x}}_{:=\cG_{31}}\label{xxxx}\\
& \underline{+4\rho_0^{\left(\frac{1}{2}-\varepsilon_0\right)\alpha}\Big(\rho_0^\alpha\partial_x^3\eta+(\rho_0^\alpha)_x\eta_{xx}-\frac{\rho_0^\alpha\eta_{xx}^2}{\eta_x}\Big)\frac{U_{xx}}{\eta_x}+2\rho_0^{\left(\frac{3}{2}-\varepsilon_0\right)\alpha}\frac{\eta_{xx}\partial_x^3U}{\eta_x}}_{:=\cG_{32}}\notag\\
&\underline{-\frac{2+2\alpha}{\alpha^2}\frac{\rho_0^{1-\frac{\alpha}{2}-\varepsilon_0\alpha}(\rho_0^\alpha)_x^2\eta_{xx}}{\eta_x}-2\rho_0^{\left(\frac{3}{2}-\varepsilon_0\right)\alpha+1}\Big(\frac{\partial_x^4\eta}{\eta_x}-\frac{3\eta_{xx}\partial_x^3\eta}{\eta_x^2}+\frac{2\eta_{xx}^3}{\eta_x^3}\Big)}_{:=\cG_{33}}\notag\\
&\underline{-\frac{2+2\alpha}{\alpha}\rho_0^{\left(\frac{1}{2}-\varepsilon_0\right)\alpha+1}\Big(\frac{(\rho_0^\alpha)_{xx}\eta_{xx}}{\eta_x}+\frac{2(\rho_0^\alpha)_x\partial_x^3\eta}{\eta_x}-\frac{2(\rho_0^\alpha)_x\eta_{xx}^2}{\eta_x^2}\Big)}_{:=\cG_{34}}\notag\\
&\underline{+\frac{2(1-\alpha)}{\alpha^3}\rho_0^{1-\frac{3\alpha}{2}-\varepsilon_0\alpha}(\rho_0^\alpha)_x^3+\frac{6}{\alpha^2}\rho_0^{1-\frac{\alpha}{2}-\varepsilon_0\alpha}(\rho_0^\alpha)_x(\rho_0^\alpha)_{xx}+\frac{2}{\alpha}\rho_0^{1+\frac{\alpha}{2}-\varepsilon_0\alpha}\partial_x^3\rho_0^\alpha}_{:=\cG_{35}}.\notag
\end{align}

For $\cG_{28}$, since $\alpha>\frac{1}{3}$, one gets from Lemmas \ref{U.B.J}, \ref{TE2}-\ref{TE3}, \ref{E.E.II}-\ref{E.E.III} and \ref{sobolev-embedding}-\ref{hardy-inequality}, and Corollary \ref{H.E.J.II} that
\begin{align}
\abs{\cG_{28}}_2&\leq C(T)\absb{\rho_0^{\left(\frac{3}{2}-\varepsilon_0\right)\alpha} \partial_tU_{xx}}_2+C(T)\left(\abs{(\rho_0^\alpha)_x}_\infty+\abs{\rho_0^\alpha\eta_{xx}}_\infty\right)\absb{\rho_0^{\left(\frac{1}{2}-\varepsilon_0\right)\alpha}U_{tx}}_2\notag\\
&\quad +C(T)\Big(\absb{\rho_0^{\left(\frac{1}{2}-\varepsilon_0\right)\alpha}\eta_{xx}}_2\abs{\eta_{xx}}_\infty+\absb{\rho_0^{\left(\frac{1}{2}-\varepsilon_0\right)\alpha}\partial_x^3 \eta}_2\Big)\abs{\rho_0^\alpha U_t}_\infty\notag\\
&\leq C(T)\Big(1+\norm{\eta_{xx}}_{1,1}+\absb{\rho_0^{\left(\frac{3}{2}-\varepsilon_0\right)\alpha}\partial_x^4 \eta}_2\Big)\sum_{j=0}^2\absb{\rho_0^\frac{3\alpha}{2}\partial_x^j U_t}_2\label{equ7.49}\\
&\leq C(T)\big(1+\absb{\rho_0^{\left(\frac{3}{2}-\varepsilon_0\right)\alpha}\partial_x^4 \eta}_2\big)\leq C(T)\Big(1+\int_0^t\absb{\rho_0^{\left(\frac{3}{2}-\varepsilon_0\right)\alpha}\partial_x^4 U}_2\,\ds\Big),\notag
\end{align}
where one has used the fact that 
\begin{equation}\label{w11}
\normf{\eta_{xx}}_{1,1}\leq C\sum_{j=2}^4\absb{\rho_0^{\left(\frac{3}{2}-\varepsilon_0\right)\alpha}\partial_x^j \eta}_2\leq C(T)+ C\absb{\rho_0^{\left(\frac{3}{2}-\varepsilon_0\right)\alpha}\partial_x^4 \eta}_2.
\end{equation}

Similar to  $\cG_{28}$, for $\cG_{29}$-$\cG_{32}$, according to \eqref{eq450}, \eqref{w11}, Lemmas \ref{L.B.J}, \ref{U.B.J}, \ref{TE2}-\ref{TE3}, \ref{E.E.II}-\ref{E.E.III} and \ref{sobolev-embedding}-\ref{hardy-inequality}, Corollary \ref{H.E.J.II}, the range of $\varepsilon_0$ defined in \eqref{varepsilon0}, and $\alpha>\frac{1}{3}$, one gets that
\begin{align}
\abs{\cG_{29}}_2&\leq C(T)\left(\abs{(\rho_0^\alpha)_{xx}}_\infty +\abs{(\rho_0^\alpha)_x}_\infty\abs{\eta_{xx}}_\infty\right)\absb{\rho_0^{\left(\frac{1}{2}-\varepsilon_0\right)\alpha}U_t}_2\notag\\
&\quad + C\abs{(\rho_0^\alpha)_{xx}}_\infty \absb{\rho_0^{\left(\frac{1}{2}-\varepsilon_0\right)\alpha}U_{xx}}_2\notag\\
&\leq C(T)\big(1+\norm{\eta_{xx}}_{1,1}\big)\leq C(T)\Big(1+\int_0^t\absb{\rho_0^{\left(\frac{3}{2}-\varepsilon_0\right)\alpha}\partial_x^4 U}_2\,\ds\Big),\notag\\
\abs{\cG_{30}}_2&\leq C\norm{\rho_0^\alpha}_3\absb{\rho_0^{\left(\frac{1}{2}-\varepsilon_0\right)\alpha}U_x}_\infty+C(T)\abs{(\rho_0^\alpha)_{xx}}_\infty\abs{\eta_{xx}}_\infty\absb{\rho_0^{\left(\frac{1}{2}-\varepsilon_0\right)\alpha}U_x}_2\notag\\
&\quad + C(T)\abs{(\rho_0^\alpha)_x}_\infty\absb{\rho_0^{\left(\frac{1}{2}-\varepsilon_0\right)\alpha}\partial_x^3\eta}_2\abs{U_x}_\infty\notag\\
&\quad + C(T)\abs{(\rho_0^\alpha)_x}_\infty\abs{\eta_{xx}}_\infty\absb{\rho_0^{\left(\frac{1}{2}-\varepsilon_0\right)\alpha}\eta_{xx}}_2\abs{U_x}_\infty\notag\\
&\leq C(T)\Big(1+\norm{\eta_{xx}}_{1,1}+\absb{\rho_0^{\left(\frac{3}{2}-\varepsilon_0\right)\alpha}\partial_x^4\eta}_2\Big)\notag\\
&\leq C(T)\Big(1+\int_0^t\absb{\rho_0^{\left(\frac{3}{2}-\varepsilon_0\right)\alpha}\partial_x^4 U}_2\,\ds\Big),\label{equ7.50}\\
\abs{\cG_{31}}_2&\leq C(T) \absb{\rho_0^{\left(\frac{3}{2}-\varepsilon_0\right)\alpha}\partial_x^4\eta}_2\abs{U_x}_\infty+C(T)\absb{\rho_0^{\left(\frac{3}{2}-\varepsilon_0\right)\alpha}\partial_x^3\eta}_2\abs{\eta_{xx}}_\infty\abs{U_x}_\infty\notag\\
&\quad +C(T)\absb{\rho_0^{\left(\frac{1}{2}-\varepsilon_0\right)\alpha} \eta_{xx}}_2\abs{\rho_0^\alpha\eta_{xx}}_\infty\abs{\eta_{xx}}_\infty\abs{U_x}_\infty\notag\\
&\leq C(T)\Big(1+\norm{\eta_{xx}}_{1,1}+\absb{\rho_0^{\left(\frac{3}{2}-\varepsilon_0\right)\alpha}\partial_x^4\eta}_2\Big)\notag\\
&\leq C(T)\Big(1+\int_0^t\absb{\rho_0^{\left(\frac{3}{2}-\varepsilon_0\right)\alpha}\partial_x^4 U}_2\,\ds\Big),\notag\\
\abs{\cG_{32}}_2&\leq C(T) \left(\abs{\rho_0^{\alpha}\partial_x^3 \eta}_\infty+\abs{(\rho_0^\alpha)_x}_\infty\abs{\eta_{xx}}_\infty+\abs{\rho_0^{\alpha}\eta_{xx}}_\infty\abs{\eta_{xx}}_\infty\right)\absb{\rho_0^{\left(\frac{1}{2}-\varepsilon_0\right)\alpha}U_{xx}}_2\notag\\
&\quad + C(T) \abs{\eta_{xx}}_\infty\absb{\rho_0^{\left(\frac{3}{2}-\varepsilon_0\right)\alpha}\partial_x^3 U}_2\notag\\
&\leq C(T)\Big(1+\norm{\eta_{xx}}_{1,1}+\absb{\rho_0^{\left(\frac{3}{2}-\varepsilon_0\right)\alpha}\partial_x^4\eta}_2\Big)\notag\\
&\leq C(T)\Big(1+\int_0^t\absb{\rho_0^{\left(\frac{3}{2}-\varepsilon_0\right)\alpha}\partial_x^4 U}_2\,\ds\Big).\notag
\end{align}

For $\cG_{33}$-$\cG_{35}$, via the similar arguments for dealing with $\cG_{28}$-$\cG_{32}$, one can obtain 
\begin{equation}\label{equ7.51}
\abs{\cG_{33}}_2+\abs{\cG_{34}}_2+\abs{\cG_{35}}_2\leq C(T)\Big(1+\int_0^t\absb{\rho_0^{\left(\frac{3}{2}-\varepsilon_0\right)\alpha}\partial_x^4 U}_2\,\ds\Big).
\end{equation}
It should be pointed out here that, when $\frac{1}{3}<\alpha<1$, the condition $\varepsilon_0<\frac{1}{\alpha}-1$ has been used to ensure that 
$\rho_0^{1-\frac{3\alpha}{2}-\varepsilon_0\alpha}\in L^2$ in obtaining the $L^2$-norm of $\cG_{35}$ .

Thus, it follows from \eqref{xxxx}-\eqref{equ7.49} and \eqref{equ7.50}-\eqref{equ7.51} that
\begin{equation*}
\begin{aligned}
&\absB{\rho_0^{\left(\frac{3}{2}-\varepsilon_0\right)\alpha}\partial_x^4 U+\Big(\frac{1}{\alpha}+2\Big)\rho_0^{\left(\frac{1}{2}-\varepsilon_0\right)\alpha}(\rho_0^\alpha)_x\partial_x^3 U}_2
\leq C(T)\Big(1+\int_0^t\absb{\rho_0^{\left(\frac{3}{2}-\varepsilon_0\right)\alpha}\partial_x^4 U}_2\,\ds\Big),
\end{aligned}
\end{equation*}
which, along with Proposition \ref{prop2.1} and Lemma \ref{E.E.II}, implies that

\begin{align}
\absb{\rho_0^{\left(\frac{3}{2}-\varepsilon_0\right)\alpha}\partial_x^4 U}_2&\leq C\absB{\rho_0^{\left(\frac{3}{2}-\varepsilon_0\right)\alpha}\partial_x^4 U+\Big(\frac{1}{\alpha}+2\Big)\rho_0^{\left(\frac{1}{2}-\varepsilon_0\right)\alpha}(\rho_0^\alpha)_x\partial_x^3 U}_2\notag\\
&\quad + C\abs{(\rho_0^\alpha)_{xx}}_\infty\absb{\rho_0^{\left(\frac{3}{2}-\varepsilon_0\right)\alpha}\partial_x^3 U}_2\label{equ7.52}\\
&\leq C(T)\Big(1+\int_0^t\absb{\rho_0^{\left(\frac{3}{2}-\varepsilon_0\right)\alpha}\partial_x^4 U}_2\,\ds\Big).\notag
\end{align}

Finally, it follows from Gr\"onwall's inequality and \eqref{equ7.52} that
\begin{equation}
\big|\rho_0^{\left(\frac{3}{2}-\varepsilon_0\right)\alpha}\partial_x^4 U(t)\big|_2\leq C(T) \ \ \text{for all }0\leq t\leq T.
\end{equation}

The proof of Lemma \ref{E.E.IV} is completed.
\end{proof}

\section{Global Existence  and Uniqueness for the Nonlinear Problem}\label{Section8}

Based on the local existence and uniqueness  results in Theorem \ref{theorem3.1} ii)
and Theorem \ref{theorem5.1} and the global estimates established in
\S \ref{Section6}-\S\ref{Section7}, we now give the proofs of
Theorems \ref{Theorem1.1} and \ref{Theorem1.4}.

\subsection{Proof of Theorem \ref{Theorem1.1}}

Let $\overline T_*$ be the maximal existence time of the unique local
classical solution $(U,\eta)$ obtained in Theorem \ref{theorem3.1} ii).
Then $\overline T_*\geq T_*>0$. We claim that
\begin{equation*}
\overline T_*=\infty.
\end{equation*}

Suppose, on the contrary, that $\overline T_*<\infty$. For every
$0<T<\overline T_*$, we apply the estimates in
Lemmas \ref{TTE1}-\ref{TE3} and
\ref{E.E.II}-\ref{E.E.IV} on $[0,T]$. Since
$T<\overline T_*$, the constants in these estimates can be bounded
in terms of $\overline T_*$, and hence
\begin{equation}\label{uniform-energy-before-maximal-time}
\sup_{0\leq t<\overline T_*}\widetilde E(t,U)
\leq C(\overline T_*).
\end{equation}
Similarly, Lemmas \ref{L.B.J} and \ref{U.B.J} imply 
\begin{equation}\label{uniform-Jacobian-before-maximal-time}
C^{-1}(\overline T_*)\leq\eta_x(t,x)\leq C(\overline T_*)
\quad\text{for all }(t,x)\in[0,\overline T_*)\times\bar I.
\end{equation}
Here and below, $C(\overline T_*)$ is a positive constant depending
only on $(\alpha,\varepsilon_0,|I|,\cC_1,\cC_2,\rho_0,u_0,\overline T_*)$ and $C(\overline T_*)$ is increasing in $\overline T_*$ and finite. Fix any $\tau\in(0,\overline T_*)$ and take
\begin{equation*}
u_*(x)=U(\tau,x),\qquad \eta_*(x)=\eta(\tau,x).
\end{equation*}
It follows from \eqref{uniform-energy-before-maximal-time}-\eqref{uniform-Jacobian-before-maximal-time} that
\begin{equation*}
C^{-1}(\overline T_*)\leq(\eta_*)_x(x)\leq C(\overline T_*)
\quad\text{for all }x\in\bar I.
\end{equation*}
Thus, the constants $(c_*,c^*)$ in Theorem \ref{theorem5.1}
can be chosen independently of $\tau$.

We next verify the weighted regularity \eqref{eta00} of $\eta_*$. It follows from
\eqref{flow-map} that
\begin{equation*}
\partial_x^k\eta_*(x)
=\delta_{1k} +\int_0^\tau\partial_x^kU(s,x)\,\ds,
\qquad k=1, 2,3,4.
\end{equation*}
Therefore, by \eqref{uniform-energy-before-maximal-time}, \eqref{uniform-Jacobian-before-maximal-time}, and the H\"older inequality,
\begin{equation*}
\absb{\rho_0^{\frac12}(\eta_*)_x}_2^2+\sum_{k=2}^4\absb{\rho_0^{(\frac32-\varepsilon_0)\alpha}
\partial_x^k\eta_*}_2^2 \leq  C(\overline T_*)|\rho_0|_1+\tau\int_0^\tau\sum_{k=2}^4\absb{\rho_0^{(\frac32-\varepsilon_0)\alpha}\partial_x^kU }_2^2\,\ds \leq C(\overline T_*).    
\end{equation*}
Since $\rho_0$ is fixed, \eqref{uniform-energy-before-maximal-time}
and the above show that there exists a
constant $\mathfrak M>0$, independent of $\tau$, such that
\begin{equation}\label{uniform-restarting-data}
\begin{aligned}
&2+\normf{\rho_0^\alpha}_3
+\absb{\rho_0^{\frac12}(\eta_*)_x}_2^2
+\sum_{k=2}^4
\absb{
\rho_0^{(\frac32-\varepsilon_0)\alpha}
\partial_x^k\eta_*
}_2^2
+\widetilde E(\tau,U)
\leq\mathfrak M.
\end{aligned}
\end{equation}
The temporal derivatives $(U_t,U_{tx},U_{txx},U_{tt})$ at $t=\tau$ entering
$\widetilde E(\tau,U)$ are the actual temporal derivatives of the
existing classical solution. Since $(U,\eta)$ satisfies
equations $\eqref{secondreformulation}_1$-$\eqref{secondreformulation}_2$, $(U_t,U_{tx},U_{txx},U_{tt})|_{t=\tau}$ satisfy \eqref{1..16z}.  Consequently, all the hypotheses of
Theorem \ref{theorem5.1} are satisfied by $(\rho_0,u_*,\eta_*)$ with
constants independent of $\tau$; in particular, the compatibility
relations at $t=\tau$ follow from those of the existing classical
solution.

Theorem \ref{theorem5.1} thus provides a time $T_0>0$, depending
only on $(c_*,c^*,\cC_1,\cC_2,\alpha,\varepsilon_0,|I|,\mathfrak M)$ and independent of $\tau$, such that the problem with the initial data $(U,\eta)\big|_{t=\tau}=(u_*,\eta_*)$ admits a unique classical solution $(U^\sharp, \eta^\sharp)$ on $[\tau,\tau+T_0]\times\bar I$.

We now choose $\tau<\overline T_*$ sufficiently close to $\overline T_*$ so that $\overline T_*-\tau<T_0$. The solution $(U^\sharp, \eta^\sharp)$ and the original solution $(U,\eta)$ solve the same system with the same reference density $\rho_0$ and the same initial data at $t=\tau$. For every $S\in(\tau,\overline T_*)$, both solutions are defined on $[\tau,S]\times\bar I$ and belong to the solution class specified in Theorem \ref{theorem5.1}. The uniqueness stated in Theorem \ref{theorem5.1} implies $(U^\sharp, \eta^\sharp)=(U, \eta)$ on $[\tau,S]\times\bar I$. Since $S<\overline T_*$ is arbitrary, $(U^\sharp, \eta^\sharp)$ thus extends $(U,\eta)$ to $[0,\tau+T_0]\times\bar I$. However, $\tau+T_0>\overline T_*$, which contradicts the maximality of $\overline T_*$. Therefore, $\overline T_*=\infty$.

Finally, for every finite $T>0$, the cited estimates, combined with
the local regularity conclusions of Theorem \ref{theorem3.1} ii)
and Theorem \ref{theorem5.1} on overlapping time intervals, yield
\eqref{b1} and \eqref{N111}-\eqref{AN111} and \eqref{continuous'} with $T_*$ replaced by $T$.
Uniqueness follows from the corresponding
local uniqueness assertion. This completes the proof of
Theorem \ref{Theorem1.1}.

\subsection{Proof of Theorem \ref{Theorem1.4}}

Based on the proof in \S \ref{proof1.3}, one can deduce from Theorem \ref{Theorem1.1} and \eqref{uty} that Theorem \ref{Theorem1.4} holds.

\section{Non-existence of global solutions with \texorpdfstring{$L^\infty$}{} decay on the velocity}\label{section9}

This section will be devoted to  the  proof of Theorem \ref{th:2.20-c}.  For some positive time $T$, let $ (\rho,u)(t,y)$ in $\overline{\mathbb{I}(T)}$ be the global  classical solution obtained in Theorem \ref{Theorem1.4}. Define the following physical quantities:
\begin{alignat*}{2}
m(t)&=\intt \rho(t,y)\,\dy \ \ \textrm{(total mass)},\ \ 
\PP(t)=\intt \rho(t,y)u(t,y)\,\dy\ \ \textrm{(momentum)},\\
E_k(t)&=\frac{1}{2}\intt \rho(t,y)u^2(t,y)\,\dy \ \ \textrm{(total kinetic energy)}.
\end{alignat*}

First, one shows that $ (\rho,u)$ satisfies  the
laws of conservation of $m(t)$ and $\PP(t)$.
\begin{Lemma}
\label{lemmak} For any $T>0$, it holds that 
$$ m(t)=m(0),\ \  \mathbb{P}(t)=\mathbb{P}(0)\ \  and \ \  E_k(t)<\infty \ \  {\rm{for}} \quad t\in [0,T]. $$
\end{Lemma}
\begin{proof}

First, one can obtain that 
\begin{equation}\label{finite}
|\mathbb{P}(t)|=\Big|\intt  \rho u\,\dy\Big| \leq \sup_{y\in I(t)}|u(t,y)|\intt \rho\,\dy<\infty.
\end{equation}

Second, let $I(t)=(a(t),b(t))$. The momentum equation $\eqref{shallow}_2$ implies
\begin{equation}\label{deng1}
\begin{aligned}
\mathbb{P}_t&=\intt (\rho u)_t\,\mathrm{d}y+b'(t)\rho u\big|_{y=b(t)}-a'(t)\rho u\big|_{y=a(t)}\\
&= \intt (\rho u_y-\rho u^2-\rho^2)_y\,\dy=0,
\end{aligned}
\end{equation}
where we have used the facts that $(\rho,\rho u^2,\rho u_y)$ vanish on the moving boundary.

The conservation of the total mass and the boundedness of the kinetic energy can be proved via the similar argument. The proof of Lemma \ref{lemmak} is completed.
\end{proof}

Then it  follows from the definitions of  $m(t)$, $\mathbb{P}(t)$ and $E_k(t)$ that
$$
 |\mathbb{P}(t)|\leq \intt \rho(t,y)|u(t,y)|\,\dy\leq  \sqrt{2m(t)E_k(t)},
$$
which, along  with Lemma \ref{lemmak}, implies that
$$
0<\frac{|\mathbb{P}(0)|^2}{2m(0)}\leq E_k(t)\leq \frac{1}{2} m(0)\sup_{y\in I(t)}|u(t,y)|^2 \quad \text{for} \quad t\in [0,T].
$$
Therefore, the proof of Theorem \ref{th:2.20-c} is completed.


\appendix
\section{Some basic lemmas}\label{appendix A}

For the convenience of readers, we list some basic facts which have been used frequently in this paper. Throughout Appendices A-E, let $I=(0,1)$, $\Gamma=\partial I$, $d=d(x)=\text{dist}(x,\Gamma)$ be the distance  function from $x\in \bar I$ to $\Gamma$. Let $0\leq \phi_0=\rho_0^\alpha\sim d(x)$ satisfying  $\phi_0\in H^3$ be defined as in \S \ref{section1}, and  $\abs{I}$ denotes the Lebesgue measure of $I$.

The first one is on the separability and density of the weighted Sobolev spaces. 
\begin{Lemma}\cite{kufner}\label{W-space}
Let $k\in \ZZ$ and $s>0$. Then 
\begin{enumerate}
\item[$\mathrm{i)}$] $H^k_{d^s}$ is a reflexive separable Banach space;
\item[$\mathrm{ii)}$] $C^\infty(\bar I)$ is dense in $H^k_{d^s}$ with respect to the norm $\norm{\cdot}_{k,d^s}$ for $k\geq 0$.
\end{enumerate}\end{Lemma}

The next lemma concerns the well-known interpolation inequality of Sobolev spaces.
\begin{Lemma}\cite{leoni}\label{GNinequality'} 
Suppose that $F\in H^p\cap H^q$ for $p,q\geq 0$. Then $F\in H^s$ for all $s=p\varepsilon+q(1-\varepsilon)$ and  $0\leq \varepsilon\leq 1$, and the following inequality holds,
\begin{equation*}
\norm{F}_{s}\leq C(p,q,\varepsilon) \norm{F}_{p}^\varepsilon\norm{F}_{q}^{1-\varepsilon},
\end{equation*}
where $C(p,q,\varepsilon)>0$ is a constant depending only on $(p, q,\varepsilon)$.
\end{Lemma}

The third lemma gives some weighted interpolation inequalities, which are useful for the analysis in the current paper.
\begin{Lemma}\label{GNinequality}
It holds that for all $k>-1$ and all $F\in 
H^1_{d^{k+1}}$,
\begin{equation}\label{G-N1}
\absb{d^\frac{k}{2}F}_{2} \leq C\left(k\right)\left(\absb{d^\frac{k+1}{2}F}_{2}+\absb{d^\frac{k+1}{2}F}_{2}^\frac{1}{2}\absb{d^\frac{k+1}{2}F_x}_{2}^\frac{1}{2}\right),
\end{equation}
and, as a consequence, for all $\varepsilon\in (0,1)$,
\begin{equation}\label{G-N2}
\absb{d^\frac{k}{2}F}_{2} \leq C(k,\varepsilon)\absb{d^\frac{k+1}{2}F}_{2}+\varepsilon\absb{d^\frac{k+1}{2}F_x}_{2}.
\end{equation}
Here, $C(k)$ and $C(k,\varepsilon)$ are positive constants depending only on $(k, \abs{I})$ and $(k,\varepsilon,\abs{I})$, respectively. In particular, the above conclusions still hold when $d$ is replaced by $\phi_0$.
\end{Lemma}
\begin{proof}
The proof for $k = 0$ can be found in \cites{LWX,LWX2}, and now we give the proof for the general $k > -1$.
 First, we consider the case that  $F\in C^\infty(\bar I)$. Based on  the symmetry of the distance function, in order to obtain \eqref{G-N1}, it suffices to derive the following inequality,
\begin{equation}\label{A..5}
\begin{aligned}
\int_0^\frac{1}{2} x^k F^2\,\dx&\leq C(k) \int_0^\frac{1}{2} x^{k+1} F^2\,\dx\\
&\quad +C(k)\Big(\int_0^\frac{1}{2} x^{k+1} F^2\,\dx\Big)^\frac{1}{2}\Big(\int_0^\frac{1}{2} x^{k+1} F_x^2\,\dx\Big)^\frac{1}{2}.
\end{aligned}
\end{equation}

Actually, it follows from integration by parts that
\begin{equation}\label{A..6}
\begin{split}
\int_0^\frac{1}{2} x^k F^2\,\dx=& \frac{2^{-k-1}}{k+1}F^2\left(\frac{1}{2}\right)- \frac{2}{k+1} \int_0^\frac{1}{2} x^{k+1} FF_x\,\dx,\\
\int_0^\frac{1}{2} x^{k+1} F^2\,\dx=& \frac{2^{-k-2}}{k+2}F^2\left(\frac{1}{2}\right)- \frac{2}{k+2} \int_0^\frac{1}{2} x^{k+2} F F_x\,\dx,
\end{split}
\end{equation}
which, along with H\"older's inequality, leads to
\begin{equation*}\label{AA4}
\begin{aligned}
\int_0^\frac{1}{2} x^k F^2\,\dx
&= \frac{2k+4}{k+1} \int_0^\frac{1}{2} x^{k+1} F^2\,\dx+ \frac{2}{k+1} \int_0^\frac{1}{2} (2x-1)x^{k+1} FF_x\,\dx\\
&\leq C(k) \Big(\int_0^\frac{1}{2}x^{k+1} F^2\,\dx\Big)+C(k)\Big(\int_0^\frac{1}{2} x^{k+1} F^2\,\dx\Big)^\frac{1}{2}\Big(\int_0^\frac{1}{2} x^{k+1} F_x^2\,\dx\Big)^\frac{1}{2}.
\end{aligned}
\end{equation*}
The proof of \eqref{A..5} is completed. 

Next we consider the case that  $F\in H^1_{d^{k+1}}$. According to Lemma \ref{W-space}, for every $F\in H^1_{d^{k+1}}$, there exists a sequence $\{F^\delta\}_{\delta>0}\subset C^\infty(\bar I)$, such that
\begin{equation}\label{AA5}
\absb{d^\frac{k+1}{2} F^\delta-d^\frac{k+1}{2} F}_2+\absb{d^\frac{k+1}{2} F_x^\delta-d^\frac{k+1}{2} F_x}_2\to 0 \quad \text{as }\delta\to 0.
\end{equation}
Then, by \eqref{G-N1}, for any $\delta,\varepsilon>0$,
\begin{equation*}
\begin{aligned}
\absb{d^\frac{k}{2}F^\delta-d^\frac{k}{2}F^\varepsilon}_{2} &\leq C \absb{d^\frac{k+1}{2}F^\delta-d^\frac{k+1}{2}F^\varepsilon}_{2}\\
&\quad +C\absb{d^\frac{k+1}{2}F^\delta-d^\frac{k+1}{2}F^\varepsilon}_{2}^\frac{1}{2}\absb{d^\frac{k+1}{2}F_x^\delta-d^\frac{k+1}{2}F_x^\varepsilon}_{2}^\frac{1}{2}\to 0,
\end{aligned}
\end{equation*}
as $(\delta,\varepsilon)\to (0,0)$, which implies that $\{d^\frac{k}{2}F^\delta\}_{\delta>0}$ is a Cauchy sequence in $L^2$, and hence converges to some limit $G$ in $L^2$. However, by \eqref{AA5}, $F^\delta$ converges to $F$ in $L^1_{\text{loc}}$; hence one can extract a subsequence $F^{\delta_\ell}$  which converges to $F$ a.e. in $I$. Therefore one has  $G=d^\frac{k}{2}F$. The proof of \eqref{G-N1} for $F\in H^1_{d^{k+1}}$ is completed.

Finally, it follows  from \eqref{G-N1} and Young's inequality that \eqref{G-N2} holds.
\end{proof}

The fourth lemma is the classical Sobolev embedding theorem.
\begin{Lemma}\cite{leoni}\label{sobolev-embedding}
It holds that
\begin{equation*}
\begin{aligned}
&\abs{F}_\infty \leq s_0\abs{F}_{1}+C\abs{F_x}_1 \ \ \text{for all }F\in W^{1,1};\\
&\abs{F}_\infty \leq s_0\abs{F}_{2}+C\abs{F_x}_2 \ \ \text{for all }F\in H^1,
\end{aligned}
\end{equation*}
where $s_0$ and $C$ are positive constants depending only on $\abs{I}$. In particular, $W^{1,1},H^1\into C(\bar I)$ continuously. Moreover, if $F|_{x\in\Gamma}=0$, one can choose $s_0=0$.
\end{Lemma}

The fifth one is on the  Hardy inequalities.

\begin{Lemma}\cites{coutand1,kufner}\label{hardy-inequality}
It holds that  
\begin{alignat}{2}
\absb{d^\frac{k}{2} F}_2&\leq C(k)\absb{d^{\frac{k}{2}+1}(F,F_x)}_2 \ \ &&\text{for all }F\in H^1_{d^{k+2}}\text{ and }k>-1;\label{A7}\\
\absf{d^k F}_1&\leq C(k,\varepsilon) \absb{d^{k+\frac{3}{2}-\varepsilon}(F,F_x)}_2 \ \ &&\text{for all }F\in H^1_{d^{2k+3-2\varepsilon}}\text{ and }k+1>\varepsilon>0;\label{A6}\\
\absf{d^k F}_\infty&\leq C(k) \absb{d^{k+\frac{1}{2}} (F,F_x)}_2 \ \ &&\text{for all }F\in H^1_{d^{2k+1}}\text{ and }k>0,\label{A8}
\end{alignat}
where $C(k)$ and $C(k,\varepsilon)$ are constants depending only on $(k,\abs{I})$ and $(k,\varepsilon,\abs{I})$, respectively.
If, additionally, $\bar d(x)\in H^3$ and $\bar d(x)\sim d(x)$, $F|_{x\in \Gamma}=0$, then for all $F\in H^{s+1}$, $s=0,1,2$, one has
\begin{equation}\label{a..8}
\normf{\bar d^{-1} F}_{s}\leq C \norm{F}_{s+1},
\end{equation}
where $C$ is a constant depending only on $(\bar d,s,\abs{I})$. In particular, \eqref{A7}-\eqref{a..8} still hold when  $d$ and $\bar d$ are  replaced by $\phi_0$.
\end{Lemma}
\begin{proof}
The proof for  \eqref{A7} can be found in   Chapter 1 of Kufner \cite{kufner}. \eqref{a..8} can be proved by following the proof of Lemma 3.1 in \cite{coutand1}. For any $k+1>\varepsilon>0$, \eqref{A6} can be  derived from \eqref{A7} by H\"older's inequality,
\begin{equation*}
\begin{aligned}
\absf{d^k F}_1\leq \absb{d^{-\frac{1}{2}+\varepsilon}}_2\absb{d^{k+\frac{1}{2}-\varepsilon} F}_2\leq C(k,\varepsilon) \absb{d^{k+\frac{3}{2}-\varepsilon}(F,F_x)}_2.
\end{aligned}
\end{equation*}

For  \eqref{A8}, according to   Lemma \ref{sobolev-embedding}, \eqref{A7}, H\"older's inequality and Young's inequality, for all $k>0$,
\begin{equation*}
\begin{aligned}
\absf{d^k F}_\infty^2&\leq C \absf{d^{2k} F^2}_1+C\absf{\big(d^{2k} F^2\big)_x}_1\\
&\leq C\absf{d^{k} F}_2^2+C\absb{d^{k-\frac{1}{2}} F}_2^2+C\absb{d^{k-\frac{1}{2}}F}_2\absb{d^{k+\frac{1}{2}}F_x}_2\\
&\leq C(k)\big(\absb{d^{k+\frac{1}{2}} F}_2^2+\absb{d^{k+\frac{1}{2}} F_x}_2^2\big).
\end{aligned}
\end{equation*}
Thus, the proof is completed.
\end{proof}

The sixth one  is on  some basic  properties of    the Hilbert basis $\{e_j\}_{j=1}^\infty$ of $H^1$,  and the corresponding proof can be found in Chapter 9 of \cite{Sayas}.

\begin{Lemma}\cite{Sayas}\label{hilbert}
Let   $\{e_j\}_{j=1}^\infty$ be  the Hilbert basis of $H^1$, which is constructed by solving the eigenvalue problem $-\Delta e +e =\lambda e$ with the Neumann boundary condition $e_x|_{x\in \Gamma}=0$. Then $\{e_j\}_{j=1}^\infty$ is orthonormal in $L^2$ and orthogonal in $H^1$,  $e_j\in C^\infty(\bar I)$ for $j\in \NN^*$, and 
\begin{enumerate}
\item[$\mathrm{i)}$]  for all $f\in L^2$, it holds that
\begin{equation*}
\sum_{j=1}^n \langle f,e_j\rangle e_j \to f \ \ \text{in }L^2;
\end{equation*}
\item[$\mathrm{ii)}$]  for all $f\in H^1$, it holds that
\begin{equation*}
\sum_{j=1}^n \langle f,e_j\rangle e_j \to f \ \ \text{in }H^1.
\end{equation*}
\end{enumerate}
Here, $\langle f,g\rangle$ stands for the $L^2$-inner product of functions $f$ and $g$, i.e., $\langle f,g\rangle :=\int fg \,\dx$.
\end{Lemma}


In addition, to get the weighted time continuity for the velocity in our analysis, one needs the following evolution triple embedding.
\begin{Lemma}\label{Aubin}
Let $T>0$, $s>0$, $F\in L^2([0,T];H^1_{\phi_0^s})$ and  $\phi_0^sF_t\in L^2([0,T];H^{-1}_{\phi_0^s})$. Then, $\phi_0^\frac{s}{2}F\in C([0,T];L^2)$, and the mapping $t\mapsto \absb{\phi_0^\frac{s}{2}F(t)}_2^2$ is absolutely continuous, with 
\begin{equation*}
\frac{\mathrm{d}}{\dt} \absb{\phi_0^\frac{s}{2}F(t)}_2^2=2\left<\phi_0^sF_t, F\right>_{H^{-1}_{\phi_0^s}\times H^1_{\phi_0^s}}\quad \text{for a.e. $t\in (0,T)$}.
\end{equation*}

Furthermore, it holds that
\begin{equation*}
\big\|\phi_0^\frac{s}{2}F\big\|_{C_t(L^2)}\leq C(T)\norm{F}_{L^2_t(H^1_{\phi_0^s})}+\norm{\phi_0^sF_t}_{L^2_t(H^{-1}_{\phi_0^s})}.
\end{equation*}
\end{Lemma}
\begin{proof}
\underline{\textbf{Step 1.}} We first establish the weighted time continuity by regularization in time. Here, $\phi_0^sF$ is viewed as an $H^{-1}_{\phi_0^s}$-valued function by integration against elements of $H^1_{\phi_0^s}$. By the definition of the dual norm and the H\"older inequality, one has
\begin{equation*}
\normf{\phi_0^sF}_{-1,\phi_0^s}\leq\absb{\phi_0^{\frac{s}{2}}F}_2\leq\normf{F}_{1,\phi_0^s}.
\end{equation*}
Since $\phi_0$ is independent of $t$, the assumptions imply that
\begin{equation*}
\phi_0^sF\in H^1([0,T];H^{-1}_{\phi_0^s})\subset C([0,T];H^{-1}_{\phi_0^s}).
\end{equation*}
Thus, the endpoint traces of $\phi_0^sF$ are well defined in $H^{-1}_{\phi_0^s}$. 

Next, we extend $F$ to $(-T,2T)$ by reflection, setting $F(t,x)=F(-t,x)$ for $-T<t<0$ and $F(t,x)=F(2T-t,x)$ for $T<t<2T$, and retain the same notation. In particular, the extension satisfies
\begin{equation*}
F\in L^2([-T,2T];H^1_{\phi_0^s}),\qquad \phi_0^sF_t\in L^2([-T,2T];H^{-1}_{\phi_0^s}).
\end{equation*}

Let $\omega_\delta$ be a standard time mollifier supported in $(-\delta,\delta)$, and set $F^\delta=F*\omega_\delta$ on $[0,T]$ for $0<\delta<\frac{T}{2}$. Then $F^\delta\in C^\infty([0,T];H^1_{\phi_0^s})$, and $\phi_0^sF_t^\delta=(\phi_0^sF_t)*\omega_\delta$. It follows from the approximation properties of convolution that
\begin{equation*}
\normf{F^\delta-F}_{L^2_t(H^1_{\phi_0^s})}+\normf{\phi_0^sF_t^\delta-\phi_0^sF_t}_{L^2_t(H^{-1}_{\phi_0^s})}\to0\quad\text{as }\delta\to0,
\end{equation*}
where the time norms are taken over $(0,T)$. 

Now, for any $0<\varepsilon,\delta<\frac{T}{2}$, the smoothness of $F^\delta$ and $F^\varepsilon$ gives
\begin{equation*}
\frac{\mathrm{d}}{\dt}\absb{\phi_0^{\frac{s}{2}}(F^\delta-F^\varepsilon)(t)}_2^2=2\bigl\langle\phi_0^s(F_t^\delta-F_t^\varepsilon)(t),(F^\delta-F^\varepsilon)(t)\bigr\rangle_{H^{-1}_{\phi_0^s}\times H^1_{\phi_0^s}}.
\end{equation*}
Integrating this identity from any $\tau\in[0,T]$ to $t\in[0,T]$ and estimating the time integral over $(0,T)$, one obtains
\begin{equation*}
\absb{\phi_0^{\frac{s}{2}}(F^\delta-F^\varepsilon)(t)}_2^2\leq\absb{\phi_0^{\frac{s}{2}}(F^\delta-F^\varepsilon)(\tau)}_2^2+2\int_0^T\normf{\phi_0^s(F_t^\delta-F_t^\varepsilon)}_{-1,\phi_0^s}\normf{F^\delta-F^\varepsilon}_{1,\phi_0^s}\,\dt.
\end{equation*}
Integrating the above inequality with respect to $\tau$ over $(0,T)$, dividing by $T$, and applying H\"older's inequality yield
\begin{equation*}
\begin{aligned}
&\sup_{t\in[0,T]}\absb{\phi_0^{\frac{s}{2}}(F^\delta-F^\varepsilon)(t)}_2^2\\
&\leq\frac1T\normf{\phi_0^{\frac{s}{2}}(F^\delta-F^\varepsilon)}_{L^2_t(L^2)}^2+2\normf{\phi_0^s(F_t^\delta-F_t^\varepsilon)}_{L^2_t(H^{-1}_{\phi_0^s})}\normf{F^\delta-F^\varepsilon}_{L^2_t(H^1_{\phi_0^s})}\to0
\end{aligned}
\end{equation*}
as $(\varepsilon,\delta)\to(0,0)$. Thus, $\{\phi_0^{\frac{s}{2}}F^\delta\}$ is a Cauchy family in $C([0,T];L^2)$. Since it also converges to $\phi_0^{\frac{s}{2}}F$ in $L^2([0,T];L^2)$, its uniform limit agrees with $\phi_0^{\frac{s}{2}}F$ for a.e. $t\in (0,T)$. Hence,
\begin{equation*}
\phi_0^{\frac{s}{2}}F\in C([0,T];L^2),\qquad\phi_0^{\frac{s}{2}}F^\delta\to\phi_0^{\frac{s}{2}}F\quad\text{in }C([0,T];L^2).
\end{equation*}

Next, for every $t,\tau\in[0,T]$, the same calculation for $F^\delta$ gives
\begin{equation*}
\absb{\phi_0^{\frac{s}{2}}F^\delta(t)}_2^2=\absb{\phi_0^{\frac{s}{2}}F^\delta(\tau)}_2^2+2\int_\tau^t\bigl\langle\phi_0^sF_t^\delta,F^\delta\bigr\rangle_{H^{-1}_{\phi_0^s}\times H^1_{\phi_0^s}}\,\mathrm{d}t'.
\end{equation*}
Moreover, by H\"older's inequality and the convergence of the regularizations,
\begin{equation*}
\begin{aligned}
&\int_0^T\bigl|\langle\phi_0^sF_t^\delta,F^\delta\rangle_{H^{-1}_{\phi_0^s}\times H^1_{\phi_0^s}}-\langle\phi_0^sF_t,F\rangle_{H^{-1}_{\phi_0^s}\times H^1_{\phi_0^s}}\bigr|\,\dt\\
&\leq\normf{\phi_0^sF_t^\delta-\phi_0^sF_t}_{L^2_t(H^{-1}_{\phi_0^s})}\normf{F^\delta}_{L^2_t(H^1_{\phi_0^s})}+\normf{\phi_0^sF_t}_{L^2_t(H^{-1}_{\phi_0^s})}\normf{F^\delta-F}_{L^2_t(H^1_{\phi_0^s})}\to0.
\end{aligned}
\end{equation*}
Consequently, letting $\delta\to0$ in the energy identity for $F^\delta$ yields
\begin{equation}\label{A102}
\absb{\phi_0^{\frac{s}{2}}F(t)}_2^2=\absb{\phi_0^{\frac{s}{2}}F(\tau)}_2^2+2\int_\tau^t\bigl\langle\phi_0^sF_t,F\bigr\rangle_{H^{-1}_{\phi_0^s}\times H^1_{\phi_0^s}}\,\mathrm{d}t'
\end{equation}
for all $t,\tau\in[0,T]$. The integrand belongs to $L^1(0,T)$ by the assumptions and H\"older's inequality. Hence, $t\mapsto\absb{\phi_0^{\frac{s}{2}}F(t)}_2^2$ is absolutely continuous, and
\begin{equation*}
\frac{\mathrm{d}}{\dt}\absb{\phi_0^{\frac{s}{2}}F(t)}_2^2=2\bigl\langle\phi_0^sF_t,F\bigr\rangle_{H^{-1}_{\phi_0^s}\times H^1_{\phi_0^s}}\quad\text{for a.e. }t\in(0,T).
\end{equation*}

Finally, integrating \eqref{A102} with respect to $\tau$ over $(0,T)$ and applying H\"older's inequality, we deduce that
\begin{equation*}
\sup_{t\in[0,T]}\absb{\phi_0^{\frac{s}{2}}F(t)}_2^2\leq\frac1T\normf{\phi_0^{\frac{s}{2}}F}_{L^2_t(L^2)}^2+2\normf{\phi_0^sF_t}_{L^2_t(H^{-1}_{\phi_0^s})}\normf{F}_{L^2_t(H^1_{\phi_0^s})}.
\end{equation*}
Since $\absb{\phi_0^{\frac{s}{2}}F}_2\leq\normf{F}_{1,\phi_0^s}$, Young's inequality gives
\begin{equation*}
\sup_{t\in[0,T]}\absb{\phi_0^{\frac{s}{2}}F(t)}_2^2\leq(1+T^{-1})\normf{F}_{L^2_t(H^1_{\phi_0^s})}^2+\normf{\phi_0^sF_t}_{L^2_t(H^{-1}_{\phi_0^s})}^2,
\end{equation*}
which yields the asserted estimate with $C(T)=(1+T^{-1})^{\frac{1}{2}}$. This completes the proof.
\end{proof}

Finally, in \S\ref{Section3}, one needs the following well-known Hahn-Banach Theorem.
\begin{Lemma}\cite{brezis}\label{Banach}
Suppose that $X$ is a normed vector space and $Y\subset X$ is a linear subspace. If $g : Y\to \RR$ is a continuous linear functional and $Y$ is dense in $X$, then there exists a unique $\bar g\in X^*$ that extends $g$, namely, $\bar g(x)=g(x)$, for all $x\in Y$, satisfying
\begin{equation*}
\norm{\bar g}_{X^*}=\sup_{\stackrel{x\in Y}{\norm{x}_X\leq 1}} \abs{g(x)}=\norm{g}_{Y^*}.
\end{equation*}
\end{Lemma}



\section{Remarks on the compatibility conditions}\label{AppB}
This appendix is devoted to giving one equivalent form of the initial condition \eqref{a1} or \eqref{a1'} in terms of $(\rho_0,u_0)$ themselves and their spatial derivatives.  

First, by \eqref{secondreformulation}, one has that all the desired initial values of time derivatives of $U$ can be  expressed by those of $(\rho_0,u_0)$ themselves and their spatial derivatives, namely,
\begin{equation}\label{1..16}
\begin{aligned}
U_t(0,x)&= \rho_0^{-1}\big( \rho_0(u_0)_{x} -\rho_0^{2} \big)_x, \\
U_{tx}(0,x)&=\Big(\rho_0^{-1}\big(\rho_0 (u_0)_{x}-\rho_0^{2}\big)_x\Big)_x,\\
U_{txx}(0,x)&=\Big(\rho_0^{-1}\big(\rho_0 (u_0)_{x}-\rho_0^{2}\big)_x\Big)_{xx}, \\
U_{tt}(0,x)&=\rho_0^{-1}\Big(\rho_0 U_{tx}(0,x)-2\rho_0 (u_0)_x^2+2\rho_0^2(u_0)_x\Big)_x. 
\end{aligned}
\end{equation}

Second, \eqref{1..16} implies  $(u_0)_x|_{x\in\Gamma}=0$. Indeed, we obtain from $\eqref{1..16}_1$ that
\begin{equation}\label{fulubianjie}
U_t(0,x)= (u_0)_{xx}+\frac{1}{\alpha}\rho_0^{-\alpha}(\rho_0^\alpha)_x(u_0)_x -\frac{2}{\alpha}\rho_0^{1-\alpha}(\rho_0^\alpha)_x.
\end{equation}
Then, if $0<\alpha\leq \frac{1}{3}$, taking the $L^\infty$ norm,  \eqref{distance}, \eqref{a1}, and Lemmas \ref{sobolev-embedding}-\ref{hardy-inequality} give 
\begin{equation*}
\begin{aligned}
\absf{\rho_0^{-\alpha}(\rho_0^\alpha)_x(u_0)_x}_\infty&\leq C\big( \abs{U_t(0)}_\infty+\abs{(u_0)_{xx}}_\infty+\abs{\rho_0^\alpha}_\infty^{\frac{1-\alpha}{\alpha}}\abs{(\rho_0^\alpha)_x}_\infty\big)\\
&\leq C  E(0,U) +C\leq C,
\end{aligned}
\end{equation*}
which, along with $(\rho_0^\alpha)_x|_{x\in \Gamma}\neq 0$, yields
\begin{equation}\label{Neu-inital}
\abs{(u_0)_x(x)}\leq C d(x)\quad \text{for all }x\in \bar I.
\end{equation}
If $\frac{1}{3}<\alpha\leq 1$, one can obtain \eqref{Neu-inital} by repeating the above calculation with $E(0,U)$ replaced by $\widetilde E(0,U)$.

Next, based on \eqref{1..16}-\eqref{Neu-inital},  we give the following auxiliary lemma to give one equivalent form of \eqref{a1} or \eqref{a1'} in terms of $(\rho_0,u_0)$ themselves and their spatial derivatives.
\begin{Lemma}\label{Reduction} Let $(\rho_0,u_0)$ be the initial data of  \eqref{secondreformulation}. 
Then, when $0<\alpha\leq\frac{1}{3}$,  $E(0,U)<\infty$ if and only if
\begin{equation}\label{compat20}
\begin{gathered}
(u_0)_x|_{x\in \Gamma}=0,\quad  u_0\in H^4_{\rho_0^{2\alpha}}, \quad \rho_0^{-1}\big(\rho_0 (u_0)_{x}-\rho_0^{2}\big)_x \in  H^2_{\rho_0^{2\alpha}},\\
\rho_0^{-1}\Big(\rho_0 \big(\rho_0^{-1}\big(\rho_0 (u_0)_{x}-\rho_0^{2}\big)_x\big)_x\Big)_x\in L^2_{\rho_0^{2\alpha}};
\end{gathered}
\end{equation}
when $\frac{1}{3}<\alpha\leq 1$, $\widetilde E(0,U)<\infty$ if and only if
\begin{equation}\label{compat2}
\begin{gathered}
(u_0)_x|_{x\in \Gamma}=0,\quad  u_0\in H^4_{\rho_0^{(3-2\varepsilon_0)\alpha}}, \quad \rho_0^{-1}\big(\rho_0 (u_0)_{x}-\rho_0^{2}\big)_x \in H^1_{\rho_0}\cap H^2_{\rho_0^{(3-2\varepsilon_0)\alpha}},\\
\rho_0^{-1}\Big(\rho_0 \big(\rho_0^{-1}\big(\rho_0 (u_0)_{x}-\rho_0^{2}\big)_x\big)_x-2\rho_0 (u_0)_x^2\Big)_x\in L^2_{\rho_0}.
\end{gathered}
\end{equation}
\end{Lemma}
\begin{proof}
For brevity, we only give the proof for \eqref{compat20}, and \eqref{compat2} can be derived similarly. If $E(0,U)<\infty$, clearly, by \eqref{1..16}-\eqref{Neu-inital}, we have 
\begin{equation}\label{sss4}
\begin{aligned}
(u_0)_x|_{x\in \Gamma}=0,\quad  u_0\in H^4_{\rho_0^{2\alpha}}, \quad \rho_0^{-1}\big(\rho_0 (u_0)_{x}-\rho_0^{2}\big)_x \in  H^2_{\rho_0^{2\alpha}},\\
\rho_0^{-1}\Big(\rho_0 \big(\rho_0^{-1}\big(\rho_0 (u_0)_{x}-\rho_0^{2}\big)_x\big)_x-2\rho_0 (u_0)_x^2+2\rho_0^2(u_0)_x\Big)_x\in L^2_{\rho_0^{2\alpha}}.
\end{aligned}
\end{equation}

Since $0<\alpha\leq\frac{1}{3}$, it follows from $u_0\in H^4_{\rho_0^{2\alpha}}$, $\rho_0^\alpha\sim d(x)$, and Lemma \ref{hardy-inequality} that
\begin{equation}\label{sss3}
\begin{aligned}
\rho_0^{-1}\big(\rho_0 (u_0)_x^2\big)_x=\frac{1}{\alpha}\rho_0^{-\alpha}(\rho_0^\alpha)_x(u_0)_x^2 +2 (u_0)_x(u_0)_{xx}\in L^2_{\rho_0^{2\alpha}},\\
\rho_0^{-1}\big(\rho_0^2 (u_0)_x\big)_x=\frac{2}{\alpha}\rho_0^{1-\alpha} (\rho_0^\alpha)_x (u_0)_x +\rho_0 (u_0)_{xx}\in L^2_{\rho_0^{2\alpha}}.
\end{aligned}
\end{equation}

Conversely, if \eqref{compat20} holds, then we can first obtain \eqref{sss3} from $u_0\in H^4_{\rho_0^{2\alpha}}$, $\rho_0^\alpha\sim d(x)$, and Lemma \ref{hardy-inequality}, and then deduce \eqref{sss4}. Hence, $E(0,U)<\infty$.
\end{proof}

Finally, based on Lemma  \ref{Reduction},  we show that the example of the initial data $(\rho_0,u_0)$ given by \eqref{initalexample1} (or \eqref{initalexample2}) in  Remark \ref{initialexample} satisfies the initial conditions \eqref{distance}  and  \eqref{a1} (or \eqref{a1'}) in Theorem \ref{theorem3.1}.
\begin{Remark}\label{initalexample3}
To check that \eqref{initalexample1} satisfies the initial conditions \eqref{distance}  and  \eqref{a1}, we only need to notice that $0<\alpha\leq \frac{1}{3}$ and
\begin{equation*}
\rho_0^{-1}(\rho_0^2)_x\in H^2_{\rho_0^{2\alpha}},\qquad \rho_0^{-1}\Big(\rho_0\big(\rho_0^{-1}(\rho_0^2)_x\big)_x\Big)_x\in L^2_{\rho_0^{2\alpha}}.
\end{equation*}

Next, we check that \eqref{initalexample2} satisfies the initial conditions \eqref{distance}  and  \eqref{a1'}. Clearly, $\rho_0$ satisfies \eqref{distance}; it suffices to prove \eqref{compat2}. First, via direct calculations, one has
\begin{align}
(u_0)_x&=(f_0)_x+ \rho_0 ,\quad  (u_0)_{xx}=(f_0)_{xx}+ \frac{1}{\alpha}\rho_0^{1-\alpha}(\rho_0^\alpha)_x,\notag\\
(u_0)_{xxx}&=(f_0)_{xxx}+ \frac{ (1-\alpha)}{\alpha^2}\rho_0^{1-2\alpha}(\rho_0^\alpha)_x^2+ \frac{1}{\alpha}\rho_0^{1-\alpha}(\rho_0^\alpha)_{xx},\label{BQ1}\\
(u_0)_{xxxx}&=(f_0)_{xxxx}+ \frac{ (1-\alpha)(1-2\alpha)}{\alpha^3}\rho_0^{1-3\alpha}(\rho_0^\alpha)_x^3\notag\\
&\quad + \frac{3 (1-\alpha)}{\alpha^2}\rho_0^{1-2\alpha}(\rho_0^\alpha)_{x}(\rho_0^\alpha)_{xx}+ \frac{1}{\alpha}\rho_0^{1-\alpha}(\rho_0^\alpha)_{xxx}.\notag
\end{align}
Then, based on $f_0\in C^\infty_c(I)$, \eqref{distance}, and \eqref{varepsilon0}, one can check that $(u_0)_x|_{x\in\Gamma}=0$ and $u_0\in H^4_{\rho_0^{(3-2\varepsilon_0)\alpha}}$. Next, since $\rho_0 (u_0)_{x}-\rho_0^{2}=\rho_0 (f_0)_{x}$, it holds that 
\begin{align}
\rho_0^{-1}\big(\rho_0 (u_0)_{x}-\rho_0^{2}\big)_x& = (f_0)_{xx}+\frac{1}{\alpha}\rho_0^{-\alpha}(\rho_0^\alpha)_x(f_0)_x\notag\\
\Big(\rho_0^{-1}\big(\rho_0 (u_0)_{x}-\rho_0^{2}\big)_x\Big)_x& = (f_0)_{xxx}+\frac{1}{\alpha}\rho_0^{ -\alpha}(\rho_0^\alpha)_x(f_0)_{xx}\notag\\
&\quad+ \frac{1}{\alpha}\rho_0^{ -\alpha}(\rho_0^\alpha)_{xx}(f_0)_x-\frac{1}{\alpha }\rho_0^{-2\alpha}(\rho_0^\alpha)_{x}^2(f_0)_x,\notag \\ 
\Big(\rho_0^{-1}\big(\rho_0 (u_0)_{x}-\rho_0^{2}\big)_x\Big)_{xx}
=& (f_0)_{xxxx}+\frac{1}{\alpha}\rho_0^{ -\alpha}(\rho_0^\alpha)_x(f_0)_{xxx}-\frac{ 2}{\alpha }\rho_0^{ -2\alpha}(\rho_0^\alpha)_{x}^2(f_0)_{xx}\label{BQ2}\\
&\quad   +\frac{2}{\alpha}\rho_0^{ -\alpha}(\rho_0^\alpha)_{xx}(f_0)_{xx}+\frac{1}{\alpha}\rho_0^{ -\alpha}(\rho_0^\alpha)_{xxx}(f_0)_{x}\notag\\
&\quad -\frac{3}{\alpha }\rho_0^{ -2\alpha}(\rho_0^\alpha)_{x}(\rho_0^\alpha)_{xx} (f_0)_{x} +\frac{2}{\alpha }\rho_0^{ -3\alpha}(\rho_0^\alpha)_{x}^3 (f_0)_{x},\notag
\end{align}
which, along with the fact {\rm(}notice that $f_0\in C^\infty_c(I)$ and $\rho^\alpha_0 \sim d(x)${\rm)} that 
\begin{equation}\label{BQ3}
|(f_0)_{xx}|+|(f_0)_{xxx}|+|(f_0)_{xxxx}|\leq C(f_0,\alpha)\rho_0^{100(1+\alpha)} \ \ \text{for }x\in \bar I,
\end{equation}
for some constant $C(f_0,\alpha)>0$, leads to $\rho_0^{-1}\big(\rho_0 (u_0)_{x}-\rho_0^{2}\big)_x \in H^1_{\rho_0}\cap H^2_{\rho_0^{(3-2\varepsilon_0)\alpha}}$.

Finally, thanks to \eqref{BQ1}-\eqref{BQ2}, one gets
\begin{equation}\label{BQ4}
\begin{aligned}
&\rho_0^{-\frac{1}{2}}\Big(\rho_0 \big(\rho_0^{-1}\big(\rho_0 (u_0)_{x}-\rho_0^{2}\big)_x\big)_x-2\rho_0 (u_0)_x^2\Big)_x\\
=&\rho_0^{-\frac{1}{2}}\Big( \rho_0 (f_0)_{xxx}+\frac{1}{\alpha}\rho_0^{1-\alpha}(\rho_0^\alpha)_x(f_0)_{xx}\Big)_x\\
&+ \rho_0^{-\frac{1}{2}}\Big(\frac{1}{\alpha}\rho_0^{1-\alpha}(\rho_0^\alpha)_{xx}(f_0)_x-\frac{1}{\alpha}\rho_0^{1-2\alpha}(\rho_0^\alpha)_{x}^2(f_0)_x -2 \rho_0 \big((f_0)_x+ \rho_0 \big)^2\Big)_x. 
\end{aligned}
\end{equation}
We get from $\rho_0^\alpha\sim d(x)$ and \eqref{BQ3}  that the right-hand side of $\eqref{BQ4}$ belongs to $L^2$. This implies that \eqref{compat2} holds, and hence $\widetilde E(0,U)<\infty$.
\end{Remark}

\section{Cross-derivatives embedding}\label{subsection2.2}

The following embedding theorem will be frequently used throughout the whole paper.
\begin{Proposition}\label{prop2.1}
Let $s>0$ and $\kappa>0$ be given constants, such that
\begin{equation}\label{con2.9}
   \frac{1}{2}<s\leq \frac{\kappa+1}{2}, 
\end{equation}
and $F\in L^1_{\mathrm{loc}}$ be a function defined on $I$. If $F$ satisfies
\begin{equation}\label{con2.10}
\abs{\phi_0^sF_x+\kappa\phi_0^{s-1}(\phi_0)_xF}_2+\abs{\phi_0^s F}_2<\infty,
\end{equation}
and $\phi_0^rF_x\in L^2$ for some $r\in \big[s, \frac{\kappa+1}{2}\big]$, then it holds that
\begin{equation}\label{con2.11}
\abs{\phi_0^sF_x}_2\leq C \left(\abs{\phi_0^sF_x+\kappa\phi_0^{s-1}(\phi_0)_xF}_2+\abs{(\phi_0)_{xx}}_\infty\abs{\phi_0^s F}_2\right),
\end{equation}
where $C >0$ is a constant depending only on $(s,\kappa,\cC_1,\cC_2,|I|)$.
\end{Proposition}
\begin{proof}
Throughout this proof, assume $C >0$ is a constant depending only on $(s,\kappa,\cC_1,\cC_2,|I|)$.

\underline{\textbf{Step 1: Case $F\in C^\infty(\bar I)$.}}
We first consider the case when $F\in C^\infty(\bar I)$. Via integration by parts, it follows from \eqref{con2.9}-\eqref{con2.10}, Lemma \ref{GNinequality} and Young's inequality that
\begin{align}
\abs{\phi_0^sF_x}_2^2=&\abs{\phi_0^sF_x+\kappa\phi_0^{s-1}(\phi_0)_xF}_2^2-\kappa^2\abs{\phi_0^{s-1}(\phi_0)_xF}_2^2 - \kappa\int \phi_0^{2s-1}(\phi_0)_x (F^2)_x\,\dx\notag\\
=&\abs{\phi_0^sF_x+\kappa\phi_0^{s-1}(\phi_0)_xF}_2^2 \underline{- \kappa\phi_0^{2s-1}(\phi_0)_xF^2\Big|_{x=0}^{x=1}}_{=0}\notag\\
&\underline{+ (2s-1-\kappa)\kappa\int \phi_0^{2s-2}((\phi_0)_x)^2 F^2\,\dx}_{\leq 0} + \kappa\int \phi_0^{2s-1}(\phi_0)_{xx} F^2\,\dx\label{2004}\\
\leq & \abs{\phi_0^sF_x+\kappa\phi_0^{s-1}(\phi_0)_xF}_2^2+\kappa\abs{(\phi_0)_{xx}}_\infty\absb{\phi_0^{s-\frac{1}{2}}F}_2^2\notag\\
\leq & C \left(\abs{\phi_0^sF_x+\kappa\phi_0^{s-1}(\phi_0)_xF}_2^2+\abs{(\phi_0)_{xx}}_\infty^2\abs{\phi_0^s F}_2^2\right)+\frac{1}{2}\abs{\phi_0^{s}F_x}_2^2,\notag
\end{align}
which yields that 
\begin{equation}
\abs{\phi_0^sF_x}_2\leq C \big(\abs{\phi_0^sF_x+\kappa\phi_0^{s-1}(\phi_0)_xF}_2+\abs{(\phi_0)_{xx}}_\infty\abs{\phi_0^s F}_2\big) .
\end{equation}

\underline{\textbf{Step 2: Case $r=s$.}}
Note that for $F\in H^1_{\phi_0^{2s}}$, one can repeat the above calculation, but still needs to check the rationality of integration by parts of $\int \phi_0^{2s-1}(\phi_0)_x (F^2)_x\,\dx$. Indeed, via Lemma \ref{W-space}, there exists a smooth sequence $\{F^\delta\}_{\delta>0}\subset C^\infty(\bar I)$, such that
\begin{equation}\label{b6b}
\absf{\phi_0^s F^\delta-\phi_0^s F}_2+\absf{\phi_0^s F_x^\delta-\phi_0^s F_x}_2\to 0 \ \ \text{as }\delta\to 0,
\end{equation}
which, along with Lemma \ref{hardy-inequality}, yields
\begin{equation}\label{b7b}
\absf{\phi_0^{s-1}F^\delta-\phi_0^{s-1}F}_2+\absb{\phi_0^{s-\frac{1}{2}}F^\delta-\phi_0^{s-\frac{1}{2}}F}_\infty \to 0 \ \ \text{as }\delta\to 0.
\end{equation}
Thus, according to \eqref{b6b}-\eqref{b7b}  and integration by parts for $F^\delta\in C^\infty(\bar I)$,
\begin{equation*}
\begin{aligned}
&-2\int \phi_0^{2s-1}(\phi_0)_x F^\delta F^\delta_x\,\dx
=  (2s-1)\int \phi_0^{2s-2}((\phi_0)_x)^2 (F^\delta)^2\,\dx
 + \int \phi_0^{2s-1}(\phi_0)_{xx} (F^\delta)^2\,\dx,  
\end{aligned}
\end{equation*}
one has that for $F\in H^1_{\phi_0^{2s}}$,
\begin{equation*}
\begin{aligned}
&-2\int \phi_0^{2s-1}(\phi_0)_x F F_x\,\dx
=  (2s-1)\int \phi_0^{2s-2}((\phi_0)_x)^2 F^2\,\dx
 + \int \phi_0^{2s-1}(\phi_0)_{xx} F^2\,\dx.
\end{aligned}
\end{equation*}

The proof of \eqref{con2.11} when $r=s$ is completed.

\underline{\textbf{Step 3: General case.}}
For general $r$, it suffices to show the case when $s<r=\frac{\kappa+1}{2}$, since $H^1_{\phi_0^p}\subset H^1_{\phi_0^q}$ whenever $p\leq q$. Note that, in this case, integration by parts in \eqref{2004} fails since we do not have $\phi_0^{s-1}(\phi_0)_x F\in L^2$. 

To overcome this difficulty, setting  $\iota:=\frac{\kappa+1}{2}-s$ and choosing $0<\varepsilon<1$, we first show a variant of Lemma \ref{GNinequality}, that is,
\begin{equation}\label{2..8}
\absB{\frac{\phi_0^\frac{\kappa}{2}}{(\phi_0+\varepsilon)^{\iota}}F}_2^2\leq C(\kappa,s)\bigg(\absB{\frac{\phi_0^\frac{\kappa+1}{2}}{(\phi_0+\varepsilon)^{\iota}}F}_2^2+\absB{\frac{\phi_0^\frac{\kappa+1}{2}}{(\phi_0+\varepsilon)^{\iota}}F}_2\absB{\frac{\phi_0^\frac{\kappa+1}{2}}{(\phi_0+\varepsilon)^{\iota}}F_x}_2\bigg).
\end{equation}

Based on the proof in Lemma \ref{GNinequality}, it suffices to show \eqref{2..8} holds for $F\in C^\infty(\bar I)$ and $\phi_0=x$, $0<x\leq \frac{1}{2}$. Via integration by parts, it holds that
\begin{align*}
\int_0^\frac{1}{2} \frac{x^\kappa F^2}{(x+\varepsilon)^{2\iota}}\,\dx
&=\frac{1}{\kappa+1}\frac{2^{2\iota}}{(1+2\varepsilon)^{2\iota}2^{\kappa+1}} F^2\Big(\frac{1}{2}\Big)-\frac{2}{\kappa+1} \int_0^\frac{1}{2} \frac{x^{\kappa+1} FF_x}{(x+\varepsilon)^{2\iota}}\,\dx\\
&\quad +\frac{2\iota}{\kappa+1} \int_0^\frac{1}{2} \frac{x^{\kappa+1}F^2}{(x+\varepsilon)^{2\iota+1}}\,\dx\\
\leq & \frac{1}{\kappa+1}\frac{2^{2\iota}}{(1+2\varepsilon)^{2\iota}2^{\kappa+1}} F^2\Big(\frac{1}{2}\Big)-\frac{2}{\kappa+1} \int_0^\frac{1}{2} \frac{x^{\kappa+1} FF_x}{(x+\varepsilon)^{2\iota}}\,\dx\\
& +\frac{2\iota}{\kappa+1} \int_0^\frac{1}{2} \frac{x^{\kappa}F^2}{(x+\varepsilon)^{2\iota}}\,\dx,
\end{align*}
which gives
\begin{equation}\label{2..9}
\int_0^\frac{1}{2}\frac{x^\kappa F^2}{(x+\varepsilon)^{2\iota}}\,\dx\leq \frac{1}{s}\frac{2^{2\iota}}{(1+2\varepsilon)^{2\iota}2^{\kappa+1}} F^2\Big(\frac{1}{2}\Big)-\frac{2}{s} \int_0^\frac{1}{2} \frac{x^{\kappa+1} FF_x}{(x+\varepsilon)^{2\iota}}\,\dx.
\end{equation}

The same calculation implies that
\begin{equation}\label{2..10}
\begin{aligned}
\int_0^\frac{1}{2} \frac{x^{\kappa+1} F^2}{(x+\varepsilon)^{2\iota}}\,\dx&=\frac{1}{\kappa+2}\frac{2^{2\iota}}{(1+2\varepsilon)^{2\iota}2^{\kappa+2}} F^2\Big(\frac{1}{2}\Big)-\frac{2}{\kappa+2} \int_0^\frac{1}{2} \frac{x^{\kappa+2} FF_x}{(x+\varepsilon)^{2\iota}}\,\dx\\
&\quad +\frac{2\iota}{\kappa+2} \int_0^\frac{1}{2} \frac{x^{\kappa+2}F^2}{(x+\varepsilon)^{2\iota+1}}\,\dx.
\end{aligned}
\end{equation}

Multiplying \eqref{2..10} by $\frac{2(\kappa+2)}{s}$, then substituting \eqref{2..10} into \eqref{2..9} to cancel the constant term, one obtains that
\begin{equation*}
\begin{aligned}
\int_0^\frac{1}{2}\frac{x^\kappa F^2}{(x+\varepsilon)^{2\iota}}\,\dx&\leq \frac{2(\kappa+2)}{s}\int_0^\frac{1}{2} \frac{x^{\kappa+1} F^2}{(x+\varepsilon)^{2\iota}}\,\dx-\frac{4\iota}{s}\int_0^\frac{1}{2} \frac{x^{\kappa+2}F^2}{(x+\varepsilon)^{2\iota+1}}\,\dx\\
&\quad +\frac{4}{s}\int_0^\frac{1}{2} \frac{x^{\kappa+2} FF_x}{(x+\varepsilon)^{2\iota}}\,\dx-\frac{2}{s}\int_0^\frac{1}{2} \frac{x^{\kappa+1} FF_x}{(x+\varepsilon)^{2\iota}}\,\dx\\
&\leq C \bigg(\absB{\frac{\phi_0^\frac{\kappa+1}{2}}{(\phi_0+\varepsilon)^{\iota}}F}_2^2+\absB{\frac{\phi_0^\frac{\kappa+1}{2}}{(\phi_0+\varepsilon)^{\iota}}F}_2\absB{\frac{\phi_0^\frac{\kappa+1}{2}}{(\phi_0+\varepsilon)^{\iota}}F_x}_2\bigg),
\end{aligned}
\end{equation*}
which completes the proof of \eqref{2..8}.

Now, we continue to prove \eqref{con2.11}. It follows from multiplying \eqref{con2.10} by $\frac{\phi_0^\iota}{(\phi_0+\varepsilon)^\iota}$ that
\begin{equation}\label{con2.15}
\absB{\frac{\phi_0^{\frac{\kappa+1}{2}}}{(\phi_0+\varepsilon)^\iota}F_x+\kappa\frac{\phi_0^{\frac{\kappa-1}{2}}}{(\phi_0+\varepsilon)^\iota}(\phi_0)_x F}_2\leq \abs{\phi_0^sF_x+\kappa\phi_0^{s-1}(\phi_0)_xF}_2<\infty .
\end{equation}
It is worth noting that, according to Lemma \ref{hardy-inequality}, each term in \eqref{con2.15} is meaningful for every $\varepsilon>0$. Moreover, similar to the density arguments in \textbf{Step 2}, integration by parts still holds in this case, that is, for all $F\in H^1_{\phi_0^{\kappa+1}}$ and $\varepsilon>0$,
\begin{equation*}
\begin{aligned}
-2\int \frac{\phi_0^{\kappa}(\phi_0)_x F F_x}{(\phi_0+\varepsilon)^{2\iota}}\,\dx&=  \kappa\int \frac{\phi_0^{\kappa-1}((\phi_0)_x)^2 F^2}{(\phi_0+\varepsilon)^{2\iota}}\,\dx-2\iota\int \frac{\phi_0^{\kappa}((\phi_0)_x)^2 F^2}{(\phi_0+\varepsilon)^{2\iota+1}}\,\dx \\
&\quad+ \int \frac{\phi_0^{\kappa}(\phi_0)_{xx} F^2}{(\phi_0+\varepsilon)^{2\iota}}\,\dx.
\end{aligned}
\end{equation*}
Via integration by parts, one can deduce from \eqref{2..8}, \eqref{con2.15} and Young's inequality that
\begin{align*}
\absB{\frac{\phi_0^{\frac{\kappa+1}{2}}}{(\phi_0+\varepsilon)^\iota}F_x}_2^2&=\absB{\frac{\phi_0^{\frac{\kappa+1}{2}}}{(\phi_0+\varepsilon)^\iota}F_x+\kappa\frac{\phi_0^{\frac{\kappa-1}{2}}}{(\phi_0+\varepsilon)^\iota}(\phi_0)_x F}_2^2-\kappa^2\absB{\frac{\phi_0^{\frac{\kappa-1}{2}}}{(\phi_0+\varepsilon)^\iota}(\phi_0)_x F}_2^2\notag\\
&\quad - 2\kappa\int \frac{\phi_0^{\kappa}(\phi_0)_x F F_x}{(\phi_0+\varepsilon)^{2\iota}}\,\dx\notag\\
&=\absB{\frac{\phi_0^{\frac{\kappa+1}{2}}}{(\phi_0+\varepsilon)^\iota}F_x+\kappa\frac{\phi_0^{\frac{\kappa-1}{2}}}{(\phi_0+\varepsilon)^\iota}(\phi_0)_x F}_2^2-\kappa^2\absB{\frac{\phi_0^{\frac{\kappa-1}{2}}}{(\phi_0+\varepsilon)^\iota}(\phi_0)_x F}_2^2\notag\\
&\quad  \underline{-2\iota\kappa\int \frac{\phi_0^{\kappa}((\phi_0)_x)^2 F^2}{(\phi_0+\varepsilon)^{2\iota+1}}\,\dx}_{\leq 0}+\kappa^2\int \frac{\phi_0^{\kappa-1}((\phi_0)_x)^2 F^2}{(\phi_0+\varepsilon)^{2\iota}}\,\dx\\
&\quad + \kappa\int \frac{\phi_0^{\kappa}(\phi_0)_{xx} F^2}{(\phi_0+\varepsilon)^{2\iota}}\,\dx\notag\\
&\leq \abs{\phi_0^sF_x+\kappa\phi_0^{s-1}(\phi_0)_xF}^2_2 +\kappa\abs{(\phi_0)_{xx}}_\infty\absB{\frac{\phi_0^\frac{\kappa}{2}}{(\phi_0+\varepsilon)^{\iota}}F}_2^2\notag\\
&\le C\left(
\abs{\phi_0^sF_x+\kappa\phi_0^{s-1}(\phi_0)_xF}_2^2
+\abs{(\phi_0)_{xx}}_\infty^2\abs{\phi_0^sF}_2^2
\right) +\frac12
\absB{\frac{\phi_0^{(\kappa+1)/2}}
{(\phi_0+\varepsilon)^\iota}F_x}_2^2,
\end{align*}
which yields that 
\begin{equation*}
\absB{\frac{\phi_0^{\frac{\kappa+1}{2}}}{(\phi_0+\varepsilon)^\iota}F_x}_2\leq C \left(\abs{\phi_0^sF_x+\kappa\phi_0^{s-1}(\phi_0)_xF}_2+\abs{(\phi_0)_{xx}}_{\infty}\abs{\phi_0^sF}_2\right).
\end{equation*}
Then one can extract a subsequence (still denoted by $\varepsilon$)  such that
\begin{equation}\label{equ213}
\frac{\phi_0^{\frac{\kappa+1}{2}}}{(\phi_0+\varepsilon)^\iota}F_x\to  \cQ\quad \text{weakly in }L^2 \quad \text{as }\varepsilon\to 0,
\end{equation}
for a function $\cQ\in L^2$, and 
\begin{equation*}
\abs{\cQ}_2\leq \liminf_{\varepsilon\to 0}\absB{\frac{\phi_0^{\frac{\kappa+1}{2}}}{(\phi_0+\varepsilon)^\iota}F_x}_2\leq C . 
\end{equation*}

However, it follows from $F_x\in L^1_{\mathrm{loc}}$ that
\begin{equation}\label{equ214}
\frac{\phi_0^{\frac{\kappa+1}{2}}}{(\phi_0+\varepsilon)^\iota}F_x\to \phi_0^s F_x\quad \text{in }L^1_{\mathrm{loc}} \ \ \text{as }\varepsilon\to 0.
\end{equation}
Then, comparing with \eqref{equ213}-\eqref{equ214}, one gets from the uniqueness of the limits that $\cQ=\phi_0^s F_x$, which completes the proof.
\end{proof}

\section{Leibniz formula in Sobolev spaces}\label{difference-quotient}
In this appendix, we give a brief summary about the operations of derivatives that will be used in \S \ref{Section3}. 
First, we give the Leibniz formula in Sobolev spaces.
\begin{Lemma}\cite{evans}\label{Leibniz}
Let $F\in H^1_{\mathrm{loc}}$ and  $G\in H^1_{\mathrm{loc}}$. Then $FG\in H^1_{\mathrm{loc}}$, and the Leibniz formula holds, that is,
\begin{equation*}
(FG)_x=F_x G+ F G_x \ \ \text{for a.e. }x\in I.
\end{equation*}  
\end{Lemma}

Lemma \ref{Leibniz} has the following applications.
\begin{Lemma}\label{qiudao}
Suppose that $\phi_0(x)$, $F(t,x)$ and $\cQ(t,x)$ satisfy the following equation:
\begin{equation}\label{equBB0}
\phi_0^s F_x + \kappa \phi_0^{s-1}(\phi_0)_x F = \cQ \ \ \text{for a.e. } (t,x)\in (0,T)\times I,
\end{equation}
where $s,\kappa\in \RR$, $\cQ,\, F,\,F_x\in L^2_{\mathrm{loc}}$ for a.e. $t\in (0,T)$. It holds that
\begin{enumerate}
\item[$\mathrm{i)}$]
if $\cQ_t,\,F_t\in L^2_{\mathrm{loc}}$ for a.e. $t\in (0,T)$, then $F_{tx}\in L^2_{\mathrm{loc}}$ and for a.e. $(t,x)\in (0,T)\times I$,
\begin{equation}\label{equBB1}
\phi_0^s F_{tx} + \kappa \phi_0^{s-1}(\phi_0)_x F_t = \cQ_t;
\end{equation}
\item[$\mathrm{ii)}$]
if $\cQ_x\in L^2_{\mathrm{loc}}$ for a.e. $t\in (0,T)$, then $F_{xx}\in L^2_{\mathrm{loc}}$ and for a.e. $(t,x)\in (0,T)\times I$,
\begin{equation}\label{equBB2}
\phi_0^s F_{xx} + (\kappa+s) \phi_0^{s-1}(\phi_0)_x F_{x}+\kappa(s-1)\phi_0^{s-2}((\phi_0)_x)^2 F+\kappa\phi_0^{s-1}(\phi_0)_{xx} F = \cQ_x.
\end{equation}
\end{enumerate}
\end{Lemma}
\begin{proof}
To get $\mathrm{i)}$, multiplying \eqref{equBB0} by $\phi_0^{-s}$, one has
\begin{equation}\label{equBB3}
F_x = \phi_0^{-s}\cQ -\kappa \phi_0^{-1}(\phi_0)_x F,
\end{equation}
which, along with the facts that  $\cQ_t,\,F_t\in L^2_{\mathrm{loc}}$ and $\phi_0\sim d(x)$, yields that $F_x$ has the weak derivative $F_{tx}\in L^2_{\rm loc}$ and \eqref{equBB1} holds.

To get $\mathrm{ii)}$, since $\cQ,F\in H^1_{\mathrm{loc}}$, $\phi_0\sim d(x)$ and $\phi_0\in H^3$, it follows from Lemma \ref{Leibniz} that each term in the right hand side of \eqref{equBB3} belongs to $H^1_{\mathrm{loc}}$ and the following identities hold for a.e. $(t,x)\in (0,T)\times I$,
\begin{equation*}
\begin{aligned}
&\left(\phi_0^{-s}\cQ\right)_x=-s\phi_0^{-s-1}(\phi_0)_x \cQ+\phi_0^{-s} \cQ_x;\\
&\left(\kappa \phi_0^{-1}(\phi_0)_x F\right)_x= -\kappa\phi_0^{-2}((\phi_0)_x)^2 F+\kappa\phi_0^{-1}(\phi_0)_{xx} F+\kappa\phi_0^{-1}(\phi_0)_x F_x,
\end{aligned}
\end{equation*}
which, along with  \eqref{equBB3}, yield that  $F_x\in H^1_{\mathrm{loc}}$ and
\begin{equation}\label{equBB5}
F_{xx}= -s\phi_0^{-s-1}(\phi_0)_x \cQ+\phi_0^{-s} \cQ_x+\kappa\phi_0^{-2}((\phi_0)_x)^2 F-\kappa\phi_0^{-1}(\phi_0)_{xx} F-\kappa\phi_0^{-1}(\phi_0)_x F_x.
\end{equation}

Finally, substituting \eqref{equBB0} into \eqref{equBB5}, one obtains \eqref{equBB2}, which completes the proof.
\end{proof}

\section{Coordinate Transformation}\label{section-CT}
In this appendix, we give the transformation relations between $(\rho(t,y),u(t,y),\Gamma(t))$ and $(H(t,x),U(t,x))$. First,  for every $t$ and $y\in I(t)$, define the inverse flow mapping $\tilde\eta$ by
\begin{equation*}
x=\tilde\eta(t,y):\quad I(t)\to I,\quad (t,y)\mapsto (t,x),
\end{equation*}
and set
\begin{equation}\label{equ1.28}
\rho(t,y)=\frac{\rho_0(\tilde \eta(t,y))}{\eta_x(t,\tilde\eta(t,y))},\quad u(t,y)=U(t,\tilde\eta(t,y)),\quad \Gamma(t)=\{\eta(t,0),\eta(t,1)\}.
\end{equation}
Note that $(\eta,\tilde\eta)$ satisfies the following relations:
\begin{equation}\label{relation-matrix}
\frac{\partial(x,t)}{\partial(y,t)}=
\begin{bmatrix}
\eta_x^{-1}(t,\tilde\eta(t,y))& -u\eta_x^{-1}(t,\tilde\eta(t,y))\\
\quad\\
0&1
\end{bmatrix}.
\end{equation}

Then it follows from \eqref{equ1.28}-\eqref{relation-matrix} and the Leibniz formula that
\begin{align}
(\rho^\alpha)_y(t,y)&=\left(\frac{(\rho_0^\alpha)_x}{\eta_x^{\alpha+1}}-\frac{\alpha\rho_0^\alpha\eta_{xx}}{\eta_x^{\alpha+2}}\right)(t,\tilde\eta(t,y)),\notag\\
(\rho^\alpha)_t(t,y)&=\left(-\frac{\alpha\rho_0^\alpha U_x}{\eta_x^{\alpha+1}}-\frac{(\rho_0^\alpha)_xU}{\eta_x^{\alpha+1}}+\frac{\alpha \rho_0^\alpha\eta_{xx}U}{\eta_x^{\alpha+2}}\right)(t,\tilde\eta(t,y)),\notag\\
(\rho^\alpha)_{yy}(t,y)&=\left(\frac{(\rho_0^\alpha)_{xx}}{\eta_x^{\alpha+2}}-\frac{(2\alpha+1)(\rho_0^\alpha)_x\eta_{xx}}{\eta_x^{\alpha+3}}-\frac{\alpha\rho_0^\alpha\partial_x^3\eta}{\eta_x^{\alpha+3}} +\frac{\alpha(\alpha+2)\rho_0^\alpha \eta_{xx}^2}{\eta_x^{\alpha+4}}\right)(t,\tilde\eta(t,y)),\notag\\
(\rho^\alpha)_{ty}(t,y)&=\left(-\frac{(\alpha+1)(\rho_0^\alpha)_xU_x}{\eta_x^{\alpha+2}}-\frac{\alpha\rho_0^\alpha U_{xx}}{\eta_x^{\alpha+2}}+\frac{\alpha(\alpha+2)\rho_0^\alpha U_{x}\eta_{xx}}{\eta_x^{\alpha+3}}-\frac{(\rho_0^\alpha)_{xx} U}{\eta_x^{\alpha+2}}\right.\notag\\
&\quad \left.+\frac{(2\alpha+1)(\rho_0^\alpha)_x U \eta_{xx}}{\eta_x^{\alpha+3}}+\frac{\alpha\rho_0^\alpha \partial_x^3 \eta U}{\eta_x^{\alpha+3}} -\frac{\alpha(\alpha+2)\rho_0^\alpha \eta_{xx}^2 U}{\eta_x^{\alpha+4}}\right)(t,\tilde\eta(t,y)),\notag\\
\partial_y^3(\rho^\alpha)(t,y)&=\left(\frac{\partial_x^3\rho_0^\alpha}{\eta_x^{\alpha+3}}-\frac{(3\alpha+3)(\rho_0^\alpha)_{xx}\eta_{xx}}{\eta_x^{\alpha+4}}-\frac{(3\alpha+1)(\rho_0^\alpha)_{x}\partial_x^3 \eta}{\eta_x^{\alpha+4}}\right.\notag\\
&\quad\left.+\frac{(3\alpha^2+9\alpha+3)(\rho_0^\alpha)_x\eta_{xx}^2}{\eta_x^{\alpha+5}}-\frac{\alpha\rho_0^\alpha \partial_x^4\eta}{\eta_x^{\alpha+4}}\right.\notag\\
&\quad \left.+\frac{(3\alpha^2+7\alpha)\rho_0^\alpha \eta_{xx}\partial_x^3\eta}{\eta_x^{\alpha+5}}-\frac{(\alpha^3+6\alpha^2+8\alpha)\rho_0^\alpha \eta_{xx}^3}{\eta_x^{\alpha+6}}\right)(t,\tilde\eta(t,y)),\notag\\
\partial_t(\rho^\alpha)_{yy}(t,y)&=\left(-\frac{(\alpha+1)(\rho_0^\alpha)_{xx}U_x}{\eta_x^{\alpha+3}}-\frac{(2\alpha+1)(\rho_0^\alpha)_x U_{xx}}{\eta_x^{\alpha+3}}-\frac{\alpha \rho_0^\alpha \partial_x^3 U}{\eta_x^{\alpha+3}}-\frac{(\rho_0^\alpha)_{xx}U_x}{\eta_x^{\alpha+3}}\right.\label{uty}\\
&\quad \left.+\frac{(2\alpha^2+7\alpha+3)(\rho_0^\alpha)_x U_x\eta_{xx}}{\eta_x^{\alpha+4}}+\frac{2\alpha(\alpha+2)\rho_0^\alpha U_{xx}\eta_{xx}}{\eta_x^{\alpha+4}}\right.\notag\\
&\quad \left. +\frac{\alpha(\alpha +3)\rho_0^\alpha U_x \partial_x^3\eta}{\eta_x^{\alpha+4}}-\frac{\alpha(\alpha+2)(\alpha+4)\rho_0^\alpha U_x\eta_{xx}^2}{\eta_x^{\alpha+5}}-\frac{\partial_x^3\rho_0^\alpha U}{\eta_x^{\alpha+3}}\right.\notag\\
&\quad \left. +\frac{3(\alpha+1)(\rho_0^\alpha)_{xx} U\eta_{xx}}{\eta_x^{\alpha+4}}+\frac{(3\alpha+1)(\rho_0^\alpha)_x U \partial_x^3\eta}{\eta_x^{\alpha+4}}+\frac{\alpha\rho_0^\alpha U \partial_x^4\eta}{\eta_x^{\alpha+4}}\right.\notag\\
&\quad\left. -\frac{3(\alpha^2+3\alpha+1)(\rho_0^\alpha)_x U \eta_{xx}^2}{\eta_x^{\alpha+5}}-\frac{\alpha(3\alpha+7)\rho_0^\alpha U \eta_{xx}\partial_x^3\eta}{\eta_x^{\alpha+5}}\right.\notag\\
&\quad \left.+\frac{\alpha(\alpha+2)(\alpha+4)\rho_0^\alpha U \eta_{xx}^3}{\eta_x^{\alpha+6}}\right)(t,\tilde\eta(t,y)),\notag\\
u_y(t,y)&=\frac{U_x}{\eta_x}(t,\tilde\eta(t,y)),\quad u_t(t,y)=\left(U_t-\frac{UU_x}{\eta_x}\right)(t,\tilde\eta(t,y)),\notag\\
u_{yy}(t,y)&=\left(\frac{U_{xx}}{\eta_x^2}-\frac{U_x\eta_{xx}}{\eta_x^3}\right)(t,\tilde\eta(t,y)),\notag\\
u_{ty}(t,y)&=\left(\frac{U_{tx}}{\eta_x}-\frac{U_{x}^2}{\eta_x^2}-\frac{UU_{xx}}{\eta_x^2}+\frac{UU_{x}\eta_{xx}}{\eta_x^3}\right)(t,\tilde\eta(t,y)),\notag\\
\partial_y^3u(t,y)&=\left(\frac{\partial_x^3 U}{\eta_x^3}-\frac{3U_{xx}\eta_{xx}}{\eta_x^4}-\frac{U_x\partial_x^3\eta}{\eta_x^4}+\frac{3U_{x}\eta_{xx}^2}{\eta_x^5}\right)(t,\tilde\eta(t,y)),\notag\\
\Gamma'(t)&=\{U(t,0),U(t,1)\},\quad \Gamma''(t)=\{U_{t}(t,0),U_{t}(t,1)\}.\notag
\end{align}

\bigskip

{\bf Acknowledgment:} The authors are grateful to Prof. Qin Duan,  Prof. Hailiang Li and Prof. Yuexun Wang  for helpful discussion on the local-in-time existence and uniqueness  of classical solutions 
to the vacuum free boundary problem of the \text{1-D} viscous Saint-Venant system.
This research is partially supported by National Key R$\&$D Program of China (No. 2022YFA1007300), Zheng Ge Ru Foundation, Hong Kong RGC Earmarked Research Grants CUHK-14301421, CUHK-14301023, CUHK-14302819, and CUHK-14300819. Xin's research is also supported in part by the key projects of National Natural Science Foundation of China (No.12131010 and No.11931013) and Guangdong Province Basic and Applied Basic Research Foundation 2020B1515310002. Zhu's research is also  supported in part by the National Natural Science Foundation of China under Grant  12471212.
\bigskip

{\bf Conflict of Interest:} The authors declare  that they have no conflict of
interest.

\bigskip

{\bf Data availability:}
Data sharing is  not applicable to this article as no data sets were generated or analysed during the current study.

\end{document}